%% file: thesis.tex
\RequirePackage{silence}
\documentclass[oneside,numbers=noenddot,headinclude,footinclude,cleardoublepage=empty,abstract=on,BCOR=5mm,paper=letter,fontsize=11pt]{scrreprt}

\input{thesisconfig}

\begin{document}

\frenchspacing
\raggedbottom

\selectlanguage{english}

\pagenumbering{roman}
\pagestyle{plain}
\include{titlepage}
\cleardoublepage
\include{abstract}
\cleardoublepage
\include{acknowledgements}
\cleardoublepage
\include{dedication}
\clearpage
\include{contents}

\cleardoublepage
\pagestyle{scrheadings}
\pagenumbering{arabic}
\cleardoublepage
\include{0_introduction}

\cleardoublepage
\include{1_weightings}

\cleardoublepage
\include{3_linear_weightings}

\cleardoublepage
\include{2_multiplicative_weightings}
\cleardoublepage
\include{4_IM_weightings}

\cleardoublepage
\include{5_graded_bundle_description}

\cleardoublepage
\include{6_outlook}

\appendix
\cleardoublepage
\include{A-equivalence_of_weighted_groupoids}
\cleardoublepage
\include{A-spray_exponential}

\cleardoublepage
\include{A-technical_proofs}

\cleardoublepage

\cleardoublepage
\include{bibliography}
\cleardoubleevenpage
\include{colophon}

\end{document}

%% file: thesisconfig.tex
\PassOptionsToPackage{utf8}{inputenc}
  \usepackage{inputenc}

\PassOptionsToPackage{T1}{fontenc}
  \usepackage{fontenc}

\PassOptionsToPackage{
  drafting=false,    
  tocaligned=false, 
  dottedtoc=true,  
  eulerchapternumbers=true, 
  linedheaders=false,       
  floatperchapter=true,     
  eulermath=false,  
  beramono=true,    
  palatino=true,    
  parts=false,
  style=classicthesis
}{classicthesis}

\newcommand{\myTitle}{On Multiplicative Weightings for Lie Groupoids and Lie Algebroids\xspace}
\newcommand{\myName}{Daniel Hudson\xspace}
\newcommand{\myDegree}{Doctor of Philosophy\xspace}
\newcommand{\myDepartment}{Graduate Department of Mathematics\xspace}
\newcommand{\myUni}{University of Toronto\xspace}
\newcommand{\myTime}{2024\xspace}

\PassOptionsToPackage{english}{babel}
\usepackage{babel} 
\usepackage{enumitem} 
\usepackage{amsmath,amsthm,amssymb} 
\usepackage{thmtools} 
\usepackage[onehalfspacing]{setspace} 
\usepackage{graphicx} 
\usepackage{xspace} 
\usepackage{calc} 
\usepackage[all]{xy}
\usepackage{mathrsfs}
\usepackage{wrapfig}

\usepackage{classicthesis}

\hypersetup{%
  colorlinks=true, linktocpage=true, pdfstartpage=3, pdfstartview=FitV,%
  colorlinks=false, linktocpage=false, pdfstartpage=3, pdfstartview=FitV, pdfborder={0 0 0},%
  breaklinks=true, pageanchor=true,%
  pdfpagemode=UseNone, %
  plainpages=false, bookmarksnumbered, bookmarksopen=true, bookmarksopenlevel=1,%
  hypertexnames=true, pdfhighlight=/O,
  urlcolor=CTurl, linkcolor=CTlink, citecolor=CTcitation, 
  pdftitle={\myTitle},%
  pdfauthor={\textcopyright\ \myName, \myDepartment, \myUni},%
  pdfsubject={},%
  pdfkeywords={},%
  pdfcreator={pdfLaTeX},%
  pdfproducer={LaTeX with hyperref and classicthesis}%
}

\makeatletter
\@ifpackageloaded{babel}%
  {%
    \addto\extrasenglish{%
    }%
    }{\relax}
\makeatother

\makeatletter
\ifthenelse{\boolean{ct@linedheaders}}%
{
    \titleformat{\chapter}[display]%
    {\relax}{\raggedleft{\color{CTsemi}\chapterNumber\thechapter} \\ }{0pt}%
    {\titlerule\vspace*{.9\baselineskip}\raggedright\Large\color{CTtitle}\spacedallcaps}[\normalsize\vspace*{.8\baselineskip}\titlerule]%
}{
    \titleformat{\chapter}[display]%
    {\relax}{\mbox{}\oldmarginpar{\vspace*{-3\baselineskip}\color{CTsemi}\chapterNumber\thechapter}}{0pt}%
    {\raggedright\huge\color{CTtitle}\spacedallcaps}[\normalsize\vspace*{.8\baselineskip}\titlerule]%
}
\makeatother

\newtheorem{theorem}{Theorem}[chapter]
\newtheorem*{theorem*}{Theorem}
\newtheorem{lemma}[theorem]{Lemma}

\newtheorem{corollary}[theorem]{Corollary}
\newtheorem{proposition}[theorem]{Proposition}

\theoremstyle{definition}
\newtheorem{definition}[theorem]{Definition}
\newtheorem*{definition*}{Definition}
\newtheorem{definitions}[theorem]{Definitions}

\newtheorem{problem}[theorem]{Problem}

\newtheorem{example}[theorem]{Example}
\newtheorem{examples}[theorem]{Examples}

\theoremstyle{remark}
\newtheorem{remark}[theorem]{Remark}
\newtheorem{remarks}[theorem]{Remarks}

\newenvironment{itemize*}
  {\begin{itemize}[topsep=-\parskip+\jot,itemsep=-\parskip-\jot]}
  {\end{itemize}}
  
\newenvironment{enumerate*}
  {\begin{enumerate}[label=(\alph*),topsep=-\parskip+\jot,itemsep=-\parskip-\jot]}
  {\end{enumerate}}
  
\newenvironment{enumerate**}
  {\begin{enumerate}[label=(\roman*),topsep=-\parskip+\jot,itemsep=-\parskip-\jot]}
  {\end{enumerate}}
  
\newenvironment{enumerate***}
  {\begin{enumerate}[label=(\alph*'),topsep=-\parskip+\jot,itemsep=-\parskip-\jot]}
  {\end{enumerate}}

\newcommand{\into}{\hookrightarrow}

\newcommand{\toto}{\rightrightarrows}

\newcommand{\C}{\mathbb{C}}
\newcommand{\R}{\mathbb{R}}

\newcommand{\Z}{\mathbb{Z}}
\newcommand{\N}{\mathbb{N}}

\newcommand{\G}{\Gamma}

\newcommand{\T}{\mathbb{T}}

\newcommand{\End}{\mathrm{End}}
\newcommand{\A}{\mathbb{A}}

\newcommand{\aHom}{\mathrm{Hom}_{\mathrm{alg}}}
\newcommand{\rees}{\mathrm{Rees}}
\newcommand{\Cl}{\mathbb{C}\mathrm{l}}
\newcommand{\ger}[1]{\mathfrak{#1}}

\newcommand{\corr}[1]{\overset{#1}{\longrightarrow}}

\newcommand{\la}{\langle}
\newcommand{\ra}{\rangle}
\newcommand{\bd}{\partial}

\newcommand{\sset}{\subseteq}
\newcommand*{\van}[1]{\mathcal{I}_{#1}}
\newcommand{\ed}{\mathrm{d}}
\newcommand{\nuw}{\nu_\mathcal{W}}
\newcommand{\defw}{\delta_\mathcal{W}}
\newcommand{\gr}{\mathrm{gr}}
\newcommand{\W}{\mathcal{W}}
\newcommand{\DO}{\mathrm{DO}}
\newcommand{\WT}[1]{\widetilde{#1}}

%% file: titlepage.tex
\begin{titlepage}
    \pdfbookmark[1]{\myTitle}{titlepage}

    \begin{center}
        \large

        \hfill

        \vfill

        \begingroup
            \color{CTtitle}\spacedallcaps{\myTitle}
        \endgroup

		\vfill
		\spacedlowsmallcaps{by}\bigskip
		
        \spacedlowsmallcaps{\myName}

		\vfill
        
        A thesis submitted in conformity with\\
        the requirements for the degree of\smallskip
        
        \myDegree \smallskip
        
        \myDepartment\\
        \myUni\bigskip
        
        \textcopyright\ Copyright by \myName \myTime

    \end{center}
\end{titlepage}

%% file: abstract.tex
\pdfbookmark[1]{Abstract}{Abstract}

\begingroup
\let\clearpage\relax
\let\cleardoublepage\relax
\let\cleardoublepage\relax

\chapter*{Abstract}

\doublespacing

\begin{center}
	\myTitle\\
\myName\\
\myDegree\\
\myDepartment\\
\myUni\\
\myTime
\end{center}

We present a thorough study of the differential geometry of weightings and develop the theory of weightings for vector bundles, Lie groupoids, and Lie algebroids. 

We begin by extending the work of Loizides and Meinrenken in~\cite{loizides2023differential} on weighted manifolds. We define weighted submanifolds, weighted immersions, and weighted embeddings, and prove normal form theorems for these objects. We also study characterizations of weighted morphisms in terms of their graphs and in terms of weighted paths. We further extend the theory of weighted manifolds by developing a theory of linear weightings for vector bundles. Our work on linear weightings is the content of the pre-print~\cite{hudson2023linear}.

Following this, we give three equivalent definitions of a multiplicative weighting for a Lie groupoid $G\toto M$: one involving the structure maps for the Lie groupoid, one involving the graph of the groupoid multiplication, and one involving the weighted deformation space. We also include a discussion of weighted VB-groupoids, and prove some basic theorems involving these objects. 

In the last two chapters of this thesis, we study infinitesimally multiplicative weightings for Lie algebroids. We characterize these in terms of linear Poisson structures and homological vector fields. We show that multiplicative weightings differentiate to infinitesimally multiplicative weightings and solve the integration problem for infinitesimally multiplicative weightings along wide Lie subalgebroids. In particular, we classify multiplicative weightings of a Lie groupoid along its units in terms of Lie filtrations of its Lie algebroid (cf.~\cite[Definition 67]{van2019groupoid}). We make progress towards a solution for the general integration problem by giving a condition for when an infinitesimally multiplicative weighting along a general Lie subalgebroid integrates. 

\endgroup

%% file: acknowledgements.tex
\pdfbookmark[1]{Acknowledgements}{acknowledgements}

\begingroup
\let\clearpage\relax
\let\cleardoublepage\relax

\chapter*{Acknowledgements}

The task of writing this thesis, in addition to completing everything else that goes along with doing a PhD, has been a tremendous undertaking - one that I could not have accomplished alone. I am greatly indebted to everyone I have had the pleasure of spending time with over the past five years. Although you might not see it in the following pages, all of you have contributed to this work in various ways. 

To my soul mate and life partner, Jane Paul - if I were to fully convey how much it meant to have you throughout this endeavour it would double the length of the thesis, so I will have to content myself with these few words. Your constant commitment and support during this experience has been the greatest gift that I have ever been given. Your rationality and ability to calm me down as I catastrophize every mistake and misstep throughout this has been absolutely paramount to the completion of this project. Moving forward, my sole ambition is to repay you and make our dream life a reality.

To my supervisor, Eckhard Meinrenken - thank you for demonstrating extraordinary patience, encouragement, and guidance throughout my PhD. Thanks to your support I have been able to pursue my passion, travel the world, and meet many wonderful and interesting people. You have not only made this work possible, but the entire experience as rewarding and enjoyable as it could have been. Above all, I want to thank you for teaching me to seek beauty and simplicity in all my creative pursuits. 

To my master's supervisor, Heath Emerson - thank you for you guidance throughout my mathematical journey, and in particular for pointing me towards the University of Toronto for my PhD. Your strong sense of mathematical taste and style has greatly benefited me. I'm happy to have you as a friend. 

To my supervisory committee, Boris Khesin and Marco Gualtieri - thank you for your earnest advice and attentive listening during our meetings. Your insights have been invaluable in completing this work.

To my parents, Frank and Theresa Hudson - thank you for your unwavering encouragement and belief in me throughout my entire post-secondary career. I am so fortunate to have parents who have shown unconditional support while I chase my dreams involving something they do not understand. I hope I've lived up to your expectations. 

To my mathematical brother, Ethan Ross - thank you for being the superlative office mate and sounding board. I cannot express how much I appreciate you listening to me ramble about nonsense and your unending willingness to go for coffee. Your passion for the esoteric and commitment to your interests has been truly inspirational to me. 

To my Alberta friends, Adam Morgan, Carrie Clark, Curtis Brown, Eli Vigneron, and Kristen C\^{o}t\`{e}, and Mississauguah friends Nick Plati, Julia Bonavita, and Turner Silverthorne - thank you for welcoming Jane and me into your friends group and including us in all your activities. I can't imagine better people to have found upon moving to Toronto. Thank you for letting me punish you endlessly with shop talk and death metal facts at our various gatherings - at least I had fun. 

To my friends at the University of Toronto and beyond, in particular Alice Rolf, Chloe Lampman, Eva Politou, Flora Bowditch, Mackenzie Wheeler, and Valia Liontou - thank you for making my time here so enjoyable with our discussions on mathematics, teaching, and cats. I secretly believe that U of T has a personality screening before admitting grad students, and you are all evidence of this. 

To the \#1 rockers, Noel Peters and Ian Stephenson - thank you for welcoming Jane and me into the Inertia Entertainment family. Outside of my PhD, the best part of being in Toronto has been the ability to go to shows, take pictures, and immerse ourselves in the music scene here.

To my aunts, Annie Brown and Shelley Emery - thank you for storing my belongings during my PhD. The financial burden you have lifted from Jane and me has not gone unnoticed or unappreciated. 

To the University of Toronto, Government of Ontario, and NSERC - thank you for the financial support throughout my PhD, including various OGS scholarships and an NSERC CGS-D. I think you will find that it was money well spent.

\endgroup

%% file: dedication.tex
\phantomsection
\pdfbookmark[1]{Dedication}{Dedication}

\vspace*{3cm}

\begingroup
\setlength{\parindent}{0cm}
\textit{To Jane,\vspace{3mm} \\
\hspace*{5mm} ``You are the sun, \\
\hspace*{6mm} I've been reflecting your light \\
\hspace*{6mm} since the day that we met." \vspace{3mm} \\
\hspace*{35mm} - adapted from "Solar Flare" by Linea Aspera}
\endgroup

\newpage

\noindent
\begin{minipage}[c][\textheight][c]{\textwidth}
    \centering
    \includegraphics[scale=0.17]{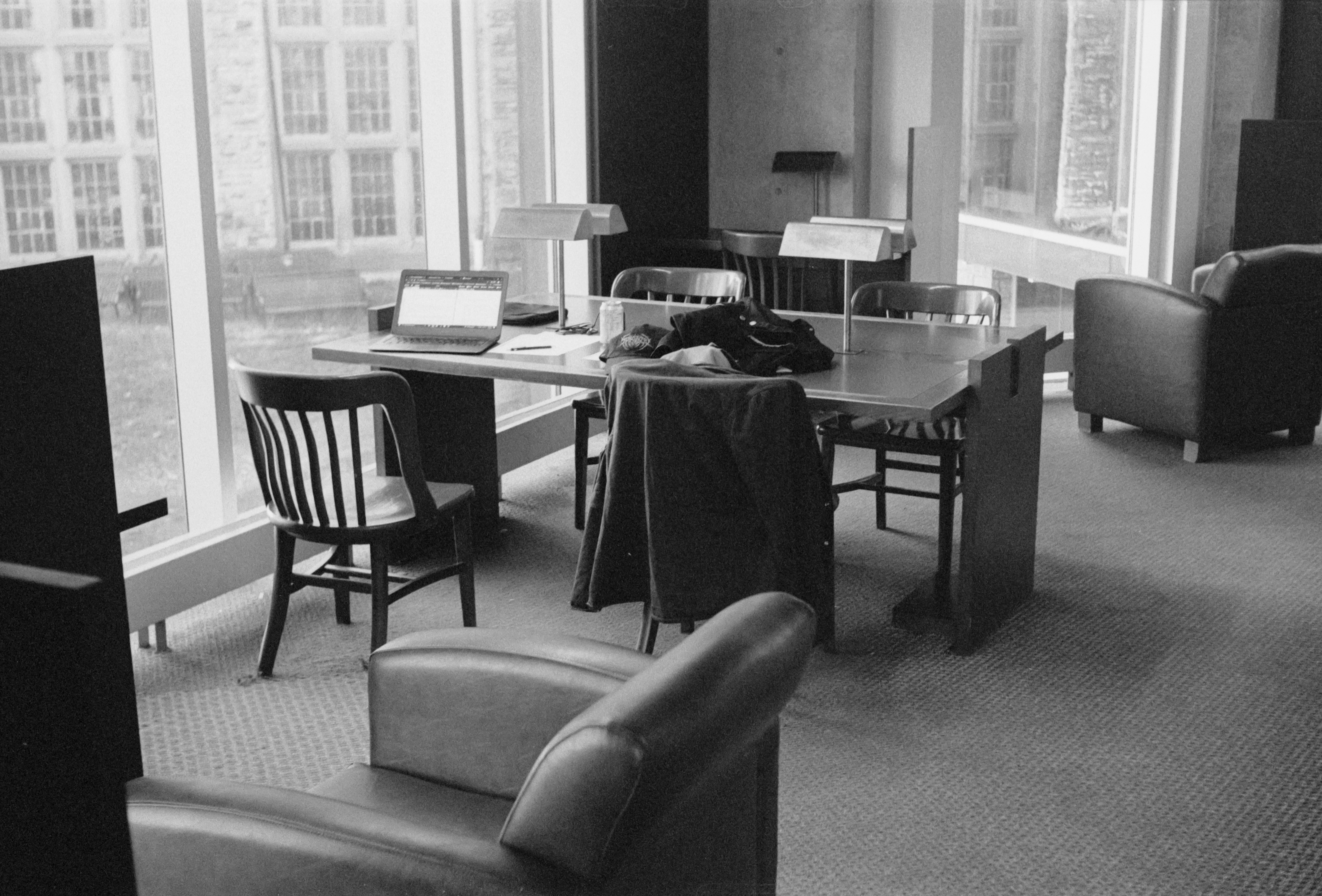}
\end{minipage}





%% file: contents.tex
\setlength{\cftbeforetoctitleskip}{100em}
\begingroup
\pagestyle{scrheadings}
\pdfbookmark[1]{\contentsname}{tableofcontents}
\setcounter{tocdepth}{1} 
\setcounter{secnumdepth}{2} 
\manualmark
\markboth{\spacedlowsmallcaps{\contentsname}}{\spacedlowsmallcaps{\contentsname}}
\tableofcontents
\automark[section]{chapter}
\renewcommand{\chaptermark}[1]{\markboth{\spacedlowsmallcaps{#1}}{\spacedlowsmallcaps{#1}}}
\renewcommand{\sectionmark}[1]{\markright{\textsc{\thesection}\enspace\spacedlowsmallcaps{#1}}}
\endgroup

%% file: 0_introduction.tex
\chapter{Introduction}
\label{chapter: introduction}

The purpose of this thesis is to give a detailed and thorough study of the theory of \emph{weightings} in the categories of manifolds, vector bundles, Lie groupoids, and Lie algebroids, with the goal of unifying, clarify, and generalizing constructions of Connes~\cite{connes1994}, Debord and Skandalis~\cite{debord2014adiabatic}, Higson and Yi~\cite{higson2019spinors}, Haj and Higson~\cite{sadegh2018euler}, van Erp and Yuncken~\cite{van2017tangent}, and \v{S}evera~\cite{vsevera2017letters}.  

\section{Background}

\subsection{Connes' tangent groupoid}

The story of weightings can be said to have started with the so-called \emph{tangent groupoid} introduced by Connes in his book~\cite{connes1994}. Given a smooth manifold $M$, the tangent groupoid $\mathbb{T}M$ is defined as a set to be the disjoint union
    \begin{equation*}
    \label{equation: tangent groupoid}
        \mathbb{T}M = TM \sqcup (M\times M\times \R^\times).
    \end{equation*}
The topology on $\mathbb{T}M$ is specified by declaring that $M\times M\times \R^\times$ is an open set, and a sequence $(x_n, y_n, t_n) \in M\times M\times \R^\times$ with $t_n \to 0$ converges to a tangent vector $X_p\in T_pM$ if and only if 
    \[ x_n \to p, \quad y_n \to p, \quad \frac{x_n-y_n}{t_n} \to X_p, \]
where the last term is understood to take place in a local coordinate system. This admits the structure of a Lie groupoid $\mathbb{T}M\toto M\times\R$. Alternatively, it can be understood as a family of Lie groupoids $\mathbb{T}M_{t} \toto M$ parameterized $t \in \R$, where 
    \[ \mathbb{T}M_{t} = \left\{
        \begin{array}{ll}
             \mathrm{Pair}(M) & t\neq 0   \\
             TM & t = 0. 
        \end{array}
    \right.\]
Connes uses this construction together with K-theory for groupoid $C^*$-algebras to give a geometric reformulation of the analytic index map of Atiyah and Singer and a proof of the Atiyah-Singer index theorem (see~\cite{higson2023fields} for a more thorough account of Connes' argument). 

\subsection{The groupoid approach to pseudodifferential operators}

There is a smooth action  $\alpha: \R^\times \times \mathbb{T}M \to \mathbb{T}M$ (the \emph{zoom action}) defined on the open set $M\times M\times \R^\times$ by
    \[ \alpha_\lambda(p,q,t) = (p,q, \lambda^{-1}t) \]
and on $TM$ by scalar multiplication. Debord and Skandalis~\cite{debord2014adiabatic} observed that any order $k$ differential operator $D$ on $M$ extends to a family of differential operators $\mathbb{D}_t$ on $\mathbb{T}M$ which is homogeneous of degree $k$, as follows. Suppose that in local coordinates $x_a$ on $M$ one has 
    \[ D = \sum_{|\alpha| \leq k}a_\alpha(x)\bd^\alpha. \]
Let $D_0$ be the differential operator on $TM$ which, for $p\in M$, acts on the fibre $T_pM$ as the constant coefficient differential operator 
    \[ D_{0, p} = \sum_{|\alpha| = k}a_\alpha(p)\bd^\alpha. \]
Then $\mathbb{D}_t$ is given by
    \[ \mathbb{D}_t = \left\{
        \begin{array}{ll}
             t^kD & t\neq 0  \\
             D_0 & t = 0, 
        \end{array}
    \right.\]
where $t^kD$ acts on the first component of $\mathbb{T}M_t = M\times M \times \{t\}$. More generally, if $D$ is a pseudodifferential operator on $M$ of order $k$ with Schwartz kernel $d\in \mathcal{D}'(M\times M)$, then $d$ extends to a distribution $\mathbb{D}$ on $\mathbb{T}M$ such that $\alpha_\lambda^*\mathbb{D}-\lambda^k\mathbb{D}$ is a smooth density. It was then shown by van Erp and Yuncken that this property \emph{characterizes} the pseudodifferential operators on $M$ (\cite[Theorem 2]{van2019groupoid}). 

An attractive feature of this result is that it allows for a coordinate free definition of pseudodifferential operators. In particular, it provides an avenue for defining pseudodifferential operators and their principal symbols in situations where the Fourier transform is unavailable. The setting in which van Erp and Yuncken apply this principle is that of \emph{filtered} (or \emph{Carnot}) manifolds, first considered by Melin~\cite{melin82lie}. A filtered manifold is a manifold $M$ together with a filtration 
    \[ TM = F_{-r} \supseteq F_{-r+1} \supseteq \cdots \supseteq F_{-1} \supseteq 0 \]
of $TM$ by subbundles $F_{-i}$ such that 
    \begin{equation*}
    \label{equation: lie bracket condition}
        [ \G(F_{-i}), \G(F_{-j}) ] \sset \G(F_{-i-j}).
    \end{equation*}
The Lie bracket condition ensures that the graded vector bundle 
    \[ \ger{t}_FM = \bigoplus_{i=0}^rF_{-i}/F_{-i+1} \to M \]
has the structure of a family of nilpotent Lie algebras. The corresponding family $T_FM = \exp(\ger{t}_FM)$ of simply connected nilpotent Lie groups fits into a smooth family of Lie groupoids 
    \[ \mathbb{T}_FM = T_FM \sqcup (\mathrm{Pair}(M)\times \R^\times)  \]
generalizing Connes' tangent groupoid (\cite{choi2015privileged, van2017tangent, sadegh2018euler}). The action of $\R^\times$ on $\ger{t}_FM$ defined by 
    \[ \alpha_\lambda(\xi) = \lambda^i\xi, \quad \text{for } \xi \in F_{-i}/F_{-i+1}, \]
is by Lie algebroid automorphisms, hence integrates to an action on $T_FG$ by Lie groupoid automorphisms, and this extends to a smooth action on $\mathbb{T}_FM$ by Lie groupoid automorphisms; this is the zoom action on $\T_FM$. Using their characterization of pseudodifferential operators, van Erp and Yuncken then \emph{define} the $F$-pseudodifferential calculus for a filtered manifold $M$ by replacing $\mathbb{T}M$ with $\mathbb{T}_FM$. 

\subsection{Euler-like vector fields and deformation to the normal cone}

Around the same time as van Erp and Yuncken's work on pseudodifferential operators, Bischoff, Bursztyn, Lima, and Meinrenken (\cite{bursztyn2019splitting, meinrenken2021euler, bischoff2020deformation}) had discovered that many linearization and normal form results could be deduced from the fact that a specific class of vector fields, called \emph{Euler-like vector fields}, could be linearized. 

Recall that the Euler-vector field $E$ on a vector bundle $V\to M$ is the vector field whose flow is given by scalar multiplication by $e^t$; if $x_a, p_b$ are local vector bundle coordinates for $V$ then 
    \[ E = \sum_b p_b\frac{\bd}{\bd p_b}.  \]
A vector field $X\in \ger{X}(M)$ is called \emph{Euler-like} with respect to a submanifold $N \sset M$ if it vanishes along $N$ and the induced vector field $\nu(X)$ on the normal bundle $\nu(M,N)$ is equal to the Euler vector field; if $x_a, y_b$ are local coordinates on $M$ such that $N$ is cut out by setting $y_b = 0$, then $X\in \ger{X}(M)$ is Euler-like with respect to $N$ if and only if 
    \[ X = \sum_af_a(x,y)\frac{\bd}{\bd x_a} + \sum_b(y_b+g_b(x,y))\frac{\bd}{\bd y_b} \]
where $f_a$ vanishes along $N$ and $g_b$ vanishes to order 2 along $N$. The linearization theorem of Bursztyn, Lima, and Meinrenken (\cite[Proposition 2.7]{bursztyn2019splitting}) is that there is a 1-1 correspondence between (germs of) tubular neighbourhoods of $N$ in $M$ and vector fields on $M$ which are Euler-like with respect to $N$. The connection between the linearization theorem of Bursztyn, Lima, and Meinrenken and Connes' tangent groupoid was established by Haj and Higson (\cite{sadegh2018euler}), who use $\mathbb{T}M$ to give a beautiful geometric explanation of this result.

\subsection{Weightings}

More generally, Haj and Higson also considered filtered manifolds. If $M$ is a filtered manifold, then they define (cf.~\cite[Definition 7.1]{sadegh2018euler}) \emph{filtered submanifolds} to be submanifolds $N\sset M$ with the property that the intersections $TN\cap (F_{-i})|_N$ are vector subbundles of $TM|_N$. Haj and Higson construct a smooth deformation to the normal cone is this setting, and explain how this construction generalizes the aforementioned $F$-tangent groupoid (~\cite[Section 9]{sadegh2018euler}). Motivated by this work, Loizides and Meinrenken (cf.~\cite{loizides2023differential}) determined what additional structure was needed along a submanifold to define these ``exotic'' deformation spaces. To this end, they introduced the concept of a \emph{weighting} (cf.~\cite[Definition 2.2]{loizides2023differential}); a similar definition was given by Melrose in~\cite{melrose1996differential} under the name of a ``quasi-homogeneous structure".

A weighting of a manifold $M$ along a closed, embedded submanifold $N$ in defined terms of a multiplicative filtration 
    \begin{equation}
    \label{introeq: weighting}
        C^\infty_{M} = C^\infty_{M, (0)} \supseteq C^\infty_{M, (1)} \supseteq C^\infty_{M, (2)} \supseteq \cdots
    \end{equation}
of the sheaf of smooth functions on $M$ such that $C^\infty_{M, (1)} = \van{N}$ is the vanishing ideal of $N$ (see~\autoref{definition: weighting}). Given a weighting of $M$ along $N$, Loizides and Meinrenken define a fibre bundle
    \[ \nuw(M,N) \to N,\]
called the \emph{weighted normal bundle} (\autoref{definition: weighted normal bundle}), generalizing the normal bundle of $M$ in $N$. Additionally, they explain how the weighted normal bundle fits into a \emph{weighted deformation space}
    \[ \defw(M,N) = \nuw(M,N) \sqcup (M\times \R^\times) \]
(\autoref{defintion: weighted deformation space}). Loizides and Meinrenken explain in~\cite{loizides2022singular} how their constructions generalize the work of Haj and Higson, by showing that if $M$ is a filtered manifold then $M\times M$ has a canonical weighting along the diagonal and with respect to this one has $\defw(M\times M, M)  = \mathbb{T}_FM$. 

\section{Overview and Summary of Results}

A key feature of Connes' construction is that the tangent groupoid is, as the name suggests, a \emph{groupoid}. Therefore a motivating question for this work is when a weighting of a Lie groupoid $G$ along a subgroupoid $H$ is compatible with the Lie groupoid structure in such a way that the groupoid structure of $G\times \R^\times$ extends to  $\defw(G,H)$. 

This question turns out to be more subtle than one might expect. Recall (or see~\autoref{section: groupoid prelims}) that the data of a Lie groupoid $G\toto M$ includes the submanifold $M \sset G$ of units, two surjective submersions $s,t:G\to M$, a partially defined multiplication $\mathrm{mult}:G\times_{M} G \to G$, and an inversion map $G \to G$. Therefore, the reasonable first attempt at a definition of a ``multiplicative weighting'' would be to simply add the adjective ``weighted'' in front of everything. However, this already begs the question of what exactly one means by this: what is a weighted submanifold? What is a weighted submersion? In order to ask that multiplication be a weighted morphism, we need the fibre product $G\times_M G$ to be weighted - is this automatic? Once the notion of weighted submanifold is understood, another reasonable definition of multiplicative weighting would be the requirement that the graph of the groupoid multiplication be a weighted submanifold of $G^3$. Do these two approaches amount to the same thing?

\subsection{Overview of Chapter 2}

This question thus demands a careful investigation into the weighted analogs of the basic notions of differential geometry, and to this we devote~\autoref{chapter: Weightings}. We begin in~\autoref{section: basics of weightings} by reviewing the basics of weightings, as described in~\cite{loizides2023differential}. In particular, we recall how a weighting of $M$ along $N$ defines a filtration 
    \begin{equation}
    \label{introeq: filtration of tangent bundle}
        TM|_N = (TM|_N)_{(-r)} \supseteq \cdots \supseteq (TM|_N)_{(-1)} \supseteq (TM|_N)_{(0)} = TN
    \end{equation}
by subbundles. In~\autoref{section: weighted submanifolds}, we define the notion of weighted submanifolds, give several examples, and explain how weighted submanifolds become weighted manifolds in their own right. 

A weighted morphism is a smooth map between weighted manifolds whose pull-back respects the filtrations of smooth functions~\eqref{introeq: weighting}. We devote~\autoref{section: weighted morphisms} to an in depth study of weighted morphisms, where we define weighted embeddings (\autoref{definition: weighted embedding}) and weighted submersions (\autoref{definition: weighted submersion}) using the filtration~\eqref{introeq: filtration of tangent bundle}. The weighted analogues of the classical normal form theorems for these maps are the content of~\autoref{theorem: weighted embedding characterization} and~\autoref{theorem: weighted submersion coordinates}. 

It is not always easy to check whether or not a map is a weighted morphism, and so an important result in this chapter is the following. 

\begin{theorem*}[=~\autoref{theorem: characterization of weighted morphisms in terms of their graphs}]
    Suppose that $(M,N)$ and $(M', N')$ are weighted pairs and $F:(M,N) \to (M', N')$ is a smooth map of pairs. Then $F$ is a weighted morphism if and only if
        \begin{enumerate}
            \item[(a)] the graph $\G(F)\sset M'\times M$ is a weighted submanifold  and 
            \item[(b)] the tangent map $TF:TM|_N \to TM'|_{N'}$ is filtration preserving. 
        \end{enumerate}
\end{theorem*}

This theorem combined with our work on weighted embeddings is a valuable tool for our work on multiplicative weightings. We also include another characterization of weighted morphisms, this time using \emph{weighted paths} (\autoref{definition: weighted paths}), which is based on valuable discussions with Beiner, Loizides, and Meinrenken.

\begin{theorem*}[=~\autoref{A-proposition: characterization of weighted morphisms}]
    \begin{enumerate}
        \item[(a)] We have $f \in C^\infty(M)_{(i)}$ if and only if  
            \[ f(\gamma(t)) =  O(t^i)\]
        for every weighted path $\gamma:\R \to M$.

        \item[(b)] A smooth map $F:(M, N) \to (M', N')$ between weighted pairs is a weighted morphism if and only if it takes weighted paths to weighted paths. 
    \end{enumerate}
\end{theorem*}

In~\autoref{section: weighted normal bundles} and~\autoref{section: weighted deformation space}, we review the aforementioned weighted normal bundle and weighted deformation space and their functorial properties. In particular, we show that our definitions of weighted immersions and weighted submersions are exactly the correct ones:  

\begin{theorem*}[=~\autoref{theorem: Characterization of Weighted Immersions and Submersions}]
    Suppose that $F:(M,N)\to (M', N')$ is a weighted morphism. Then
\begin{itemize}
    \item[(a)] $F$ is a weighted immersion if and only if $\defw(F):\defw(M,N)\to \defw(M', N')$ is an immersion. 
    \item[(b)] $F$ is a weighted submersion if and only if $\defw(F):\defw(M,N)\to \defw(M', N')$ is a submersion. 
\end{itemize}
\end{theorem*}

This chapter ends with~\autoref{section: singular Lie filtrations}, which is a review of singular Lie filtrations; these constitute an important method of constructing weightings.

\subsection{Overview of Chapter 3}

With a more thorough understanding of the differential geometry of weightings, we turn our attention to linear weightings for vector bundles. In addition to Connes' tangent groupoid and the $F$-tangent bundle of Choi-Ponge and van Erp-Yuncken, another example of a weighted deformation space appearing in the literature is the \emph{re-scaled spinor bundle} of Higson and Yi~\cite{higson2019spinors}, which we now recall. Suppose that $M$ is a Riemannian spin manifold with spinor bundle $S\to M$. In order to understand Getzler's approach (\cite{berline2003heat, getzler1983pseudodifferential}) to the index theorem from the perspective Connes' tangent groupoid, Higson and Yi introduce their rescaled spinor bundle $\mathbb{S}$, which is a vector bundle over the tangent groupoid $\T M = TM \sqcup (M\times M\times \R^\times)$. It can be understood as a family of vector bundles parameterized by $\R$, given by 
    \begin{equation*}
        \mathbb{S}_t = \left\{
            \begin{array}{ll}
                S\boxtimes S^* & t\neq 0  \\
                \pi^*(\wedge^\bullet TM) & t = 0, 
            \end{array}
        \right.
    \end{equation*}
where $\pi:TM \to M$ is the vector bundle projection and $S\boxtimes S^*$ is the vector bundle over $M\times M$ with fibre $(S\boxtimes S^*)_{(m_1,m_2)} = S_{m_1}\otimes S^*_{m_2}$. 

We unify the constructions of the rescaled spinor bundle $\mathbb{S}$ and the weighted deformation space $\defw(M,N)$ by introducing \emph{linear weightings}. If $V \to M$ is a vector bundle over a weighted manifold, then a linear weighting of $V$ is defined in terms of a $\Z$-graded filtration 
    \begin{equation}
    \label{introeq: linear weighting}
        \cdots \supseteq \G_{V, (i)} \supseteq \G_{V, (i+1)} \supseteq \cdots
    \end{equation}
of the sheaf of sections of $V$ which is compatible with the filtration of $C^\infty_M$, see~\autoref{definition: linear weighting}. In particular, we allow the filtration~\eqref{introeq: linear weighting} to be non-trivial in negative degree, in contrast with~\eqref{introeq: weighting}; this was originally suggested by Beiner~\cite{beiner2022linear}.

We give the basic definitions, examples, and constructions of linear weightings in~\autoref{section: weighted vector bundles}. A specific result we obtain  (\autoref{theorem: linear weightings in terms of polynomials}) says that linear weightings may equivalently be defined in terms of a multiplicative filtration 
    \[ \cdots \supseteq C^\infty_{pol, V, (i)} \supseteq C^\infty_{pol, V, (i+1)} \supseteq \cdots,  \]
where $C^\infty_{pol, V}$ denotes the sheaf of fiber-wise polynomial functions on $V$. We use this result in~\autoref{section: linear weighted normal bundle} and~\autoref{section: linear weighted deformation bundle} to extend the definition of the weighted normal bundle and weighted deformation space to linear weightings. This gives \emph{vector bundles} 
    \[ \nuw(V) \to \nuw(M,N) \quad \text{and} \quad \defw(V) \to \defw(M,N)  \]
such that $\defw(V)$ can be understood as a family of vector bundles 
    \begin{equation*}
     \xymatrix{
        \defw(V) = \nuw(V) \sqcup (V \times \R^\times) \ar[d] \\
        \defw(M,N) = \nuw(M,N) \sqcup (M\times \R^\times);
    }
    \end{equation*} 
see~\autoref{theorem: linear weighted normal bundle} and~\autoref{theorem: linear weighted deformation bundle}. We also give the following explicit description of the sections of $\nuw(V) \to \nuw(M,N)$ and $\defw(V)\to \defw(M,N)$.

\begin{theorem*}[=~\autoref{theorem: sections of the weighted normal bundle} and~\autoref{theorem: sections of deformation bundle}]
If $V\to M$ is a linearly weighted vector bundle then
\begin{itemize}
    \item[(a)] $\G(\nuw(V)) = C^\infty(\nuw(M,N))\otimes_{\gr(C^\infty(M))}\gr(\G(V))$ and
    \item[(b)] $\G(\defw(V)) = C^\infty(\defw(M,N))\otimes_{\rees(C^\infty(M))}\rees(\G(V))$. 
\end{itemize}
\end{theorem*}

In~\autoref{section: examples in the literature} how our construction captures the rescaled spinor bundle of Higson and Yi (\cite{higson2019spinors}), as well as a related construction by \v{S}evera in his letters to Weinstein (\cite{vsevera2017letters}). 

\subsection{Overview of Chapter 4}

In~\autoref{chapter: Multiplicative Weightings} we return to the issue of multiplicative weightings for Lie groupoids. We begin by reviewing the basics of Lie groupoids, including VB-groupoids. After this, equipped with correct definitions of weighted submanifolds and weighted submersions, we give the following definition:

\begin{definition*}[=~\autoref{definition: multiplicative weighting}]
    A weighting of $G$ along $H\sset G$ is said to be \emph{multiplicative} if 
        \begin{enumerate}
            \item[(a)] The units $M\sset G$ are a weighted submanifold, 
            \item[(b)] the source map is a weighted submersion,
            \item[(c)] multiplication $m:G^{(2)} \to G$ is a weighted morphism, and 
            \item[(d)] inversion is a weighted morphism. 
        \end{enumerate}
\end{definition*}

We devote~\autoref{section: definition of multiplicative weightings} to exploring this definition and examples thereof. We prove, in particular, that it is \emph{not} enough to simply ask that the graph of multiplication be weighted.

\begin{theorem*}[=~\autoref{theorem: characterization of mult weightings}]
    A weighting of $G\toto M$ along $H\sset G$ is multiplicative if and only if 
        \begin{itemize}
            \item[(a)] $M$ is a weighted submanifold, 
            \item[(b)] the graph of multiplication is a weighted submanifold, and 
            \item[(c)] the filtration of $TG|_H$ is by subgroupoids 
                \[ (TG|_H)_{(i)} \toto (TM|_N)_{(i)} \]
        \end{itemize}
\end{theorem*}

This characterization of multiplicative weightings is particularly useful when applied to VB-groupoids, which we shall use for the problem of differentiating multiplicative weightings (\autoref{theorem: weightings can be differentiated}). We close this chapter by showing that our definition of multiplicative weightings is exactly the right one, by proving the following theorem:

\begin{theorem*}[=~\autoref{theorem: weighted normal and weighted deformation groupoids}]
    A weighting of a Lie groupoid $G$ along $H$ is multiplicative if and only if the groupoid structure $G\times \R^\times \to M\times \R^\times$ extends to 
        \[ \defw(G,H) \toto \defw(M,N).  \]  
\end{theorem*}

\subsection{Overview of Chapter 5}

\autoref{chapter: Infinitesimally Multiplicative Weightings} is devoted to the study of the infinitesimal analogue of multiplicative weightings, called \emph{infinitesimally multiplicative weightings}. We being with a review of the basics of Lie algebroids, including the Lie functor from Lie groupoids to Lie algebroids, as well as the descriptions of Lie algebroids as linear Poisson manifolds and 1-shifted vector bundles equipped with a homological vector field. 

In~\autoref{section: IM weighting definition} we give the following definition:

\begin{definition*}[=~\autoref{definition: IM weighting}]
    An \emph{infinitesimally multiplicative weighting} of $A\Rightarrow M$ is a linear weighting of $A$ with the additional properties that 
        \begin{itemize}
            \item[(a)] the anchor $a:A\to TM$ is a weighted morphism, and 
            \item[(b)] for all $\sigma \in \G(A)_{(i)}$ and $\tau \in \G(A)_{(j)}$, we have 
                \[ [\sigma, \tau] \in \G(A)_{(i+j)}.  \]
        \end{itemize}
\end{definition*}

In~\autoref{section: IM characterizations} we prove the following alternative characterizations of infinitesimally multiplicative weightings. 

\begin{theorem*}[=~\autoref{theorem: equvalent characterizations of IM weightings}]
    Let $A\Rightarrow M$ be a Lie algebroid endowed with a linear weighting. Then the following are equivalent. 
        \begin{itemize}
            \item[(a)] $A$ is a weighted Lie algebroid, 
            \item[(b)] the Poisson bivector field $\pi \in \ger{X}^2(A^*)$ has filtration degree zero, 
            \item[(c)] the differential $\ed_A : \G(\wedge A^*) \to \G(\wedge A^*)$ is filtration preserving. 
        \end{itemize}
\end{theorem*}

We use this, as well as the theory of weighted VB-groupoids,  in~\autoref{section: differentiation of M weightings} to prove that multiplicative weightings can be differentiated to infinitesimally multiplicative weightings. 

\begin{theorem*}[=~\autoref{theorem: weightings can be differentiated}]
    Let $G\toto M$ be a weighted Lie groupoid with Lie algebroid $\mathrm{Lie}(G) \Rightarrow M$. Then 
        \begin{equation}
            \G(\mathrm{Lie}(G)|_U)_{(i)} = \{ \sigma \in \G(\mathrm{Lie}(G)|_U) : \sigma^L\in \mathfrak{X}^L(G|_U)_{(i)} \}
        \end{equation}
    defines an infinitesimally multiplicative weighting of $\mathrm{Lie}(G)$ such that 
        \[ \mathrm{Lie}(\nuw(G,H)) = \nuw(\mathrm{Lie}(G)) \quad \text{and} \quad \mathrm{Lie}(\defw(G,H)) = \defw(\mathrm{Lie}(G)).  \]
\end{theorem*}

In~\autoref{section: wide integration} we prove the following partial converse to the previous theorem. 

\begin{theorem*}[=~\autoref{theorem: wide integration}]
    Suppose that $G\toto M$ is a Lie groupoid and 
        \begin{equation}
        \label{introeq: Wide IM weighting to integrate}
            A=A_{-r} \supseteq A_{-r+1} \supseteq \cdots A_{-1} \supseteq 0
        \end{equation}
    is a Lie filtration of $A = \mathrm{Lie}(G)$. Suppose that $H\sset G$ is a wide, $s$-connected Lie subgroupoid. If $B = \mathrm{Lie}(H)$ is such that
        \begin{enumerate}
            \item[(a)] $[\G(B), \G(A_{-i})] \sset \G(A_{-i})$ for all $i$

            \item[(b)] the assignment $m \mapsto \dim(B_m+A_{-i}|_m)$ is constant as a function on $M$
        \end{enumerate}
    then~\eqref{introeq: Wide IM weighting to integrate} defines a multiplicative weighting of $G$ along $H$ such that the induced weighting of $A$ defined by~\autoref{theorem: weightings can be differentiated} is given by the filtration  
        \[ A=A_{-r} + B \supseteq A_{-r+1} + B \supseteq \cdots A_{-1} + B \supseteq B. \]
\end{theorem*}

Using this, we show that a filtered Lie groupoid (cf.~\cite[Definition 67]{van2019groupoid}) is the same thing as a Lie groupoid with a multiplicative weighting along its objects (\autoref{theorem: wide weighted groupoids are filtered}). 

\subsection{Overview of Chapter 6}

In~\autoref{chapter: Weightings and Higher Tangent Bundles} we use a description of weightings as graded subbundles of the higher jet bundles to prove another partial converse the differentiation theorem,~\autoref{theorem: weightings can be differentiated}.

We begin by reviewing the $r$-th order tangent bundle $T_rM \to M$, and review Loizides and Meinrenken's work (\cite[Section 7]{loizides2023differential} relating weightings of $M$ along $N$ and graded subbundles $Q\sset T_rM$ over $N$. We characterize multiplicative and infinitesimally multiplicative weightings in terms of the graded bundle $Q$. 

\begin{theorem*}[=~\autoref{proposition: equivalence of groupoids weighting definitions} and~\autoref{theorem: LA weightings in terms of Q}]
    \begin{enumerate}
        \item[(a)] Let $G\toto M$ be weighted along $H \sset G$. The weighting is multiplicative if and only if the graded subbundle $Q_G \sset T_rG$ is a Lie subgroupoid $Q_G\toto Q_M$ of $T_rG \toto T_rM$.

        \item[(b)] Let $A\Rightarrow M$ be linearly weighted along $B\Rightarrow N$. The weighting is infinitesimally multiplicative if and only if the corresponding graded subbundle $Q_A\sset T_rA$ is a Lie subalgebroid $Q_A \Rightarrow Q_M$ of $T_rA \Rightarrow T_rM$.
    \end{enumerate}    
\end{theorem*}

Taking advantage of this characterization of multiplicative and infinitesimally multiplicative weightings, we show that the spray exponential (\cite[Definition 3.20]{cabrera2020local}) corresponding to any Lie algebroid spray of filtration degree zero is a (partially defined) weighted morphism. The main result of~\autoref{chapter: Weightings and Higher Tangent Bundles} is following theorem. 

\begin{theorem*}[=~\autoref{theorem: integration of multiplicative weightings}]
    Suppose that $G\toto M$ is a $s$-connected Lie groupoid and $H\toto N$ is a $s$-connected Lie subgroupoid with Lie algebroids $A = \mathrm{Lie}(G)$ and $B = \mathrm{Lie}(H)$, respectively, and suppose that $A$ is infinitesimally multiplicatively weighted along $B$. Let $Q_A \Rightarrow Q_M$ be the graded subbundle of $T_rA \Rightarrow T_rM$ corresponding to the weighting of $A$ along $B$. If $Q_A$ integrates to an $s$-connected subgroupoid $Q_G \toto Q_M$ of $T_rG\toto T_rM$, then $Q_G$ is a graded subbundle of $T_rG \toto T_rM$ which defines a multiplicative weighting of $G$ along $H$. 
\end{theorem*}

\subsection{Overview of Chapter 7}

We end this thesis with some possible future directions we would like to pursue. First, we propose a definition of \emph{linear singular Lie filtrations} (\autoref{definition: linear singular Lie filtrations}) and put forward the problem of determining when a linear singular Lie filtration defines a linear weighting, extending the work of Loizides and Meinrenken on singular Lie filtrations (\cite{loizides2022singular}) to vector bundles. 

Next, we briefly review \emph{associative algebroids} in the sense of \v{S}evera (\cite[Letter 6]{vsevera2017letters}) and present the task of studying deformation spaces in this setting, with a possible application the work of van Erp and Yuncken to pseudodifferential operators acting on sections of vector bundles. 

Following this, we summarily explain $k$-multifiltered manifolds and multiweightings. We ask for a connection between the two in a way that generalizes~\autoref{theorem: wide integration}. In particular, we make a conjecture that the right object for the "tangent groupoid" for a 2-multifiltered manifold should be a double groupoid of the form
    \begin{equation*}
    \xymatrix{
        \defw(\mathrm{Pair}^2(M), \Delta^h, \Delta^v) \ar@<-.5ex>[d] \ar@<.5ex>[d] \ar@<-.5ex>[r] \ar@<.5ex>[r] & \T_{F^h}M \ar@<-.5ex>[d] \ar@<.5ex>[d] \\
        \T_{F^v}M \ar@<-.5ex>[r] \ar@<.5ex>[r] & M\times \R,
    }
    \end{equation*}
where $F^v$ and $F^h$ are Lie filtrations of $M$ coming from the 2-multifiltration. We conclude by asking what type of $C^*$-algebraic information could be obtained from such an object.

%% file: 1_weightings.tex
\chapter{The Differential Geometry of Weightings}
\label{chapter: Weightings}

Weightings were first introduced by Melrose under the name of ``quasi-homogeneous structures" in order to do analysis on manifolds with corners (\cite{melrose1996differential}). Independently, they were later considered by Loizides and Meinrenken (\cite{loizides2023differential}) in order to understand the local information necessary to construct the exotic deformation spaces appearing in the work of Haj and Higson on linearization theorems for Euler-like vector fields with weights (\cite{sadegh2018euler}). Our exposition follows the work by Loizides and Meinrenken. 

\section{Basics of Weightings}
\label{section: basics of weightings}

Let $M$ be a smooth manifold and $N\sset M$ an embedded submanifold. Associated to $N$ is a filtration of the sheaf of smooth functions $C^\infty_{M}$ by ideals
    \[ C^\infty_{M} \supseteq \van{N} \supseteq \van{N}^2 \supseteq \cdots  \]
where 
    \begin{equation*}
        \van{N}(U) = \{ f \in C^\infty(U) : f|_{U\cap N} = 0 \};
    \end{equation*}
that is, $\van{N}^k$ is the subsheaf of smooth functions on $M$ vanishing to order $k$ along $N$. This filtration has the following special property: if $x_1, \dots, x_m$ is a coordinate system defined on $U\sset M$ such that 
    \[ N\cap U = \{ x_{1} = \cdots = x_n = 0\} \]
then $\van{N}^k(U)$ is the ideal generated by the monomials 
    \begin{equation}
    \label{equation: vanishing coordinates}
        x_{1}^{s_{1}} \cdots x_n^{s_n} \quad \text{where} \quad  \sum s_i \geq k.
    \end{equation}
A \emph{weighting} of a manifold $M$ along a submanifold $N$ is a modified notion of order of vanishing along $N$. More specifically, a weighting of $M$ along $N$ multiplicative filtration of the sheaf smooth functions on $M$ satisfying a local property generalizing~\eqref{equation: vanishing coordinates}. We recall their definition now. 

\subsection{Definition of a Weighting} A \emph{weight vector} $w = (w_1, \dots, w_m) \in \Z_{\geq 0}^m$ is an $m$-tuple of non-negative integers. Any integer $r \in \Z$ satisfying $r \geq \max_i\{w_i\}$ will be called the \emph{order} of the weight vector. Given an open subset $U\sset R^m$, let $C^\infty(U)_{(i)} \sset C^\infty(U)$ be the ideal generated by the monomials 
    \[ x_1^{s_1} \cdots x_m^{s_m} \quad \text{where} \quad  \sum_a w_as_a \geq i. \]
This defines a filtration 
    \begin{equation}
    \label{equation: local weighting filtration}
        C^\infty(U) = C^\infty(U)_{(0)} \supseteq C^\infty(U)_{(1)} \supseteq C^\infty(U)_{(2)} \supseteq \cdots 
    \end{equation}

\begin{definition}[{\cite[Definition 2.2]{loizides2023differential}}]
\label{definition: weighting}
    A \emph{weighting} of $M$ is a multiplicative filtration
        \begin{equation}
        \label{equation: weighting as a filtration of sheaf}
            C^\infty_M = C^\infty_{M,(0)} \supseteq C^\infty_{M,(1)} \supseteq C^\infty_{M,(2)} \supseteq \cdots
        \end{equation}
    of the sheaf of smooth functions on $M$ with the property that each point $p\in M$ has an open neighbourhood $U\sset M$ with coordinates $x_1,\dots, x_n$ defined on $U$ so that the filtration of $C^\infty_M(U) = C^\infty(U)$ is given by~\eqref{equation: local weighting filtration}.
\end{definition}

The coordinates $x_a$ will be called \emph{weighted coordinates}. Given a weighting of $M$, let $N \sset M$ denote the set of points for which the filtration~\eqref{equation: weighting as a filtration of sheaf} is non-trivial. Then~\cite[Lemma 2.4]{loizides2023differential} implies that $N$ is a closed submanifold and $C^\infty_{M,(1)} = \van{N}$. If $N$ is given in advance, then we will say that $M$ is \emph{weighted along} $N$ and refer to $(M,N)$ as a \emph{weighted manifold pair}. We note that, by definition, weighted coordinates serve as submanifold coordinates for $N$. 

If $(M,N)$ and $(M',N')$ are weighted pairs, a \emph{weighted morphism} from $(M,N)$ to $(M',N')$ is a smooth map $F:M\to M'$ such that 
    \[ F^*C^\infty_{M',(i)} \sset C^\infty_{M, (i)} \]
for all $i \geq 0$. 

\begin{remark}
     In this thesis we will work in the $C^\infty$-category unless explicitly stated otherwise. Therefore we may take advantage of the existence of partitions of unity to avoid the use of sheaves, working instead with the filtration of global functions 
        \[ C^\infty(M) = C^\infty(M)_{(0)} \supseteq C^\infty(M)_{(1)} \supseteq C^\infty(M)_{(2)} \supseteq \cdots . \]
\end{remark}

\subsubsection{Basic Examples}

We now give several examples of weightings. 

\begin{examples}
\label{example: examples of weightings}
    \begin{enumerate}
        \item[(a)] As already discussed, if $N$ is a closed submanifold of $M$ then order of vanishing defines a weighting of $M$ along $N$:
            \[ C^\infty_M \supseteq \van{N} \supseteq \van{N}^2 \supseteq \cdots. \]
        We refer to this weighting as the \emph{trivial weighting} of $M$ along $N$. There are two extreme cases we want to point out: $N = M$ and $N = \emptyset$. In these cases, the trivial weighting is given by 
            \[ C^\infty(M) \supseteq 0 \supseteq 0 \supseteq \cdots, \]
        and 
            \[ C^\infty(M) \supseteq C^\infty(M) \supseteq C^\infty(M) \supseteq \cdots, \]
        respectively. 

        \item[(b)] Define a weighting of $\R^2$ along the origin by declaring that $x$ has weight 1 and $y$ has weight 3. The resulting filtration is given by 
            \[ C^\infty(\R^2) \supseteq \la x,y \ra \supseteq \la x^2, y \ra \supseteq \la x^3, y \ra \supseteq \la x^4, xy, y^2 \ra \supseteq \cdots.  \]
        As simple as it is, this example will shed a considerable amount of light on some of the concepts later in this chapter. 

        \item[(c)] \cite[Examples 2.9 (c)]{loizides2023differential} A nested sequence
            \[ M = N_{r+1} \supseteq N_{r} \supseteq \cdots \supseteq N_{1} \supseteq N_0 = N \]
        of embedded submanifolds $N_i$ (with $N$ closed) defines a weighting of order $r$ by 
            \[ C^\infty_{M, (i)} = \sum_{k\geq 0} \sum_{i_1 + \cdots + i_k = i} \van{N_{i_1}}\cdots \van{N_{i_k}}. \]
        Weighted coordinates near point $p\in N$ are given by coordinates $x_a$ which are submanifold coordinates for each of the $N_i$.   

        \item[(d)] \cite[Secton 2.2]{loizides2023differential} If $(M, N)$ and $(M',N')$, then $M\times M'$ is naturally weighted along $N\times N'$ by letting $C^\infty(M\times M')_{(i)}$ be the ideal generated by 
            \[ \sum_{i_1+i_2 = i} C^\infty(M)_{(i_1)} \otimes C^\infty(M)_{(i_2)}.  \]
        If $x_a$ are local weighted coordinates for $M$ and $y_b$ are local weighted coordinates for $M'$, then $x_a, y_b$ for a local weighted coordinate system for $M\times M'$. 
    \end{enumerate}
\end{examples}

\subsection{Filtration of the Tangent and Cotangent Bundles}

It was observed (\cite[Proposition 2.6]{loizides2023differential}) the a weighting of $M$ along $N$ determines a filtration 
    \begin{equation}
    \label{equation: normal filtration}
        \nu(M,N) = F_{-r} \supseteq F_{-r+1} \supseteq \cdots \supseteq F_0 = 0 
    \end{equation}
of the normal bundle of $N$ in $M$ by subbundles $F_{-i}\to N$. This filtration is induced by a filtration of $TM|_N$, which is in turn defined by ``dualizing'' a filtration of $T^*M|_N$. Through our study of weighted manifolds, we found it to be fruitful to consider the filtrations of $TM|_N$ and $T^*M|_N$ in their own right. We recall their constructions now. 

Let 
    \[ (T^*M|_N)_{(i)} = \mathrm{span}_{C^\infty(N)}\{ \ed f|_N : f \in C^\infty(M)_{(i)} \}. \]
If $x_a$ is a weighted coordinate system defined on $U\sset M$, then $\{\ed x_a|_{U\cap N} : w_a \geq i\}$ is a frame for  $T^*M|_{U\cap N}$, hence each $(T^*M|_N)_{(i)}$ is a subbundle of $T^*M|_N$. This defines the filtration
    \begin{equation}
    \label{equation: filtration of cotangent bundle}
        T^*M|_N = (T^*M|_N)_{(0)} \supseteq (T^*M|_N)_{(1)} \supseteq \cdots \supseteq (T^*M|_N)_{(r)} \supseteq 0.
    \end{equation}
We remark, in particular, that  $(T^*M)_{(1)} = \mathrm{ann}(TN)$ and 
    \[ \mathrm{rank}(\gr(T^*M|_N)_{(i)}) = \# \{ a : w_a = i\} , \]
where 
    \[ \gr(T^*M|_N)_{(i)} = (T^*M|_N)_{(i)}/(T^*M|_N)_{(i+1)}.  \]
    
\begin{lemma}
\label{lemma: extension of weighted coordinates}
    Let $(M,N)$ be a weighted pair and $p\in N$. Let $x_1, \dots, x_k$ be functions of filtration degrees $w_1, \dots, w_k$ defined near $p$. Suppose that for all $i$ the image of the set $\{\ed_p x_a : w_a = i\}$ in $\gr(T^*_pM)_{(i)}$ is linearly independent. Then the functions $x_a$ are contained in a system of weighted coordinates on an open neighbourhood $U$ of $p$. 
\end{lemma}
\begin{proof}
    Let $y_a$ be any system of weighted coordinates near $p$. The image of $\{\ed_p y_a : w_a = i\}$ in $\gr(T^*_pM)_{(i)}$ is a basis for $\gr(T^*_pM)_{(i)}$, hence by the replacement theorem from linear algebra there is an index set $J_i$ such that the image of 
        \[ \{ \ed_p x_a : w_a = i\} \cup \{\ed_p y_a : w_a = i, a \in J_i\} \]
    in $\gr(T^*_pM)_{(i)}$ is a basis for $\gr(T^*_pM)_{(i)}$. Letting $J = \bigcup_i J_i$ it follows that $\{\ed_p x_a \} \cup \{ \ed_p y_a : a \in J\}$
    is a basis for $T_p^*M$, hence 
        \[ \{x_a\} \cup \{y_b : b \in J\}\]
    is the required coordinate system. 
\end{proof}

The filtration of $TM|_N$,
    \begin{equation}
    \label{equation: filtration of tangent bundle}
        TM|_N = (TM|_N)_{(-r)} \supseteq \cdots \supseteq (TM|_N)_{(0)} \supseteq 0,
    \end{equation}
is the ``dual'' of~\eqref{equation: filtration of cotangent bundle}, given by
    \[ (TM|_N)_{(-i)} = \mathrm{ann}((T^*M|_N)_{(i+1)}). \]
In a local weighted coordinate system $x_a$ defined near $p\in N$ we have 
    \[ (T_pM)_{(-i)} = \mathrm{span}\left\{ \frac{\bd}{\bd x_a}\bigg|_p : w_a \leq i \right\}; \]
note in particular that $(TM|_N)_{(0)} = TN$. As we will see in the forthcoming sections, the tangent bundle for the weighted pair $(M,N)$ should be $TM$ together with the filtration of $TM|_N$. 

\begin{examples}
    \begin{enumerate}
        \item[(a)] If $M$ is given the trivial weighting along $N$, then the filtration~\eqref{equation: filtration of tangent bundle} is just 
            \[ TM|_N \supseteq TN.  \]

        \item[(b)] If $\R^2$ is weighted by declaring that $\mathrm{wt}(x) = 1$ and $\mathrm{wt}(y) = 3$ as in~\autoref{example: examples of weightings} (b), then the filtration of $T_0\R^2 = \R^2$ is 
            \[ \R^2 \supseteq \{y = 0\} \supseteq \{y = 0\} \supseteq \{0\} \]

        \item[(c)] If $M$ is given the weighting defined by a nested sequence of submanifolds $M = N_{r+1} \supseteq N_{r} \supseteq \cdots \supseteq N_{1} \supseteq N_0 = N$ as in~\autoref{example: examples of weightings} (c), the filtration of $TM|_N$ is given by 
            \[ TM|_N \supseteq TN_{r}|_N \supseteq TN_{r-1}|_N \supseteq \cdots \supseteq TN. \]

        \item[(d)] If $(M_i, N_i)$, $i=1,2$ are weighted pairs, then the filtration of $T(M_1\times M_2) = TM_1\times TM_2$ is given by 
            \[ (T(M_1\times M_2)|_{N_1\times N_2})_{(i)} = (TM_1|_{N_1})_{(i)} \times (TM_2|_{N_2})_{(i)} \]
    \end{enumerate}
\end{examples}

\section{Weighted Submanifolds}
\label{section: weighted submanifolds}

Recall that submanifolds of $M$ are subsets which locally look like coordinate subspaces. Modifying this definition by requiring one use weighted coordinates gives the definition of a submanifold in the category of weighted manifolds. 

\begin{definition}
    Let $(M,N)$ be a weighted pair. A submanifold $R\sset M$ is called a \emph{weighted submanifold} if there exists a weighted atlas of submanifold charts. Such a choice of coordinates will be called \emph{weighted submanifold coordinates}.
\end{definition}

That is, at each point $p\in N\cap R$ there exist local coordinates which are simultaneously submanifold coordinates for $R$ and weighted coordinates for $N$. 

\begin{examples}
\label{examples: weighted submanifolds}
    \begin{enumerate}
        \item[(a)] Let $(M,N)$ be a weighted pair. If $R\cap N = \emptyset$ then $R$ is a weighted submanifold of $(M,N)$. Any submanifold $R \sset N$ is a weighted submanifold.
    
        \item[(b)] Any point $p\in M$ is a weighted submanifold. 
        
        \item[(c)] Recall submanifolds $N, R \subseteq M$ are said to \emph{intersect cleanly} if $N\cap R$ is a submanifold and $T(N\cap R) = TN\cap TR$. This is equivalent to $M$ admitting an atlas of submanifold charts for $N$ and $R$ simultaneously (see~\cite[Proposition C.3.1]{hormander2007analysis}, for instance). Thus, if $M$ is trivially weighted along $N$, then the weighted submanifolds are exactly the submanifolds of $M$ intersecting $N$ cleanly. 

        \item[(d)] If $(M, N)$ is a weighted pair then any submanifold $R$ which is transverse to $N$ is a weighted submanifold. Indeed, let $p\in R\cap N$. Choosing weighted coordinates $x_a$ near $p$, then $N$ is, by definition, given near $p$ by $\{x_a = 0 : w_a \geq 1\}$. If $R$ is given near $p$ by $\{y_b = 0\}$, then transversality ensures that each $y_b \in C^\infty(M)_{(0)}$ and the image of the set $\{\ed_py_b\}$ in  $\gr(T^*_pM)_{(0)}$ is independent. Thus, 
            \[ \{y_b\} \cup \{x_a : w_a \geq 1\} \]
        can be completed to a weighted coordinate system near $p$ by~\autoref{lemma: extension of weighted coordinates}.

        \item[(e)] If $\R^n$ is weighted by assigning weights to the standard coordinates $x_a$, then any linear subspace $V$ is a weighted submanifold. To see this, let $e_i$ be the standard basis for $\R^n$. The weighting defines a filtration
            \[ \R^n = W_{(-r)} \supseteq W_{(-r+1)} \supseteq \cdots \supseteq W_{(0)} \supseteq 0, \]
        where 
            \[ W_{(-i)} = \mathrm{span}\{e_a : w_a \leq i \}.  \]
        Letting $V_{(-i)} = V \cap W_{(-i)}$, we can inductively find subspaces
            \[ V'_{(-r)} \supseteq V'_{(-r+1)} \supseteq 
            \cdots \supseteq V'_{(-1)} \supseteq V'_{(0)} \supseteq 0\]
        such that $W_{(-i)} = V_{(-i)} \oplus V'_{(-i)}$. Let 
            \[ \mathcal{B}_{(i)} = \{ v_a^{(i)} : a = 0, \dots, \dim(V_{(-i)}/V_{(-i+1)})\}   \]
        be a linearly independent set with the property that $\bigcup_{i=0}^j \mathcal{B}_{(i)}$ is a basis for $V_{(j)}$; in particular, $\mathcal{B} = \bigcup_{i=0}^r\mathcal{B}_{(i)}$ is a basis for $V$. Letting $\mathcal{B}'$ be an analogously constructed basis for $V'$, the linear coordinates for $\R^n$ defined by $\mathcal{B}\cup \mathcal{B}'$ define weighted submanifold coordinates for $V$.  

        \item[(f)] Define a weighting of $\R^2$ along the origin by declaring that $\mathrm{wt}(x) = 1$ and $\mathrm{wt}(y) = 3$, as in~\autoref{example: examples of weightings} (b). Then the weighted submanifolds of $(\R^2,\{0\})$ are precisely those which: 
            \begin{enumerate}
                \item do not pass through the origin 

                \item submanifolds $S$ passing through the origin which are not tangent to the x-axis. 

                \item submanifolds $R$ passing through the origin with at least third order tangency with the $x$-axis. In particular, the curve $y=x^2$ is \emph{not} a weighted submanifold. 
            \end{enumerate}
        The first case is obvious. For the second case, by the implicit function theorem we can locally describe $S$ as $x=f(y)$ for some function $f$ with $f(0)=0$. In this case, the weighted submanifold coordinates can be given by $\Tilde{x} = x-f(y)$ and $\Tilde{y} = y$. For the third case, the submanifold $R$ is given near the origin by $y=f(x)$ with $f(0)=0$ and $f'(x) = 0$ (i.e. tangent to the $x$-axis). Suppose that $\Tilde{x}$ and $\Tilde{y}$ are weighted coordinates so that the curve $y = f(x)$ has the standard form $\Tilde{y} = 0$. Then
            \[ \Tilde{y} = a(y-f(x)) + O_f(4) \]
        where $O_f(4)$ are terms of filtration order at least 4. However, this has filtration order 3 if and only if only if $f = O(x^3)$. 
    \end{enumerate}
\end{examples}

\begin{remarks}
    \begin{itemize}
        \item[(a)] Note that~\autoref{examples: weighted submanifolds} shows that for a weighted pair $(M,N)$, submanifolds $R\sset M$ containing $N$ \emph{need not} be weighted submanifolds. 

        \item[(b)] A weighted submanifold $R$ of a weighted pair $(M,N)$ is \emph{not} simply a submanifold with a weighting along $R\cap N$ such that the inclusion $R \into M$ is a weighted morphism. 
        
        For example, consider the weighting of $\R^2$ defined by declaring $x$ to have weight 1 and $y$ to have weight 3. If we give the curve $y=x^2$ the doubled trivial weighting along $\{0\}$ then the inclusion into $\R^2$ is a weighted morphism, but $y=x^2$ is \emph{not} a weighted submanifold. 
    \end{itemize}
        
\end{remarks}

\begin{proposition}
\label{proposition: weighted submanifolds are weighted}
    If $R$ is a weighted submanifold of the weighted pair $(M,N)$, then $R$ inherits a natural weighting along $R\cap N$. With respect to this weighting, the inclusion $R \into M$ is a weighted morphism, and we have moreover that 
        \begin{equation}
        \label{equation: tangent filtration of weighted subbundles}
            (TR|_{R\cap N})_{(i)} = TR\cap (TM|_N)_{(i)}.
        \end{equation}
\end{proposition}
\begin{proof}
    Recall that restriction to $R$ determines an isomorphism 
        \[ C^\infty(M)/\van{R} \cong C^\infty(R).  \]
    Using this, we define 
        \[ C^\infty(R)_{(i)} = C^\infty(M)_{(i)}/(C^\infty(M)_{(i)} \cap \van{N}); \]
    the restriction of weighted submanifold coordinates to $R$ yields weighted coordinates for $R$, hence this defines a weighting. Since the inclusion $R\into M$ induces the restriction map $C^\infty(M) \to C^\infty(R)$, it is clear that it is a weighted morphism. 

    For the statement about filtration of tangent spaces, let $p\in N\cap R$ and let 
    $x_a, y_b$ be a weighted submanifold coordinate system defined on $U\sset M$ containing $p$ so that $R\cap U = \{y_b = 0\}$. As a subbundle of $T^*_pM$, 
        \[ \mathrm{ann}((T_pR)_{(i)}) = \mathrm{span}\{\ed_px_a : w_a \geq i\} + \mathrm{span}\{\ed_py_b \}.  \]
    On the other hand, we have that 
        \begin{align*}
             \mathrm{ann}(T_pR) & = \mathrm{span}\{\ed_py_a\} \quad \text{and} \\
                 T^*_pM_{(i)} & = \mathrm{span}\{\ed_px_a,\ \ed_py_b : w_a \geq i, w_b \geq i\},
        \end{align*}
    hence $\mathrm{ann}((TR|_{R\cap N})_{(i)}) = \mathrm{ann}(TR) + (T^*M|_N)_{(i)}$. Therefore, 
        \begin{align*}
            (TR|_{R\cap N})_{(i)} & = \mathrm{ann}(\mathrm{ann}((TR|_{R\cap N})_{(-i+1)}) \\
            & = \mathrm{ann}(\mathrm{ann}(TR) + (T^*M|_N)_{(-i+1)}) = TR \cap (TM|_N)_{(i)}. \qedhere
        \end{align*}
\end{proof}

\begin{proposition}
    Let $(M,N)$ be a weighted pair. Given $p \in M$, any subspace of $T_pM$ is realized as the tangent space to a weighted submanifold $S$.
\end{proposition}
\begin{proof}
    A choice of weighted coordinates $x_a$ on $U\subset M$ allows us to identify $U$ with $T_pU = \R^n$ as weighted manifolds, where $\R^n$ is weighted as in~\autoref{examples: weighted submanifolds} (e), from which the result follows. 
\end{proof}

We close our discussion with a weighted version of~\cite[Theorem 1.13]{kolar2013natural}.

\begin{proposition}
\label{proposition: weighted projections and weighted submanifolds}
    Let $(M,N)$ be a weighted pair, and suppose that $p:M\to M$ is a weighted morphism satisfying $p\circ p = p$. Then $R = p(M)$ is a weighted submanifold.
\end{proposition}

Before proceeding to the proof, we need a lemma from linear algebra. 

\begin{lemma}
\label{lemma: linear algebra projection fact}
    Let $V$ be a finite dimensional vector space and $P:V\to V$ be a projection. If $\{v_1, \dots, v_n\}$ is any basis for $V$, then there exists an index set $\mathcal{I} \sset \{1, 2, \dots, n\}$ with $|\mathcal{I}| = \dim (\mathrm{Im}(P))$ such that 
        \[ \{ Pv_a : a \in \mathcal{I}\} \cup \{(1-P)v_b : b \in \mathcal{I}^c \} \]
    is a basis for $V$. 
\end{lemma}
\begin{proof}
    For the existence of $\mathcal{I}$, by an inductive argument it is enough to show that $\{Pv_1, v_2, \dots v_n\}$ or $\{(1-P)v_1, v_2, \dots, v_n\}$ is a basis for $V$. If $Pv_1 = 0$ or $(1-P)v_1 = 0$ then there is nothing to prove, so assume that $Pv_1 \neq 0$ and $(1-P)v_1 \neq 0$. If neither of these sets are linearly independent then $P v_1 \in \mathrm{span}\{v_2, \dots v_n\}$ and $(1-P)v_1 \in \mathrm{span}\{v_2, \dots v_n\}$. However, this implies that $v_1 \in \mathrm{span}\{v_2, \dots v_n\}$, which is a contradiction. 
    
    To see that $|\mathcal{I}| = \dim (\mathrm{Im}(P))$, note that we obviously have $|\mathcal{I}| \leq \dim (\mathrm{Im}(P))$. On the other hand, since $P$ is a projection,  
        \[n - |\mathcal{I}| = |\mathcal{I}^c| \leq \dim (\mathrm{Im}(1-P)) = n-\dim(\mathrm{Im}(P)) \]
    hence $\dim(\mathrm{Im}(P)) \leq |\mathcal{I}|$.
\end{proof}

\begin{proof}[Proof of~\autoref{proposition: weighted projections and weighted submanifolds}]
    By~\cite[Theorem 1.13]{kolar2013natural}, $p$ has constant rank near $R = p(M)$. Let $q\in R$ and let $x_a$ be a weighted coordinate system defined near $q$. Since the cotangent map $T^*_qp:T^*_qM \to T^*_qM$ is a projection,~\autoref{lemma: linear algebra projection fact} implies that there is an index set $\mathcal{I}$ with $|\mathcal{I}| = \mathrm{rank}_q(p) = \dim(R)$ such that 
        \begin{equation}
        \label{equation: weighted submanifold coordinates for weighted projection}
            \{p^*x_a : a \in \mathcal{I}\} \cup \{x_b-p^*x_b: b \in \mathcal{I}^c\}
        \end{equation}
    is a coordinate system. In fact, since $p$ is a weighted morphism, it is a weighted coordinate system near $q$. Note as well that for any $q' \in R$ sufficiently close to $q$ we have that 
        \[ (x_b-p^*x_b)(q') = x_b(q') - x_b(p(q')) = x_b(q') - x_b(q') = 0; \]
    since $|\mathcal{I}^c| = \mathrm{codim}(R)$ it follows that~\eqref{equation: weighted submanifold coordinates for weighted projection} is a weighted submanifold coordinate system for $R$. 
\end{proof}

\section{Weighted Morphisms}
\label{section: weighted morphisms}

We now turn our attention to weighted morphisms. In this section we will establish the weighted analogues of embeddings, submersions, and transversality. In practice, it can be hard to check whether a given map is a weighted morphism or not, so we also give two characterizations of weighted morphisms - one in terms of its graph and one in terms of ``weighted paths''. Let us begin be giving some examples of weighted morphisms. 

\begin{examples}
    \begin{itemize}
        \item[(a)] Let $(M,N)$ and $(M',N')$ be weighted pairs, and let $F:M\to N$ be a smooth map. If $M$ and $M'$ are both trivially weighted, then $F$ is a weighted morphism if and only if $F(N) \sset N'$. In general, if $F$ is weighted then $F(N)\sset N'$. 

        \item[(b)] With the set up as above, if $N = \emptyset$ or $N' = M'$, then $F$ is a weighted morphism. 

        \item[(c)] Suppose that $(M,N)$ and $(M',N')$ are weighted manifolds. Then both of the projections $M\times M' \to M$ and $M\times M' \to M$ are weighted morphisms. If $p \in M'$, then the inclusion $M \to M\times M'$, $m \mapsto (m,p)$ is a weighted morphism if and only if $p \in N'$. On the other hand, the inclusion $M \times \{p\} \into M\times M'$ \emph{is} a weighted morphism for any $p\in M'$ since $M\times \{p\}$ is a weighted submanifold of $M\times M'$; in the case when $p\notin N'$, $M\times \{p\}$ is weighted along the empty set, hence is \emph{not} isomorphic to $M$ as a weighted manifold. 
    \end{itemize}
\end{examples}

\subsection{Weighted Embeddings}

We now discuss \emph{weighted embeddings}. Recall that an immersion is a smooth map whose tangent map is an isomorphism onto its image. Taking into account the philosophy that the tangent bundle of a weighted pair $(M,N)$ is $TM$ together with the filtration~\eqref{equation: filtration of tangent bundle} of $TM|_N$, we arrive at the following definition.

\begin{definition}
\label{definition: weighted embedding}
    A weighted morphism $f:(M,N) \to (M',N')$ is called a \emph{weighted immersion} if 
        \begin{enumerate}
            \item[(a)] $f:M\to M'$ is an immersion, 

            \item[(b)] $Tf:TM \to TM'$ restricts to injections $(TM|_N)_{(i)} \to (TM'|_{N'})_{(i)}$ with image 
                \[ Tf(TM)\cap (TM'|_{N'})_{(i)}. \]
        \end{enumerate}
    A weighted \emph{embedding} is defined analogously. 
\end{definition}

In particular, the intersections $Tf(TM)\cap (TM'|_{N'})_{(i)}$ are subbundles of $TM'|_{N'}$ and 
    \[ Tf:TM \to Tf(TM)\cap (TM'|_{N'})_{(i)}\]
is a vector bundle isomorphism. 

\begin{examples}
\label{examples: weighted embeddings}
    \begin{enumerate}
        \item[(a)] If $R$ is a weighted submanifold of the weighted pair $(M,N)$, then the inclusion $\iota : (R, R\cap N) \into (M,N)$ is a weighted embedding by~\autoref{proposition: weighted submanifolds are weighted}.
        
        \item[(b)] If $(M,N)$ and $(M',N')$ are both given the trivial weighting, then a weighted morphism $f:(M,N) \to (M',N')$ is a weighted embedding if and only if $f$ is an embedding and $f(M)$ intersects $N'$ cleanly, with intersection $f(N)$. Indeed, for $i=1$ the bundles in question are the zero bundles, so we have a diffeomorphism  
            \[ N \cong f(M) \cap N', \]
        which, in particular, says that $f(M)\cap N'$ is a manifold with $T(f(M)\cap N') = Tf(TN)$. For $i=0$ this says that 
            \[ TN \cong Tf(TM)\cap TN',  \]
        hence $f(M)$ intersects $N'$ cleanly. 

        \item[(c)] Let $M=\R$ with $\mathrm{wt}(x) = 2$ and $M'=\R$ with $\mathrm{wt}(x) = 1$. Then the identity map $M\to M'$ is a weighted morphism, but it is \emph{not} a weighted embedding, since $(T_0M)_{(-1)} = 0$ but $(T_0M')_{(-1)} = \R$. 
    \end{enumerate}
\end{examples}

\begin{theorem}[Normal Form for Weighted Immersions]
\label{theorem: weighted embedding characterization}
    Let $i: M\to M'$ be a smooth map between weighted manifolds of dimensions $n\leq n'$. Then $i$ is a weighted immersion if and only if for all $p\in M$, with image $p_0 = i(p)$, there exist weighted coordinates $x_1, \dots, x_n$ near $p_0$ such that $i^*x_1, \dots, i^* x_n$ are weighted coordinates near $p$ and $i^*x_{n+1} = \cdots = i^*x_{n'}=0$.  
\end{theorem}

We refer to the coordinates $x_1, \dots, x_{n'}$ in~\autoref{theorem: weighted embedding characterization} as \emph{immersion coordinates} for the map $i$.

\begin{proof}
    The direction $(\Leftarrow)$ follows by direct examination in weighted immersion coordinates. For the converse, suppose that $\iota$ is a weighted immersion. To build weighted submanifold coordinates near $p_0 \in  f(M) \cap N'$, choose a system of weighted coordinates $x_1, \dots, x_{n'}$. We may choose a subset of coordinates, re-indexed as $x_1, \dots, x_n$, such that the image of the set $\{\ed_{p_0} x_a : w_a = i\}$ in $\mathrm{gr}(T^*_{p_0}M')_{(i)}$ pulls back to a basis of $\mathrm{gr}(T^*_pM)_{(i)}$. In particular, $y_a = \iota^*x_a$, for $a=1,\dots, n$, are weighted coordinates on $M$. If the pull-backs of remaining coordinates vanish, we are done. Otherwise, consider a coordinate $x_a$ with $a > n$ and $i = w_a$ as large as possible that such that $\iota^*x_a \neq 0$. Then $\iota^*x_a$ has filtration degree $i$, so it can be written as $\iota^*x_a = h(y_1, \dots, y_n)$, where $h$ is a function with $h(t^{w_1}c_1, \dots, t^{w_n}c_n) = O(t^i)$ for all $c\in \R^n$. By replacing $x_a$ with $x_a - h(x_1, \dots, x_n)$, we can arrange that $\iota^*x_a = 0$. Proceeding in this way, we obtain the required weighted coordinate system.
\end{proof}

Applying this to the case that our weighted immersion is a weighted embedding yields the following.

\begin{corollary}
    If $i:(M,N)\to (M', N')$ is a weighted embedding, then $i(M)$ is a weighted submanifold of $(M', N')$ such that $i(M) \cap N' = i(N)$ and $(M,N)$ is isomorphic to $(i(M), i(N))$ as weighted manifolds. 
\end{corollary}

Applying this to the inclusion of a submanifold $R\sset M$ yields the following characterization of weighted submanifolds, which generalizes clean intersection. 

\begin{corollary}
\label{corollary: criteria for weighted submanifolds}
    If $(M, N)$ is a weighted pair, then a submanifold $R\sset M$ is a weighted submanifold if and only if 
        \begin{enumerate}
            \item[(a)] $R$ is weighted along $R\cap N$,
            \item[(b)] the inclusion $R\into M$ is a weighted morphism, and
            \item[(c)] $(TR|_{R\cap N})_{(i)} = TR\cap (TM|_N)_{(i)}$. 
        \end{enumerate}
\end{corollary}

\begin{remark}
    Note that this says that we could have taken the conclusion of~\autoref{proposition: weighted submanifolds are weighted} as a definition of weighted submanifolds. In particular, we could use this to give a define immersed weighted submanifolds. 
\end{remark}

\subsection{Weighted Submersions}

We now discuss \emph{weighted submersions}. Since a submersion is a smooth map whose tangent map is fibrewise surjective, the discussion at the beginning of the last section leads us to the following definition. 

\begin{definition}
\label{definition: weighted submersion}
    Let $(M,N)$ and $(M',N')$ be weighted pairs. A submersion $\pi:M\to M'$ is called a \emph{weighted submersion} if it is a weighted morphism such that for each $i$ the map 
        \[ (TM|_N)_{(i)} \to (TM'|_{N'})_{(i)} \]
    is fibrewise surjective. 
\end{definition}

\begin{examples}
    \begin{enumerate}
        \item[(a)] If $(M_1,N_1)$ and $(M_2,N_2)$ are weighted pairs, then both of the projections $\pi_i: M_1\times M_2\to M_i$ are weighted submersions.

        \item[(b)] If $M = \R^2$ is weighted by $\mathrm{wt}(x) = 1$ and $\mathrm{wt}(y) = 3$, and $M' = \R^2$ is given the trivial weighting along the curve $y = x^2$, then the identity map $M\to M'$ is a weighted morphism and a submersion, but it is not a weighted submersion. Indeed, we have that $(T_0M)_{(-i)} = \mathrm{span}\{\bd_x\}$, but $(T_0M')_{(-1)} = \mathrm{span}\{\bd_x, \bd_y\}$

        \item[(c)] If $(M,N)$ and $(M',N')$ are both trivially weighted, then the weighted submersions are the submersions $\pi:M\to M'$ which restrict to submersions $\pi|_N : N\to N'$. 
    \end{enumerate}
\end{examples}

\begin{theorem}[Normal Form for Weighted Submersions]
\label{theorem: weighted submersion coordinates}
     Let $\pi: M \to M'$ be a smooth map between weighted manifolds of dimensions $n \geq  n'$. Then $\pi$ is a weighted submersion if and only if for all $p \in M$, with image $p_0 = \pi(p)$, there are weighted coordinates $x_1, \dots, x_n$ around $p$ such that $x_1, \dots, x_{n'}$ are $\pi$-basic and descend to weighted coordinates near $p_0$.
\end{theorem}

We refer to the coordinates $x_1, \dots, x_n$ in~\autoref{theorem: weighted submersion coordinates} as \emph{submersion coordinates} for the map $\pi$.

\begin{proof}
    The direction $(\Leftarrow)$ follows by direct examination in weighted submersion coordinates. For the opposite direction, note that the condition implies that for $p \in N$, with image $p_0 = \pi(p)$, the map
        \[ \mathrm{gr}(T^*_{p_0}M')_{(i)} \to \gr(T^*_pM)_{(i)} \]
    is injective for each $i$. Let $\{y_b\}$ be coordinates on an open neighborhood $U_0$ of $p_0$. Then the set $\{ \ed_{p_0} y_b : \mathrm{wt}(y_b) = i\}$ defines a basis of $\gr(T^*_{p_0}M')_{(i)}$. Their pullbacks under $\pi$ are hence linearly independent in $\gr(T^*_pM)_{(i)}$. By~\autoref{lemma: extension of weighted coordinates}, the functions $\pi^*y_b$ are part of a system of weighted coordinates near $p$.
\end{proof}

Put differently, this result says that given a weighted submersion we can always find weighted coordinates which are simultaneously submersion coordinates. In light of this, we have the following weighted versions of the standard results from differential geometry. 

\begin{corollary}
\label{corollary: inverse image of mfld is mfld}
    Let $(M,N)$ and $(M',N')$ be weighted pairs and $\pi : (M,N) \to (M',N')$ a weighted submersion. Then the pre-image of any weighted submanifold is a weighted submanifold. In particular, the fibres of a weighted submersion are weighted submanifolds. 
\end{corollary}


\subsection{Weighted Transversality}

We may also consider the weighted analogue of transversality. Let $(M, N)$, $(M', N')$, and $(M'', N'')$ be weighted pairs and $f:M\to M''$, $g:M'\to M''$ be weighted morphisms. 

\begin{definition}
\label{definition: weighted transverse}
    We say that $f$ and $g$ are \emph{weighted transverse} if they are transverse and
        \[ T_pf(\gr(T_pM)_{(i)})+T_qg(\gr(T_qM')_{(i)}) = \gr(T_{f(p)}M'')_{(i)} \]
    for all $(p,q)\in N\times N'$ such that $f(p) = g(q)$. 
\end{definition}

\begin{example}
    A weighted submersion is weighted transverse to any weighted morphism.
\end{example}

\begin{remark}
    If $f$ and $g$ are weighted transverse, then their restrictions to $N$ and $N'$ remain transverse. Hence $N\times_{N''} N'$ is a smooth manifold.
\end{remark}

\begin{theorem}
\label{theorem: weighted fibre products}
    If $f:M_0\to M$ and $g:M_1\to M$ are weighted transverse, then their fibre product 
        \[ M_1\times_{M}M_0 = \{ (q, p) \in M_1\times M_0 : f(p) = g(q) \} \]
    is a weighted submanifold of $M_1\times M_0$, with induced weighted along $N_1\times_N N_0$. Moreover, the projections 
        \[ \mathrm{pr}_i : M_1\times_M M_0 \to M_i, \quad i = 0, 1, \]
    are both weighted morphisms.
\end{theorem}
\begin{proof}
    Since $f$ and $g$ are transverse, it follows that $M_1\times_{M}M_0$ is a submanifold of $M_1\times M_0$. Therefore, we must produce weighted submanifold coordinates near any $(p,q) \in (M_1\times_{M}M_0) \cap (N_1\times N_0)$. 

    Let $U\sset M$ be an open neighbourhood of the point $r = f(p) = g(q) \in M$ and let $x_1, \dots, x_m$ be a weighted coordinate system defined on $U$. Consider the functions 
        \begin{align*}
            y_a : M_1\times M_0 & \to \R \\
            (q,p) & \mapsto (f^*x_a)(p) - (g^*x_a)(q).
        \end{align*}
    Since $f$ and $g$ are weighted morphisms, and since $x_a$ has weight $w_a$, it follows that $y_a \in C^\infty(M_1\times M_0)_{(w_a)}$. Moreover, since 
        \begin{align*}
            \ed y_a = f^*\ed x_a - g^*\ed x_a 
        \end{align*}
    the weighted transversality assumption ensures that the image of the set $\{\ed_{(q,p)}y_a : w_a = i\}$ in $\gr(T_{(q,p)}^*(M_1\times M_0))_{(i)}$ is linearly independent. Thus,~\autoref{lemma: extension of weighted coordinates} ensures that $y_a$ can be extended to a weighted coordinate system. Since $M_1\times_M M_1$ is locally cut out by the $y_a$, it follows that the completed coordinate system are the required weighted submanifold coordinates. Since the projections factor as the composition 
        \[ M_1\times_M M_0 \into M_1\times M_0 \to M_i,\]
    it is clear that they are weighted morphisms.
\end{proof}

\begin{corollary}
    If two weighted submanifolds intersect in a weighted transverse manner (i.e. the inclusions are weighted transverse), then their intersection is again a weighted submanifold. 
\end{corollary}

\begin{remark}
    In fact, if the intersection of two weighted submanifolds $R$ and $R'$ is weighted transverse, then by modifying the standard argument appropriately one finds that any point $p \in R \cap R'$ is contained in an open neighbourhood $U\sset M$ on which there is a local weighted coordinate system $x_a$ such that 
        \[ R\cap U = \{ x_1 = \cdots = x_{r} = 0 \} \quad \text{and} \quad R'\cap U = \{ x_{r+1} = \cdots = x_{r'} = 0 \}.  \]
\end{remark}

\subsection{Characterizations of Weighted Morphisms}

It can be difficult in practice to verify if a given map between weighted manifolds is a weighted morphism. It is therefore desirable to have alternative characterizations that one can appeal to in specific examples. 

Let $(M,N)$ and $(M', N')$ be weighted manifolds. One might guess that weighted morphisms are the smooth maps of pairs $f:(M,N) \to (M', N')$ for which the tangent map $Tf:TM|_N \to TM'|_{N'}$ is filtration preserving. However, this is not sufficient, as the following example shows. 

\begin{example}
    Consider $\R$ with the trivial weighting along the origin, and $\R^2$ weighted along the origin by declaring that $\mathrm{wt}(x) = 1$ and $\mathrm{wt}(y) = 3$. Consider the function
        \begin{align*}
            f: \R & \to \R^2 \\
            x & \mapsto (x, x^2);
        \end{align*}
    this is not a weighted morphism because $y$ has weight 3, but $f^*y = x^2$ has weight 2. However, $Tf_0 : T_0\R \to T_0\R^2$ is a filtration preserving map.  
\end{example}

For another possible approach, recall that morphisms of manifolds with some additional structure can be characterized in terms of their graphs. For instance, 
    \begin{itemize}
        \item[(a)] linear maps between vector spaces are exactly the maps $\varphi: V\to W$ whose graph 
            \[ \G(\varphi) = \{(\varphi(v), v) : v\in V \} \sset W\times V \]
        is a linear subspace of $W\times V$;

        \item[(b)] Lie algebroid morphisms are the vector bundle morphisms $\varphi:A \to B$ between Lie algebroids whose graphs are Lie subalgebroids of $B\times A$;

        \item[(c)]  Lie groupoid morphisms are exactly the smooth maps $\varphi : G \to H$ between Lie groupoids whose graphs are Lie subgroupoids of $H\times G$.
    \end{itemize}
In light of this, one might guess that weighted morphisms are exactly the smooth maps whose graph is a weighted submanifold of the product. It turns out that this is also not sufficient, as the following example shows.

\begin{example}
    Let $(M,N) = (\R^2, \{0\})$ be the trivially weighted pair, and let $(M', N') = (\R^2, \{0\})$ be weighting pair given by assigning $x$ weight $1$ and $y$ weight 3. Then the identity map $\mathrm{id} : (M,N)\to (M', N')$ is not a weighted morphism, but its graph is a weighted submanifold of $(M'\times M, N'\times N)$. Indeed, this follows because $\mathrm{id} : (M',N')\to (M, N)$ \emph{is} a weighted morphism. 
\end{example}

The following theorem says that these two conditions together are enough to characterize weighted morphisms.

\begin{theorem}
\label{theorem: characterization of weighted morphisms in terms of their graphs}
    Suppose that $(M,N)$ and $(M', N')$ are weighted pairs and $f:(M,N) \to (M', N')$ is a smooth map of pairs. Then $f$ is a weighted morphism if and only if
        \begin{enumerate}
            \item[(a)] the graph $\G(f)\sset M'\times M$ is a weighted submanifold  and 
            \item[(b)] the tangent map $Tf:TM|_N \to TM'|_{N'}$ is filtration preserving. 
        \end{enumerate}
\end{theorem}
\begin{proof}
    First, suppose that $f:(M,N) \to (M', N')$ is a weighted morphism. Let $x_a$ be weighted coordinates near $p\in N$ and $y_b$ be weighted coordinates near $q = f(p) \in N'$. Then weighted submanifold coordinates for $\G(f)$ near $(q, p)$ are given by $\Tilde{x_a} = x_a$ and $\Tilde{y_b} = y_b - f^*y_b$. To see that $Tf|_N$ is filtration preserving, let $p\in N$. Since $\bd_{x_a}(f^*y_b) \in C^\infty(M)_{(w'_b - w_a)}$ we have that 
        \[ w'_b > w_a \implies \frac{\bd (f^*y_b)}{\bd x_a}\bigg|_p = 0.  \]
    Therefore, 
        \[ T_pf\left( \frac{\bd}{\bd x_a}\bigg|_p \right) = \sum_{b\ :\ w'_b \leq w_a} \frac{\bd (y_b^*f)}{\bd x_a}\bigg|_p\frac{\bd}{\bd y_b}\bigg|_{f(p)} \in (T_{f(p)}M')_{(-w_a)}. \]

    For the converse, suppose that (a) and (b) are satisfied. The map $f$ factors as a composition of the diffeomorphism 
        \begin{align}
        \label{equation: weighted diffeo}
            M \to \G(f), \quad  x \mapsto (f(x), x),
        \end{align}
    the inclusion $\G(f) \to M'\times M$, and the projection $M'\times M \to M'$. Since the composition of weighted morphisms is a weighted morphism, it suffices to show that the map~\eqref{equation: weighted diffeo} is a weighted diffeomorphism (i.e.  a weighted morphism with weighted inverse). The inverse map $\G(f) \to M$ is a weighted morphism, since it is the restriction of the projection $M' \times M \to M$ to a weighted submanifold. By~\autoref{theorem: weighted submersion coordinates} it is enough to show that the map 
        \[ T\G(f)|_{\G(f) \cap N} \to TM|_N \]
    is an isomorphism of filtered vector bundles. By~\autoref{proposition: weighted submanifolds are weighted}, the filtration on $T\G(f) = T\G(f) \cap (N'\times N)$ is defined by the intersections 
        \[ (T\G(f)|_{\G(f)\cap N})|_{(i)} = T\G(f) \cap ((TM'|_{N'})_{(i)}\times (TM|_N)_{(i)}). \]
    This intersection consists of all $(Tf(v), v)$ such that $v\in (TM|_N)_{(i)}$ and $Tf(v) \in (TM'_{N'})_{(i)}$, hence it maps isomorphically onto $(TM|_N)_{(i)}$ if and only if $Tf$ is filtration preserving. 
\end{proof}

\begin{example}
\label{example: diagonal is weighted submanifold}
    For any weighted manifold $M$, the diagonal $\Delta_M\sset M\times M$ is a weighted submanifold. 
\end{example}

Another characterization of weighted morphisms is given by the notion of a \emph{weighted path}. Let $(M,N)$ be a weighted pair.

\begin{definition}
\label{definition: weighted paths}
    A path $\gamma:\R \to M$ is called a \emph{weighted path} if it is a weighted morphism for the trivial weighting of $\R$ along the origin.
\end{definition}

If $x_a$ is a system of weighted coordinates, with weight $w_a$, then $\gamma$ is of the form 
    \[ \gamma(t) = (x_1(t), \dots, x_n(t)) = (O(t^{w_1}), \dots, O(t^{w_n})). \]
We can completely recover the weighting of $M$ by knowledge of the weighted paths. 

\begin{theorem}
\label{A-proposition: characterization of weighted morphisms}
    \begin{enumerate}
        \item[(a)] We have $f \in C^\infty(M)_{(i)}$ if and only if  
            \[ f(\gamma(t)) =  O(t^i)\]
        for every weighted path $\gamma:\R \to M$.

        \item[(b)] A smooth map $\phi:(M, N) \to (M', N')$ between weighted pairs is a weighted morphism if and only if it takes weighted paths to weighted paths. 
    \end{enumerate}
\end{theorem}
\begin{proof}
    \begin{enumerate}
        \item[(a)] We have to show the "only if" direction.  Suppose that $f\in C^\infty(M)$ has the property that $f(\gamma(t)) =  O(t^i)$ for every weighted path $\gamma :\R \to M$ has filtration degree $i$. 

        Let $p \in N$ and let $x_a$ be weighted coordinates defined on $U\sset M$. Consider the Taylor expansion of $f$ with respect to these coordinates:
            \[ f(x) = \sum_{I \cdot w < i} c_Ix^I + g(x) = \sum_{j=1}^{i-1}p_j(x) + g(x), \]
        where $I = (i_1, \dots, i_n)$ is a multi-index, $g\in C^\infty(M)_{(i)}$, and $p_j$ are weighted homogeneous of degree $j$. Let $\lambda = (\lambda_1, \dots, \lambda_n) \in \R^n$ and consider the weighted path $\gamma(t) = (\lambda_1t^{w_1}, \dots, \lambda_nt^{w_n})$. Using that the $p_j$ are weighted homogeneous of degree $j$, we have, by assumption, that 
            \[ f(\gamma(t)) = \sum_{j=0}^{i-1}p_j(\gamma(t))+g(\gamma(t)) = \sum_{j=0}^{i-1}t^jp_j(\lambda)+g(\gamma(t)) = O(t^i). \]
        Since $g(\gamma(t)) = O(t^i)$, this implies that $t^jp_j(\lambda) = O(t^i)$, whence $p_j(\lambda) = 0$ for all $\lambda \in \R^n$. Thus, $f = g \in C^\infty(M)_{(i)}$.

        \item[(b)] Since the composition of weighted morphisms is weighted, a weighted morphism takes weighted paths to weighted paths. For the converse, suppose that $\phi$ takes weighted paths to weighted paths. Let $f \in C^\infty(M')_{(i)}$ and let $\gamma : \R \to M$ be a weighted path. Then $\phi\circ \gamma : \R \to M'$ is a weighted path and
            \[ (\phi^*f)(\gamma(t)) = f(\phi(\gamma(t))) = O(t^i), \]
        hence $\phi^*f \in C^\infty(M)_{(i)}$ by part (a). \qedhere
    \end{enumerate}
\end{proof}

\section{The Weighted Normal Bundle}
\label{section: weighted normal bundles}

Let us motivate this section by the following observation, which we learned from Haj and Higson (\cite{sadegh2018euler}). Suppose that $N\sset M$ is a closed submanifold, and let
    \[ \nu(M,N) = TM|_N/TN \]
be the normal bundle of $N$ in $M$. Any $f\in C^\infty(M)$ defines a fibrewise constant function $f^{[0]} \in C^\infty_{[0]}(\nu(M,N)) = C^\infty(N)$ by restriction to $N$ followed by pull back, and the assignment $f \mapsto f^{[0]}$ defines an isomorphism 
    \[ C^\infty(M)/\van{N} \to C^\infty_{[0]}(\nu(M,N)).  \]
If $f \in \van{N}$, the differential $\ed f$ vanishes along $TN$, hence defines a fibrewise linear map 
    \[ f^{[1]} \in C^\infty_{[1]}(\nu(M,N)) = \G(\nu(M,N)^*), \quad X_p \mapsto \ed_pf(X_p),\]
and the assignment $f \mapsto f^{[1]}$ defines an isomorphism
    \[ \van{N}/\van{N}^2 \to C^\infty_{[1]}(\nu(M,N)).  \]
This extends uniquely to an isomorphism of graded algebras
    \[ \bigoplus_{k=0}^\infty \van{N}^k/\van{N}^{k+1} \to C^\infty_{pol}(\nu(M,N)), \]
where $C^\infty_{pol}(\nu(M,N))$ denotes the graded algebra of fibrewise polynomial functions on $\nu(M,N)$ (see~\eqref{equation: polynomials on graded bundles}). It can be shown (see, for instance,~\cite{grabowski2009higher}) that the map  
    \begin{align*}
        \nu(M,N) \to \aHom(C^\infty_{pol}(\nu(M,N)), \R), \quad X_p \mapsto \mathrm{ev}_{X_p}
    \end{align*}
is a bijection of sets. Even more, one can naturally define a vector bundle structure on $\aHom(C^\infty_{pol}(\nu(M,N)), \R)$ in such a way that the evaluation map is an isomorphism of vector bundles. From this observation, we deduce that we can recover the normal bundle of $N$ in $M$ from the filtration 
    \[ C^\infty(M) \supseteq \van{N} \supseteq \van{N}^2 \supseteq \cdots .\]
Loizides and Meinrenken repeated this construction with general weightings in~\cite{loizides2023differential}, calling the resulting object the \emph{weighted normal bundle} $\nuw(M,N) \to N$, which is the topic of this section. 

\subsection{Graded Bundles}
\label{subsection: graded bundles}

In contrast to the normal bundle, the weighted normal bundle is not naturally a vector bundle; instead it is a \emph{graded bundle}, in the sense of Grabowski and Rotkiewicz, \cite{grabowski2012graded}. 

\begin{definition}[\cite{grabowski2012graded}]
    A \emph{graded bundle} is a smooth manifold $E$ with a smooth monoid action of $(\R, \cdot)$, 
        \[ \kappa : \R \times E \to E, \quad (t,x) \mapsto \kappa_t(x). \]
    A \emph{morphism of graded bundles} if a smooth map between graded bundles which intertwines the $(\R, \cdot)$-action. 
\end{definition}

Note in particular that $\kappa_0\circ \kappa_0 = \kappa_0$, hence $N = \kappa_0(E)$ is a submanifold of $E$; in fact, Grabowski and Rotkiewicz showed in~\cite{grabowski2012graded} that $E$ is a locally trivial fibre bundle over $N$.

\begin{examples}
    \begin{itemize}
        \item[(a)] A vector bundle $V \to M$ is a graded bundle, with the monoid action of $\R$ given by scalar multiplication. In fact, Grabowski and Rotkiewicz showed in~\cite{grabowski2009higher} that vector bundles are precisely the graded bundle for which the map 
            \begin{equation}
            \label{equation: vecotr bundle regularity}
                V \to TE, \quad v \mapsto \frac{d}{dt}\bigg|_{t=0}\kappa_t(v)
            \end{equation}
        is injective. 

        \item[(b)] A negatively graded vector bundle $E = E_{-r}\oplus E_{-r+1} \oplus \cdots \oplus E_{-1} \to N$ is a graded bundle, where $\R$ acts on the factor $E_{-i}$ as scalar multiplication by $t^i$. If $E \to N$ is a graded bundle, then $\nu(E, N)$ is a graded vector bundle and there is a non-canonical isomorphism 
            \[ E \cong \nu(E, N) \]
        as graded bundles (see~\cite{grabowski2012graded}). Thus, any graded bundle is \emph{non-canonically} the total space of a vector bundle. 

        \item[(c)] A graded Lie groupoid is a Lie groupoid $G\toto M$ with a monoid action of $(\R, \cdot)$ by Lie groupoid morphisms. In the case that the scalar multiplication satisfies~\eqref{equation: vecotr bundle regularity}, then this recovers the notion of a VB-groupoid. On the other hand, if $G\toto M$ is a vector bundle, then this recovers the notion of a graded vector bundle.
    \end{itemize}
\end{examples}

Given a graded bundle $\pi:E \to N$, let $C^\infty_{[n]}(E) = \{f \in C^\infty(E) : \kappa_t^*f = t^nf \}$ and let 
    \begin{equation}
    \label{equation: polynomials on graded bundles}
        C^\infty_{pol}(E) = \bigoplus_{n\geq 0} C^\infty_{[n]}(E)
    \end{equation}
be the space of polynomial functions on $E$; note that $C^\infty_{[0]}(E) = C^\infty(N)$. Let $k_0 = \mathrm{dim}(N)$ and $k_i = \mathrm{rank}(\nu(E, N)_{-i})$. Given an open set $U\sset N$, coordinate systems $x_a$ defined on $\pi^{-1}(U)\sset E$ are called \emph{graded bundle coordinates} if $x_a \in C^\infty_{[i]}(E)$ for $k_{i-1}<a\leq k_i$. A \emph{graded subbundle} is a submanifold $F\sset E$ which is closed under the monoid action of $\R$; graded subbundles are always locally cut out by graded coordinates. 

\begin{example}
    If $E\to N$ is a vector bundle then $C^\infty_{[1]}(E) \cong \G(E^*)$. More generally, we have 
        \[ C^\infty_{pol}(E) \cong \mathrm{Sym}(\G(V^*)) \]
    as graded algebras. Graded coordinate systems for $E$ are vector bundle coordinates systems, consisting of coordinates for the base and linear coordinates on the fibres.
\end{example}

As shown in~\cite{grabowski2017duality}, graded bundles are completely determined by their polynomial functions, in the sense that 
    \begin{equation}
    \label{equation: graded bundles vs polynomial functions}
        E = \aHom(C^\infty_{pol}(E), \R).
    \end{equation}

\subsection{Definition of the weighted normal bundle}
\label{subsection: definition of weighted normal bundle}

In light of~\eqref{equation: graded bundles vs polynomial functions}, one could \emph{define} a graded bundle by declaring what its polynomial functions are (cf.~\cite{sadegh2018euler}). Motivated by the observation that $C^\infty_{pol}(\nu(M,N)) = \bigoplus_{k\geq 0} \van{N}^k/ \van{N}^{k+1}$, let $(M,N)$ be a weighted pair and let $\gr(C^\infty(M))$ be the graded algebra with graded components
    \[ \gr(C^\infty(M))_{(i)} = C^\infty(M)_{(i)}/C^\infty(M)_{(i+1)}. \]

\begin{definition}[\cite{loizides2023differential}]
\label{definition: weighted normal bundle}
    The \emph{weighted normal bundle} of the weighted pair $(M,N)$ is the set 
        \[ \nuw(M,N) = \aHom(\gr(C^\infty(M)), \R).  \]
\end{definition}

Given $f\in C^\infty(M)_{(i)}$, let $f^{[i]}$ denote its class in $\gr(C^\infty(M))_{(i)}$. Evaluation by $f^{[i]}$ determines a map, which we also denote by $f^{[i]}$, 
    \[ f^{[i]} : \nuw(M,N)\to \R \quad \varphi \mapsto \varphi(f^{[i]}) \]
which we refer to as the \emph{$i$-th homogeneous approximation} of $f$. The $C^\infty$-structure of $\nuw(M,N)$ is defined by declaring that $f^{[i]}$ is smooth for any $f \in C^\infty(M)_{(i)}$. 

\begin{theorem}[{\cite[Theorem 4.2]{loizides2023differential}}]
\label{theorem: weighted normal bundle is smooth}
    \begin{itemize}
        \item[(a)] The weighted normal bundle $\nuw(M,N)$ has a unique structure as a graded bundle over $N$, of dimension equal to that of $M$, in such a way that
            \[ C^\infty_{pol}(\nuw(M,N)) = \gr(C^\infty(M)) \]
        as graded algebras.

        \item [(b)] Given weighted coordinates $x_1,\dots , x_n$ on $U \sset M$, the homogeneous approximations $x_a^{[w_a]}$ for $a = 1, \dots ,n$ serve as graded bundle coordinates on $\nuw(M,N)|_{U\cap N} = \nuw(U, U\cap N)$. Moreover, $C^\infty_{pol}(\nuw(U, U\cap N))$ is generated as an algebra over $C^\infty(U\cap N)$ by $x_a^{[w_a]}$ with $w_a \geq 1$. 
    \end{itemize}
\end{theorem}

This construction is functorial for weighted morphisms. That is, any weighted morphism $\varphi:(M,N) \to (M',N')$ induces a morphism of graded bundles 
    \[ \nuw(\varphi) : \nuw(M,N) \to \nuw(M',N'). \]
Note as well that for any $f\in C^\infty(M')_{(i)}$ we have 
    \[ \nuw(\varphi)^*f^{[i]} = (\varphi^*f)^{[i]}.  \]

\begin{examples}
    \begin{itemize}
        \item[(a)] If $M$ is trivially weighted along $N$ then $\nuw(M,N) = \nu(M,N)$, as explained at the beginning of this section. In particular, $\nuw(M,M) = M$ and $\nuw(M, \emptyset) = \emptyset$

        \item[(b)] If $(M,N)$ and $(M',N')$ are weighted pairs then the canonical map 
            \[ \nuw(M\times M', N\times N') \to \nuw(M,N) \times \nuw(M',N') \]
        is an isomorphism of graded bundles, as can be seen by considering local graded bundle coordinates. Moreover if $\varphi:(M, N) \to (R,Q)$ and $\psi:(M',N') \to (R', Q')$ are weighted morphisms, then under this identification $\nuw(\varphi\times \psi) = \nuw(\varphi)\times \nuw(\psi)$. 

        \item[(c)] If $\R^2$ is weighted be declaring that $\mathrm{wt}(x) = 1$ and $\mathrm{wt}(y) = 3$, then $\nuw(\R, \{0\}) = \R^2$. The monoid action of $\R$ on $\R^2$ is given by 
            \[ \kappa_t(x,y) = (tx, t^3y).  \]

        \item[(d)] If $E\to N$ is a graded bundle, then $E$ is canonically weighted along $N$ by 
            \[ C^\infty(E)_{(i)} = \{f \in C^\infty(E) : \kappa_t^*f = O(t^i) \};  \]
        weighted coordinates are then given by graded bundle coordinates. With respect to this weighted, $\nuw(E, N) = E$ canonically. 
    \end{itemize}
\end{examples}

\subsection{Properties of the weighted normal bundle}

\begin{proposition}
\label{proposition: weighted embeddings define embeddings}
    If $\iota: (R, Q)\into (M,N)$ is a weighted embedding then
        \[ \nuw(\iota):\nuw(R, Q) \to \nuw(M,N) \]
    is an embedding. 
\end{proposition}
\begin{proof} 
    By~\autoref{theorem: weighted embedding characterization} we may assume that $R$ is a weighted submanifold of $M$ and $Q = R\cap N$. We will show that any choice of weighted submanifold coordinates for $R$ define submanifold coordinates for the image of $\nuw(R, R\cap N)$ in $\nuw(M,N)$. 
    
    Recall that the weighting on $R$ is defined so that $\iota: R\into M$ defines a surjection $\gr(C^\infty(M)) \to \gr(C^\infty(R))$, hence an injection $\nuw(R, R\cap N) \to \nuw(M, N)$. The image of $\nuw(R, R\cap N)$ in $\nuw(M,N)$ consists of algebra morphisms $\varphi:\gr(C^\infty(M)) \to \R$ with the property that the value of $\varphi(f^{[i]})$ depends only on $(f|_R)^{[i]} \in C^\infty(R)_{(i)}/C^\infty(R)_{(i+1)}$. Therefore, if $x_a$ are weighted submanifold coordinates defined on $U\sset M$ such that 
        \[ R \cap U = \{ x_1 = \cdots = x_r = 0\},  \]
    and $\varphi \in \nuw(U, U\cap N)$ is in the image of $\nuw(R\cap U, R\cap N \cap U)$, then since $(x_a|_R)^{[w_a]} = 0$ for $a=1, \dots, r$ we have $x_a^{[w_a]}(\varphi) = 0$. This shows that 
        \[ \nuw(R\cap U, R\cap N \cap U) \sset \{ x_1^{[w_1]} = \cdots = x_r^{[w_r]} = 0 \}. \]
    For the reverse inclusion, let $\varphi \in \{ x_1^{[w_1]} = \cdots = x_r^{[w_r]} = 0 \}$. Given $f \in C^\infty(U)_{(i)}\cap \van{R\cap U}$, we want to show that $\varphi(f^{[i]}) = 0$. Since $f$ vanishes along $R\cap U$ we can write $f = \sum_{a=1}^r f_ax_a$ with $f_a \in C^\infty(U)_{(i-w_a)}$, hence
        \[ \varphi(f^{[i]}) = \sum_{a=1}^rx_a^{[w_a]}(\varphi)\cdot f_a^{[i-w_a]}(\varphi) = 0, \]
    as needed. 
\end{proof}

\begin{proposition}
\label{proposition: tangent bundle of weighted normal bundle}
    There is a canonical isomorphism  
        \[ T\nuw(M,N)|_N = \gr(TM|_N).  \]
    If $\varphi:(M,N) \to (M', N')$ is a weighted morphism, then with respect to this identification one has 
        \[ T\nuw(\varphi)|_N = \gr(T\varphi|_N) \]
\end{proposition}
\begin{proof}
    Using that $\nuw(M,N)$ is a graded bundle and~\cite[Proposition 4.4]{loizides2023differential}, we have that 
        \begin{align*}
            T\nuw(M,N)|_N = TN \oplus \gr(\nu(M,N)). 
        \end{align*}
    The filtration of $\nu(M,N)$ is given by $\nu(M,N)_{(i)} = (TM|_N)_{(i)}/TN$, so that
        \begin{align*}
            T\nuw(M,N)|_N & = TN \oplus \bigoplus_{i=1}^r \nuw(M,N)_{(-i)}/ \nuw(M,N)_{(-i+1)} \\
            & = TN \oplus \bigoplus_{i=1}^r(TM|_N)_{(-i)}/(TM|_N)_{(-i+1)} \\
            & = \gr(TM|_N),
        \end{align*}
    as claimed. 
\end{proof}

\begin{corollary}
\label{corollary: weighted submersions induce submersions}
    If $F:M\to M'$ is a weighted submersion between weighted pairs $(M,N)$ and $(M',N')$, then $\nuw(F)$ is a submersion. 
\end{corollary}
\begin{proof}
    By~\autoref{proposition: tangent bundle of weighted normal bundle}, $\nuw(F)$ is a submersion along $N \sset \nuw(M,N)$ hence near $N$. Since $\nuw(F)$ is a morphism of graded bundles it follows that it is a submersion everywhere. 
\end{proof}

\begin{proposition}
\label{proposition: weighter normal functor and fibre products}
    Suppose that $F: M\to M''$ and $G:M' \to M''$ are weighted transverse. Then 
        \begin{itemize}
            \item[(a)] $\nuw(F): \nuw(M,N) \to \nuw(M'', N'')$ and $\nuw(G):\nuw(M',N') \to \nuw(M'', N'')$ are transverse, and 
            \item[(b)] the canonical map
                \[ \nuw(M\times_{M''} M', N\times_{N''} N')
                \to \nuw(M, N) \times_{\nuw(M'',N'')} \nuw(M', N') \]
            is an isomorphism of graded bundles.
        \end{itemize}
\end{proposition}

Here the graded bundle structure on the fibre product is as a graded subbundle of $\nuw(M, N) \times \nuw(M', N') \to N\times N'$,

\begin{proof}
    \begin{itemize}
        \item[(a)] By modifying the standard argument appropriately, one finds that $F$ and $G$ are weighted transverse if and only if the graph $\G(F\times G)$ of $F\times G$ is weighted transverse to the diagonal $\Delta_{M''\times M'' \times M \times M'} = \{ (m'',m'',m,m')\} \sset M''\times M'' \times M \times M'$. It follows by considering weighted coordinates adapted to this transverse intersection that $\G(\nuw(F\times G)) = \G(\nuw(F)\times \nuw(G))$ is transverse to the diagonal in $\nuw(M''\times M'' \times M \times M', N''\times N'' \times N \times N')$, which establishes the claim. 

        \item[(b)] This follows by considering graded bundle coordinates, since it is a graded bundle morphism which is the identity along the base. \qedhere
    \end{itemize}
\end{proof}


\section{The Weighted Deformation Space}
\label{section: weighted deformation space}

Motivated by the construction of Haj and Higson~\cite{sadegh2018euler}, Loizides and Meinrenken show that a weighted manifold $M$ can be "deformed" to its weighted normal bundle in a smooth way. The construction is known in algebraic geometry as the deformation to the normal cone, which we now review. 

\subsection{Construction of the Weighted Deformation Space}

Let $(M,N)$ be a weighted pair, and let 
    \[ \rees(C^\infty(M)) = \left\{ \sum_{i\in \Z} f_iz^{-i} : f \in C^\infty(M)_{(i)} \right\} \sset C^\infty(M)[z^{-1}, z] \]
denote the Rees algebra associated to the weighting of $M$. 

\begin{definition}[\cite{loizides2023differential}]
\label{defintion: weighted deformation space}
    The \emph{weighted deformation space} of the weighted pair $(M,N)$ is the set 
        \[ \defw(M,N) = \aHom(\rees(C^\infty(M)), \R). \]
\end{definition}

We now explain how to endow $\defw(M,N)$ with a smooth structure. Given $f\in C^\infty(M)_{(i)}$, its \emph{$i$-th homogeneous interpolation} is the function 
    \[ \widetilde{f}^{[i]}:\defw(M,N)\to \R, \quad \varphi \mapsto \varphi(fz^{-i});\]
note that if $f \in C^\infty(M)_{(i)}$ and $g\in C^\infty(M)_{(j)}$ then $\widetilde{fg}^{[i+j]} = \widetilde{f}^{[i]}\widetilde{g}^{[j]}$. The map $\pi = \widetilde{1}^{[-1]} : \defw(M,N) \to \R$ is a surjection with fibres 
    \begin{equation}
    \label{equation: decomposition of deformation space}
        \defw(M,N)_t = \pi^{-1}(\{t\}) = \left\{
        \begin{array}{ll}
            M & t\neq 0  \\
            \nuw(M,N) & t = 0, 
        \end{array}
    \right.
    \end{equation}
see~\cite{sadegh2018euler, loizides2023differential}. Given $f\in C^\infty(M)_{(i)}$ one has 
    \begin{align}
    \begin{split}
    \label{equation: deformation approximation}
        \widetilde{f}^{[i]}|_{\pi^{-1}(\{t\})} : \delta_\W(M,N)_t  \to \R, \quad 
        x \mapsto \left\{
            \begin{array}{ll}
                 t^{-i}f(x) & t \neq 0 \\
                 f^{[i]}(x) & t = 0.
            \end{array}
        \right.
    \end{split}
    \end{align}
    
Loizides and Meinrenken prove the following. 

\begin{theorem}[{\cite[Theorem 5.1]{loizides2023differential}}]
    \begin{itemize}
        \item[(a)] There is a unique $C^\infty$-structure on $\delta_\W(M,N)$, as a manifold of dimension $\dim(M)+1$, such that the maps~\eqref{equation: deformation approximation} are smooth for each $f\in C^\infty(M)_{(i)}$. In terms of this structure, the map $\pi:\delta_\W(M,N) \to \R$ is a surjective submersion and the identifications~\eqref{equation: decomposition of deformation space} are diffeomorphisms.

        \item[(b)] If $x_a$ is a weighted coordinate system on $U\sset M$, then the homogeneous interpolations $\widetilde{x}_{a}^{[w_a]}$, together with the coordinate $t = \pi:\defw(M,N) \to \R$ are coordinates on $\defw(U, U\cap N) = \nuw(U, U\cap N)\sqcup (U\times \R^\times)$.
    \end{itemize}
\end{theorem}

\begin{examples}
\label{examples: weighted deformation spaces}
    \begin{enumerate}
        \item[(a)] If $M$ is trivially weighted along $N$, then $\delta_\W(M,N)$ is the standard deformation to the normal cone. In particular, $\delta_\W(M\times M, M)$ is the tangent groupoid $\mathbb{T}M$ of Connes (cf.~\cite{connes1994, sadegh2018euler}).

        \item[(b)] If $M$ is weighted along itself then $\defw(M,M) = M\times \R$. 
        
        \item[(c)] If $E$ is a graded bundle over $N$, then $\delta_W(E,N)$ is canonically isomorphic to $E\times \R$. The isomorphism is given by the family of maps 
            \begin{align*}
                \delta_\W(E,N)_t & \to E \\
                e & \mapsto \left\{
                    \begin{array}{ll}
                        \kappa_{t^{-1}}e & t \neq 0  \\
                        e & t = 0,
                    \end{array}
                \right.
            \end{align*}
        where on the $t=0$ we are identifying $\nu_\W(E,N)$ with $E$. 

        \item[(d)] If $(M, N)$ and $(M', N')$ are weighted pairs, then $\defw(M\times M', N\times N') = \defw(M,N)\times_{\R} \defw(M', N')$. 
    \end{enumerate}
\end{examples}

\subsection{Properties of the Weighted Deformation space}
\label{subsection: properties of weighted deformation space}

We now list some properties of the weighted normal bundle. Most of the proofs are similar as for the weighted normal bundle, so we will be brief. 

\begin{itemize}
    \item[(a)] The weighted deformation space is functorial for weighted morphisms. That is, any weighted morphism $F : (M,N)\to (M', N')$ induces a smooth map $\defw(F):\defw(M,N) \to \defw(M', N')$. In terms of the identifications~\eqref{equation: decomposition of deformation space}, one has 
        \[ \defw(F)|_{\pi^{-1}(\{t\})} : x \mapsto \left\{
            \begin{array}{ll}
                F(x) & t\neq 0 \\
                 \nuw(F)(x) &  t = 0. 
            \end{array}
        \right.\]
    Moreover, given $f\in C^\infty(M')_{(i)}$ we have that $\defw(F)^*\widetilde{f}^{[i]} = \widetilde{F^*f}^{[i]}$. 

    \item[(b)] If $R$ is a weighted submanifold of the weighted pair $(M,N)$, then $\defw(R, R\cap N)$ is a submanifold of $\defw(M,N)$. In particular, $\defw(N,N) = N\times \R$ is a submanifold of $\defw(M,N)$. As with the weighted normal bundle, if $x_a$ are weighted submanifold coordinates for $R$ on $U\sset M$, then $\widetilde{x}_a^{[w_a]}$ together with $t$, are submanifold coordinates for $\defw(R, R\cap N)$ on $\defw(U, U\cap N)$. 
    
    \item[(c)] If $F:(M,N) \to (M'', N'')$ and $G:(M', N') \to (M'', N'')$ are weighted transverse, then $\defw(F) : \defw(M,N) \to \defw(M'', N'')$ and $\defw(G) : \defw(M',N') \to \defw(M'', N'')$ are transverse and the canonical map 
        \[ \defw(M\times_{M''} M', N\times_{N''}N') \to \defw(M, N)\times_{\defw(M'', N'')} \defw(M', N') \]
    is a diffeomorphism. 

    \item[(d)] (\cite[Section 5.2]{loizides2023differential}) Given $u\in \R^\times$, consider the algebra morphism of $\rees(C^\infty(M))$
        \[ \kappa_u : \sum_i f_iz^{-i} \mapsto  \sum_i f_iu^iz^{-i}. \]
    This defines a smooth action of the group $\R^\times$ on $\defw(M,N)$ called the \emph{zoom action}. Given $f \in C^\infty(M)_{(i)}$, the function $\widetilde{f}^{[i]}$ is homogeneous of degree $i$; in particular, $\pi = \widetilde{1}^{[-1]}$ is homogeneous of degree -1. Thus, 
        \[ \kappa_u : \defw(M,N)_{t} \to \defw(M,N)_{u^{-1}t}. \]
    On the open set $M\times \R^\times = \pi^{-1}(\R^\times) \sset \defw(M,N)$ the action is given by 
        \[ \kappa_u : (m, t) \mapsto (m, u^{-1}t),  \]
    whereas on $\nuw(M,N) = \defw(M,N)_{0}$ the action is the graded bundle multiplication. If $F:(M,N) \to (M', N')$ is a weighted morphism, then $\defw(F) : \defw(M,N) \to \defw(M', N')$ is $\R^\times$-equivariant. 
\end{itemize}

We have, furthermore, the following theorem which justifies our definitions of weighted immersions and weighted submersions. 

\begin{theorem}[Characterization of Weighted Immersions and Submersions]
\label{theorem: Characterization of Weighted Immersions and Submersions}
    A weighted morphism $F:(M,N) \to (M', N')$ is a weighted immersion (resp. submersion) if and only if the induced map 
        \[ \defw(F):\defw(M,N) \to \defw(M',N') \]
    is an immersion (resp. submersion). 
\end{theorem}
\begin{proof}
    We prove the statement for weighted immersions, the proof for weighted submersions being entirely analogous. 
    
    If $F$ is a weighted immersion, then $\defw(F)$ is an immersion as can be seen by considering weighted immersion coordinates. On the other hand, if $\defw(F)$ is an immersion, then clearly $F$ is a immersion. To see that the map $(TM|_N)_{(i)}\to (TM'|_{N'})_{(i)}$ is fibrewise injective, we remark that 
        \[ T\defw(M,N)|_{N\times \{0\}} = \gr(TM|_N)\times \R\]
    and with respect to this 
            \begin{equation*}
                T_p\defw(F) = \begin{bmatrix}
                T_p\nuw(F) & \bd_t\widetilde{(F^*y)}^{[w']}(p) \\
                0 & 1
            \end{bmatrix}
        \end{equation*}
    for any $p \in N \sset \defw(M,N)|_{\pi^{-1}(\{0\})}$, where 
        \[ \bd_t\widetilde{(F^*y)}^{[w']}(p) = [ \bd_t\widetilde{(F^*y_1)}^{[w'_1]}(p), \dots, \bd_t\widetilde{(F^*y_m)}^{[w'_m]}(p)]^\intercal. \]
    In particular, $T_p\defw(F)$ is injective only if $T_p\nuw(F)$ is, which happens if only if the map $(TM|_N)_{(i)}\to (TM'|_{N'})_{(i)}$ is fibrewise injective. It remains to show that $TF(TM)\cap (TM'|_{N'})_{(i)}\sset TF((TM|_N)_{(i)})$ (the other inclusion is automatic since $F$ is a weighted morphism). Let $p\in N$ and $X_{F(p)} \in TF(T_pM)\cap (T_{F(p)}M')_{(i)}$. Since $F$ is an immersion, there exists a unique $Y_p\in T_pM$ such that $TF(Y_p) = X_{F(p)}$, and furthermore, because $T\nuw(F)|_N = \gr(TF|_N)$ is injective, we see that $Y_p\in (T_pM)_{(i)}$,  as needed. 
\end{proof}

\section{Singular Lie Filtrations}
\label{section: singular Lie filtrations}

In this section we review \emph{singular Lie filtrations}, which are the main technique for constructing weightings. Singular Lie filtrations are filtrations of the sheaf of vector fields on a smooth manifold by locally finitely generated $C^\infty_M$-submodules which are compatible with Lie brackets (see~\autoref{definition: singular Lie filtration}), generalization both Lie filtrations and singular foliations in the sense of Androulidakis and Skandalis~\cite{androulidakis2009holonomy}. They recently made an appearance in the work of Androulidakis, Mohsen, and Yuncken on maximally hypoelliptic operators~\cite{androulidakis2022pseudodifferential}.

Every weighted manifold has a canonical singular Lie filtration associated to the weighting. Conversely, Loizides and Meinrenken showed in~\cite{loizides2022singular} that a singular Lie filtration on a manifold $M$ together with a submanifold $N$ satisfying a cleanness assumption canonically defines a weighting of $M$ along $N$, generalizing the work of Haj and Higson~\cite{sadegh2018euler}. This has the practical application that, in examples, it is more natural to define a singular Lie filtration than a weighting. In this section, we review their work. 

\subsection{Filtration of vector fields}

Let $M$ have an order $r$ weighting along $N$. Define $\mathfrak{X}(M)_{(i)}$ to be the collection of vector fields $X\in \mathfrak{X}(M)$ whose Lie derivative shifts filtration degree by $i$; that is, 
    \[ f \in C^\infty(M)_{(j)} \implies Xf \in C^\infty(M)_{(i+j)}.  \]
This defines a filtration 
    \begin{equation}
    \label{equation: filtration of vector fields}
        \mathfrak{X}(M) = \mathfrak{X}(M)_{(-r)} \supseteq \mathfrak{X}(M)_{(-r+1)} \supseteq \cdots 
    \end{equation}
which is compatible with Lie brackets in the sense that if $X\in \mathfrak{X}(M)_{(i)}$ and $Y \in \mathfrak{X}(M)_{(j)}$ then  $[X,Y] \in \mathfrak{X}(M)_{(i+j)}$. If $x_a$ is a system of weighted coordinates on $U$, then $X\in \mathfrak{X}(U)_{(i)}$ if and only if 
    \[ X = \sum_a f_a\frac{\bd}{\bd x_a}, \]
with $f \in C^\infty(U)_{(i-w_a)}$. In particular, for any $p \in N$ we have 
    \[ T_pM_{(i)} = \mathrm{span}\{ X_p : X \in \mathfrak{X}(M)_{(i)}\}. \]
More generally, the weighting of $M$ determines a filtration 
    \[ \cdots \supseteq \DO(M)_{(q)} \supseteq \DO(M)_{(q+1)} \supseteq \cdots \]
of the graded algebra of differential operators on $M$, where $\DO(M)_{(q)}$ is the operators $D \in \DO(M)$ with the property that
    \[ f\in C^\infty(M)_{(i)} \implies Df\in C^\infty(M)_{(i+q)}.  \]
One has that $D\in \mathrm{DO}(M)_{(q)}$ if locally $D$ is given by a linear combination of terms of the form 
    \[ X_1\cdots X_k, \]
where $X_j \in \mathfrak{X}(U)_{(i_j)}$ with $i_1 + \cdots + i_k = q$; in particular, we note that the filtration of $\DO^k(M)_{(q)}$ starts in degree $-kr$
    \[ \DO^k(M) = \DO^k(M)_{(-kr)} \supseteq \DO^k(M)_{(-kr+1)} \supseteq \cdots  \]
where $r$ is the order of the weighting. With respect to the filtration of differential operators we have 
    \begin{equation}
    \label{equation: recover weighting from differential operators}
        C^\infty(M)_{(i)} = \{f \in C^\infty(M) : -q < i,\ D \in \mathrm{DO}(M)_{(q)} \implies Df|_N = 0\}.
    \end{equation}

\subsection{Singular Lie filtrations}

\autoref{equation: recover weighting from differential operators} implies, in particular, that we can recover the weighting of $M$ from the filtration of $\mathfrak{X}(M)$ and the submanifold $N$. It is reasonable, therefore, to ask when a filtration of $\mathfrak{X}(M)$ defines a weighting of $M$. An answer to this question was given by Loizides and Meinrenken in~\cite{loizides2022singular}, which we now summarize. 

\begin{definition}[\cite{loizides2022singular}]
\label{definition: singular Lie filtration}
    Let $M$ be a smooth manifold.
    \begin{enumerate*}
        \item A \emph{singular distribution} on $M$ is a sheaf $\mathcal{D}$ of $C^\infty_M$-submodules of $\mathfrak{X}_M$ which is locally finitely generated. That is, each point $p\in M$ is contained in an open neighbourhood $U$ such that $\mathcal{D}(U)$ is a finitely generated $C^\infty(U)$-module. 

        \item A \emph{singular foliation} is a singular distribution $\mathcal{F}$ on $M$ which is involutive: 
            \[ [ \mathcal{F}, \mathcal{F} ] \sset \mathcal{F}. \]

        \item A \emph{singular Lie filtration} is a filtration 
            \[ \mathfrak{X}_M = \mathcal{F}_{-r} \supseteq \mathcal{F}_{-r+1} \supseteq \cdots \supseteq \mathcal{F}_{0} \]
        of $\mathfrak{X}_M$ by singular distributions $\mathcal{F}_i$ with the property that 
            \[ [ \mathcal{F}_i, \mathcal{F}_j] \sset \mathcal{F}_{i+j}. \]
    \end{enumerate*}
\end{definition}

\begin{examples}
    \begin{itemize}
        \item[(a)] If $(M, N)$ is a weighted pair then $\mathfrak{X}(M)_{(0)}$ is a singular foliation and the filtration of vector fields~\eqref{equation: filtration of vector fields} is a singular Lie filtration. 

        \item[(b)] A filtered manifold is a manifold $M$ together with a filtration of $TM$ by subbundles 
            \[ TM = F_{-r} \supseteq F_{-r+1} \supseteq \cdots \supseteq F_{-1} \supseteq 0 \]
        satisfying $[\G(F_{i}), \G(F)_{j}] \sset \G(F_{i+j})$; this called a (regular) \emph{Lie filtration}. Filtered manifolds appear in CR geometry, contact geometry, and parabolic geometry (see~\cite{choi2015privileged, sadegh2018euler, van2019groupoid,yuncken2018pseudodifferential} and the references therein). 
    \end{itemize}
\end{examples}

\begin{definition}[\cite{loizides2022singular}]
\label{defintion: clean submanfiolds}
Let $M$ be a smooth manifold. 
     \begin{enumerate}
         \item[(a)] Let $\mathcal{D}$ be a singular distribution on $M$. A (closed) submanifold $N\sset M$ is called \emph{$\mathcal{D}$-clean} if the function 
            \begin{align*}
                N  \to \N, \quad  p  \mapsto \dim(\mathcal{D}_p + T_pN)
            \end{align*}
        is a constant. 

        \item[(b)] Let $\mathcal{F}_\bullet$ be a singular Lie filtration on $M$. A (closed) submanifold $N\sset M$ is called \emph{$\mathcal{F}_\bullet$-clean} if it is $\mathcal{F}_i$-clean for each $i$.  
     \end{enumerate}
\end{definition}

\begin{examples}
    \begin{itemize}
        \item[(a)] If $(M,N)$ is a weighted pair and $\mathfrak{X}_{(\bullet)}$ is the corresponding singular Lie filtration, then $N$ is $\mathfrak{X}(M)_{(\bullet)}$-clean. 

        \item[(b)] If $\mathcal{H}_{\bullet}$ is a singular Lie filtration on $M$, then $\mathcal{F}_\bullet\times  \mathcal{F}_\bullet$ is a singular Lie filtration on $M\times M$. The diagonal $\Delta_M \sset M\times M$ is $\mathcal{F}_\bullet\times \mathcal{F}_\bullet$-clean if and only if the $\mathcal{F}_\bullet$ is a regular Lie filtration.

        \item[(c)] Let $\mathcal{D}$ be the submodule of $\mathfrak{X}(\R^2)$ generated by the vector fields $x\frac{\bd}{\bd x}$ and $\frac{\bd}{ \bd y}$. Then the $x$-axis is $\mathcal{D}$-clean, whereas the $y$-axis is not.
    \end{itemize}    
\end{examples}

The theorem relating singular Lie filtrations with weightings is the following. 

\begin{theorem}[{\cite[Theorem 4.1]{loizides2022singular}}]
\label{theorem: singular Lie filtrations and weightings}
    Let $\mathcal{F}_\bullet$ be a singular Lie filtration on a manifold $M$, and let $N\sset M$ be an $\mathcal{F}_\bullet$-clean, closed submanifold. The filtration defined by $C^\infty_{M,(1)} = \van{N}$ and 
        \[ C^\infty(U)_{(i)} = \{ f\in C^\infty(U) : X \in \mathcal{F}_{j}(U),\ 0<j<i \implies Xf \in C^\infty(U)_{(i-j)} \} \]
    for $i>1$ defines a weighting of $M$ along $N$ such that 
        \[ (TM|_N)_{(i)} = \mathcal{F}_i|_N + TN. \]
\end{theorem}

Before moving on, we record the following simple proposition about weighted morphisms, the proof of which is obvious from the construction of the weighting in~\autoref{theorem: singular Lie filtrations and weightings}. 

\begin{proposition}
\label{proposition: weighted morphisms of weighted lie filtrations}
    Let $\mathcal{F}_\bullet$ be a singular Lie filtration on  $M$ and let $N\sset M$ be an $\mathcal{F}_\bullet$-clean submanifold. Suppose that $(M', N')$ is a weighted pair, and that $F:(M, N)\to (M', N')$ is a map of pairs with the property that for all $X \in \mathcal{F}(U)_{(i)}$ there exists some $Y\in \ger{X}(M')_{(i)}$ such that $X\sim_F Y$. Then $F$ is a weighted morphism with respect to the weighting of $M$ defined in~\autoref{theorem: singular Lie filtrations and weightings}. 
\end{proposition}

\section{Inner Automorphisms of Weighted Manifolds}

Let $(M,N)$ be a weighted pair. We close this chapter by showing that vector fields $X\in \mathfrak{X}(M)_{(0)}$ are exactly those whose flow is an automorphism of the weighted manifold. We begin with a lemma.

\begin{lemma}[{\cite[Propositions 1.6, 3.1, 1.6.9]{androulidakis2009holonomy, garmendia2019inner, laurent-gengoux2022singular}}]
\label{lemma: infinitesimal automorphisms are automorphisms}
    Let $\mathcal{D}$ be a singular distribution on $M$ and let $X\in \ger{X}(M)$ be a vector field whose flow $\phi_t$ is defined for $|t|<\epsilon$. If $[X, \mathcal{D}] \sset \mathcal{D}$, then 
        \[ \phi_t^*(\mathcal{D}) \sset \mathcal{D}. \]
\end{lemma}

\begin{proposition}
Let $(M,N)$ be a weighted pair, and let $X \in \ger{X}(M)$ be a vector field whose time flow $\phi_t$ is defined for $|t|<\epsilon$. Then the following are equivalent:
    \begin{itemize}
        \item[(a)] $X\in \ger{X}(M)_{(0)}$, 

        \item[(b)] $[X, \ger{X}(M)_{(i)}] \sset \ger{X}(M)_{(i)}$ for all $i$,
        
        \item[(c)] $\phi_t$ is a weighted diffeomorphism for all $|t|<\epsilon$.
    \end{itemize}
\end{proposition}
\begin{proof}
    Recall that the filtration of $\ger{X}(M)$ is compatible with brackets in the sense that if $Y\in \mathfrak{X}(M)_{(i)}$ and $Z \in \mathfrak{X}(M)_{(j)}$ then  $[Y,Z] \in \mathfrak{X}(M)_{(i+j)}$. Thus, (a) implies (b). 

    Now suppose that (b) holds. By applying~\autoref{lemma: infinitesimal automorphisms are automorphisms} to the singular distribution $\mathcal{D} = \ger{X}(M)_{(i)}$, we have that 
        \[ \phi_t^*(\ger{X}(M)_{(i)}) \sset \ger{X}(M)_{(i)}  \]
    for all $|t|<\epsilon$. By~\autoref{proposition: weighted morphisms of weighted lie filtrations}, this implies that $\phi_t$ is a weighted diffeomorphism. 
    
    Finally, suppose that (c) holds and let $f\in C^\infty(M)_{(i)}$. Since
        \[ \phi_t^*Xf = \frac{d}{dt}\phi_t^*f \in C^\infty(M)_{(i)}, \]
    it follows that $Xf \in C^\infty(M)_{(i)}$, hence $X\in \ger{X}(M)_{(0)}$.  
\end{proof}

%% file: 3_linear_weightings.tex
\chapter{Linear weightings}
\label{chapter: linear weightings}

\section{Weighted Vector Bundles}
\label{section: weighted vector bundles}

We now turn our attention to weightings for vector bundles. In~\cite{loizides2023differential}, Loizides and Meinrenken define a \emph{linear weighting} of a vector bundle $V$ to be one for which scalar multiplication is a weighted morphism. We found this definition to be too restrictive in the sense the it excludes some basic examples. For instance, there is no natural weighting of the cotangent bundle of a weighted pair $(M,N)$ even though the tangent bundle is canonically weighted. Moreover, since vector bundles can be defined in terms of their sheaf of sections, it is desirable to have a definition of linear weighting in these terms.

We define linear weightings in terms of a $\Z$-graded filtration 
    \begin{equation*}
        \cdots \supseteq \G_{V, (i)} \supseteq \G_{V, (i+1)} \supseteq \cdots
    \end{equation*}
filtration of the sheaf of sections of a vector bundle, satisfying a local condition modelled after~\autoref{definition: weighting}. To ensure our definition contains naturally arising examples, we allow for the filtration of the sections of $V$ to be non-trivial in both positive and negative degree. If $V\to M$ is a linearly weighted vector bundle, then $M$ is a weighted manifold and the constructions of the weighted normal bundle and the weighted deformation space still go through yielding vector bundles over the weighted normal bundle and weighted deformation space of $M$, respectively. 

We let $\G(V)$ denote the (smooth) sections of $V$. If $W\to N$ is a subbundle of $V\to M$ then $\G(V, W) = \{ \sigma \in \G(V) : \sigma|_N \in W \}$.

\subsection{Definition of linear weightings}

A linear weighting of a vector space $V$ is a filtration of $V$ by subspaces. Therefore, we define a weighted vector bundle to be one which is locally the product of a weighted manifold with a filtered vector space. This is formalized as follows. 

Let $(M,N)$ be a weighted pair and $M \times \R^k\to M$ the trivial bundle of rank $k$ over $M$. A \emph{vertical weight vector} for $M \times \R^k$ is a $k$-tuple of integers $(v_1, \dots, v_k) \in \Z^k$. A choice of vertical weight vector determines a filtration of $\G(M \times \R^k)$ by $C^\infty(M)$-submodules
    \begin{equation}
    \label{equation: local model for linear weighting}
        \G(M\times \R^k)_{(i)} = \sum_{a=1}^k C^\infty(M)_{(i-v_a)}\sigma_a, 
    \end{equation}
where $\sigma_1, \dots, \sigma_k$ is the standard basis for $\R^k$.  
	
\begin{definition}
\label{definition: linear weighting}
    A \emph{linear weighting} of a rank $k$ vector bundle $V$ over the weighted pair $(M,N)$ is a $\Z$-graded filtration 
        \[ \cdots \supseteq \G_{V, (i)} \supseteq \G_{V,(i+1)} \supseteq \cdots \]
    of the sheaf of sections $\G_V$ by $C^\infty_M$-submodules such that for every point $p\in M$ there exists an open neighbourhood $U\sset M$ containing $p$ and a frame $\sigma_1, \dots, \sigma_k \in \G(V|_U)$ such that $\G(V|_U)_{(i)}$ is given by~\eqref{equation: local model for linear weighting}. The frame $\sigma_a$ is called a \emph{weighted frame}. We refer to a vector bundle with a linear weighting as a \emph{weighted vector bundle}.
\end{definition}

\begin{remarks}
    \begin{enumerate}
        \item[(a)] If $v_1, \dots, v_k$ is the vertical weight sequence for $V$, then  $\G_{V, (i)} = \G_V$ for $i \leq \min_a\{v_a\}$.

        \item[(b)]  By definition, $\G_V$ is a filtered module over the filtered algebra $C^\infty_M$. That is, for all $i, j \in \Z$,
            \[ C^\infty_{M,(i)}\cdot \G_{V, (j)} \sset \G_{V, (i+j)}.\]

        \item[(c)] Using that we are working in the $C^\infty$-category, we may take advantage of the existence of partitions of unity to avoid the use of sheaves, working instead with the filtration of global sections  
            \[ \cdots \supseteq \G(V)_{(i)} \supseteq \G(V)_{(i+1)} \supseteq \cdots \]
    \end{enumerate}
\end{remarks}

\begin{examples}
\label{examples: examples of linear weightings}
    \begin{itemize}
        \item[(a)] If $M$ is weighted along itself, then a filtration of $V$
            \[ \cdots \supseteq V_{(i)} \supseteq V_{(i+1)} \supseteq \cdots  \]
        by subbundles $V_i \to M$ via defines a linear weighting of $V$ via correspondence 
            \[ \G(V)_{(i)} = \G(V_{(i)}).  \]
            
        \item[(b)] If $M$ is a weighted manifold and $V$ is a weighted vector space, then the trivial bundle $M\times V$ is a weighted vector bundle. 
        
        \item[(c)] Let $(M,N)$ be a weighted pair and let 
            \begin{equation}
            \label{equation: TM linear weighting}
                \cdots \supseteq \mathfrak{X}(M)_{(i)} \supseteq \mathfrak{X}(M)_{(i+1)} \supseteq \cdots
            \end{equation}
         be the induced filtration of $\ger{X}(M)$ (see~\autoref{section: singular Lie filtrations}). If $x_a$ is a local weighted coordinate system on $M$, then the local coordinate vector fields $\frac{\bd}{\bd x_a}$ define a local weighted frame for $TM$. Thus,~\eqref{equation: TM linear weighting} is a linear weighting of $TM$. If $w_1, \dots, w_n$ are the weights for $M$, then the vertical weights for $TM$ are given by $-w_1, -w_2, \dots, -w_n$.  

         \item[(d)] With the same setting as in (c), let $\Omega^1(M)_{(i)}$ denote the 1-forms $\omega \in \Omega^1(M)$ on $M$ with the property that 
            \[ X \in \ger{X}(M)_{(j)} \implies \la \omega, X \ra \in C^\infty(M)_{(i+j)}. \]
        If $x_a$ is a weighted coordinate system on $M$, then the differentials $\ed x_a$ form a weighted frame for $T^*M$. Thus, $T^*M$ is a weighted vector bundle. If $w_1, \dots, w_n$ are the weights for $M$, then the vertical weights for $T^*M$ are given by $w_1, \dots, w_n$ as well. Note that the map $\ed:C^\infty(M)\to \Omega^1(M)$ is filtration preserving.

         \item[(e)] Let $N\sset M$ be a closed submanifold with vanishing ideal $\van{N}$. The trivial weighting of $M$ along $N$ is given by order of vanishing, 
                \[ C^\infty(M)_{(j)} = \van{N}^j.  \]    
         If $W\to N$ is a subbundle of $V\to M$, then the filtration 
            \[ \G(V)_{(i)} = \left\{
                \begin{array}{ll}
                     \G(V) & i \leq -1,  \\
                    \G(V, W) & i = 0, \\
                    \van{N}^i\cdot \G(V) & i \geq 1, 
                \end{array}
            \right.\]
        determines a linear weighting of $V$. If $M$ is trivially weighted along $N$, this recovers the tangent weighting of $TM$ from the previous example for $W = TN$
    \end{itemize}
\end{examples}

Recall from~\autoref{section: basics of weightings} that a weighting of $M$ along $N$ determines a filtration of $TM|_N$ be subbundles. The analogue of this in the context of linear weightings is the following. 

\begin{proposition}
\label{proposition: linear weighting determines filtration by subbundles}
    Let $(M,N)$ be a weighted pair. A linear weighting of $V\to M$ with vertical weights $v_1, \dots, v_k$ determines a filtration of $V|_N$  
        \[ \cdots \supseteq (V|_N)_{(i)} \supseteq (V|_N)_{(i+1)} \supseteq \cdots  \]
    by subbundles $(V|_N)_i \to N$ of rank $k_i = \#\{ a : v_a \geq i\}$. In fact, $\G((V|N)_{(i)})$ is given by the image of $\G(V)_{(i)}$ in 
        \[\G(V)/(\van{N}\cdot \G(V)) = \G(V|_N) \]
    under the quotient map. 
\end{proposition}
\begin{proof}
    Let $p\in N$ be contained in the open neighbourhood $U\sset M$, and let $\sigma_a$ be a weighted frame for $V|_U$. Then the image of $\G(V|_U)_{(i)}$ in the quotient $\G(V|_U)/(\van{N\cap U}\cdot \G(V|_U))$ is freely generated by 
        \[ \{ \sigma_a|_{N\cap U} : v_a \geq i \}. \qedhere \]
\end{proof}

\begin{remark}
\label{remark: wide weighting and filtered vector bundles}
    This shows that if $M$ is weighted along itself then a linear weighting of $V\to M$ is \emph{equivalent} to a filtration of $V$ by subbundles, as in~\autoref{examples: examples of linear weightings} (a). In particular, a \emph{weighted vector space} is equivalent to a filtered vector space. 
\end{remark}

\begin{examples}
    \begin{itemize}
        \item[(a)] If $(M,N)$ is a weighted pair and $TM$ and $T^*M$ are weighted as in~\autoref{examples: examples of linear weightings} (c) and (d), then the filtrations of $TM|_N$ and $T^*M|_N$ are the ones defined in~\autoref{section: basics of weightings}. 

        \item[(b)] If $V\to M$ is trivially weighted along $W\to N$ as in~\autoref{examples: examples of linear weightings} (e), then the filtration of $V|_N$ is 
            \[ V|_N \supseteq W \supseteq 0. \]
        
        \item[(c)] Let $\R$ be given the trivial weighting along the origin, so that the coordinate $x$ has weight $1$. Let $V = \R \times \R^2 \corr{\mathrm{pr}_1} \R$ be the trivial bundle of rank 2 over $\R$, and let $\sigma_1, \sigma_2 \in \G(V)$ be standard frame. The linear weightings
	       \begin{align*}
                \G(V)_{(i)} & = C^\infty(\R)_{(i)}\cdot  \sigma_1 + C^\infty(\R)_{(i+2)}\cdot \sigma_2 \\
                \G(V)'_{(i)} & = C^\infty(\R)_{(i)}\cdot (\sigma_1 + x\sigma_2) + C^\infty(\R)_{(i+2)}\cdot \sigma_2 
	       \end{align*}
        determine the same filtration of $V_{0} = \{0\}\times \R^2$, but they are different weightings. Indeed, $\sigma_1+x\sigma_2 \in \G(V)'_{(0)}$, but $\sigma_1+x\sigma_2 \notin\G(V)_{(0)}$. In particular, this example shows that the filtration of $V|_N$ alone is not enough to recover the weighting of $V$.
    \end{itemize}
\end{examples}

\subsection{Constructions}
\label{subsection: linear constructions}

We now explain how various constructions with vector bundles work in the weighted setting.

\begin{itemize}
    \item[(a)] (\emph{Dual weighting)} The dual of a weighted vector bundle is linearly weighted by 
        \[ \tau \in \G(V^*)_{(i)} \iff \forall \sigma\in \G(V)_{(j)},\ \la \tau, \sigma \ra \in C^\infty(M)_{(i+j)}. \]
    If $\sigma_a$ is a weighted frame for $V|_U$ then the corresponding dual frame $\tau_a$ is a weighted frame for $V^*|_U$. In particular, if $v_a$ are the vertical weights for $V$ then $-v_a$ are the vertical weights for $V^*$ and $V=(V^*)^*$ as weighted vector bundles. Furthermore, the corresponding filtration of $V^*|_N$ is given by 
        \[ (V^*|_N)_{(i)} = \mathrm{ann}((V|_N)_{(-i+1)}). \]
    This generalizes the weighting of $T^*M$ defined in~\autoref{examples: examples of linear weightings} (d). 

    \item[(b)](\emph{Direct sums, tensor products, etc.)} If $V\to M$ and $W\to M$ are weighted vector bundles over the weighted pair $(M, N)$, then $V\oplus W$, $V\otimes W$, $\mathrm{Hom}(V,W)$, $\wedge^n V$, and $\mathrm{Sym}^n(V)$ all inherit linear weightings in a canonical way. 

    For example, the linear weighting on $V\otimes W$ is given by 
        \[ \G(V\otimes W)_{(k)} = \sum_{i+j = k}\G(V)_{(i)} \otimes_{C^\infty(M)} \G(W)_{(j)}.   \]
    The linear weighting on $\mathrm{Hom}(V, W)$ is given by the identification $\mathrm{Hom}(V,W) \cong V^*\otimes W$. With respect to this weighting we have that, for all $i, j \in \Z$,
        \[ \G(\mathrm{Hom}(V, W))_{(i)} \times \G(V)_{(j)} \to \G(W)_{(i+j)}. \]

    \item[(c)](\emph{Shifted weighting}) Given a linear weighting of $V\to M$ we denote by $V[k]$ the vector bundle $V$ linearly weighted by 
        \[ \G(V[k])_{(i)} = \G(V)_{(i+k)}. \]
    Note that: 
        \begin{itemize}
            \item[(i)] $(V[k]|_N)_{(i)} = (V|_N)_{(i+k)}$, and 

            \item[(ii)] $V[k]^* = V^*[-k]$.
        \end{itemize}

    \item[(d)](\emph{Pullbacks}) Suppose that $(M,N)$ and $(M', N')$ are weighted pairs and that $\varphi:M' \to M$ is a weighted morphism. 

\begin{proposition}
    If $V\to M$ is a weighted vector bundle, then 
        \begin{equation}
        \label{equation: pullback weighting}   
            \G(\varphi^*V)_{(i)} = \sum_{j\geq 0}C^\infty(M')_{(j)}\cdot \varphi^*\G(V)_{(i-j)}. 
        \end{equation} 
    defines a linear weighting of the pullback bundle $\varphi^*V\to M'$. 
\end{proposition}
\begin{proof}
    Let $U\sset M$ be open and let $\sigma_1, \dots, \sigma_k$ be a weighted frame for $V|_U$. We claim that the pullbacks $\varphi^*\sigma_a \in \G(\varphi^*(V|_{U}))$ define a weighted frame for $\varphi^*(V|_U)$. Indeed, we have that 
        \begin{align*}
            \G(\varphi^*(V)&|_{\varphi^{-1}(U)})_{(i)}  = \sum_{j\geq 0} C^\infty(\varphi^{-1}(U))_{(j)}\cdot\varphi^*\G(V|_U)_{(i-j)} \\
            & = \sum_a \left( \sum_{j\geq 0}C^\infty(\varphi^{-1}(U))_{(j)}\cdot \varphi^*C^\infty(U)_{(i-j-v_a)}\right)\varphi^*\sigma_a.
        \end{align*}
    To complete the proof, we must show that
        \[ \sum_{j\geq 0}C^\infty(\varphi^{-1}(U))_{(j)}\cdot \varphi^*C^\infty(U)_{(i-j-v_a)} = C^\infty(\varphi^{-1}(U))_{(i-v_a)}; \]
    note that there is nothing to show if $i-v_a < 0$. We have that $\varphi^*C^\infty(U)_{(i-j-v_a)} \sset C^\infty(\varphi^{-1}(U))_{(i-j-v_a)}$ as $\varphi:M'\to M$ is a weighted morphism, hence the left hand side is contained in the right hand side. The other inclusion follows by putting $j = i-v_a$ in the sum. 
\end{proof}
\end{itemize}

\subsection{Linear Weightings as a Filtration of Polynomial Functions} 

We now relate linear weighting to the weightings discussed in~\autoref{chapter: Weightings}. For this, we replace the algebra of smooth functions on $M$ with the algebra of \emph{fibrewise polynomial} functions on the total space of $V$. Recall that the space of space of fibrewise polynomial functions on a vector bundle $V\to M$ is the direct sum 
    \[ C^\infty_{pol}(V) = \bigoplus_{n\geq 0}C^\infty_{[n]}(V), \]
where $C^\infty_{[n]}(V)$ are the functions of homogeneity $n$. This perspective allows for a quick definition of the weighted normal bundle and weighted deformation bundles in~\autoref{section: linear weighted normal bundle} and~\autoref{section: linear weighted deformation bundle}.

\begin{theorem}[Polynomial Characterization of Linear Weightings]
\label{theorem: linear weightings in terms of polynomials}
     Let $V$ be a rank $k$ vector bundle over the $m$-dimensional manifold $M$. There is a one to one correspondence between linear weightings of $V$ and multiplicative filtrations 
        \[ \cdots \supseteq C^\infty_{pol}(V)_{(i)} \supseteq C^\infty_{pol}(V)_{(i+1)} \supseteq \cdots \]
    of the sheaf of polynomial functions on $V$ with the following property: there exist tuples $(w_1, \dots, w_m) \in \Z^m_{\geq 0}$ and $(v_1, \dots, v_k) \in \Z^k$ such that for each $\xi\in V$, there is an open set $U\sset M$ with $\xi\in V|_U$ and vector bundle coordinates $x_a \in C^\infty_{[0]}(V|_U)$, $p_b\in C^\infty_{[1]}(V|_U)$ such that $C^\infty_{pol}(V|_U)_{(i)}$ is generated as a $C^\infty(U)$-module by the monomials 
        \begin{equation}
        \label{equation: local vector bundle coordinate generation}
            x^sp^t = x_1^{s_1}\cdots x_m^{s_m}\cdot p_1^{t_1} \cdots p_k^{t_k}
        \end{equation}
    such that $s\cdot w - t\cdot v = \sum_{i=1}^ms_iw_1 - \sum_{i=1}^kt_iv_i \geq i$. 
\end{theorem}

We refer to the coordinate system $x_a, p_b$ described in~\autoref{theorem: linear weightings in terms of polynomials} as \emph{weighted vector bundle coordinates}. We stress that while the base coordinates $x_a$ have non-negative weight, the fibre coordinates $y_b$ can have both positive \emph{and} negative weight. 

\begin{proof}
    Suppose that $V$ is linearly weighted. For each $n \geq 0$ we will define a filtration of $C^\infty_{[n]}(V)$ and define filtration of $C^\infty_{pol}(V)$ by taking the direct sum. Given $n\geq 0$, recall that there is an identification 
        \begin{equation}
        \label{equation: polynomials as symmetric algebra}
            C^\infty_{[n]}(V) \cong \G(\mathrm{Sym}^n(V^*)).
        \end{equation}
    Since $\mathrm{Sym}^n(V^*)$ inherits a linear weighting from the weighting of $V$ (see~\autoref{subsection: linear constructions} (b)), we use the identification~\eqref{equation: polynomials as symmetric algebra} to define the filtration of $C^\infty_{[n]}(V)$. Explicitly, for any open set $U\sset M$, elements of  $C^\infty_{[n]}(V|_U)_{(i)}$ are functions $f\in C^\infty_{[n]}(V|_U)$ such that for any $\sigma \in \G(V|_U)_{(j)}$ one has
        \[ f\circ \sigma \in C^\infty(U)_{(i+nj)}.  \]
    In particular, $C^\infty_{[0]}(V)_{(i)} = C^\infty(M)_{(i)}$ and $C^\infty_{[1]}(V)_{(i)} = \G(V^*)_{(i)}$. Let 
        \[ C^\infty_{pol}(V)_{(i)} = \bigoplus_{n\geq 0} C^\infty_{[n]}(V)_{(i)}.  \]
    To get weighted vector bundle coordinates, let $U\sset M$ be an open subset over which $V|_U$ is trivialized, let $x_a$ be weighted coordinates on $U$ and let $\sigma_b$ be a weighted frame for $V|_U$. Let $\tau_b\in \G(V^*|_U)$ denote the dual frame and $p_b\in C^\infty_{[1]}(V|_U)$ the corresponding linear coordinates. The $n$-fold symmetric products of the $\tau_b$ form a weighted frame for $\mathrm{Sym}^n(V|_U^*)$, hence the functions $x_a, p_b$ define the necessary vector bundle coordinate system. Note that if the section $\sigma_b$ has weighted $v_b$, then the corresponding fibre coordinate $p_b$ has weight $-v_b$.

    Conversely, suppose that we are given a filtration 
        \[ \cdots \supseteq C^\infty_{pol}(V)_{(i)} \supseteq C^\infty_{pol}(V)_{(i+1)} \supseteq \cdots  \]
    as described in the statement of the theorem. For $U\sset M$ open, define
        \begin{align}
        \label{equation: induced weighting on M}
            C^\infty(U)_{(i)} & = C^\infty_{[0]}(V|_U) \cap C^\infty_{pol}(V|_U)_{(i)} \quad \text{and} \\ \label{equation: induced filtration on sections}
            \G(V|_U^*) & = C^\infty_{[1]}(V|_U)\cap C^\infty_{pol}(V|_U)_{(i)}.
        \end{align}
    The requirement that $C^\infty_{pol}(V|_U)$ be generated as a $C^\infty(U)$-module by terms of the form~\eqref{equation: local vector bundle coordinate generation} ensures that~\eqref{equation: induced weighting on M} defines a weighting of $M$ and~\eqref{equation: induced filtration on sections} defines a linear weighting of $V^*$. By taking the dual weighting, we have thus also defined a linear weighting of $V$. These constructions are clearly inverse to one another, and so we have proved the theorem. 
\end{proof}

\begin{example}
    If $W\to N$ is a vector subbundle of $V\to M$, then order of vanishing defines a linear weighting of $V$, which we refer to as the trivial weighting along $W$. This agrees with the linear weighting defined in~\autoref{examples: examples of linear weightings} (c).
\end{example}

\begin{remarks}
    \begin{itemize}
        \item[(a)] One can use to conclusion of~\autoref{theorem: linear weightings in terms of polynomials} to define \emph{graded bundle weightings}, using graded bundle coordinates in place of vector bundle coordinates.

        \item[(b)] Let $V\to M$ be a linearly weighted vector bundle. Recall that the weighting of $V$ determines a filtration
        \[ \cdots \supseteq (V|_N)_{(i)} \supseteq (V|_N)_{(i+1)} \supseteq \cdots \]
    of $V|_N$ by wide subbundles. In local vector bundle coordinates $x_1, \dots, x_m \in C^\infty(U)$, $p_1, \dots, p_k\in C^\infty_{[1]}(V|_U)$, $(V|_{N\cap U})_{(i)}$ is cut out by the equations 
        \begin{align*}
            x_a = 0 & \text{ for } w_a > 0,  \\
            p_b = 0 & \text{ for } v_b < i.
        \end{align*}

        \item[(c)] As a consequence of~\autoref{theorem: linear weightings in terms of polynomials}, one finds that a linear weighting of $V$ naturally defines a filtration of $TV|_N$ by subbundles. In terms of the canonical decomposition $TV|_N = TM|_N \oplus V|_N$, one has 
            \[ (TV|_N)_{(i)} = (TM|_N)_{(i)} \oplus (V|_N)_{(i)}. \]

        \item[(d)] We say that a linear weighting of $V\to M$ is \emph{concentrated in non-positive degree} if 
            \[ (V|_N)_{(i)} = 0 \quad \text{for} \quad i>0. \]
        In this case,~\autoref{theorem: linear weightings in terms of polynomials} shows that a linear weighting is a weighting in the sense of~\autoref{definition: weighting} along the subbundle 
            \[ W = (V|_N)_{(0)}, \]
        and we say that $V$ is linearly weighted along $W$. For example, if $(M,N)$ is a weighted pair then the linear weighting of $TM$ is concentrated in non-positive degree and $TM$ is linearly weighted along $TN$. 
    \end{itemize}
\end{remarks}

\subsubsection{Weighted subbundles} 
\label{subsection: weighted subbundles}

Now that we have a characterization of linear weightings in terms of polynomial functions, we can define weighted subbundles. 

\begin{definition}
    Let $V\to M$ be a weighted vector bundle. A subbundle $W\to R$ is a \emph{weighted subbundle} if there exists a weighted atlas of subbundle coordinates for $W$. 
\end{definition}

That is, $W$ is a weighted subbundle if and only if for every $w\in W$ there exist weighted vector bundle coordinates $x_a\in C^\infty_{[0]}(V|_U)$ and $p_b \in C^\infty_{[1]}(V|_U)$ such that $W$ is locally defined by the vanishing of a subset of the coordinates. Coordinates with this property are called \emph{weighted subbundle coordinates}.

\begin{examples}
    \begin{itemize}
        \item[(a)] If $V$ is a linearly weighted vector bundle over the weighted pair $(M,N)$, then $V|_N$ is a weighted subbundle. Any system of weighted vector bundle coordinates are weighted subbundle coordinates for $V|_N$.

        \item[(b)] If $V$ is linearly weighted vector space then any subspace is a weighted subspace (cf.~\autoref{examples: weighted submanifolds} (e)). 
    \end{itemize}
\end{examples}

\begin{proposition}
    Let $V \to M$ be a linearly weighted vector bundle. If $W \to R$ is a weighted subbundle, 
        \begin{itemize}
            \item[(a)] $R$ is a weighted submanifold of $(M,N)$, 
            \item[(b)] $W$ inherits a linear weighting and the filtration of $W|_{R\cap N}$ is given by 
                \[ (W|_{R\cap N})_{(i)} = W \cap (V|_N)_{(i)}.  \]
        \end{itemize}
\end{proposition}
\begin{proof}
    One proceeds as in the proof of~\autoref{proposition: weighted submanifolds are weighted} using weighted subbundle coordinates. 
\end{proof}

\subsection{Weighted vector bundle morphisms}

We now discuss morphisms of weighted vector bundles. The easiest definition uses the characterization of linear weightings in terms of polynomial functions. 

\begin{definition}
    A vector bundle morphism $\varphi:V\to V'$ between linearly weighted vector bundles $V\to M$ and $V'\to M'$ is called \emph{weighted} if $\varphi^*C^\infty_{pol}(V')_{(i)} \sset C^\infty_{pol}(V)_{(i)}$ for all $i\in \Z$. 
\end{definition}

Recall that a vector bundle map $\varphi:V\to V'$ induces a module map $\varphi^*:\G((V')^*) \to \G(V^*)$. In terms of this perspective, one has the following characterization of weighted vector bundle morphisms. 
	
\begin{proposition}
\label{proposition: VB-morphisms in terms of sections}
    Let $\varphi:V\to V'$ be a vector bundle morphism between linearly weighted vector bundles $V\to M$ and $V'\to M'$, respectively. Then $\varphi$ is weighted if and only if 
        \begin{enumerate}
            \item[(a)] the base map $\varphi_M:M\to M'$ is weighted and 
            \item[(b)] for all $i\in \Z$, $\varphi^* : \G((V')^*)_{(i)} \to \G(V^*)_{(i)}$.
        \end{enumerate}
\end{proposition}
\begin{proof}
    Suppose that $\varphi:V \to V'$ is a weighted vector bundle morphism. Then, 
        \begin{align*}
            \varphi_M^*C^\infty(M')_{(i)} & = \varphi_M^*(C^\infty_{pol}(V')_{(i)} \cap C^\infty_{[0]}(W)) \\
            & \sset C^\infty_{pol}(V)_{(i)} \cap C^\infty_{[0]}(V) = C^\infty(M)_{(i)},
        \end{align*}
    whence $\varphi_M:M\to M'$ is weighted. Similarly, for any $\sigma \in \G((V')^*)$, let $f_\sigma \in C^\infty_{[1]}(V)$ be the corresponding function under the identification $C^\infty_{[1]}(V) = \G(V^*)$. Then
        \[ \varphi^*(\sigma) = \varphi^*f_\sigma \in C^\infty_{[1]}(V)_{(i)} = \G(V^*)_{(i)}.  \]

    The converse follows by a similar argument, using that $\varphi^*: C^\infty_{pol}(V')\to C^\infty_{pol}(V)$ is an algebra morphism and $C^\infty_{pol}(V)_{(i)}$ is locally generated as a filtered algebra by $C^\infty(M)$ and $C^\infty_{[1]}(V) = \G(V^*)$. 
\end{proof}

\begin{example}
    If $f:M\to M'$ is a weighted morphism between weighted pairs $(M, N)$ and $(M', N')$, then the tangent map $Tf:TM \to TM'$ is a weighted vector bundle morphism. 
\end{example}
	
In the case that $V$ and $V'$ are vector bundles over a common base the definition of a weighted vector bundle morphism reduces to the one that one might expect. 
	
\begin{proposition}
    Suppose that that $V$ and $V'$ and linearly weighted vector bundles over a common weighted pair $(M, N)$ and $\varphi : V\to V'$ is a vector bundle morphism covering the identity. Then $\varphi$ is a weighted vector bundle morphism if and only if $\varphi(\Gamma(V)_{(i)}) \sset \Gamma(V')_{(i)}$ for all $i$. 
\end{proposition}
\begin{proof}
    Suppose that $\varphi^*(\Gamma((V')^*)_{(i)}) \sset \Gamma(V^*_{(i)})$, and let $\sigma \in \Gamma(V)_{(i)}$. Let $\tau \in \Gamma((V')^*)_{(j)}$ be arbitrary, and note that 
		\[ \la \tau, \varphi(\sigma) \ra = \la \varphi^*(\tau), \sigma \ra \in C^\infty(M)_{(i+j)},  \]
    since $\varphi^*(\tau)\in \Gamma(V^*)_{(j)}$; here the angular brackets denote the pairing between a vector bundle and its dual. This implies that $\varphi(\sigma)\in \Gamma(V')_{(i)}$, proving one inclusion. The opposite inclusion follows similarly.
\end{proof}

We state the following observation, whose proof is obvious, which allows us to apply the results of~\autoref{section: weighted morphisms} to vector bundle morphisms. 

\begin{lemma}
\label{lemma: weighted morphisms are shift invariant}
    Let $V\to M$ and $V'\to M'$ be weighted vector bundles, and $k \in \Z$. A vector bundle morphism $\varphi:V\to V'$ is a weighted vector bundle morphism if and only if $\varphi:V[k]\to V'[k]$ is weighted. 
\end{lemma}

Note that by taking $k$ to be sufficiently large, the weighting of $V[k]$ is concentrated in negative degree. In particular, $V[k]$ is weighted along a closed submanifold $N\sset M$ in the sense of~\autoref{chapter: Weightings}. In light of~\autoref{lemma: weighted morphisms are shift invariant}, all of the definitions and results in~\autoref{section: weighted morphisms} hold for vector bundle morphisms, replacing $TM|_N$ with $TV|_N = V|_N\oplus TM|_N$. Let us illustrate this with the following result (cf.~\autoref{theorem: characterization of weighted morphisms in terms of their graphs}). 

\begin{theorem}
\label{theorem: linear weighted morphism in terms of graph}
    Suppose that $V \to M$ and $V' \to M'$ are weighted vector bundles and $\varphi:V \to V'$ is a vector bundle morphism whose base map is a smooth map of pairs $(M, N)\to (M', N')$. Then $\varphi$ is a weighted VB-morphism if and only if
        \begin{enumerate}
            \item[(a)] the graph $\G(\varphi)\sset V'\times V$ is a weighted subbundle  and 
            \item[(b)] the maps $V|_N \to V'|_{N'}$ and $TM|_N \to TM'|_{N'}$ are filtration preserving. 
        \end{enumerate}
\end{theorem}
\begin{proof}
    By replacing $V$ and $V'$ with $V[k]$ and $V'[k]$ for sufficiently large $k$, we may assume the weightings are concentrated in negative degree. By~\autoref{theorem: characterization of weighted morphisms in terms of their graphs}, $\varphi:V\to V'$ is a weighted morphism if and only if its graph is a weighted subbundle of $V'\times V$ and the map $T\varphi|_N:TV|_N \to TV'|_{N'}$ is filtration preserving. With respect to the canonical decomposition $TV|_N = TM|_N \oplus V|_N$, the map $T\varphi|_N$ decomposes as $T\varphi|_N = T\varphi_M|_N\oplus\varphi|_N$ hence $T\varphi|_N$ is filtration preserving if and only if the maps $V|_N \to V'|_{N'}$ and $TM|_N \to TM'|_{N'}$ are filtration preserving. Therefore, $\varphi: V \to V'$ is a weighted morphism.  
\end{proof}

\section{The Weighted Normal Bundle}
\label{section: linear weighted normal bundle}

We now define the weighted normal bundle of a linearly weighted vector bundle. Since we have defined weightings in terms of filtrations of fibrewise polynomial functions on $V$, much of the discussion in~\autoref{section: weighted normal bundles} carries over, almost verbatim, to linear weightings. 

\subsection{Definition}

Let $V\to M$ be a linearly weighted vector bundle. Let $\gr(C^\infty_{pol}(V))$ be the graded algebra associated to the filtered algebra, with graded components given by 
    \[ \gr(C^\infty_{pol}(V))_{i} = C^\infty_{pol}(V)_{(i)}/C^\infty_{pol}(V)_{(i+1)}.  \]

\begin{definition}
    The \emph{weighted normal bundle} of the linearly weighted vector bundle $V\to M$ is the character spectrum 
        \[ \nuw(V) = \aHom(\gr(C^\infty_{pol}(V)), \R) \]
\end{definition}

Given a function $f\in C^\infty_{pol}(V)_{(i)}$, let $f^{[i]}$ denote the class of $f$ in $\gr(C^\infty_{pol}(V))_{i}$. We think of $f^{[i]}$ as a function on $\nuw(V)$, defined by evaluation
    \begin{align*}
        f^{[i]} : \nuw(V)  \to \R, \quad \varphi  \mapsto \varphi(f^{[i]}). 
    \end{align*}
Let $\pi:V\to M$ and $\kappa_\lambda: V \to V$ denote the bundle projection and scalar multiplication by $\lambda \in \R$, respectively. Since these maps are filtration preserving, they induce maps 
    \[ \nuw(\pi) : \nuw(V) \to \nuw(M,N) \quad \text{and} \quad \nuw(\kappa_\lambda) : \nuw(V) \to \nuw(V), \]
which we take to be our vector bundle projection and scalar multiplication, respectively, in the following theorem. 

\begin{theorem}
\label{theorem: linear weighted normal bundle}
    Let $V\to M$ be a weighted vector bundle. 
        \begin{itemize}
            \item[(a)] The weighted normal bundle $\nuw(V)$ has the unique structure of a $C^\infty$-vector bundle over $\nuw(M,N)$ of rank equal to that of $V$, such that for all $n\geq 0$ 
                    \[\gr(C^\infty_{[n]}(V)) \sset C^\infty_{[n]}(\nuw(V)). \]

            \item[(b)] Given weighted vector bundle coordinates $x_a$ and $p_b$ on $V|_U$, the functions $x_a^{[w_a]}$, $p_b^{[v_b]}$ serve as vector bundle coordinates on $\nuw(V|_U) = \nuw(V)|_{\nuw(U, U\cap N)}$.

            \item[(c)] This construction is functorial: any weighted vector bundle morphism $\varphi:V\to W$ between linearly weighted vector bundles defines a vector bundle morphism  $\nuw(\varphi):\nuw(V)\to \nuw(W)$.
        \end{itemize}
\end{theorem}  
The proof mirrors that of~\cite[Theorem 4.2]{loizides2023differential}. We include it for completeness and to explain the use of polynomial functions.
\begin{proof}
    Let $U\sset M$ and let $x_1, \dots, x_m, p_1, \dots, p_k \in C^\infty(V|_U)$ be a system of weighted vector bundle coordinates. We will show that the homogeneous approximations $y_a = x_a^{[w_a]}$, $q_b = p_b^{[-v_b]}$ serve as vector bundle coordinates on $\nuw(\pi)^{-1}(\nuw(U, U\cap N)) = \nuw(V|_U)$. By~\autoref{theorem: weighted normal bundle is smooth}, the map 
        \[ (y_1, \dots, y_m):\nuw(U, U\cap N) \to \R^m \]
    is a diffeomorphism onto an open set $\overline{U} \sset \R^m$. Since $\gr(C^\infty_{pol}(V|_U))$ is generated as an algebra over $\gr(C^\infty(U))$ by the elements $q_1, \dots, q_k$ it follows that the map 
        \[ (y_1, \dots, y_m, q_1, \dots, q_k) : \nuw(\pi)^{-1}(\nuw(U, U\cap N)) \to \R^{m+k} \]
    is a bijection onto an open set of the form $\overline{U}\times \R^k$. This defines a smooth structure on $\nuw(V|_U)$. We now show that if $f\in C^\infty_{pol}(V|_U)_{(i)}$, then $f^{[i]}: \nuw(V|_U)\to \R$ is smooth. We may assume that $f(x,p) = g(x)p_1^{s_1}\cdots p_k^{s_k}$ with $g\in C^\infty(U)_{(j)}$, $s\cdot v = \sum s_bv_b = i-j$, and $\sum_b s_b = n$, in which case 
        \[ f^{[i]}(y,q) = g^{[j]}(y)q_1^{s_1} \cdots q_k^{s_k}  \]
    which is evidently smooth and homogeneous of degree $n$ with respect to the scalar multiplication $\nuw(\kappa_t)$. In particular, if $y'_1, \dots, y'_k$ is a different choice of weighted linear coordinates on $V|_U$ then $(y'_b)^{[-v_b]}$ is smooth as a function of $y_a$ and $q_b$. The uniqueness assertion and the claim about functoriality are clear. 
\end{proof}

\begin{remark}
    This proof shows that if $[f_{ab}]$ is the transition matrix for $V|_U$ corresponding to two weighted vector bundle coordinate systems, then $f_{ab} \in C^\infty(U)_{(v_b-v_a)}$ and the transition matrix for $\nuw(V|_U)$ is given by $[f_{ab}^{[v_b-v_a]}] \in C^\infty(\nuw(U, U\cap N), \mathrm{GL}(\R^k))$. 
\end{remark}

\begin{examples}
\label{examples: examples of weighted normal bundles for linear weightings}
    \begin{itemize}
        \item[(a)] Suppose that $W\to N$ is a subbundle of $V \to M$. If $V$ is given the trivial weighting along $W$, then 
            \[ \nuw(V) = \nu(V, W),  \]
        as a vector bundle over $\nu(M,N)$. 

        \item[(b)] If $V$ is linearly weighted by a filtration of wide subbundles 
            \[ V = V_{(-r)} \supseteq V_{(-r+1)} \supseteq \cdots \supseteq V_{(q)} \supseteq 0, \]
        as in~\autoref{examples: examples of linear weightings} (a), then $\nuw(V) = \gr(V) \to M$. In particular, if $V = V_{(-r)}\oplus \cdots \oplus V_{(q)}$ is a graded vector bundle, then $\nuw(V) = V$.

        \item[(c)] If $(M,N)$ is a weighted pair and $TM$ is linearly weighted as in~\autoref{examples: examples of linear weightings} (b), then 
        \[ \nuw(TM) = T\nuw(M,N). \]
    \end{itemize}
\end{examples}

\begin{remark}
    The weighted normal bundle is invariant under shifting the weighting: for any $k \in \Z$ one has
        \[ \nuw(V[k]) = \nuw(V). \]
    In particular, shifting the weighting of $V$ so that it is concentrated in negative degrees, we can see that the weighted normal bundle as a bundle of homogeneous spaces of nilpotent Lie groups (cf.~\cite[Proposition 7.7]{loizides2023differential}.)
\end{remark}

\subsection{Action of $\R^\times$}

Let $\R^\times$ act on $\gr(C^\infty_{pol}(V))_{i}$ by multiplication by $t^i$. This defines a smooth action $\alpha : \R^\times\times \nuw(V)\to \nuw(V)$ on the weighted normal bundle which commutes with the vector bundle multiplication and extends the action of $\R^\times$ on $\nuw(M,N)$ described in~\autoref{subsection: definition of weighted normal bundle}; note in particular that since we are allowing the fibre coordinates to have negative weights multiplication by zero is \emph{not} defined, so $\nuw(V)$ fails to be a graded bundle in general. As in~\autoref{subsection: definition of weighted normal bundle}, given any $f\in C^\infty_{pol}(V)_{(i)}$ the function $f^{[i]} \in C^\infty(\nuw(V))$ is homogeneous of degree $i$. Letting $C^\infty_{[i,j]}(\nuw(V))$ denote the functions with are homogeneous of degree $i$ with respect to the vector bundle multiplication and homogeneous of degree $j$ with respect to the $\R^\times$-action, we have that 
    \[ C^\infty_{[i,j]}(\nuw(V)) = \gr(C^\infty_{[i]}(V))_j. \]

\begin{example}
\label{example: zoom action on graded VB}
    Suppose that $V=V_{(-r)}\oplus \cdots \oplus V_{(q)}\to M$ is a graded vector bundle over $M$.With respect to the identification $\nuw(V) = V$ explained in~\autoref{examples: examples of weighted normal bundles for linear weightings} (b), we have that 
        \begin{equation*}
        \label{equation: zoom action on graded VB}
            \alpha_\lambda(v_{-r}, \dots, v_{q}) = (\lambda^rv_{-r}, \dots, \lambda^{-q} v_{q}).
        \end{equation*}
\end{example}

\subsection{Sections of the weighted normal bundle}
\label{subsection: sections of the normal bundle}

Recall from~\autoref{subsection: definition of weighted normal bundle} that 
    \[ \gr(C^\infty(M)) = C^\infty_{pol}(\nuw(M,N)).  \]
In this section, we are going to define a canonical identification 
    \[ \gr(\G(V)) \to \G_{pol}(\nuw(V)) \]
of graded $\gr(C^\infty(M))$-modules, where 
    \[ \gr(\G(V)) = \bigoplus_{i\in \Z} \G(V)_{(i)}/\G(V)_{(i+1)} \]
and $\G_{pol}(\nuw(V))$ are the \emph{polynomial} sections of $\nuw(V)$, defined as follows. For $i\in \Z$, let $\G_{[i]}(\nuw(V))$ denote the sections of $\nuw(V)$ which are homogeneous of degree $i$, i.e. sections $\sigma \in \G(\nuw(V))$ such that  
    \[ \alpha_{\lambda, *}(\sigma) = \alpha_\lambda \circ \sigma \circ \alpha_{\lambda^{-1}} = \lambda^i\sigma.  \]
Then 
    \[ \G_{pol}(\nuw(V)) = \bigoplus_{i \in \Z} \G_{[i]}(\nuw(V)). \]
Before we begin, recall that 
    \[ \nuw(M,N) = \aHom(\gr(C^\infty(M)), \R) \]
and that for any $\sigma \in \G(V)_{(i)}$ and $f \in C^\infty_{[n]}(V)_{(j)}$ one has $f\circ \sigma \in C^\infty(M)_{(j+ni)}$. This motivates the following definition.

\begin{definition}
    Given $\sigma \in \G(V)_{(i)}$, the \emph{$i$-th homogeneous approximation} is the map $ \sigma^{[i]} : \nuw(M,N) \to \nuw(V)$ defined by 
        \begin{equation}
        \label{equation: homogeneous approximation definition}
            (\sigma^{[i]}(\varphi))(f^{[j]}) = \varphi((f\circ \sigma)^{[j+ni]}),
        \end{equation}
    where $\varphi \in \nuw(M,N)$ and $f \in C^\infty_{[n]}(V)_{(j)}$.
\end{definition}

\begin{example}
    Suppose that $V=V_{(-r)}\oplus \cdots \oplus V_{(q)}\to M$ is a graded vector bundle over $M$. A section $\sigma \in \G(V)$ has filtration degree $-i$ if and only if it can be written as a sum 
        \[ \sigma = \sigma_{-i} + \sigma_{-i+1} + \cdots + \sigma_{q}, \quad \text{where} \quad \sigma_{-j} \in \G(V_{(-j)}). \]
    Then $\nuw(V) = V$ and $\sigma^{[-i]} = \sigma_{-i}$ with respect to this identification. Note that $\sigma^{[-i]}\in \G_{[i]}(\nuw(V))$, by~\autoref{example: zoom action on graded VB}.
\end{example}

\begin{lemma}
\label{lemma: graded modules gives sections of weighted normal bundle}
    For $\sigma \in \G(V)_{(i)}$, the $i$-th homogeneous approximation is a smooth section of $\nuw(V)$ which is homogeneous of degree $-i$ and which depends only on the class of $\sigma$ in $\G(V)_{(i)}/\G(V)_{(i+1)}$. Moreover, for any $f\in C^\infty_{[n]}(V)_{(j)}$ and $g\in C^\infty(M)_{(k)}$ one has 
        \begin{align}
        \label{equation: homogeneous approximation properties}
        \begin{split}
            f^{[j]}\circ \sigma^{[i]} & = (f\circ \sigma)^{[j+ni]} \in C^\infty(\nuw(M,N)) \quad \text{and} \\
            g^{[k]}\sigma^{[i]} & = (g\sigma)^{[i+k]} \in \G(\nuw(V)).
        \end{split}
        \end{align}
\end{lemma}
\begin{proof}
   See~\autoref{A-section: graded modules gives sections of weighted normal bundle}.
\end{proof}

This establishes a map $\gr(\G(V)) \to \G_{pol}(\nuw(V))$. By abusing notation, given $\sigma \in \G(V)_{(i)}$ we denote both its class in $\G(V)_{(i)}/\G(V)_{(i+1)}$ and its $i$-homogeneous approximation by $\sigma^{[i]}$. 

\begin{theorem}
\label{theorem: sections of the weighted normal bundle}
    If $\sigma_b$ is a weighted frame for $V|_U$, then the homogeneous approximations $\sigma_b^{[v_b]}$ define a frame for $\nuw(V|_U) = \nuw(V)|_{\nuw(U, U\cap N)}$. In particular, the map 
        \begin{align}
        \label{equation: polynomial section identification}
            \gr(\G(V)) \to \G_{pol}(\nuw(V))
        \end{align}
    is an isomorphism of graded $\gr(C^\infty(M))$-modules and 
        \[ \G(\nuw(V)) = C^\infty(\nuw(M,N))\otimes_{\gr(C^\infty(M))}\gr(\G(V)).  \]
\end{theorem}
\begin{proof}
Let $p_b \in C^\infty_{[1]}(V|_U)_{(-v_b)}$ be the linear vector bundle coordinates defined by the frame $\sigma_a$ and note that for any non-zero $\varphi_0 \in \nuw(U, U\cap N)$ one has 
    \begin{align*}
    \label{equation: vector bundle coordinate pairing}
    \begin{split}
            \sigma_b^{[v_b]}(\varphi_0)\left(p_a^{[-v_a]}\right) = \varphi_0 \left((p_a \circ \sigma_b))^{[v_b-v_a]}\right) = \left\{
            \begin{array}{cc}
                 1 & \text{if $a=b$} \\
                 0 & \text{else,}
            \end{array}
        \right.
    \end{split}
    \end{align*}
which shows that $\sigma_a^{[v_a]}$ form a linearly independent set. In particular, this shows that the map~\eqref{equation: polynomial section identification} is injective.  
    
Since $\gr(C^\infty_{pol}(V|_U))$ is generated as an algebra over $\gr(C^\infty(U))$ by the elements $p_b^{[-v_b]}$, any $\varphi \in \nuw(V|_U)$ is determined by its value on $\gr(C^\infty(U))$ and $p_b^{[-v_b]}$, $b=1, \dots, k$. In particular, if  $\nuw(\pi)(\varphi) = \varphi_0$ then 
        \[ \varphi = \sum_b \varphi\left(p_b^{[-v_b]}\right)\sigma_b^{[v_b]}(\varphi_0),\] 
so $\sigma_b^{[v_b]}$ spans as well.

To see that~\eqref{equation: polynomial section identification} is an isomorphism, it suffices to show that it is surjective as we have already established injectivity. Suppose that $\sigma \in \G(\nuw(V|_U))$ can be written as $\sigma = \sum_b f_b\sigma_b^{[v_b]}$. The pairing 
    \[ \G(\nuw(V))\otimes \G(\nuw(V^*))\to C^\infty(\nuw(M,N)) \]
respects homogeneity, hence 
    \[ \sigma \in \G_{[i]}(\nuw(V|_U)) \iff \quad  f_b \in C^\infty_{[i-v_b]}(\nuw(U,N\cap U)) \]
for $b=1, \dots, k$. This implies that $\gr(\G(V)) \to \G_{pol}(\nuw(V))$ is surjective, which completes the proof. 
\end{proof}

\subsection{Weighted Subbundles}

Recall that a subbundle $W\to R$ of a linearly weighted vector bundle $V\to M$ is called a weighted subbundle if there exist a weighted atlas of subbundle coordinates for $W$. In this case, it is explained in~\autoref{subsection: weighted subbundles} that $W$ inherits a linear weighting. 

\begin{proposition}
\label{proposition: subbundles of weighted normal bundle}
    If $W\to R$ is a weighted subbundle of $V\to M$, then $\nuw(W)$ is a subbundle of $\nuw(V)$. 
\end{proposition}
\begin{proof}
    If $x_a, p_b$ are weighted subbundle coordinates for $W|_{R\cap U}$, then the homogeneous approximations $x_a^{[w_a]}, p_b^{[-v_b]} \in C^\infty(\nuw(V|_U))$ define subbundle coordinates for $\nuw(W|_{R\cap U})$, as in~\autoref{proposition: weighted embeddings define embeddings}.
\end{proof}

Let $(V|_N)_{(i)} \to N$ be the filtration of $V|_N$ defined by the linear weighting (see~\autoref{proposition: linear weighting determines filtration by subbundles}).~\autoref{proposition: subbundles of weighted normal bundle} allows us to describe the weighted normal bundle of $V$ in terms of the associated graded bundle 
    \[ \gr(V|_N) = \bigoplus_i (V|_N)_{(i)}/(V|_N)_{(i+1)}. \]

\begin{corollary}
    Suppose that $V \to M$ is linearly weighted. If $\pi : \nuw(M,N)\to N$ denotes the graded bundle projection, then one has a (non-canonical) identification 
        \[ \nuw(V) \cong \pi^*\gr(V|_N).\]
\end{corollary}
\begin{proof}
    Recall that $V|_N$ is a weighted subbundle of $V$ and that the induced weighting of $V|_N$ is given by the filtration $(V|_N)_{(i)}$. As in~\autoref{examples: examples of weighted normal bundles for linear weightings} (b), we have
        \[ \nuw(V|_N) = \gr(V|_N).  \]
    On the other hand $\nuw(V|_N)$ is a subbundle of $\nuw(V)|_N$ of equal rank, and so $\nuw(V|_N) = \nuw(V)|_N$. Since $\pi:\nuw(M, N)\to N$ is a smooth homotopy inverse of the inclusion $\iota:N\into \nuw(M,N)$ it follows that 
        \[ \nuw(V) \cong \pi^*\iota^*\nuw(V) = \pi^*(\nuw(V)|_N) = \pi^*\gr(V|_N), \]
    as claimed. 
\end{proof}

\begin{example}
    If $V\to M$ is given the trivial weighting along the subbundle $W\to R$, then this proposition amounts to the identification 
        \[ \nu(V,W) \cong \pi^*(V|_N/W \oplus W) \]
    as vector bundles over $\pi:\nu(M,N)\to N$. 
\end{example}

\section{The Weighted Deformation Bundle}
\label{section: linear weighted deformation bundle}

Let $V\to M$ be a linearly weighted vector bundle. As explained above, the weighting of $M$ determines a weighted deformation space $\defw(M,N)$ which admits a set-theoretic decomposition 
    \[ \defw(M,N) = \nuw(M, N) \sqcup (M\times \R^\times).  \]
The purpose of this section is to explain how this construction carries over to the context of linearly weighted vector bundles, yielding a \emph{weighted deformation bundle} $\defw(V)$ which is itself a vector bundle over $\defw(M,N)$. Since many of the proofs in this section are very similar to those in the previous section we will omit them for brevity. 

\subsection{Definitions} 

We use the same definition for the weighted deformation bundle as in~\autoref{section: weighted deformation space}. 

\begin{definition}
    The \emph{weighted deformation bundle} of the linearly weighted vector bundle $V\to (M,N)$ is the character spectrum
        \[ \defw(V) = \aHom(\rees(C^\infty_{pol}(V)), \R) \]
\end{definition}

Given $f\in C^\infty_{pol}(V)_{(i)}$, let $\WT{f}^{[i]}:\defw(V) \to \R$ be as in~\autoref{section: weighted deformation space} and $\pi_\delta = \WT{1}^{[-1]}: \defw(V) \to \R$. If $\pi:V \to M$ and $\kappa_\lambda : V \to V$ denote the vector bundle projection and scalar multiplication by $\lambda \in \R$, respectively, it follows that we get maps 
    \[ \defw(\pi) : \defw(V) \to \defw(M,N) \quad \text{and} \quad \defw(\kappa_\lambda) : \defw(V)\to \defw(V).  \]
As before, these give $\defw(V)$ the structure of a family of vector spaces over $\defw(M,N)$. The following can be proved by the same argument as~\cite[Theorem 5.1]{loizides2023differential} (cf.~\autoref{theorem: weighted normal bundle is smooth} and~\cite[Theorem 4.2]{loizides2023differential}). 

\begin{theorem}
\label{theorem: linear weighted deformation bundle}
Let $V\to M$ be a weighted vector bundle. 
    \begin{itemize}
        \item[(a)] The weighted deformation bundle has the unique structure of a $C^\infty$-vector bundle over $\defw(M,N)$ of rank equal to that of $V$ over $\defw(M,N)$ such that $\pi_\delta:\defw(V)\to \R$ is a surjective submersion and that for all $n\geq 0$ 
            \[\rees(C^\infty_{[n]}(V)) \sset C^\infty_{[n]}(\defw(V)). \]

        \item[(b)] Given weighted vector bundle coordinates $x_a$ and $p_b$ on $V|_U$, the functions $\WT{x_a}^{[w_a]}$, $\WT{p_b}^{[-v_b]}$ serve as vector bundle coordinates on $\defw(V|_U) = \defw(V)|_{\defw(U, U\cap N)}$.

        \item[(c)] This construction is functorial: any weighted vector bundle morphism $\varphi:V\to W$ between linearly weighted vector bundles defines a vector bundle morphism  $\defw(\varphi):\defw(V)\to \defw(W)$.
    \end{itemize}
\end{theorem}   

As before, the surjective submersion $\pi_\delta : \defw(V) \to \R$ defines a decomposition $\defw(V) = \nuw(V) \sqcup (V\times \R^\times)$. The following diagram commutes
    \begin{equation*}
    \xymatrix{
        \defw(V) \ar[rr]^-{\defw(\pi)} \ar[dr]_{\pi_\delta} & & \defw(M,N) \ar[dl]^{\pi_\delta}\\
        & \R & 
    }
    \end{equation*}
so we can think of $\defw(V)$ as a family of vector bundles over $\defw(M,N)$, with 
    \begin{equation}
    \label{equation: decomposition of weighted deformation bundle}
        \defw(V)|_{\pi_\delta^{-1}(t)} = \left\{
        \begin{array}{ll}
             V & t\neq 0,   \\
            \nuw(V) & t = 0.  
        \end{array}
    \right.
    \end{equation}

\begin{examples}
\label{examples: deformation bundle}
    \begin{itemize}
        \item[(a)] If $V =  V_{(-r)}\oplus \cdots \oplus V_{(q)}$ is a graded vector bundle, then there is a canonical diffeomorphism $\defw(V) \to V\times \R$ given by the family of maps 
            \[ \defw(V)|_{\pi_\delta^{-1}(t)} \to V, \quad 
                \left\{
                    \begin{array}{ll}
                        (v_{-r}, \dots, v_{q}) \mapsto (t^{-r}v_{-r}, \dots, t^{q}v_{q}), & t \neq 0 \\
                        (v_{-r}, \dots, v_{q}) \mapsto (v_{-r}, \dots, v_{q}), & t = 0,
                    \end{array}
                \right.  \]
        where we are identifying $\nuw(V)$ with $V$ for $t=0$. 

        \item[(b)] If $(M,N)$ is a weighted pair and $TM$ is linearly weighted as in~\autoref{examples: examples of linear weightings}~(b), then 
        \[ \defw(TM) = \bigcup_{t\in \R} T\defw(M,N)|_{\pi_\delta^{-1}(t)}. \]
    \end{itemize}
\end{examples}

\subsection{Zoom action of $\R^\times$}

The action of $\R^\times$ on $\nuw(V)$ extends smoothly to a linear action of $\R^\times$ on $\defw(V)$ such that for any $f\in C^\infty_{pol}(V)_{(i)}$ the homogeneous interpolation $\WT{f}^{[i]} \in C^\infty(\defw(V))$ is homogeneous of degree $i$. On the open dense set $V\times \R^\times \sset \defw(V)$, this action is given by 
    \begin{equation*}
        \alpha_\lambda(v,t) = (v, \lambda^{-1}t)
    \end{equation*}

\begin{example}
    If $V =  V_{(-r)}\oplus \cdots \oplus V_{(-1)}$ is a graded vector bundle then, with respect to the identification $\defw(V) \to V\times \R$ described in~\autoref{examples: deformation bundle}, the zoom action is given by 
        \[ \alpha_\lambda(v_{-r}, \dots, v_{q}, t) = (\lambda^rv_{-r}, \dots, \lambda^{-q} v_q, \lambda^{-1}t). \]
\end{example}

\subsection{Sections of the weighted deformation bundle}

Analogous to~\autoref{subsection: sections of the normal bundle}, let $\G_{[i]}(\defw(V))$ denote the sections of $\defw(V)$ which are homogeneous of degree $i$ with respect to the zoom action of $\R^\times$, and let 
    \[ \G_{pol}(\defw(V)) = \bigoplus_{i\in \Z} \G_{[i]}(\defw(V)).  \]
We are going to define an inclusion 
    \[ \rees(\G(V)) \to \G_{pol}(\defw(V)) \]
of $\rees(C^\infty(M))$-modules, where
    \[ \rees(\G(V)) = \left\{ \sum_{i \in \Z} \sigma_iz^{-i} : \sigma_i \in \G(V)_{(i)} \right\} \sset \G(V)[z,z^{-1}].  \]

\begin{definition}
    Given $\sigma \in \G(V)_{(i)}$, the \emph{$i$-th homogeneous interpolation} is the map $\WT{\sigma}^{[i]} : \defw(M,N) \to \defw(V)$ defined by 
        \begin{equation}
        \label{equation: homogeneous section definition}
            (\WT{\sigma}^{[i]}(\varphi))(fz^{-j}) = \varphi((f\circ \sigma)z^{-j-ni}),
        \end{equation}
    where $\varphi \in \defw(M,N)$ and $f \in C^\infty_{[n]}(V)_{(j)}$.
\end{definition}

\begin{lemma}
\label{lemma: rees modules gives sections of weighted deformation bundle}
    For $\sigma \in \G(V)_{(i)}$, the $i$-th homogeneous interpolation is a smooth section of $\defw(V)$ of homogeneity $i$ such that, for any $f\in C^\infty_{[n]}(V)_{(j)}$ and $g\in C^\infty(M)_{(k)}$, one has 
        \begin{align}
        \label{equation: homogeneous interpolation properties}
        \begin{split}
            \WT{f}^{[j]}\circ \WT{\sigma}^{[i]} & = \WT{f\circ \sigma}^{[j+ni]} \in C^\infty(\defw(M,N)) \quad \text{and} \\
            \WT{g}^{[k]}\WT{\sigma}^{[i]} & = \WT{g\sigma}^{[i+k]} \in \G(\defw(V)).
        \end{split}
        \end{align}
\end{lemma}

As was the case for the weighted normal bundle, this yields a map
    \[ \rees(\G(V)) \to \G_{pol}(\defw(V)), \quad \sum_i \sigma_iz^{-i} \mapsto \sum_i \WT{\sigma}_i^{[i]} \]
In terms of the decomposition~\eqref{equation: decomposition of weighted deformation bundle} one finds, using the relations~\eqref{equation: homogeneous interpolation properties}, that if $\sigma \in \G(V)_{(i)}$ then 
    \begin{equation*}
        \WT{\sigma}^{[i]}|_{\pi_\delta^{-1}(t)} = \left\{
            \begin{array}{ll}
                t^{-i}\sigma & t \neq 0 \\
                \sigma^{[i]} & t = 0. 
            \end{array}
        \right. 
    \end{equation*}

\begin{example}
    Suppose that $V =  V_{(-r)}\oplus \cdots \oplus V_{(q)}$ is a graded vector bundle, and let 
        \[ \sigma = \sigma_{-i}+\sigma_{-i+1} + \cdots + \sigma_{q} \in \G(V)_{(-i)}.  \]
    With respect to the identification $\defw(V)\to V\times \R$ described in~\autoref{examples: deformation bundle} (a), we have that 
        \[ \WT{\sigma}^{[-i]} = \sigma_{-i}+t\sigma_{-i+1}+ \cdots + t^{i+q}\sigma_{q}, \]
    which clearly extends to $\sigma_{-i} = \sigma^{[-i]}$ as $t\to 0$. 
\end{example} 

\begin{theorem}
\label{theorem: sections of deformation bundle}
    If $\sigma_b$ is a weighted frame for $V|_U$, then the homogeneous interpolations $\WT{\sigma}_b^{[v_b]}$ define a frame for $\defw(V|_U) = \defw(V)|_{\defw(U, U\cap N)}$. In particular, the map 
        \[ \rees(\G(V)) \to \G_{pol}(\defw(V)) \]
    is an inclusion of $\rees(C^\infty(M))$-modules and  
        \[ \G(\defw(V)) = C^\infty(\defw(M,N))\otimes_{\rees(C^\infty(M))}\rees(\G(V)).  \]
\end{theorem}

\begin{remark}
    In general, the map $\rees(\G(V)) \to \G_{pol}(\defw(V))$ fails to be surjective because the map $\rees(C^\infty(M)) \to C^\infty_{pol}(\defw(M,N))$ need not be surjective. 
\end{remark}

For completeness, we record the following.

\begin{proposition}
\label{proposition: subbundles of weighted deformation bundle}
    If $W\to R$ is a weighted subbundle of $V\to M$, then $\defw(W)$ is a subbundle of $\defw(V)$. 
\end{proposition}

\section{The Rescaled Spinor Bundle and \v{S}evera's Algebroid}
\label{section: examples in the literature}

We now take a some time to explain how our constructions put two separate constructions in the literature, namely the \emph{rescaled spinor bundle} of Higson and Yi introduced in~\cite{higson2019spinors} and the "associative algebroid" introduced by \v{S}evera in~\cite{vsevera2017letters}, under the same umbrella. The results of this section are the result of discussions with Gabriel Beiner, Yiannis Loizides, and Eckhard Meinrenken. See~\cite{beiner2022linear} for another account and more information on the connections to index theory. 

\subsection{The rescaled spinor bundle}

Let $M$ be an even dimensional spin Riemannian manifold with spinor bundle $S\to M$ and Clifford connection $\nabla$ on $S$. Higson and Yi define a filtration of differential operators acting on sections of $S$ by declaring that Clifford multiplication and covariant differentiation have order $-1$. Thus, $D\in \mathrm{DO}(S)_{(-q)}$ if and only if it can be locally expressed as a sum of terms of the form 
    \[ fD_1\cdots D_q \]
where $f\in C^\infty(M)$ and each $D_i$ is either a covariant derivative $\nabla_X$, Clifford multiplication $c(X)$, or the identity operator (cf.~\cite[Definition 3.3.1]{higson2019spinors}). Higson and Yi say that $D$ has \emph{Getzler order} $-q$ if $D\in \DO(S)_{(-q)}$.

The Getzler filtration of $\DO(S)$ determines a filtration of $\G(S\boxtimes S^*)$ as follows. Recall that $\Cl(TM)$ is naturally filtered 
    \begin{equation}
    \label{equation: filtration of Clifford algebra}
        \Cl(TM) = \Cl_{-\dim(TM)}(TM) \supseteq \cdots \supseteq \Cl_0(TM) = \C,
    \end{equation}
hence $S\boxtimes S^*|_M = S\otimes S^* \cong\Cl(TM)$ is naturally filtered. Define 
    \begin{align}
    \label{equation: getzler weighting}
    \begin{split}
        \G(S\boxtimes S^*)_{(i)} = \{ \sigma \in \G(S\boxtimes S^*&) : D\in \DO(S)_{(-q)} \\
        & \implies D\sigma \in \G(S\boxtimes S^*, \Cl_{i-q}(TM)) \},
    \end{split}
    \end{align}
where we are using the identification $\Cl(TM) = S\boxtimes S^*|_{M}$.

\begin{theorem}[{\cite[Lemma 3.4.10]{higson2019spinors}}].
    If $M\times M$ is given the trivial weighting along the diagonal then the filtration~\eqref{equation: getzler weighting} defines a linear weighting of $S\boxtimes S^*$ such that the induced filtration of $S\boxtimes S^*|_M$ is given by~\eqref{equation: filtration of Clifford algebra}. 
\end{theorem}

Applying the weighted deformation bundle to $S\boxtimes S^*$ with this linear weighting yields the rescaled spinor bundle. 

\begin{definition}(cf.~\cite[Section 3.4]{higson2019spinors}).
    The \emph{rescaled spinor bundle} is the weighted deformation bundle 
        \[ \mathbb{S} = \defw(S\boxtimes S^*) \to \T M. \]
\end{definition}

\subsection{\v{S}evera's algebroid}

Let $V\to M$ be a rank $k$ vector bundle with inner product and assume that the principal $\mathrm{SO}(k)$-frame bundle admits a lift to a principal $\mathrm{Spin}(k)$-bundle $P$. Let $\Cl(\R^k)$ be the complexified Clifford algebra, and consider the action of $\mathrm{Pair}(\mathrm{Spin}(k))$ on $\mathrm{Pair}(P) \times \Cl(\R^k)$ given by
    \[ (g_1, g_2).(f_1,f_2, v) = (f_1.g_1, f_2.g_2, g_1vg_2^{-1}).  \]

\begin{definition}[cf. \cite{vsevera2017letters}]
    If $V$ is a rank $k$ vector bundle equipped with inner product and spin structure, the \emph{Clifford algebroid} is the associated bundle
        \[ \mathscr{C}l(V) = \mathrm{Pair}(P)\times_{\mathrm{Pair}(\mathrm{Spin}(k))}\Cl(\R^k) \to \mathrm{Pair}(M) \]
\end{definition}

We will show that there is a canonical linear weighting of $\mathscr{C}l(V)$, starting with a technical lemma. The proof of this lemma uses weighted paths so in order to keep our exposition short, we postpone the proof to the appendix. 

\begin{lemma}
\label{lemma: action is weighted morphism}
    If $\mathrm{Pair}(P)$, $\mathrm{Pair}(\mathrm{Spin}(k))$ are given the doubled trivial weighting along the diagonal, and $\Cl(\R^k)$ is given the linear weighting defined by its filtration by subspaces, then the group action 
        \begin{equation}
        \label{equation: severa action is weighted}
            \mathrm{Pair}(\mathrm{Spin}(k)) \times \mathrm{Pair}(P) \times \Cl(\R^k) \to \mathrm{Pair}(P) \times \Cl(\R^k)
        \end{equation}
    is a weighted morphism.
\end{lemma}
\begin{proof}
    See~\ref{A-section: weighted action proof} (cf.~\cite{beiner2022linear}).
\end{proof}

A consequence of this lemma is the following theorem. 

\begin{theorem}
    Let $\pi: \mathrm{Pair}(P)\times \Cl(\R^k) \to \mathscr{C}l(V)$ be the quotient map. Then 
        \[ C^\infty_{pol}(\mathscr{C}l(V))_{(i)} = \{ f \in C^\infty_{pol}(\mathscr{C}l(V)) : \pi^*f \in C^\infty_{pol}(\mathrm{Pair}(P)\times \Cl(\R^k))_{(i)} \}  \] 
    defines a linear weighting of $\mathscr{C}l(V)$, where $\mathrm{Pair}(P)\times \Cl(\R^k)$ is given the product weighting. 
\end{theorem}
\begin{proof}
    We have to find weighted vector bundle coordinates. Any point $p\in \mathrm{Pair}(M)$ is contained in an open neighbourhood $U\sset \mathrm{Pair}(M)$ such that $\mathrm{Pair}(P)|_U$ is isomorphic to $U\times \mathrm{Pair}(\mathrm{Spin}(k))$ as weighted manifolds. Choosing $U$ small enough, we may assume that there exist weighted coordinates $x_a \in C^\infty(U)$. If $p_b$ are linear weighted coordinates on $\Cl(\R^k)$ then since the action map~\eqref{equation: severa action is weighted} is weighted and the following diagram commutes 
        \begin{equation}
        \xymatrixcolsep{0.25pc}
         \vcenter{\hbox{\xymatrix{
            \mathrm{Pair}(P)|_U \times \Cl(\R^k) \cong U\times \mathrm{Pair}(\mathrm{Spin}(k)) \times \Cl(\R^k) \ar[dr] \ar[d]_{\pi} & \\
            \mathscr{C}l(V)|_U \ar[r]_\cong & U\times \Cl(\R^k),
        }}}
        \end{equation}  
    it follows that we can take $x_a, p_b$ to be our weighted vector bundle coordinates on $\mathscr{C}l(V)|_U$. 
\end{proof}

\begin{definition}
    The \emph{\v{S}evera algebroid} is the weighted deformation bundle
        \[^\tau \mathscr{C}l(V) = \defw(V).  \]
\end{definition}

%% file: 2_multiplicative_weightings.tex
\chapter{Multiplicative Weightings}
\label{chapter: Multiplicative Weightings}

In this chapter we will develop the theory of multiplicative weightings for Lie groupoids $G\toto M$. We will show that if $G\toto M$ is multiplicatively weighted along $H\toto N$, then one has Lie groupoids 
    \[ \nuw(G, H) \toto \nuw(M,N) \quad \text{and} \quad \defw(G,H) \toto \defw(M,N). \]
We also discuss the appropriate notion of equivalence of weighted Lie groupoids.

\section{Preliminaries on Lie groupoids}
\label{section: groupoid prelims}

\subsection{Definition and basic examples}

A Lie groupoid $G\rightrightarrows M$ is a smooth manifold $G$ with a partially defined associative binary operation together with two surjective submersions $t,s:G\to M$ onto a smooth submanifold $M \sset G$. Elements of $G$ are called \emph{arrows}, elements of $M$ are called \emph{units}, and the maps $t, s$ are called the \emph{target} and \emph{source}, respectively. The space of \emph{$k$-arrows} is the $k$-fold fibre product
    \[ G^{(k)} = \{ (g_1, \dots, g_k) \in G^k : s(g_{i}) = t(g_{i+1}),\ i=1, \dots, k-1 \}; \]
we refer to $G^{(2)}$ as the space of \emph{composable arrows}. The partially defined multiplication is a map 
    \[ \mathrm{mult}_G:G^{(2)} \to G,\ (g_1,g_2) \mapsto g_1\circ g_2 \]
defined on composable arrows, satisfying 
    \begin{itemize}
        \item[(a)] $g_1\circ (g_2\circ g_3) = (g_1\circ g_2)\circ g_3$, 

        \item[(b)] $t(g)\circ g = g\circ s(g) = g$.
    \end{itemize}
Moreover, we assume that for every $g \in G$ there exists some $h\in G$ such that $(g,h) \in G^{(2)}$, $(h,g) \in G^{(2)}$ and that both $g\circ h$ and $h \circ g$ are units. It is automatic that for a given $g$, the corresponding element $h$ is unique and denoted $g^{-1}$; the map corresponding map $g\mapsto g^{-1}$ is denoted $\mathrm{inv}_G$.

A smooth map $F:H\to G$ between Lie groupoids $H\toto N$ and $G\toto M$ is called a \emph{morphism} of Lie groupoids if 
    \[ F(g_0\circ g_1) = F(g_0)\circ F(g_1). \]
If $F$ is the inclusion of a submanifold then $H\toto N$ is called a Lie subgroupoid. We will often denote a Lie groupoid $G\rightrightarrows M$ together with a Lie subgroupoid $H \rightrightarrows N$ as a pair $(G,H)\rightrightarrows (M,N)$. If $(G,H)\rightrightarrows (M,N)$ is a Lie groupoid pair with $N=M$ then we call $H$ a \emph{wide} subgroupoid.  

\begin{remark}
    In general it is not required that Lie groupoids be Hausdorff, asking only that the source fibres and units be Hausdorff. However, for simplicity, we will assume that all groupoids \emph{are} Hausdorff. 
\end{remark}

\begin{examples}
    \begin{enumerate}
        \item[(a)] A Lie group is a Lie groupoid whose object space is a single point.
        
        \item[(b)] Any smooth manifold $M$ defines a trivial Lie groupoid with $s = t = \mathrm{id}_M$.
        
        \item[(c)] The \emph{pair groupoid} of a manifold $M$ is the groupoid 
            \[ \mathrm{Pair}(M) = M\times M \rightrightarrows M \]
        where $s(m,m') = m'$ and $t(m,m') = m$. Composition is defined as 
            \[(m,m')\circ(m',m'') = (m,m''). \]
            
        \item[(d)] Let $G$ be a Lie group, $M$ be a smooth manifold $M$, and 
            \[ \alpha : G\times M \to M, \quad (g,m) \mapsto g\cdot m \]
        be a smooth action. The \emph{action groupoid} is the defined to be the manifold $G\times M$ with source map $s(g,m) = m$, target $t(g,m) = g\cdot m$, and composition law
            \[ (g,m) = (g',m')\circ (g'',m'') \iff m = m'',\ m' = g''\cdot m'', \text{ and } g = g'g''. \]
        By identifying $G\times M$ with the graph of the group action, the action groupoid can be realized as a Lie subgroupoid of $G \times \mathrm{Pair}(M)$ 
        
        \item[(e)] A Lie groupoid $G\rightrightarrows M$ such that $s=t$ is called a \emph{family of Lie groups} over $M$. Note that this differs slightly from a Lie group bundle over $M$, since the Lie group structure may differ from fibre to fibre. In particular, vector bundles are examples of Lie groupoids. 
    \end{enumerate}
\end{examples}

An important observation is that the groupoid structure of $G$ is completely determined by the graph of multiplication, 
    \[ \G(\mathrm{mult}_G) =  \{ (g,g_0, g_1) \in G^3 : g = g_0\circ g_1\} \sset G^3. \]
Indeed, 
    \begin{itemize}
        \item[(a)] the objects of $G$ are exactly the elements $m\in G$ with the property that $(m,m,m)\in \G(\mathrm{mult}_G)$
        \item[(b)] given $g\in G$, $s(g)$ is the unique element of $G$ such that $(g,g,s(g)) \in \G(\mathrm{mult}_G)$; $t(g)$ is described similarly,
        \item[(c)] $g^{-1}$ is the unique element for which $(s(g),g^{-1}, g)\in \G(\mathrm{mult}_G)$ and $(t(g),g, g^{-1})\in \G(\mathrm{mult}_G)$.
    \end{itemize}

\subsection{The group of bisections}

One consequence of the partially defined multiplication operation is that there is no map 
    \[ G \to \mathrm{Diff}(G) \]
generalizing left or right translation on a Lie group. Instead, one needs to consider certain submanifolds $S\sset G$ called \emph{bisections}. 

\begin{definition}
    A \emph{bisection} of a Lie groupoid $G \toto M$ is a submanifold $S \sset G$ such that both $t$ and $s$ restrict to diffeomorphisms $s|_S, t|_S:S \to M$. The set of bisections of the groupoid $G$ is denoted $\G(G)$. 
\end{definition}

The set of bisections is naturally a group, with multiplication given by 
    \[ S_1\circ S_2 = \{ g_1\circ g_2 : (g_1, g_2) \in (S_1\times S_2)\cap G^{(2)} \}.  \]
The group $\G(G)$ acts on $G$ in three ways:
    \begin{itemize}
        \item[(a)] the \emph{left action} 
            \[ \mathcal{A}^L:\G(G) \times G \to G, \quad (S, g) \mapsto \mathcal{A}^L_S(g) = h\circ g, \]
        where $h \in S$ is the unique element such that $s(h) = t(g)$;

        \item[(b)] the \emph{right action}
            \[ \mathcal{A}^R:\G(G) \times G \to G, \quad (S, g) \mapsto \mathcal{A}^R_S(g) = g\circ h^{-1}, \]
        where $h \in S$ is the unique element so that $s(h) = s(g)$;

        \item[(c)] the \emph{adjoint action}, $\mathrm{Ad}_S = \mathcal{A}^R_S\mathcal{A}^L_S$. 
    \end{itemize}

For a general Lie groupoid, a global bisection through an element $g\in G$ might not exist, unless the groupoid has connected source fibres (\cite[Theorem 3.1]{zhong2009existence}). On the other hand, there always exists a \emph{local} bisection through any $g\in G$. 

\subsection{VB-groupoids}

A useful tool for the problem of differentiating weighted groupoids is the notion of a so-called "\emph{vector bundle groupoid}", or VB-groupoid. Using the Grabowski-Rotkiewicz characterization of vector bundles in terms of their scalar multiplication (cf.~\cite{grabowski2009higher}), Bursztyn, Cabrera, and  del Hoyo showed  that VB-groupoids can be defined rather simply as groupoids with a vector bundle scalar multiplication by groupoid morphisms~\cite[Theorem 3.2.3]{bursztyn2016vector}. Following their lead, we define \emph{graded bundle groupoids}, or GB-groupoids. 

\begin{definition}[{cf.~\cite[Theorem 3.2.3]{bursztyn2016vector}}]
\label{definition: GB groupoid}
    A Lie groupoid $V \toto E$ is called a \emph{GB-groupoid} if it is also a graded bundle such that the monoid action $\kappa:\R\times V \to V$ is by groupoid morphisms. If $V$ is a vector bundle then it is called a \emph{VB-groupoid}. A \emph{GB-subgroupoid} is a Lie subgroupoid which is closed under scalar multiplication. 
\end{definition}

If $V\toto E$ is a VB-groupoid, then the zero section $G\sset V$ inherits a groupoid structure over the zero section $M\sset E$. We depict a VB-groupoid as a diagram
    \begin{equation}
    \xymatrix{
        V \ar@<-.5ex>[r] \ar@<.5ex>[r] \ar[d] & E \ar[d] \\
        G \ar@<-.5ex>[r] \ar@<.5ex>[r] & M
    }
    \end{equation}
If $V\toto E$ is any VB-groupoid, then the dual bundle $V^*$ is canonically a VB-groupoid with units $\mathrm{ann}(E)$ (cf.~\cite[Theorem 4.8]{meinrenken2017Lie}; see also~\cite[Section 11.2]{mackenzie2005general} and~\cite{pradines1988remarque}). The groupoid structure on $V^*$ is uniquely characterized by the requirement that $\mu = \mu_1\circ \mu_2$ if and only if 
    \[ \la \mu, v\ra = \la \mu_1, v_1\ra + \la \mu_2, v_2\ra \]
whenever $v = v_1\circ v_2$. The graph of multiplication of $V^*$ is 
    \[ \G(\mathrm{mult}_{V^*}) = \mathrm{ann}^\sharp(\G(\mathrm{mult}_V)), \]
where $\mathrm{ann}^\sharp(\G(\mathrm{mult}_V))$ is the subbundle defined by 
    \[ (\mu_2, \mu_1) \in \mathrm{ann}^\sharp(\G(\mathrm{mult}_V)) \iff \la \mu_2, v_2 \ra = \la \mu_1, v_1 \ra\ \ \  \forall
        (v_2, v_1) \in \G(\mathrm{mult}_V).\]


\begin{example}
\label{example: structure of tangent groupoid}
    The main example of a VB-groupoid is given by applying the tangent functor to a Lie groupoid $G\toto M$. This yields the \emph{tangent groupoid}, 
    \begin{equation*}
    \xymatrix{
        TG \ar@<-.5ex>[r] \ar@<.5ex>[r] \ar[d] & TM \ar[d] \\
        G \ar@<-.5ex>[r] \ar@<.5ex>[r] & M.
    }
    \end{equation*}
    An explicit formula for multiplication on the tangent groupoid can be given using bisections. Indeed, let $X_{g_0} \in T_{g_0}G$ and $Y_{g_1}\in T_{g_1}G$ be composable. If $S_0$ is a bisection through $g_0$ and $S_1$ is a bisection through $g_1$, then 
        \begin{equation}
        \label{equation: multiplication on the tangent groupoid}
            X_{g_0} \circ Y_{g_1} = (T_{g_0}\mathcal{A}^R_{S_1})(X_{g_0}) + (T_{g_1}\mathcal{A}^L_{S_0})(Y_{g_1}) - (T_{g_1}(\mathcal{A}^L_{S_0}\mathcal{A}^R_{S_1}t))(Y_{g_1});
        \end{equation}
    see~\cite[Proposition 3.1]{mackenzie1997classical}.
\end{example}

We now record two lemmas that we will need later on. 

\begin{lemma}[{\cite[Corollary C.4]{li2014dirac}}]
\label{lemma: annhilator of VB-subgroupoid}
    Let $V\toto E$ be a VB-groupoid. If $W$ is a VB-subgroupoid, then $\mathrm{ann}(W) \sset V^*$ is a VB-subgroupoid with objects $\mathrm{ann}(W)\cap \mathrm{ann}(E)$. 
\end{lemma}

Suppose that $V\toto W$ and $W\toto F$ are VB-groupoids with base groupoid $G\toto M$. Their direct sum is the VB-groupoid
    \begin{equation*}
        \xymatrix{
        V\oplus W \ar@<-.5ex>[r] \ar@<.5ex>[r] \ar[d] & E\oplus F \ar[d] \\
        G \ar@<-.5ex>[r] \ar@<.5ex>[r] & M
    }
    \end{equation*}
with the obvious source and target maps, and multiplication defined by
    \[ (v_0, w_0) \circ (v_1, w_1) = (v_0\circ v_1, w_0\circ w_1) \]
for $((v_0, w_0),(v_1, w_1)) \in V^{(2)}\oplus W^{(2)} = (V\oplus W)^{(2)} $. 

\begin{lemma}
    If
        \begin{equation*}
        \xymatrix{
            V \ar@<-.5ex>[r] \ar@<.5ex>[r] \ar[d] & E \ar[d] \\
            G \ar@<-.5ex>[r] \ar@<.5ex>[r] & M
        }
        \end{equation*}
    is a VB-groupoid, then $TV|_G$ is the direct sum VB-groupoid $TG\oplus V \toto TM\oplus E$.
\end{lemma}
\begin{proof}
    The map 
        \[ TG\oplus V \to TV|_G, \quad (\xi, v)\mapsto \xi + v, \]
    where we are identifying $V$ with the vertical bundle in $TV|_G$, is an isomorphism of vector bundles, so it remains to show that it is a groupoid morphism. But this follows since $\mathrm{mult}_V:V^{(2)}\to V$ is a linear map with base map $\mathrm{mult}_G$, so  
        \[  T\mathrm{mult}_V|_G = T\mathrm{mult}_G\oplus \mathrm{mult}_V. \qedhere \]
\end{proof}

\section{Definition of Multiplicative Weightings}
\label{section: definition of multiplicative weightings}

We are now move on to \emph{multiplicative weightings} for Lie groupoids $G\toto M$. The obvious definition is given by inserting the adjective "weighted" in the standard definition. However, we recall from~\autoref{section: weighted submanifolds} and~\autoref{section: weighted morphisms} that it is \emph{not} obvious how to do this!

\begin{definition}
\label{definition: multiplicative weighting}
    A weighting of $G$ along $H\sset G$ is said to be \emph{multiplicative} if 
        \begin{enumerate}
            \item[(a)] The units $M\sset G$ are a weighted submanifold, 
            \item[(b)] the source map is a weighted submersion,
            \item[(c)] multiplication $m:G^{(2)} \to G$ is a weighted morphism, and 
            \item[(d)] inversion is a weighted morphism. 
        \end{enumerate}
    A groupoid with a multiplicative weighting is called a \emph{weighted Lie groupoid}. 
\end{definition}

\begin{remarks}
\begin{itemize}
    \item[(a)] Since inversion is its own inverse, it is automatically a weighted \emph{diffeomorphism}.

    \item[(b)] Since $t = s\circ \mathrm{inv}$, it follows that the target map is also a weighted submersion. 

    \item[(c)] Since the source and target are weighted submersions, if follows from~\autoref{theorem: weighted fibre products} that the space of $k$-arrows is a weighted submanifold of $G^k$. In particular, the space of composable arrows is naturally a weighted manifold, and therefore~\autoref{definition: multiplicative weighting} (c) makes sense. 

    \item[(d)] All of the face maps $\bd_i : G^{(k)} \to G^{(k-1)}$,
        \[ \bd_i(g_1, \dots, g_k) = \left\{
            \begin{array}{ll}
                 (g_2, \dots, g_k) & i = 0,  \\
                 (g_1, \dots, g_i\circ g_{i+1}, \dots, g_k) & 0 < i < k, \\ 
                (g_1, \dots, g_{k-1}) & i=k,
            \end{array}
        \right.\]
    and all the degeneracy maps 
        \[ \epsilon_i:G^{(k)} \to G^{(k+1)} \quad (g_1, \dots, g_k) \mapsto (g_1, \dots, g_i, s(g_i), g_{i+1}, \dots, g_k) \]
    are automatically weighted. 
\end{itemize}
\end{remarks}

Note that in the definition we did \emph{not} assume that $H \sset G$ was a Lie subgroupoid. It turns out that this is automatic. 

\begin{proposition}
    Suppose that a weighting of $G \toto M$ along $H$ is multiplicative. Then $H$ is a Lie subgroupoid of $G$ with units $N = H\cap M$. 
\end{proposition}
\begin{proof}
    Since $M\sset G$ is a weighted submanifold, it follows by~\autoref{proposition: weighted submanifolds are weighted} that $H$ intersects $M$ cleanly; in particular, $N = H\cap M$ is a submanifold. We have furthermore that $N\neq \emptyset$ so long as $H\neq \emptyset$ because $s:(G,H) \to (M,N)$ is a map of pairs (it is a weighted morphism). 
    
    Similarly, both 
        \[ \mathrm{mult}_G:(G^{(2)}, H^{(2)}) \to (G,H) \quad \text{and} \quad \mathrm{inv}_G: (G,H)\to (G, H) \]
    are maps of pairs, hence $H$ is also closed under multiplication and inversion. Finally, it follows from the normal form for weighted submersions (\autoref{theorem: weighted submersion coordinates}) that the restrictions of $s$ and $t$ to $H$ are submersions.    
\end{proof}

\begin{lemma}
\label{lemma: vector bundle submersions}
    Let $V\to M$ and $W\to N$ be smooth vector bundles and let $\varphi: V\to W$ be a vector bundle morphism. Then $\varphi$ is a submersion if and only if it is fibrewise surjective and the base map $\varphi_M:M\to N$ is a submersion.
\end{lemma}
\begin{proof}
    Suppose that $\varphi$ is a submersion, and let $p\in M$. Since $\varphi$ is a vector bundle morphism, the tangent map $T_p\varphi$ splits as $T_p\varphi = T_p\varphi_M \oplus \varphi$ with respect to the splitting $T_pV = T_pM \oplus V_p$. Thus $\varphi_M$ is a submersion and $\varphi$ is fibrewise surjective. 

    Conversely, if $\varphi$ is fibrewise surjective and the base map $\varphi_M:M\to N$ is a submersion then by the previous paragraph it is a submersion near $M\sset V$. Since it is equivariant with respect to the $\R$-action, it follows that it is a submersion everywhere. 
\end{proof}

In practice it can be difficult to verify that all the structure maps are indeed weighted in the appropriate sense. As an application of~\autoref{theorem: characterization of weighted morphisms in terms of their graphs} we have the following equivalent characterization of multiplicative weightings. 

\begin{theorem}
\label{theorem: characterization of mult weightings}
    A weighting of $G\toto M$ along $H\sset G$ is multiplicative if and only if 
        \begin{itemize}
            \item[(a)] $M$ is a weighted submanifold of $G$, 
            \item[(b)] the graph of multiplication is a weighted submanifold of $G^3$, and 
            \item[(c)] the filtration of $TG|_H$ is by subgroupoids 
                \[ (TG|_H)_{(i)} \toto (TM|_N)_{(i)}. \]
        \end{itemize}
\end{theorem}
\begin{proof}
    Suppose that $G$ is multiplicatively weighted along $H$. Since multiplication is a weighted morphism, its graph is a weighted submanifold, establishing $(b)$. To establish (c), note that since $s$, $t$, $\mathrm{mult}_G$, and $\mathrm{inv}_G$ are all weighted morphisms their tangent maps are filtration preserving and so each 
        \[ (TG|_H)_{(i)} \toto (TM|_N)_{(i)}, \quad i=-r,-r+1, \dots, 0, \]
    are set-theoretic subgroupoids. To show that they are \emph{Lie} subgroupoids, it remains to show that $Ts$ and $Tt$ are submersions. To see this, note that since $s$ is a weighted submersion its tangent map $Ts : (TG|_H)_{(i)} \toto (TM|_N)_{(i)}$ is a fibrewise surjection whose base map is a submersion. Therefore, it follows by~\autoref{lemma: vector bundle submersions} that $Ts:(TG|_H)_{(i)} \toto (TM|_N)_{(i)}$ is a submersion. Similarly, $Tt:(TG|_H)_{(i)} \toto (TM|_N)_{(i)}$ is a submersion, establishing (c). 

    For the converse, suppose that (a), (b), and (c) hold and let $\pi_1:G^{(2)} \to G$ be projection onto the first factor. It is sufficient to show that $G^{(2)}$ is a weighted submanifold of $G^2$ and that $\mathrm{mult}_G, \pi_1: G^{(2)}\to G$ are weighted submersions for the following reason. Since the filtration of $TG|_H$ is by subgroupoids
        \[ (TG|_H)_{(i)} \toto (TM|_N)_{(i)} \]
    it follows that $T\mathrm{inv}_G$ is filtration preserving and $Ts$ is filtration preserving and fibrewise surjective in each filtration degree. Therefore, to establish conditions (b) and (d) in~\autoref{definition: multiplicative weighting} it suffices by~\autoref{theorem: characterization of weighted morphisms in terms of their graphs} to show that their graphs are weighted submanifolds. This will follow because the graphs of $\mathrm{inv}_G$ and $s$ are the inverse images of $M\sset G$ under the weighted submersions $\mathrm{mult}_G:G^{(2)}\to G$ and $\pi_1:G^{(2)}\to G$. 
        
    Let $f:\G(\mathrm{mult}_G)\to G^2$ be the restriction of the map $G^3\to G^2$ given by projection to the last two factors. In order to show that $G^{(2)}$ is a weighted submanifold, we will show that $f$ is a weighted embedding. Since this map is an embedding as well as a weighted morphism, it suffices, by~\autoref{theorem: weighted embedding characterization}, to check that the maps 
        \[ (T\G(\mathrm{mult}_G)|_{\mathrm{mult}_H})_{(i)} \to (TG^2|_{TH^2})_{(i)} \]
    are injective, with range equal to the intersection with $\mathrm{ran}(Tf)$. This is clear since $(TG|_H)_{(i)}$ is a subgroupoid: the graph of its multiplication is a subbundle of 
        \[ T\G(\mathrm{mult}_G) = \G(\mathrm{mult}_{TG}). \]
    This establishes the claim that $G^{(2)}$ is a weighted subbundle and that $f:\G(\mathrm{mult}_G)\to G^{(2)}$ is a weighted diffeomorphism. The maps $\mathrm{mult}_G$ and $\pi_1$ are weighted morphisms since they factor as compositions 
        \[ G^{(2)} \to \G(\mathrm{mult}_G) \into G^3 \to G. \]
    To see that they are weighted submersions we only need to check that the maps 
        \[ ((TG|_H)_{(i)})^{(2)} = ((TG^{(2)}|_{H^{(2)}})_{(i)} \to (TM|_N)_{(i)} \]
    are fibrewise surjective bundle maps, which follows since the $(TG|_H)_{(i)}$ are subgroupoids. 
\end{proof}

\begin{corollary}
\label{corollary: weighted subgroupoids}
    Suppose that $G\toto M$ is a weighted Lie groupoid and $K\toto P$ is a Lie subgroupoid. If $K$ is a weighted submanifold of $G$ then the induced weighting is multiplicative. 
\end{corollary}
\begin{proof}
    It is clear that inversion $\mathrm{inv}_{K}:K\to K$ is a weighted morphism because it is the restriction of a weighted morphism to a weighted submanifold. It remains to show that $P$ is a weighted submanifold of $K$, that $s:K\to P$ is a weighted submersion, and that multiplication on $K$ is a weighted morphism. 

    Consider the map $K\to K$ defined as the restriction of the composition
        \[G \corr{s} M \into G \]
    to $K$. This is a weighted projection of $K$ whose image is $P$, hence $P$ is a weighted submanifold of $K$ by~\autoref{proposition: weighted projections and weighted submanifolds}. 

    We now show that $s|_K:K\to P$ is a weighted submersion. Since $K$ is a weighted submanifold, it follows by considering weighted submanifold coordinates that $TK$ intersects $(TG|_H)_{(i)}$ cleanly. By~\autoref{theorem: characterization of mult weightings}, each $(TG|_H)_{(i)}$ is a Lie subgroupoid of $TG|_H$, so it follows by~\cite[Theorem 4.20]{meinrenken2017Lie} that the intersections $TK\cap (TG|_H)_{(i)} = (TK|_{K\cap H})_{(i)}$ are Lie subgroupoids with objects $TP\cap (TM|_N)_{(i)} = (TP|_{P\cap N})_{(i)}$. This, in particular, implies that 
        \[ Ts|_{K\cap H}:(TK|_{K\cap H})_{(i)} \to (TP|_{P\cap N})_{(i)} \]
    is fibrewise surjective, hence $s|_K:K\to P$ is a weighted submersion. 
    
    Finally, we show that multiplication is weighted. Both $s|_K$ and $t|_k$ are weighted submersions, and so $K^{(2)}$ is a weighted submanifold of $G^{(2)}$. Since multiplication for $K$ is the restriction of multiplication for $G$, it is a weighted morphism. 
 \end{proof}

 \begin{remark}
    In particular,~\autoref{corollary: weighted subgroupoids} shows that "weighted subgroupoids" are simply the Lie subgroupoids which are weighted submanifolds. 
 \end{remark}

\begin{examples}
\label{examples: multiplicative weightings}
\begin{itemize}
    \item[(a)] If $H\toto N$ is a Lie subgroupoid of $G\toto M$, then the trivial weighting of $G$ along $H$ is multiplicative.

    \item[(b)] If $(M, N)$ is a weighted pair, then $M\times M = \mathrm{Pair}(M)$ is a weighted groupoid with respect to the product weighting along $\mathrm{Pair}(N)$. Indeed, the units are a weighted submanifold by~\autoref{example: diagonal is weighted submanifold}, the graph of multiplication is a weighted submanifold as it is the image of the weighted embedding
        \[ M^3 \to M^6 \quad (m_1, m_2, m_3)  \mapsto (m_1, m_3, m_1, m_2, m_2, m_3), \]
    and the filtration of $T\mathrm{Pair}(M)|_{\mathrm{Pair}(N)}$ is given by the subgroupoids 
        \[ (T\mathrm{Pair}(M)|_{\mathrm{Pair}(N)})_{(i)} = \mathrm{Pair}((TM|_N)_{(i)}).\]

    \item[(c)] If $\pi:M\to N$ is a weighted submersion, then the submersion groupoid $M\times_\pi M$ is a weighted Lie groupoid. Indeed, this follows by~\autoref{corollary: weighted subgroupoids} as it is a weighted submanifold of the pair groupoid $\mathrm{Pair}(M)$.  

    \item[(d)] A weighted Lie groupoid $G\toto M$ with $M$ equal to a point will be called a weighted Lie group. In this case, the first condition in~\autoref{definition: multiplicative weighting} is vacuous, and the second just says that $H$ contains the group unit. Hence, weighting of a Lie group along a submanifold $H$ is multiplicative if and only if $H$ contains the group unit and both multiplication and inversion are weighted morphisms. 

   \item[(e)] An action of a weighted Lie group $G$ on a weighted manifold $M$ will be called \emph{weighted} if the action map $G\times M\to M$ is a weighted morphism. In this case, the action groupoid $G\times M \toto M$ is a weighted Lie groupoid. Indeed, the action groupoid can be identified with the graph of the group action, a subgroupoid of $G\times \mathrm{Pair}(M)$. Since the group action is weighted, the graph is a weighted submanifold, and the corresponding weighting on the action groupoid is weighted by~\autoref{corollary: weighted subgroupoids}.


    \item[(f)] If $V\to M$ is a vector bundle over $M$ thought of as a Lie groupoid, then a weighting is multiplicative if and only if it is linear and concentrated in non-positive degrees. 

    \item[(g)] Let $(G, H)$ be a weighted Lie group pair and $P\to B$ a principal $G$-bundle with a principal weighting along $Q\sset P$ (see~\autoref{defintion: Morita equivalence} (b)). By~\autoref{proposition: principal weightings}, $Q$ is a principal $H$-bundle. Using an argument analogous to to proof of~\autoref{A-theorem: composition of weighted HS-morphisms} one finds that the product weighting on $\mathrm{Pair}(P)$ descends to a weighting of the Atiyah groupoid $\mathrm{At}(P) = (P\times P)/G$ along $\mathrm{At}(Q) = (Q\times Q)/H$ and that this weighting is multiplicative. 
\end{itemize}
\end{examples}

\begin{proposition}
    Let $G\toto M$ be a weighted Lie groupoid. For all $g\in G$ there is a local bisection through $g$ which is given by a weighted submanifold. 
\end{proposition}
\begin{proof}
    We may assume that $g\in H$ as otherwise the statement is obvious. The tangent space $T_gG$ is a weighted vector space (cf.~\autoref{remark: wide weighting and filtered vector bundles}) with both $\ker(T_gt)$ and $\ker(T_gs)$ as weighted subspaces. We may choose a subspace of $T_gG$ which is a weighted complement to both. This weighted subspace is realized as the tangent space to a weighted submanifold $S$. Taking $S$ smaller if needed, this is the desired local bisection. 
\end{proof}

\subsection{Weighted VB-groupoids}

Using the conclusion of~\autoref{theorem: characterization of mult weightings}, we can combine linear weightings and multiplicative weightings to give rise to a notion of weighted VB-groupoids. These will be used for the differentiation procedure discussed in~\autoref{section: differentiation of M weightings}.

\begin{definition}
    A \emph{weighted VB-groupoid} is a VB-groupoid 
        \begin{equation*}
        \xymatrix{
            V \ar@<-.5ex>[r] \ar@<.5ex>[r] \ar[d] & E \ar[d] \\
            G \ar@<-.5ex>[r] \ar@<.5ex>[r] & M
        }
        \end{equation*}
    together with a linear weighting such that 
        \begin{itemize}
            \item[(a)] $E$ is a weighted subbundle of $V$, 

            \item[(b)] the graph of multiplication is a weighted subbundle of $V^3$, and 

            \item[(c)] the filtrations of $V|_H$ and $TG|_H$ are by subgroupoids
                \[ (V|_H)_{(i)} \toto (E|_N)_{(i)} \quad \text{and} \quad (TG|_H)_{(i)} \toto (TM|_N)_{(i)},  \]
            respectively. 
        \end{itemize}
\end{definition}

Note that this definition ensures that the weighting of $G$ along $H$ is also multiplicative. 

\begin{example}
\label{examples: VB groupoids}
    If $G\toto M$ is a weighted groupoid, then $TG \toto TM$ is a weighted VB-groupoid. Indeed, since $M$ is a weighted submanifold of $G$ it follows that $TM$ is a weighted subbundle of $TG$. Similarly, we have that 
        \[ \G(\mathrm{mult}_{TG})) = T\G(\mathrm{mult}_{G}) \sset TG^3, \]
    so it is a weighted subbundle. Finally, the filtration of 
        \[ T(TG)|_H = TG|_H \oplus TG|_H \]
    is by subgroupoids because the weighting of $G$ is multiplicative.
\end{example}

\begin{lemma}
\label{lemma: annilator is a weighted subbundle}
    Let $V\to M$ be a weighted vector bundle and $W\sset V$ a weighted subbundle. Then $\mathrm{ann}(W)$ is a weighted subbundle of $V^*$. 
\end{lemma}
\begin{proof}
    See~\autoref{A-section: annihilator is weighed subbundle}. 
\end{proof}

\begin{proposition}
\label{proposition: dual of a weighted vb groupoid}
    Let $V\toto E$ be a weighted VB-groupoid. Then the dual VB-groupoid $V^*\toto \mathrm{ann}(E)$, equipped with the dual weighting (\autoref{subsection: linear constructions}), is a weighted VB-groupoid. 
\end{proposition}
\begin{proof}
    Since $E$ is a weighted subbundle of $\mathcal{V}$, its annihilator is a weighted subbundle of $\mathcal{V}^*$ by~\autoref{lemma: annilator is a weighted subbundle}. Similarly, $\mathrm{ann}^\sharp(\mathrm{Gr}(\mathrm{mult}_\mathcal{V})) = \mathrm{Gr}(\mathrm{mult}_{\mathcal{V}^*})$ is a weighted subbundle of $(\mathcal{V}^*)^3$ (see~\autoref{A_remark: sharp annihilator}). To conclude, we recall that 
        \[ (\mathcal{V}^*|_H)_{(i)} = \mathrm{ann}((\mathcal{V}|_H)_{(-i+1)})\] 
    whence the filtration of $(T\mathcal{V}^*)|_H$ is by subgroupoids, since the annihilator of a VB-subgroupoid is itself a VB-subgroupoid by~\autoref{lemma: annhilator of VB-subgroupoid}.
\end{proof}

\section{The Weighted Normal and Deformation Groupoids}
\label{section: Weighted Normal and Deformation Groupoids}

Let $G\toto M$ be a Lie groupoid. The most naive definition of a multiplicative weighting would be one with the property that the groupoid structure $G\times \R^\times \toto M\times \R^\times$ extends to a groupoid structure $\defw(G, H) \toto \defw(M,N)$. In this section we will show that this is \emph{equivalent} to our definition of multiplicative weighting. 

\begin{lemma}
\label{lemma: extensions of morphisms}
    Suppose that $(M,N)$ and $(M', N')$ are weighted pairs and $F:M\to M'$ is a smooth map. Then $F$ is a weighted morphism if and only if the map 
        \[ \WT{F} : M\times \R^\times \to M'\times \R^\times, \quad (p,t) \mapsto (F(p), t) \]
    extends to a smooth map $\defw(M,N)\to \defw(M',N')$. 
\end{lemma}
\begin{proof}
    Given $f\in C^\infty(M)_{(i)}$ we will show that the function 
        \[ \WT{f}^{[j]} : M\times \R^\times \to \R \quad p\mapsto t^{-i}f(p) \]
    extends to $\defw(M,N)$ if and only if $i \geq j$. Working locally we may assume that $x_a$ is a global weighted coordinate system on $M$. Let 
        \[ y_a = \WT{x}_a^{[w_a]} = \left\{
            \begin{array}{ll}
                 t^{-w_a}x_a & t\neq 0 \\
                x_a^{[w_a]} & t=0
            \end{array}
        \right.\]
    be the corresponding (global) coordinates on $\defw(M,N)$. Using that $f\in C^\infty(M)_{(i)}$ we may write  
        \[ f(x_1, \dots, x_m) = \chi(x_1, \dots, x_m)x^s = \chi(x_1, \dots, x_m)x_1^{s_1}x_2^{s_2}\cdots x_m^{s_m} \]
    where $\chi \in C^\infty(M)$ and $s\cdot w = \sum_as_aw_a = i$. Then, for $t\neq 0$, we have 
        \[ \WT{f}^{[j]}(y, t) = t^{-j}f(x) = t^{-j}\chi(x)x^s = t^{s\cdot w-j}\chi(t^{w_1}y_1, \dots, t^{w_m}y_m)y^s \]
    which extends smoothly to $t=0$ if and only if $i-j = s\cdot w-j \geq 0$. 

    To see how this implies the lemma, $y_b\in C^\infty(M)_{(w'_b)}$ be weighted coordinates on $M'$. Then $F$ extends to $\defw(M, N)$ if and only if each of the component functions $f_b = F^*y_b$ extend to $\defw(M,N)$. By the previous argument, this happens if and only if each $f_b$ has weight $w'_b$, which is equivalent to $F$ being a weighted morphism. 
\end{proof}

We now state the main theorem of the chapter. 

\begin{theorem}
\label{theorem: weighted normal and weighted deformation groupoids}
    Let $G\toto M$ be a Lie groupoid which is weighted along the Lie subgroupoid $H\toto N$. If the weighting of $G$ along $H$ is multiplicative, then the Lie groupoid $G\times \R^\times \to M\times \R^\times$ extends uniquely to 
        \[ \defw(G,H) \toto \defw(M,N).  \]
    Conversely, if $(M,N)$ is a weighted pair and the groupoid structure on $G\times \R^\times \toto M\times \R^\times$ extends to $\defw(G,H)\toto \defw(M,N)$, then $(G,H)$ is a weighted Lie groupoid pair.   
\end{theorem}
\begin{proof}
    Suppose that the weighting of $G$ along $H$ is multiplicative. Since $M$ is a weighted submanifold of $G$, it follows from the discussion in~\autoref{subsection: properties of weighted deformation space} (b)  that $\defw(M,N)$ is a submanifold of $\defw(G,H)$. Since both $s, t:G \to M$ are weighted submersions, the maps $\defw(s_g), \defw(t_g):\defw(G, H)\to \defw(M,N)$ are submersions by~\autoref{theorem: Characterization of Weighted Immersions and Submersions}. By~\autoref{subsection: properties of weighted deformation space} (c) we have that $\defw(G, H)^{(2)} = \defw(G^{(2)}, H^{(2)})$, and therefore define groupoid multiplication by 
        \[ \mathrm{mult}_{\defw(G,H)} = \defw(\mathrm{mult}_G). \]
    Associativity of multiplication, that $\defw(M,N)$ are the units, and existence of inverses all follow from functoriality. For example, the fact that $\mathrm{mult}_{\defw(G,H)}(\varphi, s_{\defw(G,H)}(\varphi)) = \varphi$ for any $\varphi \in \defw(G,H)$ follows since the composition
        \begin{equation*}
        \xymatrixcolsep{3pc}
        \xymatrix{
            G \ar[r]^-{(\mathrm{id}, s)}  & G^{(2)} \ar[r]^-{\mathrm{mult}_G} & G  
        }        
        \end{equation*}  
    is the identity. Uniqueness of this extension is clear since $G\times \R^\times$ is a dense open set in $\defw(G,H)$. 

    Conversely, suppose that the Lie groupoid structure of $G\times \R^\times \toto M\times \R^\times$ extends to $\defw(G,H)\to \defw(M,N)$. By~\autoref{lemma: extensions of morphisms}, it follows that inversion is a weighted morphism and, using~\autoref{theorem: Characterization of Weighted Immersions and Submersions} in addition, the source and target maps are weighted submersions. Similarly, $M$ is a weighted submanifold of $G$. Since $s$ and $t$ are weighted submersions $G^{(2)}$ is a weighted submanifold of $G^2$, and $\defw(G, H)^{(2)} = \defw(G^{(2)}, H^{(2)})$ because this holds on the dense open set $G^2\times \R \sset \defw(G^2, H^2)$. It follows from~\autoref{lemma: extensions of morphisms} that $\mathrm{mult}_G$ is a weighted morphism as well. 
\end{proof}

\begin{remark}
    If $G\toto M$ is multiplicatively weighted along a \emph{wide} subgroupoid $H \toto M$ then $\defw(M,M) = M\times \R$ and we may think of 
        \[ \defw(G,H) \toto M\times \R \]
    as a family of Lie groupoids over $M$.
\end{remark}

Recall (\autoref{definition: GB groupoid}) that a GB-groupoid is a Lie groupoid together with a monoid action of $\R$ by groupoid morphisms. Applying the weighted normal functor to a weighted Lie groupoid gives an example of such an object.  

\begin{corollary}
    If $(G,H)\toto (M,N)$ is a weighted Lie groupoid pair, then 
        \[ \nuw(G,H) \toto \nuw(M,N) \]
    is a GB-groupoid with 
        \begin{enumerate*}
            \item $s_{\nuw(G,H)} = \nuw(s)$ and $t_{\nuw(G,H)} = \nuw(t)$, 
            \item $\nuw(G, H)^{(2)} = \nuw(G^{(2)}, H^{(2)})$,
            \item $\mathrm{mult}_{\nuw(G,H)} = \nuw(\mathrm{mult}_{G})$
            \item $\mathrm{inv}_{\nuw(G,H)} = \nuw(\mathrm{inv}_{G})$
        \end{enumerate*}
\end{corollary}

\begin{examples}
\begin{itemize}
    \item[(a)] If $G \toto M$ is trivially weighted along $H\toto N$, then $\nuw(G,H)$ is just $\nu(G, H) \toto \nu(M,N)$ with its usual Lie groupoid structure. 

    \item[(b)] If $(M,N)$ is weighted pair, then the product weighting of $\mathrm{Pair}(M)$ along $\mathrm{Pair}(N)$ is multiplicative and 
        \[ \defw(\mathrm{Pair}(M), \mathrm{Pair}(N)) = \mathrm{Pair}(\defw(M,N)) \]
    by the discussion in~\autoref{subsection: properties of weighted deformation space} (c).

    \item[(c)] If $(M,N)$ and $M',N')$ are weighted pairs, and $q:M\to M'$ is a weighted submersion, then submersion groupoid $M\times_{M'} M$ is multiplicatively weighted along $N\times_{N'}N$. In this case 
        \[ \defw(M\times_{M'} M, N\times_{N'}N) = \defw(M, N)\times_{\defw(M', N')} \defw(M, N), \]
    see~\autoref{subsection: properties of weighted deformation space} (c) again. 

    \item[(d)] Suppose that $(G,H)$ is a weighted Lie group pair and $(M,N)$ is a weighted $G$-space, then $\defw(G,H)$ acts on $\defw(M,N)$ and applying the weighted deformation functor to the action groupoid gives that action groupoid for the action of $\defw(G,H)$ on $\defw(M,N)$.

    \item[(e)] Let $G$ be a weighted Lie group and $P\to B$ a weighted principal $G$-bundle. Then $\defw(P, Q)\to \defw(B, C)$ is a weighted principal $\defw(G,H)$-bundle and 
        \[ \mathrm{At}(\defw(P, Q)) = \defw(\mathrm{At}(P), \mathrm{At}(Q)).  \]

    \item[(f)](\cite{van2017tangent, sadegh2018euler}) Recall that a filtered manifold is a manifold $M$ together with a filtration 
        \[ TM = F_{-r} \supseteq F_{-r+1} \supseteq \cdots \supseteq F_{-1} \supseteq 0 \]
    by subbundles satisfying $[\G(F)_{(-i)}, \G(F_{-j})]\sset \G(F_{-i-j})$. By~\autoref{theorem: singular Lie filtrations and weightings}, the product filtration 
        \[ \mathrm{Pair}(TM)  = \mathrm{Pair}(F_{-r})\supseteq \mathrm{Pair}(F_{-r+1}) \supseteq \cdots \supseteq \mathrm{Pair}(F_{-1}) \supseteq 0 \]
    defines a weighting of $\mathrm{Pair}(M)$ along $M$. Since the weighting is along $M$, it follows in particular that it is a weighted submanifold and it is shown in~\cite{van2017tangent} that the groupoid structure
        \[ \mathrm{Pair}(M)\times \R^\times \toto M\times \R^\times \]
    extends to $\defw(\mathrm{Pair}(M), M) \toto M\times \R^\times$. By~\autoref{theorem: weighted normal and weighted deformation groupoids} it follows that this weighting is multiplicative. In the next chapter we will give a direct proof of this fact (\autoref{theorem: wide integration}). In this example, $\nuw(\mathrm{Pair}(M), M)\toto M$ is the family of simply connected Lie groups integrating the family of nilpotent Lie algebras 
        \[ \ger{t}_HM = \bigoplus H_{-i}/H_{-i+1}\to M. \]

\end{itemize}
\end{examples}


    

%% file: 4_IM_weightings.tex
\chapter{Infinitesimally Multiplicative Weightings}
\label{chapter: Infinitesimally Multiplicative Weightings}

The infinitesimal counterpart a Lie groupoid is a \emph{Lie algebroid}. In this chapter we will define \emph{infinitesimally multiplicative} weightings (\autoref{definition: IM weighting}), which are the infinitesimal analogue of multiplicative weightings. We then explain two equivalent characterizations in terms of linear Poisson manifolds and the de Rham complex of the Lie algebroid (\autoref{theorem: equvalent characterizations of IM weightings}). We will then discuss differentiation of multiplicative weightings (\autoref{theorem: weightings can be differentiated}) and integration of infinitesimally multiplicative weightings along wide subalgebroids (\autoref{theorem: wide integration}), and use this to show that Lie groupoids with a multiplicative weighting along their units are in 1-1 correspondence with filtered groupoids in the sense of van Erp and Yuncken (cf.~\cite[Definition 67]{van2019groupoid}).  

\section{Preliminaries on Lie Algebroids}

Lie algebroids were introduced by Pradines in~\cite{pradines1967theorie}. 

\begin{definition}
    A \emph{Lie algebroid} over a smooth manifold $M$ is a vector bundle $A\to M$ together with a Lie bracket $[\cdot, \cdot]:\G(A)\times \G(A) \to \G(A)$ and a vector bundle morphism $a:A\to TM$, called the \emph{anchor}, such that 
        \[ [\sigma, f\tau] = f[\sigma, \tau] + \mathcal{L}_{a(\sigma)}f\cdot \tau  \]
    for all $\sigma, \tau \in \G(A)$ and $f\in C^\infty(M)$. 
\end{definition}

We use the notation $A\Rightarrow M$ to imply that the vector bundle $A\to M$ is a Lie algebroid. 

\begin{examples}
\label{examples: Lie algebroids}
\begin{itemize}
    \item[(a)] The tangent bundle $TM$ of any smooth manifold $M$ is a Lie algebroid with anchor map given by the identity. 

    \item[(b)] A Lie algebroid over a point is a Lie algebra. 

    \item[(c)] The tangent bundle $T_\mathcal{F}M$ of a foliation $\mathcal{F}$ on $M$ is a Lie algebroid with anchor map given by the identity. 

    \item[(d)] If $\kappa:P\to M$ is a principal $G$-bundle, then $\mathrm{At}(P) = TP/G \to M$ is a Lie algebroid with Lie bracket induced by the Lie bracket on $G$-invariant vector fields on $P$ and anchor map induced by $T\kappa:TP \to TM$.
    
    In the special case the $P$ is the frame bundle of a vector bundle $V\to M$, we may identify $\G(\mathrm{At}(P))$ with the \emph{linear vector fields} on $V$, i.e. the vector fields $X\in \ger{X}(V)$ with the property that 
        \[ \mathcal{L}_X : C^\infty_{[n]}(V) \to C^\infty_{[n]}(V) \]
    for all $n\geq 0$, where $C^\infty_{[n]}(V)$ denotes the functions on $V$ which are fibrewise homogeneous of degree $n$ (see~\autoref{subsection: graded bundles}). 

    \item[(e)] A \emph{family of Lie algebras} is a vector bundle $A\to M$ with a fibre-wise Lie bracket. A family of Lie algebras is a Lie algebroid over $M$ by taking the anchor map to be zero. Conversely, any Lie algebroid with zero anchor map is given in this way. 

    \item[(f)] Let $\omega\in \Omega^2(M)$ be a 2-form on a manifold $M$ and consider the vector bundle $A = TM\times \R \to M$. Define a bracket on $\G(A) = \ger{X}(M)\oplus C^\infty(M)$ by
        \[ [X+f, Y+g] = [X,Y] + \mathcal{L}_Xg - \mathcal{L}_Yf + \omega(X,Y) \]
    and let $a:TM\times \R \to TM$ be the projection onto the first factor. Then $(A, [\cdot, \cdot], a)$ is a Lie algebroid if and only if $\omega$ is closed. 
\end{itemize}
\end{examples}

For our work on weighted Lie algebroids, we will need two equivalent characterizations of Lie algebroids. 

\subsection{Lie algebroids as linear Poisson manifolds}
\label{subsection: Lie algebroids as linear Poisson manifolds}

Recall that a Poisson structure on a manifold $M$ is a is skew-symmetric bilinear map $\{\cdot, \cdot\} : C^\infty(M) \times C^\infty(M) \to C^\infty(M)$ such that, for all $f,g,h \in C^\infty(M)$,
    \begin{itemize}
        \item[(i)] the derivation property: $\{ f,gh\} = \{f,g\}h + g\{f,h\}$, and
        \item[(ii)] the Jacobi identity: $\{ f,\{g,h\}\} = \{\{f,g\}, h\} + \{g,\{f,h\}\}$. 
    \end{itemize}
A Poisson structure on $M$ defines a bivector field $\pi \in \mathfrak{X}^2(M) = \G(\wedge^2 TM)$ by the equation
    \begin{equation}
    \label{equation: bivector field - poisson relation}
        \pi(\ed f, \ed g) = \{f,g\}.
    \end{equation}
Conversely, a bivector field $\pi \in \mathfrak{X}^2(M)$ defines a skew-symmetric bracket on $C^\infty(M)$ satisfying the derivation property by~\eqref{equation: bivector field - poisson relation}; it is called a \emph{Poisson bivector field} if the induced bracket on $C^\infty(M)$ satisfies the Jacobi identity. If $V$ is a vector bundle, then a Poisson structure on $V$ is \emph{linear} if $C^\infty_{[n]}(V)$ is a subalgebra of $C^\infty(V)$ for any $n\geq 0$. Equivalently, $\pi \in \mathfrak{X}^2_{[-1]}(V)$.

\begin{example}
\label{example: symplectic manifolds are poisson}
    Let $\omega \in \Omega^2(M)$ be a symplectic form on $M$. Since $\omega$ is non-degenerate it defines an isomorphism 
        \[ \omega^\flat : TM \to T^*M, \quad X \mapsto \iota_X\omega. \]
    Let $\pi^\sharp : T^*M \to TM$ be the inverse map and, given $f\in C^\infty(M)$, let $X_f \in \ger{X}(M)$ denote the Hamiltonian vector field 
        \[ X_f = -\pi^\sharp(\ed f). \]
    Then 
        \[ \{f,g\} = \omega(X_f, X_g) \]
    defines a Poisson structure on $M$. In this way, every symplectic manifold is canonically a Poisson manifold. If $x_1, \dots, x_m$, $y_1, \dots, y_m$ are local Darboux coordinates for $M$ so that $\omega = \sum_i \ed x_i \wedge \ed y_i$, then the corresponding Poisson bivector field is given by 
        \begin{equation}
        \label{equation: local Poisson bivector field}
            \sum_i \frac{\bd}{\bd x_i}\wedge \frac{\bd}{\bd y_i}.
        \end{equation}
    In particular, the cotangent bundle $T^*M$ of \emph{any} manifold $M$ is canonically a Poisson manifold. Moreover, from~\eqref{equation: local Poisson bivector field} we deduce that it is a \emph{linear} Poisson manifold. 
\end{example}

Given a section $\sigma \in \G(V)$ let $f_\sigma \in C^\infty_{[1]}(V^*)$ be the corresponding linear function. The following theorem explains how linear Poisson structures are related to Lie algebroid structures.  

\begin{theorem}~\cite[Theorem 2.1.4]{courant1990dirac}
\label{theorem: Poisson-Lie algebroid equivalence}
    For any Lie algebroid $A\Rightarrow M$, the total space of the dual bundle $p: A^* \to  M$ has a unique Poisson structure such that for all sections $\sigma, \tau \in \G(A)$,
        \[ \{f_\sigma, f_\tau \} = f_{[\sigma,\tau]}. \]
    The anchor map is described in terms of the Poisson bracket as
        \[ p^*(a(\sigma)f) = \{f_\sigma, p^*f\}, \]
    for $f \in C^\infty(M)$ and $\sigma \in \G(A)$, while $\{p^*f, p^*g\} = 0$ for all functions $f$, $g$. The Poisson structure on $A^*$ is linear; conversely, every linear Poisson structure on a vector bundle $V \to M$ arises in this way from a unique Lie algebroid structure on the dual bundle $V^*$.
\end{theorem}

\begin{example}
    Recall from~\autoref{examples: Lie algebroids} (a) that the tangent bundle $TM$ of any manifold $M$ is a Lie algebroid. The Poisson structure on $T^*M$ described in~\autoref{theorem: Poisson-Lie algebroid equivalence} is the one corresponding to the canonical symplectic structure, as described in~\autoref{example: symplectic manifolds are poisson}. 
\end{example}

\subsection{Super-geometric characterization of Lie algebroids}
\label{subsection: Super-geometric characterization of Lie algebroids}

A Lie algebroid structure on a vector bundle $A \to M$ induces a differential graded algebra structure on $\G(\wedge A^*)$ by setting 
    \begin{align}
    \label{equation: algebroid differential}
    \begin{split}
        (\ed_A \omega)(\sigma_1, \dots, \sigma_{k})  =  \sum_{i<j} & (-1)^{i+j} \omega([\sigma_i, \sigma_j], \sigma_1, \ldots, \widehat{\sigma_i}, \ldots, \widehat{\sigma_j}, \ldots, \sigma_{k}) \\
        & + \sum_{i=1}^{k}(-1)^{i+1}\mathcal{L}_{a(\sigma_i)}\omega(\sigma_1, \ldots, \widehat{\sigma_i}, \ldots, \sigma_{k}),
    \end{split}
    \end{align}
where $\omega \in \G(\wedge^kA^*)$ and where the hat denotes omission. 

Vaintrob~\cite{vaintrob1997lie} observed that the converse is also true. If $A\to M$ is a vector bundle and $\ed_A:\G(\wedge A^*) \to \G(\wedge A^*)$ is a degree $+1$ derivation satisfying $\ed_A\circ \ed_A = 0$, then we can define a Lie algebroid structure on $A$ as follows. For $\sigma \in \G(A)$, let $\iota_\sigma : \G(\wedge A^*) \to \G(\wedge A^*)$ be contraction by $\sigma$ and let $\mathcal{L}_\sigma = \ed_A \circ \iota_\sigma + \iota_\sigma \circ \ed_A$. Given $\sigma_1, \sigma_2 \in \G(A)$, define $[\sigma_1, \sigma_2]\in \G(A)$ by the formula 
    \[ \iota_{[\sigma_1, \sigma_2]} = [\mathcal{L}_{\sigma_1}, \iota_{\sigma_2}]. \]
Finally, define the map $a:A\to TM$ by the equation
    \[ \mathcal{L}_{a(\sigma)}f = \mathcal{L}_\sigma f, \]
where on the right hand side we are identifying $\G(\wedge^0A^*)$ with $C^\infty(M)$. Then $(A, [\cdot, \cdot], a)$ is a Lie algebroid over $M$. Moreover, these two constructions are inverse to one another. Summarizing, we have:

\begin{theorem}[\cite{vaintrob1997lie}]
\label{theorem: supergeometric description of Lie algebroids}
    Let $A \to M$ be a vector bundle, and let $\ed_A$ be a differential on $\G(\wedge A^*)$, that is, $\ed_A$ is a derivation of degree 1 with $\ed_A \circ \ed_A = 0$. Then $\ed_A$ determines a unique Lie algebroid structure on $A$ for which $\ed_A$ is the de Rham differential.
\end{theorem}

This approach has the benefit that Lie algebroid morphisms are readily defined. 

\begin{definition}
    Let $A\Rightarrow M$ and $B\Rightarrow N$ be Lie algebroids. A vector bundle morphism $\varphi:B\to A$ is a \emph{Lie algebroid} morphism if the pullback map $\G(\wedge A^*) \to \G(\wedge B^*)$ is a cochain map.  
\end{definition}

See~\cite[Section 12]{meinrenken2017Lie} for more details.

\begin{remark}
    In super-geometric terms, the graded commutative algebra $\G(\wedge A^*)$ is understood as the smooth functions on a "super-manifold" denoted\footnote{Not to be confused with the shifted weighting, as described in~\autoref{subsection: linear constructions} (c).} $A[1]$. The differential $\ed_A$ is a derivation of this graded commutative algebra, hence can be thought of as a vector field $H$ on $A[1]$. The condition that $\ed_A \circ \ed_A = 0$ is equivalent to $[H,H] = 2H^2 = 0$. Such a vector field is called a \emph{homological vector field}. One may therefore restate~\autoref{theorem: supergeometric description of Lie algebroids} as saying that a Lie algebroid is a vector bundle $A\to M$ together with a homological vector field $H$.
\end{remark}

\subsection{The Lie functor}

We now explain how to associate a Lie algebroid to a Lie groupoid $G \toto M$; our approach is somewhat non-standard, so for more details see~\cite[Section 11]{meinrenken2017poisson}. As a vector bundle, let $\mathrm{Lie}(G) = \nu(G,M) \to M$ be the normal bundle of the units. Recall that the tangent bundle $TG$ is a VB-groupoid, hence dually we obtain the cotangent groupoid
    \begin{equation*}
    \xymatrix{
        T^*G \ar@<-.5ex>[r] \ar@<.5ex>[r] \ar[d] & \mathrm{Lie}(G)^* \ar[d] \\
        G \ar@<-.5ex>[r] \ar@<.5ex>[r] & M
    }
    \end{equation*}
The canonical Poisson structure on $T^*G$ defines a Poisson structure on $\mathrm{Lie}(G)^*$ by the formula 
    \begin{equation}
    \label{equation: Libermann Poisson sructure}
        \{ s_{T^*G}^*f, s_{T^*G}^*g\}_{T^*G} = s_{T^*G}^*\{f,g\}_{\mathrm{Lie}(G)^*},
    \end{equation}
which is linear since the Poisson structure on $T^*G$ is linear. By~\autoref{theorem: Poisson-Lie algebroid equivalence}, this gives $\mathrm{Lie}(G)$ the structure of a Lie algebroid. 

We can describe this Lie algebroid structure more concretely as follows. Since the source and target maps $s_G, t_G:G\to M$ agree on $M$, it follows that $Ts-Tt:TG \to TM$ vanishes on $TM$ and therefore descends to a map $a:\mathrm{Lie}(G) \to TM$; we take this to be the anchor map. To describe the bracket on sections, let $\sigma \in \G(\mathrm{Lie}(G))$ and let $f_\sigma \in C^\infty_{[1]}(\mathrm{Lie}(G)^*)$ be the corresponding linear function. Since $s_{T^*G} : T^*G \to \mathrm{Lie}(G)^*$ is a vector bundle morphism, the pullback $s_{T^*G}^*f_\sigma \in C^\infty_{[1]}(T^*G)$ is a linear function and therefore corresponds to a vector field $\sigma^L \in \mathfrak{X}(G)$; it is called the \emph{left-invariant extension} of $\sigma \in \G(\mathrm{Lie}(G))$. The Lie bracket on $\G(\mathrm{Lie}(G))$ is the determined by the equation
    \[ [\sigma^L, \tau^L] = [\sigma, \tau]^L.  \]
Similarly, we define the \emph{right-invariant} extension of $\sigma \in \G(\mathrm{Lie}(G))$ to be the vector field $\sigma^R \in \ger{X}(G)$ corresponding to $-t_{T^*G}^*f_\sigma \in C^\infty_{[1]}(T^*G)$. 

\begin{example}[{\cite[Example 11.3.1]{mackenzie2005general}}]
    Let $G = \mathrm{Pair}(M) \toto M$. In this case, we identify $\mathrm{Lie}(\mathrm{Pair}(TM))$ with $TM$ using the map
        \[ TM\oplus TM \to TM, \quad (X,Y)\mapsto Y-X.  \]
    The cotangent groupoid of $\mathrm{Pair}(M)$ is the VB-groupoid 
        \begin{equation*}
        \xymatrix{
            T^*\mathrm{Pair}(M) \ar@<-.5ex>[r] \ar@<.5ex>[r] \ar[d] & T^*M \ar[d] \\
            \mathrm{Pair}(M) \ar@<-.5ex>[r] \ar@<.5ex>[r] & M,
        }
        \end{equation*}
    where 
        \[ s_{T^*\mathrm{Pair}(M)}(\xi, \tau) = \tau \quad \text{and} \quad t_{T^*\mathrm{Pair}(M)}(\xi, \tau) = -\xi.  \]
    Given a vector field $X\in \ger{X}(M)$, we have that 
        \[ s^*_{T^*\mathrm{Pair}(M)}f_X = f_{(0,X)} \quad \text{and} \quad t^*_{T^*\mathrm{Pair}(M)}f_X = f_{(-X,0)}, \]
    hence $X^L = (0,X)\in \ger{X}(\mathrm{Pair}(M))$ and $X^R = (X, 0)\in \ger{X}(\mathrm{Pair}(M))$. 
\end{example}

For reference, we record the following fact. 

\begin{lemma}
\label{lemma: left and right invariant relations} 
    Let $G\toto M$ be a Lie groupoid with Lie algebroid $\mathrm{Lie}(G)\Rightarrow M$. Then for all $\sigma, \tau \in \G(\mathrm{Lie}(G))$:
        \begin{itemize}
            \item[(a)] $a(\sigma) \sim_i \sigma^L-\sigma^R$, where $i:M\to G$ is the inclusion, and 
            \item[(b)] $[\sigma^L, \tau^R] = 0$ and $[\sigma^R, \tau^R] = -[\sigma, \tau]^R$. 
        \end{itemize}
\end{lemma}
    
\section{Definition of Infinitesimally Multiplicative Weightings}
\label{section: IM weighting definition}

Let $A \Rightarrow M$ be a Lie algebroid.

\begin{definition}
\label{definition: IM weighting}
    An \emph{infinitesimally multiplicative weighting} of $A\Rightarrow M$ is a linear weighting of $A$ with the additional properties that 
        \begin{itemize}
            \item[(a)] $a(\G(A)_{(i)}) \sset \mathfrak{X}(M)_{(i)}$ for all $i$, and 
            \item[(b)] for all $\sigma \in \G(A)_{(i)}$ and $\tau \in \G(A)_{(j)}$, we have 
                \[ [\sigma, \tau] \in \G(A)_{(i+j)}.  \]
        \end{itemize}
    A Lie algebroid endowed with an infinitesimally multiplicative weighting is called a \emph{weighted Lie algebroid}. 
\end{definition}

\begin{examples}
\label{examples: weighted Lie algebroids}
\begin{itemize}
    \item[(a)] If $B\Rightarrow N$ is a Lie subalgebroid of $A\Rightarrow M$, then the trivial weighting of $A$ along $B$ is infinitesimally multiplicative. 
    
    \item[(b)] If $(M,N)$ is a weighted pair, then $TM$ with its induced weighting is a weighted Lie algebroid.

    \item[(c)] If $\pi:M\to N$ is a weighted submersion, then $\ker(T\pi)$ is infinitesimally weighted via the filtration
        \[ \G(\ker(T\pi))_{(i)} = \{X\in \ger{X}(M)_{(i)} : X\sim_\pi 0 \}. \]

    \item[(d)] Recall from~\autoref{remark: wide weighting and filtered vector bundles} that a linear weighting of a vector space is simply a filtration by subspaces. Hence, an infinitesimally multiplicative weighting of a Lie algebra $\mathfrak{g}$ is given by a filtration of $\ger{g}$ by subspaces $\mathfrak{g}_{i}$ with the property that 
        \begin{equation}
        \label{equation: filtration of lie alg}
            [\mathfrak{g}_i, \mathfrak{g}_j] \sset \mathfrak{g}_{i+j}.
        \end{equation}
    If $(G, H)$ is a weighted Lie group pair then the filtration of $TG|_H$ restricts to a filtration of $\ger{g}$ satisfying~\eqref{equation: filtration of lie alg}. Hence the Lie algebra of a weighted Lie group inherits an infinitesimally multiplicative weighting. 

    \item[(e)] An action of a weighted Lie algebra $\ger{g}$ on a weighted manifold $M$ is called \emph{weighted} if the action map 
        \[ \ger{g}\times M  \to TM, \quad (\xi, p)\mapsto \xi_M(p) \]
    is a morphism of weighted vector bundles. In this case, the \emph{action algebroid}
        \[ \ger{g}\times M \Rightarrow M \]
    is a weighted Lie algebroid. 

    \item[(f)] Any linearly weighted vector bundle can be thought of as a weighted Lie algebroid by taking both the anchor and bracket to be zero. 

    \item[(g)] If $G$ is a weighted Lie group and $P\to M$ is a principal $G$-bundle with a principal weighting along $Q\to N$, then the Atiyah algebroid $\mathrm{At}(P)\Rightarrow M$ has a canonical Lie algebroid weighting along $\mathrm{At}(Q)\Rightarrow N$. The filtration of the sections is defined by the identification with $G$-invariant vector fields on $P$. Similarly, the Lie algebroid $\ger{gau}(P) = (P\times \ger{g})/G \Rightarrow M$ has a canonical infinitesimally multiplicative weighting along $\ger{gau}(Q)\Rightarrow N$. Moreover,  
        \[ 0 \corr{} \ger{gau}(P) \corr{} \mathrm{At}(P) \corr{} TM \corr{} 0 \]
    is an exact sequence of weighted vector bundles in the sense that one can find a splitting $TM\to \mathrm{At}(P)$ which is a weighted vector bundle morphism.

    \item[(h)] An infinitesimally multiplicative weighting of a Lie algebroid $A \Rightarrow M$ along the zero subalgebroid is given by a filtration 
        \[A = F_{-r} \supseteq F_{-r+1} \supseteq \cdots \supseteq 0 \]
    such that $[\Gamma(F_{-i}), \Gamma(F_{-j})] \subseteq \Gamma(F_{-i-j})$. If $A = \mathrm{Lie}(G)$ for some Lie groupoid $G$, then this is a \emph{Lie filtration} of $G$, in the sense of van Erp and Yuncken (cf.~\cite[Definition 17]{van2017tangent}). In particular, a Lie filtration of $M$ is a infinitesimally multiplicative weighting of $TM$ along the zero section.

    \item[(i)] Let $(M,N)$ be a weighted pair and let $\omega \in \Omega^2(M)$ be closed. Let $A = TM\times \R \Rightarrow M$ be the Lie algebroid defined in~\autoref{examples: Lie algebroids} (f), with 
        \[ [X+f, Y+g] = [X,Y] +\mathcal{L}_Xf - \mathcal{L}_{Y}g + \omega(X,Y) \]
    Then $\G(A)_{(i)} = \mathfrak{X}(TM)_{(i)}\oplus C^\infty(M)_{(i)}$ is an infinitesimally multiplicative weighting if and only if $\omega \in \Omega^2(M)_{(0)}$. 
\end{itemize}
\end{examples}

Let $A\Rightarrow M$ be a weighted Lie algebroid. Given a frame weighted $\sigma_a$ for $A|_U$ let $\G_{ab}^c \in C^\infty(U)$ be the corresponding structure functions for the Lie algebroid, defined by 
    \[ [\sigma_a, \sigma_b] = \sum_c \G_{ab}^c \sigma_c.  \]
Since $\sum_c \G_{ab}^c \sigma_c \in \G(V|_U)_{(v_a+v_b)}$ if and only if $\G_{ab}^c \in C^\infty(U)_{(v_a+v_b-v_c)}$, we have the following lemma. 

\begin{lemma}
\label{lemma: structure constants for Lie algebroids}
    A linear weighting of a Lie algebroid $A\Rightarrow M$ is infinitesimally multiplicative if and only if the anchor map is a weighted vector bundle morphism and for any choice of weighted frame $\sigma_a\in \G(A|_U)_{(v_a)}$ the associated structure functions satisfy $\G_{ab}^c \in C^\infty(U)_{(v_a+v_b-v_c)}$. 
\end{lemma}

\begin{remark}
    For a weighted Lie algebra $\ger{g}$ with structure constants $\G_{ab}^c$, this says that $\G_{ab}^c = 0$ whenever $v_a+v_b > v_c$. 
\end{remark}

\section{Alternative Characterizations of Lie Algebroid Weightings}
\label{section: IM characterizations}

We now explain the weighted analogue of the characterizations of Lie algebroids described in~\autoref{subsection: Lie algebroids as linear Poisson manifolds} and~\autoref{subsection: Super-geometric characterization of Lie algebroids}. The latter perspective allows for quick definition of weighted Lie algebroid morphisms, whereas the former will be used in our proof of~\autoref{theorem: weightings can be differentiated}, which says that the Lie algebroid of a weighted Lie groupoid has a canonical infinitesimally multiplicative weighting. 

Before we begin, we remind the reader of~\autoref{subsection: linear constructions} and~\autoref{theorem: linear weightings in terms of polynomials}, where we explain how a linear weighting of $V\to M$ defines a linear weighting of $V^*\to M$ and how this can be described in terms of a filtration of $C^\infty_{pol}(V^*)$. 

\begin{theorem}
\label{theorem: equvalent characterizations of IM weightings}
    Let $A\Rightarrow M$ be a Lie algebroid endowed with a linear weighting. Then the following are equivalent. 
        \begin{itemize}
            \item[(a)] $A$ is a weighted Lie algebroid, 
            \item[(b)] the Poisson bivector field $\pi \in \ger{X}^2(A^*)$ has filtration degree zero, 
            \item[(c)] the differential $\ed_A : \G(\wedge A^*) \to \G(\wedge A^*)$ is filtration preserving. 
        \end{itemize}
\end{theorem}

\begin{proof}[Proof that (a) and (b) are equivalent]
    Let $\{\cdot, \cdot \}_\pi$ be the Poisson bracket defined by $\pi$. By considering weighted vector bundle coordinates we see that $\pi \in \mathfrak{X}^2_{[-1]}(A^*)$ has filtration degree zero if and only if 
        \begin{equation}
        \label{equation: poisson bracket preserves filtration}
            \{\cdot, \cdot \}_\pi : C^\infty_{[n]}(A^*)_{(i)}\times C^\infty_{[n]}(A^*)_{(j)} \to C^\infty_{[n]}(A^*)_{(i+j)} \quad 
        \end{equation}
    for n=0,1.
    
    We start with $n=1$. Given $\sigma \in \G(A)_{(i)}$ and $\tau \in \G(A)_{(j)}$, let $f_\sigma \in C^\infty_{[1]}(V^*)_{(i)}$ and $f_\tau \in C^\infty_{[1]}(V^*)_{(j)}$ be the corresponding linear functions.  The equation 
        \[ \{f_\sigma, f_\tau \}_\pi = f_{[\sigma,\tau]} \]
    shows that $[\sigma, \tau] \in \G(V)_{(i+j)}$ if and only if $\{f_\sigma, f_\tau \}_\pi\in C^\infty_{[1]}(V^*)_{(i+j)}$, i.e.~\eqref{equation: poisson bracket preserves filtration} holds for $n=1$. 

    Now let $f\in C^\infty(M)_{(j)}$. Since the bundle projection $p:A^*\to M$ is a weighted morphism and $M$ is a weighted submanifold of $A^*$, the equation 
        \[p^*(\mathcal{L}_{a(\sigma)}f) = \{f_\sigma, p^*f\}_{\pi}\]
    shows that $\{f_\sigma, p^*f\}_{\pi} \in C^\infty_{[1]}(V^*)_{(i+j)}$ if and only if $\mathcal{L}_{a(\sigma)}f \in C^\infty(M)_{(i+j)}$. Since $f\in C^\infty(M)_{(j)}$ was arbitrary, the latter holds if and only if $a(\sigma) \in \ger{X}(M)_{(i)}$. Thus~\eqref{equation: poisson bracket preserves filtration} holds for $n=0$ if and only if $a:\G(V)_{(i)} \to \ger{X}(M)_{(i)}$, which completes the proof. 
\end{proof}

\begin{proof}[Proof that (a) and (c) are equivalent]
    Suppose that $A$ is a weighted Lie algebroid and let $\sigma_j \in \Gamma(V)_{(i_j)}$ and $\omega \in \G(\wedge^kA^*)_{(i)}$. Since $[\sigma_{\ell}, \sigma_j] \in \Gamma(V)_{(i_\ell + i_j)}$ and $a(\sigma_j) \in \ger{X}(M)_{(i_j)}$ it follows that both 
        \begin{align*}
            \omega([\sigma_\ell, \sigma_j], \sigma_1, \dots, \hat{\sigma}_\ell, \dots, \hat{\sigma}_j, \dots, \sigma_k)  \quad \text{and} \quad 
            \mathcal{L}_{a(\sigma_\ell)}\omega(\sigma_1, \dots, \hat{\sigma}_\ell, \dots, \sigma_{k+1})
        \end{align*}
    are in $C^\infty(M)_{(i + i_1 + \cdots + i_{k+1})}$. It follows by~\autoref{equation: algebroid differential} that 
        \[(\ed_A \omega)(\sigma_1, \dots, \sigma_{k+1}) \in C^\infty(M)_{(i + i_1 + \cdots + i_{k+1})}, \] 
    hence $\ed_A$ is filtration preserving. 

    For the converse, recall that the weighting of $\wedge A^*$ is defined so that $\iota_\sigma : \G(\wedge A^*)_{(j)}\to \G(\wedge A^*)_{(i+j)}$ if and only if $\sigma \in \G(A)_{(i)}$. In particular, if $\ed_A$ is filtration preserving then 
        \[ \mathcal{L}_\sigma = \ed_A \circ \iota_\sigma + \iota_\sigma \circ \ed_A : \G(\wedge A^*)_{(j)}\to \G(\wedge A^*)_{(i+j)} \]
    whenever $\sigma \in \G(A)_{(i)}$. Thus, given $\sigma_1 \in \G(A)_{(i)}$ and $\sigma_2 \in \G(A)_{(j)}$, the formula $\iota_{[\sigma_1, \sigma_2]} = [ \mathcal{L}_{\sigma_1}, \iota_{\sigma_2}]$ implies that $[\sigma_1, \sigma_2] \in \G(A)_{(i+j)}$. Similarly, given $\sigma \in \G(A)_{(i)}$, the formula $\mathcal{L}_{a(\sigma)}f = \mathcal{L}_\sigma f$ implies that $a(\sigma) \in \ger{X}(M)_{(i)}$.
\end{proof}

As promised, we give the following definition. 

\begin{definition}
    Let $A\Rightarrow M$ and $B\Rightarrow N$ be weighted Lie algebroids. A weighted Lie algebroid morphism $\varphi:B\to A$ is a vector bundle morphism such that  pullback map $\G(\wedge A^*) \to \G(\wedge B^*)$ is a filtration preserving cochain map. 
\end{definition}


\section{The Weighted Normal and Weighted Deformation Algebroids}
\label{section: IM normal and deformation algebroids}

We now explain how the weighted normal bundle and weighted deformation bundles of a weighted Lie algebroid $A\Rightarrow M$ are naturally Lie algebroids. 

\begin{theorem}
\label{theorem: weighted normal and weighted deformation algebroids}
    Let $A\Rightarrow M$ be a weighted Lie algebroid. Then both $\nuw(A) \to \nuw(M,N)$ and $\defw(A)\to \defw(M,N)$ are naturally Lie algebroids and the isomorphism 
        \[ \defw(A)|_{\pi_\delta^{-1}(t)} \to \left\{
            \begin{array}{ll}
                 A & t\neq 0 \\
                 \nuw(A) & t = 0 
            \end{array}
        \right.\]
    is an isomorphism of Lie algebroids.
\end{theorem}
\begin{proof}
        Recall from~\autoref{examples: examples of weighted normal bundles for linear weightings} (c) that $\nuw(TM) = T\nuw(M,N)$. Since $a:A\to TM$ is a weighted morphism, $\nuw(a):\nuw(A) \to \nuw(TM)$ is defined and we take this to be our anchor for $\nuw(A)$. We define the Lie bracket as follows. Given $\sigma \in \G(A)_{(i)}$ and $\tau \in \G(V)_{(j)}$ define 
            \begin{equation}
            \label{equation: Lie bracket on weighted normal groupoid}
                [\sigma^{[i]}, \tau^{[j]}] = [\sigma, \tau]^{[i+j]};
            \end{equation}
        this defines a Lie bracket $\gr(\G(A))\times \gr(\G(A))\to \gr(\G(A))$. 
        Given $f \in C^\infty(M)_{(k)}$, we compute 
            \begin{align*}
                [\sigma^{[i]}, f^{[k]}\tau^{[j]}] & = [\sigma^{[i]}, (f\tau)^{[j+k]}] \\
                & = [\sigma, f\tau]^{[i+j+k]} \\
                & = (f[\sigma, \tau] + \mathcal{L}_{a(\sigma)}f\cdot\tau)^{[i+j+k]} \\
                & = f^{[k]}[\sigma, \tau]^{[i+j]} + \mathcal{L}_{a(\sigma)^{[i]}}f^{[k]}\cdot \tau^{[j]} \\
                & = f^{[k]}[\sigma, \tau]^{[i+j]} + \mathcal{L}_{\nuw(a)(\sigma^{[i]})}f^{[k]}\cdot \tau^{[j]},
            \end{align*}
        hence~\eqref{equation: Lie bracket on weighted normal groupoid} satisfies the Leibniz rule for any $f\in C^\infty_{pol}(\nuw(M,N)) = \gr(C^\infty(M))$. Recall from~\autoref{theorem: sections of the weighted normal bundle} that 
            \[ \G(\nuw(A)) = C^\infty(\nuw(M,N))\otimes_{\gr(C^\infty(M))}\gr(\G(A)),\] 
        hence we may extend~\eqref{equation: Lie bracket on weighted normal groupoid} to $\G(\nuw(A))$ using the Leibniz rule, which gives the Lie algebroid structure for $\nuw(A)$. An analogous construction gives the Lie algebroid structure for $\defw(A)$. The last statement is clear. 
\end{proof}

\begin{remarks}
    \begin{itemize}
        \item[(a)] If $\G_{ij}^k \in C^\infty(U)_{(v_i+v_j-v_k)}$ are the local structure functions for $A|_U$ corresponding to a local weighted frame, then 
            \[ (\G_{ij}^k)^{[v_i+v_j-v_k]} \quad \text{and} \quad \WT{(\G_{ij}^k)}^{[v_i+v_j-v_k]} \]
        are the local structure functions for $\nuw(A)$ and $\defw(A)$, respectively. 

        \item[(b)] If $\pi\in \ger{X}^2(A^*)_{(0)}$ is the Poisson bivector field then $\pi^{[0]} \in \ger{X}^2(\nuw(A)^*)$ and $\WT{\pi}^{[0]}\in \ger{X}^2(\defw(A)^*)$ are the Poisson bivector fields for $ \nuw(A)^*$ and $\defw(A)^*$, respectively. . 

        \item[(c)] The de Rham differentials for $\nuw(A)$ and $\defw(A)$ are defined by the equations 
            \[ \ed_{\nuw(A)}\omega^{[i]} = (\ed_A \omega)^{[i]} \quad \text{and} \quad \ed_{\defw(A)}\WT{\omega}^{[i]} = \WT{(\ed_A \omega)}^{[i]}, \]
        respectively. 

        \item[(d)] By the last remark, we see that this construction is functorial for weighted Lie algebroid morphisms. 
    \end{itemize}
\end{remarks}

\begin{examples}
    \begin{itemize}
        \item[(a)] Let $A = A_{-r}\oplus \cdots \oplus  A_{0} \Rightarrow M$ be a graded Lie algebroid, i.e. $[A_{i}, A_{j}] \sset A_{i+j}$. If $M$ is weighted along itself then the filtration 
            \[ \G(A)_{(-i)} = \G(A_{-i}) \oplus \cdots \oplus \G(A_{0})  \]
        defines an infinitesimally multiplicative weighting. 

        In particular, we can view any Lie algebroid $A\Rightarrow M$ as a graded Lie algebroid by setting $A = A_{-r}$. The corresponding bracket on $\defw(A)$ is given by 
            \[ [\WT{\sigma}^{[-r]}, \WT{\tau}^{[-r]}] = \WT{[\sigma, \tau]}^{[-2r]} = t^{r}[\sigma, \tau].  \]
        
        \item[(b)] Let $A\Rightarrow M$ be a Lie algebroid with Lie filtration 
            \[ A = A_{-r}\supseteq A_{-r+1} \supseteq \cdots \supseteq A_{-1} \supseteq 0. \] 
        This defines an infinitesimally multiplicative weighted of $A$ and the weighted normal algebroid is given by $\gr(A) = \bigoplus_{i=1}^r A_{-i}/A_{-i+1}$. Since $\ger{X}(M) = \ger{X}(M)_{(0)}$, then anchor map for $\nuw(A)$ is zero and $\nuw(A)$ is a family of nilpotent Lie algebras. 

        \item[(c)] If $(M,N)$ is a weighted pair, $\omega\in \Omega(M)_{(0)}$, and $TM\times \R \Rightarrow M$ is the Lie algebroid defined in~\autoref{examples: weighted Lie algebroids}, then $\nuw(TM\times \R) = T\nuw(M,N)\times \R \Rightarrow \nuw(M,N)$ is the Lie algebroid defined analogously, using $\omega^{[0]} \in \Omega^2(\nuw(M,N))$. 
    \end{itemize}
\end{examples}

\section{Differentiation of Multiplicative Weightings}
\label{section: differentiation of M weightings}

Let $G\toto M$ be a weighted groupoid with Lie algebroid $A \Rightarrow M$. In this section we will show that the weighting of $G$ differentiates to an infinitesimally multiplicative weighting of $A$, and the the Lie functor commutes with both a weighted normal and weighted deformation functors. More specifically, we prove the following theorem. 

\begin{theorem}
\label{theorem: weightings can be differentiated}
    Let $G\toto M$ be a weighted Lie groupoid with Lie algebroid $A = \mathrm{Lie}(G) \Rightarrow M$. Then 
        \begin{equation}
        \label{equation: definition of differentiated weighting}
            \G(A|_U)_{(i)} = \{ \sigma \in \G(A|_U) : \sigma^L\in \mathfrak{X}^L(G|_U)_{(i)} \}
        \end{equation}
    defines an infinitesimally multiplicative weighting of $A$ such that 
        \[ \mathrm{Lie}(\nuw(G,H)) = \nuw(A) \quad \text{and} \quad \mathrm{Lie}(\defw(G,H)) = \defw(A).  \]
\end{theorem}

The idea of the proof is as follows. Since $G\toto M$ is a weighted Lie groupoid, the cotangent groupoid $T^*G \toto A^*$ is a weighted VB-groupoid, as per~\autoref{proposition: dual of a weighted vb groupoid}. In particular, $A^*$ is a weighted subbundle of $T^*G$ and is therefore weighted in its own right. We then argue that the Poisson structure on $A^*$ induced by the canonical one on $T^*G$ has filtration degree zero. By~\autoref{theorem: equvalent characterizations of IM weightings}, this implies the dual weighting on $A$ is infinitesimally multiplicative. We then verify that~\eqref{equation: definition of differentiated weighting} agrees with the dual weighting. 

To begin, note that for any weighted manifold $M$ the canonical Poisson structure on $T^*M$ has filtration degree zero. Indeed, the follows from~\autoref{theorem: equvalent characterizations of IM weightings} since the weighting of $TM$ is infinitesimally multiplicative. In particular, for a weighted Lie groupoid $G$ the canonical Poisson structure on $T^*G$ has filtration degree zero. Recall that the Poisson structure on $A^*$ is defined by the equation
    \begin{equation}
    \label{equation: Libermann Poisson sructure}
        \{ s_{T^*G}^*f, s_{T^*G}^*g\}_{T^*G} = s_{T^*G}^*\{f,g\}_{A^*},
    \end{equation}
where $s_{T^*G}:T^*G \to AG^*$ is the source map for $T^*G$. 
    
\begin{lemma}
\label{lemma: AG^* is weighted Poisson}
    The Poisson bivector field on $A^*$ defined by~\eqref{equation: Libermann Poisson sructure} has filtration degree zero. 
\end{lemma}
\begin{proof}
    We have to show that
        \[ f \in C^\infty_{pol}(A^*)_{(i)},\ g \in C^\infty_{pol}(A^*)_{(j)} \implies \{f,g\}_{A^*} \in C^\infty_{pol}(A^*)_{(i+j)}. \]
    We claim that this happens if and only if $s^*\{f,g\}_{A^*} \in C^\infty_{pol}(T^*G)_{(i+j)}$. Indeed, since $s:T^*G \to A^*$ is a weighted morphism we have that 
        \[ \{f,g\}_{A^*} \in C^\infty_{pol}(A^*)_{(i+j)} \implies s^*\{f,g\}_{A^*} \in C^\infty_{pol}(T^*G)_{(i+j)}. \]
    On the other hand, since $AG^*$ is a weighted subbundle of $T^*G$, the inclusion $i:A^* \into T^*G$ is a weighted morphism. Hence 
        \begin{align*}
            s^*\{f,g\}_{A^*} \in C^\infty_{pol}(T^*G)_{(i+j)} \implies \{f,g\}_{A^*} = i^*s^*\{f,g\}_{A^*} \in C^\infty_{pol}(A^*)_{(i+j)}
        \end{align*}
    which proves the claim. Combining this with the fact that the canonical Poisson structure on $T^*G$ is filtration preserving and using~\eqref{equation: Libermann Poisson sructure} completes the proof.  
\end{proof}

\begin{proof}[Proof of~\autoref{theorem: weightings can be differentiated}]
    By~\autoref{lemma: AG^* is weighted Poisson} and~\autoref{theorem: equvalent characterizations of IM weightings} the weighting of $A$ defined as the dual of $A^*$ is infinitesimally multiplicative. It remains to establish that this weighting is given by~\eqref{equation: definition of differentiated weighting}. 
    
    Given $\sigma \in \G(A|_U)$, let $f_\sigma \in C^\infty_{[1]}(A|_U^*)$ be the corresponding linear function. Recall that $\sigma^L\in \ger{X}(G|_U)$ is the vector field corresponding to the linear function $s^*_{T^*G}f_\sigma \in C^\infty_{[1]}(T^*G)$. Using this and the fact that $s_{T^*G}:T^*G \to A^*$ is a weighted submersion we have 
        \begin{align*}
            \sigma \in \G(A|_U)_{(i)} & \iff f_\sigma \in C^\infty_{[1]}(A|_U^*)_{(i)} \\
            & \iff s^*_{T^*G}f_\sigma \in C^\infty_{[1]}(T^*G|_U)_{(i)} \\
            & \iff \sigma^L \in \ger{X}(G|_U)_{(i)},
        \end{align*}
    as claimed. The identification $\mathrm{Lie}(\nuw(G,H)) = \nuw(A)$ is given at the level of sections by the map 
        \[ \gr(\G(A)) \to \ger{X}^L(\nuw(G,H)), \quad \sigma^{[i]} \mapsto (\sigma^L)^{[i]}. \]
    Similarly, the identification $\mathrm{Lie}(\defw(G,H)) = \defw(A)$ is given by 
        \[ \rees(\G(A)) \to \ger{X}^L(\defw(G,H)), \quad \WT{\sigma}^{[i]} \mapsto \WT{(\sigma^L)}^{[i]}. \qedhere \]
\end{proof}

\begin{remark}
\label{remark: right invariant also filtration preserving}
    Since $\mathrm{inv}_G$ is a weighted diffeomorphism and $\sigma^L \sim_{\mathrm{inv}_G} -\sigma^R$, it follows that 
        \[ \sigma \in \G(A)_{(i)} \iff \sigma^R\in \ger{X}(G)_{(i)}.  \]
\end{remark}

\section{Integration of Lie Algebroid Weightings - the Wide Case}
\label{section: wide integration}

In this section we give a partial converse to~\autoref{theorem: weightings can be differentiated}. Notice from the previous work that any infinitesimally multiplicative weighting defined by differentiating a multiplicative weighting is necessarily concentrated in non-positive degrees.  

\begin{theorem}
\label{theorem: wide integration}
    Suppose that $G\toto M$ is a Lie groupoid and 
        \begin{equation}
        \label{equaton: Wide IM weighting to integrate}
            A=A_{-r} \supseteq A_{-r+1} \supseteq \cdots A_{-1} \supseteq 0
        \end{equation}
    is a Lie filtration of $A = \mathrm{Lie}(G)$. Suppose that $H\sset G$ is an $s$-connected Lie subgroupoid. If $B = \mathrm{Lie}(H)$ is such that
        \begin{enumerate}
            \item[(a)] $[\G(B), \G(A_{-i})] \sset \G(A_{-i})$ for all $i$

            \item[(b)] the assignment $m \mapsto \dim(B_m+A_{-i}|_m)$ is constant as a function on $M$
        \end{enumerate}
    then~\eqref{equaton: Wide IM weighting to integrate} defines a multiplicative weighting of $G$ along $H$ such that the induced weighting of $A$ defined by~\autoref{theorem: weightings can be differentiated} is given by the filtration  
        \begin{equation}
        \label{equation: differentiated weighting filtration}
             A=A_{-r} + B \supseteq A_{-r+1} + B \supseteq \cdots A_{-1} + B \supseteq B.
        \end{equation}
\end{theorem}

The idea of the proof is to use the filtration~\eqref{equaton: Wide IM weighting to integrate} to define a singular Lie filtration of $G$ with respect to which $H$ is clean. By~\autoref{theorem: singular Lie filtrations and weightings}, this defines a weighting of $G$ along $H$ which we then show that this weighting is multiplicative. 

\subsection{Definition of the singular Lie filtration}

Let $F_{-i}^L\sset TG$ be the subbundle spanned by the left-invariant vector fields $\sigma^L \in \ger{X}(G)$ for $\sigma \in \G(A_{-i})$, and let $F_{-i}^R$ be defined analogously using right-invariant extensions. Motivated by Equation~\eqref{equation: definition of differentiated weighting} and~\autoref{remark: right invariant also filtration preserving}, we define
    \begin{equation}
    \label{equation: integration Lie filtration}
        \mathcal{F}_{(-i)} = \G(F_{-i}^L) + \G(F_{-i}^R)  \sset \ger{X}(G), \quad i=1, \dots, r-1 
    \end{equation}
and $\mathcal{F}_{(-r)} = \mathfrak{X}(G)$. 

\begin{lemma}
    $\mathcal{F}_\bullet$ is a singular Lie filtration. 
\end{lemma}
\begin{proof}
    Since $F_{-i}^L$ and $F_{-i}^R$ are vector bundles, it follows that each $\mathcal{F}_{(-i)}$ is locally finitely generated. To see that the bracket condition is satisfied, let $\sigma \in \G(A_{-i})$, $\tau \in \G(A_{-j})$, and $f\in C^\infty(G)$. Since~\eqref{equaton: Wide IM weighting to integrate} is a Lie filtration, we have 
        \[ [\sigma^L, f\tau^L] = f[\sigma, \tau]^L + \mathcal{L}_{\sigma^L}f\cdot \tau^L \in \mathcal{F}_{(-i-j)}, \]
    because $\mathcal{L}_{\sigma^L}f\cdot \tau^L\in \mathcal{F}_{(-j)} \sset \mathcal{F}_{(-i-j)}$ as $-i-j \leq -j$; the same computation works for right invariant vector fields. Finally, since left and right invariant vector fields commute, we have 
        \[  [\sigma^L, f\tau^R] = \mathcal{L}_{\sigma^L}f\cdot \tau^R \in \mathcal{F}_{(-i-j)}, \]
    as needed.
\end{proof}

\subsection{Showing that $H$ is $\mathcal{F}_\bullet$-clean.}

In order to show that~\eqref{equation: integration Lie filtration} defines a weighting of $G$ along $H$, we need to show that $H$ is $\mathcal{F}_\bullet$-clean. Recall (cf.~\autoref{defintion: clean submanfiolds}) that this means that the function
    \[ H \to \N, \quad h \mapsto \mathrm{dim}(T_hH + \mathcal{F}_{(-i)}|_h ) \]
is constant for each $i=1, \dots, r$. 

Before proceeding, we fix some notation. Let $\sigma \in \G_c(B)$ be a compactly supported section of $B$. Since $\sigma$ is compactly support, the corresponding left-invariant extension $\sigma^L \in \ger{X}(H)$ is complete (cf.~\cite{zhong2009existence}), and so $\sigma$ defines a bisection $\exp(\sigma) \in \G(H)$ by the formula 
    \[ \exp(\sigma) : M\to H, \quad m \mapsto \phi^1(m) \]
where $\phi^1$ denotes the time 1 flow of the left-invariant vector field $\sigma^L$. 

\begin{lemma}
    $H$ is $\mathcal{F}_\bullet$-clean. 
\end{lemma}
\begin{proof}
    To begin we claim that, for any $m\in M$, one has 
        \[ T_mH + \mathcal{F}_{(-i)}|_m = T_mH + (A_{-i})_m.\]
    Indeed, let $X_m+\sigma_m^L+\tau_m^R \in T_mH + \mathcal{F}_{(-i)}|_m$ and observe that 
        \begin{align*}
            X_m+\sigma_m^L+\tau_m^R & = X_m+\sigma_m^L - \sigma_m^R + \sigma_m^R +\tau_m^R \\
            & = (X_m+a(\sigma_m)) + (\sigma_m+\tau_m)^R \in T_mH + (A_{-i})_{m},
        \end{align*}
    where was have used~\autoref{lemma: left and right invariant relations}, which says that the restriction of $\sigma^L-\sigma^R$ to $M$ is given by $a(\sigma)$. 
    
    Next, let $h\in H$. Since $H$ is $s$-connected, we may choose compactly supported sections $\sigma_1, \dots, \sigma_k\in \G_c(B)$ such that the bisection 
        \[ S = \exp(\sigma_1)\cdots \exp(\sigma_k) \in \G(H) \]
    passes through $h^{-1}$. Let $X_h+\sigma_h^L+\tau_h^R \in T_hH + \mathcal{F}_{-i}|_h$. Since $\exp(\mathrm{ad}_{\sigma}) = \mathrm{Ad}(\exp(\sigma))$ as operators on $\G(A)$, it follows that $\mathrm{ad}_{\G(B)}(\G(A_{-i}))  \sset \G(A_{-i})$ implies $\mathrm{Ad}_{\exp(\sigma)}(A_{-i}) \sset A_{-i}$. Therefore,
        \[ T_h\mathcal{A}^L_S(X_h+\sigma_h^L+\tau_h^R) = T_h\mathcal{A}^L_S(X_h) + \sigma_{s(h)}^L+(\mathrm{Ad}_S\tau)^R_{s(h)} \in T_{s(h)}H + \mathcal{F}_{-i}|_{s(h)}. \]
    This shows that 
        \begin{align*}
            \dim(T_hH + \mathcal{F}_{-i}|_h) & = \dim(T_{s(h)}H + \mathcal{F}_{-i}|_{s(h)}) \\
            & = \dim(T_{s(h)}H + (A_{-i})_{s(h)}) \\
            & = \dim(T_{s(h)}M) + \dim(B_{s(h)}+A_{-i}|_{s(h)})
        \end{align*}
    which is constant as a function of $h\in H$ by assumption. 
\end{proof}

Therefore, by~\cite[Theorem 4.1]{loizides2023differential}, the Lie filtration $\mathcal{F}_\bullet$ defines a weighting of $G$ along $H$ such that $(TG|_H)_{(-i)} = TH + (F_{-i}^L+F_{-i}^R)|_H$.

\subsection{Verification the weighting is multiplicative}

We conclude by showing that the weighting of $G$ along $H$ is multiplicative. Since $H$ is a wide subgroupoid, it follows from~\autoref{corollary: criteria for weighted submanifolds} that $M$ is a weighted submanifold of $G$ and the induced weighting of $M$ is along itself. It remains to show two things: 
    \begin{enumerate} 
        \item[(a)] the bundles
            \[ (TG|_H)_{(-i)} = TH + (F_{-i}^L+F_{-i}^R)|_H \]
            are wide VB-subgroupoids of $TG|_H \toto TM$ and 
        \item[(b)] the graph of multiplication is a weighted submanifold of $G^3$. 
    \end{enumerate}
To see why (b) holds, recall from~\autoref{example: structure of tangent groupoid} that the composition rule for $TG$ is given by 
    \begin{equation}
    \label{equation: structure of tangent groupoid}
         X_{g_0} \circ Y_{g_1} = T_{g_0}\mathcal{A}^R_{S_1}(X_{g_0}) + T_{g_1}\mathcal{A}^L_{S_0}(Y_{g_1}) - T_{g_1}(\mathcal{A}^L_{S_0}\mathcal{A}^R_{S_1}t)(Y_{g_1}),
    \end{equation}
where $X_{g_0} \in T_{g_0}G$ and $Y_{g_1}\in T_{g_1}G$ are composable, $S_0$ is a bisection through $g_0$ and $S_1$ is a bisection through $g_1$. 

\begin{lemma}
    Each $(TG|_H)_{(-i)} \toto TM$ is a VB-subgroupoid.
\end{lemma}
\begin{proof}
    First of all, note that both $Ts, Tt : (TG|_H)_{(-i)} \toto TM$ are submersions because $TH\sset (TG|_H)_{(-i)}$ and $H$ is a wide subgroupoid. Similarly, $TM\sset (TG|_H)_{(-i)}$. We now show that $(TG|_H)_{(-i)}$ is closed under composition and inversion. 
    
    Let $X+\sigma_0^L+\tau_0^R \in (T_{h_0}G)_{(-i)}$ and $Y+\sigma_1^L+\tau_1^R \in (T_{h_1}G)_{(-i)}$ be composable. We have 
        \begin{align*}
            T_{h_0}\mathcal{A}^R_{S_1}(X+\sigma_0^L+\tau_0^R) & =  T_{h_0}\mathcal{A}^R_{S_1}X + (\mathrm{Ad}_{S_1}\sigma_0)^L + \tau_0^R \in (T_{h_0\circ h_1}G)_{(-i)}, \\ 
            T_{h_1}\mathcal{A}^L_{S_0}(Y+\sigma_1^L+\tau_1^R) & =  T_{h_1}\mathcal{A}^L_{S_0}Y + \sigma_1^L + (\mathrm{Ad}_{S_0}\tau_1)^R \in (T_{h_0\circ h_1}G)_{(-i)}.
        \end{align*}
    Moreover, since $H$ is assumed to be wide and $S_0$ and $S_1$ are bisections in $H$, 
        \[T_{h_1}(\mathcal{A}^L_{S_0}\mathcal{A}^R_{S_1}t)(Y+\sigma_1^L+\tau_1^R) \in T_{h_0\circ h_1}H \sset (T_{h_0\circ h_1}G)_{(-i)}, \]
    where we have used that $\sigma_1^L \sim_t 0$ and $\tau_1^R\sim_t a(\tau_1)$. Since the composition $(X+\sigma_0^L+\tau_0^R)\circ (Y+\sigma_1^L+\tau_1^R)$ is the sum of these three terms~\autoref{equation: structure of tangent groupoid}, it follows that $(TG|_H)_{(-i)}$ is closed under multiplication. Similarly, using that $\sigma^L \sim_{\mathrm{inv}} -\sigma^R$, it follows that $(TG|_H)_{(-i)}$ is closed under inversion hence hence is a VB-groupoid.  
\end{proof}

We conclude the proof with the following lemma. 

\begin{lemma}
    The graph of multiplication is a weighted submanifold. 
\end{lemma}
\begin{proof}
    Recall (cf.~\cite[Section 9.6]{meinrenken2017Lie}) that $\G(\mathrm{mult}_G) \sset G^3$ is the flow-out of $\{(m,m,m) : m\in M\} \sset G^3$ by the vector fields 
        \[ X_\sigma = (-\sigma^R, -\sigma^R, 0) \quad \text{and} \quad Y_\tau = (0, \tau^L, -\tau^R), \]
    with $\sigma, \tau \in \G(A)$. Therefore, analogously to the work above, we can define a singular Lie filtration of $\G(\mathrm{mult}_G)$ with respect to which $\G(\mathrm{mult}_H)$ is clean. This defines a weighting of $\G(\mathrm{mult}_G)$ along $\G(\mathrm{mult}_H)$. One has, by construction, that 
        \[ (T\G(\mathrm{mult}_G)|_{\mathrm{mult}_H})_{(-i)} = T\G(\mathrm{mult}_G)\cap (TG^3|_{H^3})_{(-i)},  \]
    since $G^3$ is weighted along $H^3$ using the product Lie filtration. Furthermore,~\autoref{proposition: weighted morphisms of weighted lie filtrations} implies that the inclusion $\G(\mathrm{mult}_G) \into G^3$ is a weighted morphism.~\autoref{corollary: criteria for weighted submanifolds} implies that $\G(\mathrm{mult}_G)$ is a weighted submanifold.
\end{proof}

This concludes the proof of~\autoref{theorem: wide integration}. As an immediate application, we have the following classification theorem, which says that \emph{any} multiplicative weighting along an $s$-connected wide subgroupoid arises in this way. In particular, a Lie groupoid with a multiplicative weighting along its objects is a filtered Lie groupoid in the sense of van Erp and Yuncken (cf.~\cite[Definition 17]{van2017tangent}), and conversely.

\begin{theorem}
\label{theorem: wide weighted groupoids are filtered}
    The differentiation procedure of~\autoref{theorem: weightings can be differentiated} and integration procedure of~\autoref{theorem: wide integration} define a 1-1 correspondence between multiplicative weightings along $s$-connected, wide Lie subgroupoids and infinitesimally multiplicative weightings along wide Lie subalgebroids.
\end{theorem}
\begin{proof}
    Let $G\toto M$ be a Lie groupoid and $H\toto M$ be an $s$-connected, wide Lie subgroupoid. An infinitesimally multiplicative weighting of $\mathrm{Lie}(G)$ along $\mathrm{Lie}(H)$ is a Lie filtration 
        \[ \mathrm{Lie}(G) = F_{-r} \supseteq \cdots \supseteq F_{-1} \supseteq \mathrm{Lie}(H) \]
    which satisfies the assumptions of~\autoref{theorem: wide integration}. Thus, it defines a multiplicative weighting of $G$ along $H$ and~\autoref{equation: differentiated weighting filtration} says that the differentiated weighting of $\mathrm{Lie}(G)$ along $\mathrm{Lie}(H)$ is the one we started with. 

    Now, suppose that $G$ is multiplicatively weighted along $H$. By~\autoref{theorem: weightings can be differentiated}, the induced weighting of $\mathrm{Lie}(G)$ along $\mathrm{Lie}(H)$ is given by 
        \[ \mathrm{Lie}(G) = F_{-r} \supseteq \cdots \supseteq F_{-1} \supseteq \mathrm{Lie}(H), \]
    where 
        \[ F_{-i} = (TG|_M)_{(-i)} / TM. \]
    Let $\WT{G}$ denote $G$ but with the weighting along $H$ defined by~\autoref{theorem: wide integration}. By~\autoref{proposition: weighted morphisms of weighted lie filtrations}, the identity map 
        \[ \WT{G} \to G \]
    is a weighted morphism. We claim that it is a weighted diffeomorphism. By the normal form for weighted submersions (\autoref{theorem: weighted submersion coordinates}), it is enough to show that $T_h\WT{G}_{(-i)} = T_hG_{(-i)}$ for all $h\in H$. But this clearly holds for all $h \in M$ and since left multiplication by any bisection in $H$ is a weighted diffeomorphism\footnote{This can be seen as a consequence of~\autoref{lemma: extensions of morphisms}. Indeed, since $S\sset H$ it is automatically a weighted submanifold of $G$. Therefore, $\defw(S,S) = S\times \R$ is a bisection of $\defw(G,H)$ and one has that $\defw(\mathcal{A}^L_S) = \mathcal{A}^L_{\defw(S,S)}$} it therefore holds for all $h \in H$.  
\end{proof} 

\begin{examples}
\begin{itemize}
    \item[(a)] Suppose that $G = \mathrm{Pair}(M)$ is the pair groupoid and $H = M$ is the units. An infinitesimally multiplicative weighting of $\mathrm{Lie}(\mathrm{Pair}(M)) = TM$ is a filtered structure on $M$,
        \[ TM = F_{-r} \supseteq F_{-r+1} \supseteq \cdots \supseteq F_{-1} \supseteq 0. \]
    The singular Lie filtration of $\mathrm{Pair}(M)$ defined by~\eqref{equation: integration Lie filtration} is regular, and is given by 
        \[ T\mathrm{Pair}(M)_{-i} = F_{-i} \times F_{-i}.  \]

    \item[(b)] Let $G$ be a Lie be a Lie group, $H\sset G$ a connected, closed subgroup, and let 
        \[ \ger{g} = \ger{g}_{-r} \supseteq \ger{g}_{-r+1} \supseteq \cdots \ger{g}_{-1} \supseteq \ger{h} = \mathrm{Lie}(H), \]
    be a Lie filtration. The weighting defined by~\autoref{theorem: wide integration} is the same as the weighting defined by the (right-invariant) Lie filtration 
        \[ TG = G\times \ger{g}_{-r} \supseteq G\times \ger{g}_{-r+1} \supseteq \cdots G \times \ger{g}_{-1} \supseteq G\times \ger{h}, \]
    given by left-translation. Indeed, this follows by a similar argument as in~\autoref{theorem: wide weighted groupoids are filtered}.
\end{itemize}  
\end{examples}

%% file: 5_graded_bundle_description.tex
\chapter{Weightings and Higher Tangent Bundles}
\label{chapter: Weightings and Higher Tangent Bundles}

One of the main goals in~\cite{loizides2023differential} is to give a ``coordinate-free" definition of a weighting. They accomplish this by showing that a weighting of $M$ along $N$ is completely described by a graded subbundle (cf.~\autoref{subsection: graded bundles}) of the $r$-th order tangent bundle $T_rM$. This perspective can be used to give an alternative characterization of multiplicative and infinitesimally multiplicative weightings. In this chapter we will explain this perspective and use it to extend~\autoref{theorem: wide integration} to arbitrary $s$-connected Lie subgroupoids (\autoref{theorem: integration of multiplicative weightings}). 

\section{Higher Tangent Bundles}

\subsection{Definition}

We begin be reviewing the definition of $T_rM$, following the algebraic perspective  presented in~\cite[Chapter VIII]{kolar2013natural}. Consider the truncated polynomial algebra $\A_r := \R[\epsilon]/\la \epsilon^{r+1} \ra$. 

\begin{definition}
	The \emph{$r$-th order tangent bundle} of $M$ is the character spectrum 
		\[ T_rM := \aHom(C^\infty(M), \A_r). \]
\end{definition}

\begin{example}
    Any $\varphi \in T_rM = \aHom(C^\infty(M), \A_r)$ can be written as 
        \[ \varphi = \sum_{i=0}^r u_i \epsilon^i \]
    for linear maps $u_i : C^\infty(M) \to \R$. We see that $u_0$ is simply an algebra morphism and therefore is given by evaluation at a point $p\in M$. Hence $T_0M = M$. Similarly, $u_1$ is a derivation with respect to $u_0$, and is therefore a tangent vector based at $p\in M$. Thus, $T_1M = TM$. 
\end{example}

The assignment $M\mapsto T_rM$ is functorial: for a smooth map $F:M\to M'$ its $r$-th order tangent lift is defined by 
	\[T_rF(u) = \sum (u_i\circ F^*)\epsilon^i.  \] 

\subsection{Lifts of functions and vector fields}

Any function $f\in C^\infty(M)$ admits $r$ tangent lifts $f^{(i)} \in C^\infty(T_rM)$, for $i=0,1,\dots, r$, defined by 
    \[ f^{(i)} : \sum_{i=0}^ru_i\epsilon^i \mapsto  u_i(f).  \]
These lifts satisfy the product rule 
	\begin{equation}
	\label{equation: product rule for lifts}
		(fg)^{(i)} = \sum_{j=0}^if^{(j)}g^{(i-j)}.
	\end{equation}
Furthermore, if $F:M\to M'$ and $g\in C^\infty(M')$ then
	\[ (T_rF)^*g^{(i)} = (F^*g)^{(i)}. \]
If $x_a$, $a=1, \dots, m$, are coordinates for $M$ defined on $U$, then the lifts $x^{(i)}_a$, $a=1, \dots, m$ and $i=0,\dots, r$, define coordinates on $T_rU \subseteq T_rM$. 

\begin{example}
    Given $f\in C^\infty(M)$, the zeroth lift $f^{(0)}\in C^\infty(TM)$ is just the pullback of $f$ to $TM$, while $f^{(1)}\in C^\infty(TM)$ is $\ed f$ thought of as a function on $TM$.  
\end{example}
 
Morimoto observed in~\cite{morimoto1970liftings} that any vector field $X\in \mathfrak{X}(M)$ also admits lifts $X^{(-i)} \in \mathfrak{X}(T_rM)$ for $i=0,1, \dots, r$. These lifts are uniquely determined by their action on lifts of functions $f\in C^\infty(M)$, which is given by
	\[X^{(-i)}f^{(j)} = (Xf)^{(j-i)}. \]
These lifts satisfy the product rule and are compatible with Lie brackets: 
    \begin{equation}
    \label{equation: lift compatibilities}
        (fX)^{(-i)} = \sum_{k=i}^rf^{(k-i)}X^{(-k)} \quad \text{and} \quad [X^{(-i)}, Y^{(-j)}] = [X,Y]^{(-i-j)},
    \end{equation}
where $[X,Y]^{(-i-j)} = 0$ if $i+j >r$. 

\subsection{Iterating the higher tangent functor}

Given non-negative integers $r_1, \dots, r_k$ one can, more generally, define
    \[T_{r_1,\dots, r_k}(M) = \aHom(C^\infty(M), \A_{r_1} \otimes \cdots \otimes \A_{r_k}). \]
This can be identified with the iterated tangent bundles $T_{r_1}T_{r_2}\cdots T_{r_k}(M)$. In particular, for any permutation $\sigma \in S_r$ the corresponding algebra isomorphism $\A_{r_1}\otimes \cdots \A_{r_k} \to \A_{r_{\sigma(1)}}\otimes \cdots \A_{r_{\sigma(k)}}$ induces an isomorphism 
    \begin{equation}
    \label{equation: isomorphism of higher tangent bundles}
        \delta_\sigma:T_{r_1}T_{r_2}\cdots T_{r_k}(M) \corr{\cong} T_{r_{\sigma(1)}}T_{r_{\sigma(2)}}\cdots T_{r_{\sigma(k)}}(M). 
    \end{equation} 

In the case $k=2$, the isomorphism $T_rT_sM \cong T_{s,r}M$ is given as follows. Let $\varphi \in T_{r}T_{s}M = \aHom(C^\infty(T_sM), \A_r)$ and define $\tilde{\varphi} \in T_{r,s} M$ by 
    \[ \tilde{\varphi}(f) = \sum_{i=0}^s\varphi(f^{(i)})\tau^i. \]
To see that the map
    \begin{align*}
        T_rT_sM \to T_{s,r}M, \quad \varphi \mapsto \tilde{\varphi}
    \end{align*}
is an isomorphism we write down its inverse. Let $\Psi \in T_{r,s} M$, and write $\Psi = \sum_{i=0}^s\psi_i\tau^i$. Define $\widehat{\Psi}\in T_rT_s(M)$ by 
    \[ (\widehat{\Psi})(f^{(i)}) = \psi_i(f)\in \A_r, \quad f\in C^\infty(M); \]
since $T_sM$ is a graded bundle, after extending as needed to make $\widehat{\Psi}$ an algebra morphism, this is sufficient to define $\widehat{\Psi}$. It is straightforward to check that these maps are inverse to one another. 

\begin{remark}
\label{remark: lifts of vector fields and their flows}
    With respect to the identification $T(T_rM) = T_r(TM)$ we have 
        \[  T_rX = X^{(0)} \quad \text{and} \quad T_r\phi_X^t = \phi_{X^{(0)}}^t, \]
    where $\phi_X^t$ represents the time $t$ flow of $X$; see~\autoref{lemma: flows and lifts}.
\end{remark}

\subsection{Higher tangent bundles of vector bundles}  

If $V\to M$ is a vector bundle with scalar multiplication given by $\kappa_t$, then $T_rV$ is a vector bundle over $T_rM$ with scalar multiplication given by $T_r\kappa_t$. In this section we will generalize the lifting procedure for vector fields to sections of an arbitrary vector bundle. 

Let $E\in \ger{X}(V)$ be the Euler vector field on $V$, and let $\ger{X}_{[n]}(V)$ denote the vector fields $X\in \ger{X}(V)$ with the property that $[E,X] = nX$. Recall that any section $\sigma \in \G(V)$ defines a vector field $X_\sigma \in \ger{X}_{[-1]}(V)$ by the formula 
    \[ X_{\sigma}(v) = \sigma(p)  \quad v \in V_p\]
where we are identifying $V_p$ with the vertical subspace of $T_vV$. The corresponding map
    \[ \G(V) \to \ger{X}_{[-1]}(V), \quad \sigma \mapsto X_\sigma \]
is an isomorphism of $C^\infty(M)$-modules. We will use this observation to define the tangent lifts of sections of $V$. 

\begin{lemma}
\label{lemma: lift of Euler vector field}
    If $E\in \ger{X}(V)$ is the Euler vector field for $V$, then $E^{(0)} \in \ger{X}(T_rV)$ is the Euler vector field for $T_rV$. 
\end{lemma}
\begin{proof}
    Recall that the Euler vector field on $V$ is the vector field whose flow is given by $\phi^t_{E} = \kappa_{e^t}$. Using~\autoref{remark: lifts of vector fields and their flows}, we find that 
        \[ \Phi^t_{E^{(0)}} = T_r\phi^t_E = T_r\kappa_{e^t}. \]
    Since scalar multiplication on $T_rV$ is given by $T_r\kappa$, the result follows. 
\end{proof}

\begin{proposition}
\label{proposition: lifts of sections}
    Any section $\sigma \in \Gamma(V)$ admits tangent lifts $\sigma^{(-i)} \in \Gamma(T_rV)$ for $i=0,1,\dots, r$. These lifts are uniquely characterized by the conditions that 
        \begin{enumerate}
            \item[(a)] $\sigma^{(0)} = T_r\sigma:T_rM \to T_rV$, 
            
            \item[(b)] $(f\sigma)^{(0)} = \sum_{k=0}^rf^{(k)}\sigma^{(-k)}$ for all $f\in C^\infty(M)$, and 
            
            \item[(c)] if $V$ is a vector space then $T_rV = \oplus_{i=0}^rV$, and we demand that $\sigma^{(-i)}$ be the copy if $\sigma$ in the 	$i$-th term of this direct sum.
    \end{enumerate}
\end{proposition}
\begin{proof}
    Given $\sigma \in \G(V)$, let $X_\sigma \in \mathfrak{X}_{[-1]}(V)$ be the vector field on $V$ it corresponds to under the identification $\Gamma(V) \cong \mathfrak{X}_{[-1]}(V)$. Using,~\autoref{remark: lifts of vector fields and their flows}, and~\autoref{equation: lift compatibilities} we find that
        \[ [E^{(0)}, X_\sigma^{(-i)}] = [E, X_\sigma]^{(-i)} = -X_\sigma^{(-i)} \]
    hence, by~\autoref{lemma: lift of Euler vector field}, it follows that 
        \[ X_\sigma^{(-i)} \in \ger{X}_{[-1]}(T_rV) \]
    for $i=0,1,\dots, r$. Using the isomorphism $\G(T_rV) \cong \ger{X}_{[-1]}(T_rV)$ we thus define $\sigma^{(-i)}$ by 
        \[ X_{\sigma^{(-i)}} = X_\sigma^{(-i)}. \]
    It follows that these lifts satisfy conditions (a), (b), and (c) because the lifts of vector fields do. Furthermore, by considering local vector bundle coordinates we see that these conditions uniquely characterize the lifts. 
\end{proof}

\begin{remark}
    We can also proceed as follows. First of all, note that applying the $r$-th order tangent functor to the non-degenerate bilinear pairing 
        \[ V\times_M V^* \to \R \]
    gives a map $T_rV \times_{T_rM} T_rV^*\to T_r\R$, and composing this with projection onto the last factor or $T_r\R = \oplus_{i=0}^r \R$ gives a non-degenerate bilinear pairing between $T_rV$ and $T_rV^*$. In particular, this shows that $T_rV^* = (T_rV)^*$.
    
    Now, given $\sigma \in \Gamma(V)$, let $f_\sigma \in C^\infty_{[1]}(V^*)$ be the corresponding linear function. For each $i=0,1,\dotsm, r$, the lift $f_\sigma^{(i)} \in C^\infty(T_rV^*)$ is linear. Using the previous observation, we let $\sigma^{(-i)} \in \Gamma(T_rV)$ be the section corresponding to the linear function $f_\sigma^{(r-i)}$. Conditions (a), (b), and (c) in~\autoref{proposition: lifts of sections} are all satisfied, hence this gives the same result as before. 
\end{remark}

\subsection{Higher tangent lifts of Poisson and Lie algebroid structures}

As an application of the previous section, we can describe how to define tangent lifts of Poisson and algebroid structures. There is (see~\cite[Section 4.1]{kouotchop2023vertical}) a vector bundle morphism 
	\[\epsilon_M^{r,q}: T_r(\Lambda^q TM) \to \Lambda^q T(T_rM) \]
and so any multi-vector field $X\in \mathfrak{X}^q(M) =  \Gamma(\Lambda^q TM)$ admits lifts $X^{(-i)} \in \mathfrak{X}^q(T_rM)$, $i=0,1,\dots, r$, defined as the composition 
	\[T_rM \corr{X^{(-i)}} T_r(\Lambda^qTM) \corr{\epsilon_M^{r,q}} \Lambda^q T(T_rM). \] 
 
\begin{proposition}[{\cite[Theorems 4.2 and 5.1]{kouotchop2023vertical}}]
	\begin{enumerate}
            \item[(a)] If $X\in \mathfrak{X}^2(M)$ is a Poisson bivector field, then $X^{(0)} \in \mathfrak{X}^2(T_rM)$ is a Poisson bivector field for $T_rM$. Furthermore, the corresponding Poisson bracket satisfies
			\[ \{f^{(r-i)}, g^{(r-j)}\} = \{f,g\}^{(r-i-j)} \quad \forall f,g \in C^\infty(M). \]
			
            \item[(b)] If $A\Rightarrow M$ is a Lie algebroid with anchor map $a:A\to TM$, then there is a unique bracket on $\Gamma(T_rA)$ such that
			\[ [f^{(i)}, g^{(j)}] = [f,g]^{(i+j)} \]
            making $T_rA\Rightarrow T_rM$ a Lie algebroid with anchor map given by the composition 
			\[ T_rA \corr{T_ra} T_r(TM) \corr{\cong} T(T_rM). \]
            Moreover, the corresponding dual Poisson structure on $T_rA^*$ is the order $0$ lift of the dual Poisson structure on $A^*$ as specified in the previous point. 
	\end{enumerate}
\end{proposition} 

\begin{example}
    Given any Lie groupoid $G \toto M$, the $r$-order tangent bundle has the structure a Lie groupoid $T_rG \toto T_rM$ by applying the $r$-th order tangent functor to all the structure maps. If $A\Rightarrow M$ is the Lie algebroid of $G\toto M$, then $T_rA \Rightarrow T_rM$ is the Lie algebroid of $T_rG \toto T_rM$. 
\end{example}

\section{Weightings as Graded Subbundles of $T_rM$}

\subsection{The graded subbundle defined by a weighting}

We now explain how an order $r$ weighting of $M$ naturally gives rise to a graded subbundle $Q\sset T_rM$ and how multiplicative and infinitesimally multiplicative weightings can be described in terms of this graded subbundle. Let $M$ be weighted to order $r$ along $N$ and consider the set
        \begin{equation}
        \label{equation: graded subbundle defined by the weighting}
            Q := \{q\in T_rM : \forall f\in C^\infty(M)_{(i)},\ j<i\leq r \implies f^{(j)}(q) = 0 \}.
        \end{equation}
We summarize the results of~\cite[Sections 7 and 8]{loizides2023differential} in the following theorem. 

\begin{theorem}[\cite{loizides2023differential}]
\label{theorem: weighting-graded bundle correspondences}
    With $Q$ as in~\eqref{equation: graded subbundle defined by the weighting}, we have the following 
        \begin{enumerate}
            \item[(a)] $Q$ is a graded subbundle of $T_rM$ with base $N$. 

            \item[(b)] $TQ|_N = TN\oplus (TM|_N)_{(-1)} \oplus(TM|_N)_{(-2)} \oplus \cdots \oplus (TM|_N)_{(-r)}$
            as a graded subbundle of $T(T_rM)|_N = TM|_N^{\oplus (r+1)}$.

            \item[(c)] The weighting of $M$ along $N$ can be recovered from $Q$ as
                \begin{equation}
                \label{equation: weighting defined by a graded subbundle}
                    C^\infty(M)_{(i)} = \{f\in C^\infty(M): j<i \implies f^{(j)}|_Q =0 \}.
                \end{equation}

            \item[(d)] $X\in \mathfrak{X}(M)_{(-i)}$ if and only if $X^{(-i)}$ is tangent to $Q$. Moreover, vector fields of the form $X^{(-i)}$ with $0\leq i\leq r$ and $X\in \ger{X}(M)_{(-i)}$ span the tangent bundle of $Q$ everywhere. 

            \item[(e)] A smooth map $F:M\to M'$ is a weighted morphism if and only if $T_rF(Q_M) \sset T_rF(Q_{M'})$.
        \end{enumerate} 
\end{theorem}

If $x_a \in C^\infty(U)$ is a weighted coordinate system then the set $Q\cap T_rU$ is given by 
    \begin{align}
    \label{equation: cut out functions for Q}
    \begin{split}
        x^{(0)}_a = 0 & \quad \text{for } w_a > 0 \\
        x^{(1)}_a = 0 & \quad \text{for } w_a > 1 \\
        & \vdots  \\
        x^{(r-1)}_a = 0 & \quad \text{for } w_a > r-1.
    \end{split}
    \end{align}
Conversely, if $Q\sset T_rM$ is a graded subbundle over $N\sset M$ with the property any $p\in N$ is contained in an open neighbourhood $U\sset M$ with coordinates $x_a\in C^\infty(U)$ such that $Q\cap T_rU$ is defined by~\eqref{equation: cut out functions for Q}, then $Q$ defines a weighting of $M$ along $N$ and the $x_a$ form a weighted coordinate system on $U$. 

\begin{proposition}
\label{proposition: weighted submanifolds in terms of $Q$}
    Let $(M,N)$ be a weighted pair. A submanifold $R \sset M$ is a weighted submanifold if and only if $T_rR$ intersects $Q$ cleanly in $T_rM$. In this case, the induced weighting of $R$ along $R\cap N$  is given by $Q\cap T_rR$. 
\end{proposition}
\begin{proof}
    Suppose that $R$ is a weighted submanifold. If $x_a, y_b \in C^\infty(U)$ are weighted coordinates such that $R$ is locally cut out by the $y_b$ coordinates, then $T_r(R\cap U)$ is cut out by the lifts $y_b^{(i)}$, $i=0, \dots, r$. Therefore, by~\eqref{equation: cut out functions for Q}, we see that the system of lifts $x_a^{(i)}$, $y_b^{(i)}$, $i=0,\dots, r$ form submanifold coordinates for $T_r(R\cap U)$ and $Q\cap T_rU$ which implies that $T_rR$ intersects $Q$ cleanly. 

    Conversely, if $Q$ intersects $T_rR$ cleanly then the intersection $Q\cap T_rR \sset T_rR$ defines an order $r$ weighting of $R$ along $R\cap N$ such that the inclusion $R\into M$ is a weighted morphism. By~\autoref{theorem: weighting-graded bundle correspondences} (b), we have that  
        \[ (TQ|_{R\cap N} \cap T(T_rR)|_{R\cap N})_{(-i)} = (TM|_{R\cap N})_{(-i)}\cap TR|_{R\cap N}\]
    Therefore, since $Q$ and $T_rR$ intersect cleanly it follows that 
        \begin{align*}
            (TR|_{R\cap N})_{(-i)} = (T(Q\cap T_rR)|_{R\cap N})_{(-i)} = (TM|_{R\cap N})_{(-i)}\cap TR|_{R\cap N}
        \end{align*}
    hence $R$ is a weighted submanifold of $M$ by~\autoref{corollary: criteria for weighted submanifolds}. 
\end{proof}

\subsection{Multiplicative and infinitesimally multiplicative weightings in terms of $Q$}

Recall that if $G \toto M$ is a Lie groupoid, then $T_rG \toto T_rM$ is a Lie groupoid as well. We can use this observation to give another characterization of multiplicative weightings in terms of the graded subbundle $Q\sset T_rG$. 

\begin{theorem}[{cf.~\cite[Section 8.5]{loizides2023differential}}]
\label{proposition: equivalence of groupoids weighting definitions}
    Let $G\toto M$ be weighted along $H \sset G$. The weighting is multiplicative if and only if the graded subbundle $Q_G \sset T_rG$ is a Lie subgroupoid $Q_G\toto Q_M$ of $T_rG \toto T_rM$. 
\end{theorem}
\begin{proof}
    Suppose that $G$ is multiplicatively weighted along $H$, so that $M \sset G$ is a weighted submanifold, $s,t:G\to M$ are weighted submersions, and $\mathrm{mult}_G:G^{(2)} \to G$, and $\mathrm{inv}:G\to G$ are weighted morphisms. By~\autoref{proposition: weighted submanifolds in terms of $Q$}, the intersection $T_rM \cap Q_G = Q_M$ is clean and is the graded subbundle associated to the induced weighting of $M$. Using~\autoref{theorem: weighting-graded bundle correspondences} (e), we see that $Q_G$ is closed under multiplication and inversion. Furthermore, by~\autoref{theorem: weighting-graded bundle correspondences} (b) and~\autoref{proposition: weighted submanifolds are weighted}, it follows that $T_rs|_{Q_G}:Q_G\to Q_M$ and $T_rt|_{Q_G}:Q_G\to Q_M$ are submersions near $H\sset Q_G$. By homogeneity they are therefore submersions on all of $Q_G$.

    Conversely, suppose that $Q_G \sset T_rG$ is a subgroupoid. It follows by reversing the above argument at inversion and multiplication are weighted morphisms, and that source and target are weighted submersions. Since $T_rM$ are the objects of $T_rG$, it follows that $Q_G$ intersects $T_rM$ cleanly, hence $M$ is a weighted submanifold of $G$ by~\autoref{proposition: weighted submanifolds in terms of $Q$}.
\end{proof}

\begin{example}
\label{example: pullback multiplicative weighting}
    Let $(M,N)$ be a weighted pair. If the graded bundle corresponding to the weighting of $M$ along $N$ is $Q_M \sset T_rM$, then the graded bundle corresponding to the product weighting of $\mathrm{Pair}(M)$ along $\mathrm{Pair}(N)$ is $\mathrm{Pair}(Q_N)$. In particular, gives another proof that the product weighting is multiplicative.

    Carry on, let $G\toto M$ be a Lie groupoid with the property that the groupoid anchor $a_G = (s,t):G \to \mathrm{Pair}(M)$ is transverse to $\mathrm{Pair}(N)$. In this case, $T_ra_G:T_rG \to \mathrm{Pair}(T_rM)$ is transverse to $\mathrm{Pair}(Q_M)$ and 
        \[ (T_ra_G)^{-1}(\mathrm{Pair}(Q_M)) \sset T_rG \]
    is a graded subbundle which defines a multiplicative weighting of $G$ along $G|_N$.
\end{example}

\begin{corollary}
\label{corollary: vector bundle weightings in terms of Q}
    Let $V\to M$ be a vector bundle, thought of as a manifold. A weighting of $V$ is linear (with weights concentrated in non-positive degree) if and only if the graded subbundle $Q_V\sset T_rV$ is a vector subbundle. In this case, we have that $\sigma \in \G(V)_{(i)}$ if and only if $\sigma^{(i)}$ restricts to a section of $Q_V$ over $Q_V\cap T_rM$. 
\end{corollary}
\begin{proof}
    Thinking of $V$ as a Lie groupoid over $M$, the weighting of $V$ is linear (with weights concentrated in non-positive degree) if and only if it is multiplicative. Hence, it is linear if and only if $Q_V\sset T_rV$ is a subgroupoid, which in this case is a vector subbundle. The statement about lifts is just a reformulation of~\autoref{theorem: weighting-graded bundle correspondences} (d). 
\end{proof}

Likewise, we have a similar characterization for Lie algebroid weightings which are concentrated in non-positive degree.

\begin{theorem}[{cf.~\cite[Section 8.5]{loizides2023differential}}]
\label{theorem: LA weightings in terms of Q}
    Let $A\Rightarrow M$ be a Lie algebroid with Lie subalgebroid $B\Rightarrow N$. A linear weighting of $A$ along $B$ concentrated in non-positive degree is infinitesimally multiplicative if and only if the corresponding graded subbundle $Q_A\sset T_rA$ is a Lie subalgebroid $Q_A \Rightarrow Q_M$ of $T_rA \Rightarrow T_rM$. 
\end{theorem}
\begin{proof}
    Suppose that $A$ is infinitesimally weighted along $B$, with corresponding graded bundle $Q_A\sset T_rA$, and let $Q_M = T_rM\cap Q$ be the graded bundle associated to the induced weighting of $M$ along $N$. To show that $Q_A$ is a Lie subalgebroid of $T_rA$, we must show 
        \begin{enumerate}
            \item[(a)] $a_{T_rA}(Q_A) \subseteq TQ_M$, and 
		\item[(b)] $\Gamma(T_rA, Q_A) \subseteq \Gamma(T_rA)$ is a Lie subalgebra. 
	\end{enumerate}
    For (a), recall that the anchor map for $T_rA$ is $T_ra$. Using~\autoref{corollary: vector bundle weightings in terms of Q} and~\autoref{theorem: weighting-graded bundle correspondences}, it is enough to show that $T_ra(\sigma^{(i)})$ is tangent to $Q_M$ for any $\sigma \in \G(A)_{(i)}$. But, as the weighting is infinitesimally multiplicative, if $\sigma \in \G(A)_{(i)}$ one has that $a(\sigma) \in \ger{X}(M)_{(i)}$ and since
        \[ T_ra(\sigma^{(i)}) = (a(\sigma))^{(i)} \]
    it follows from~\autoref{theorem: weighting-graded bundle correspondences} (d) that $T_ra(\sigma^{(i)})$ is tangent to $Q_M$. 
    
    For (b), it suffices to show that if $f\sigma^{(i)} \in \Gamma(T_rA, Q_A)$ and $g\tau^{(j)} \in  \Gamma(T_rA, Q_A)$, then $[f\sigma^{(i)}, g\tau^{(j)}] \in  \Gamma(T_rA, Q_A)$, where $\sigma, \tau \in \Gamma(V)$ and $f, g\in C^\infty(T_rM)$.  Using the Jacobi identity
	\begin{align}
	\label{equation: lie bracket computation}
		[f\sigma^{(i)}, g\tau^{(j)}] = fg[\sigma,\tau]^{(i+j)}+f\mathcal{L}_{T_ra(\sigma^{(i)})}g\cdot \tau^{(j)} - g\mathcal{L}_{T_ra(\tau^{(j)})}f \cdot \sigma^{(i)},
	\end{align}
    we see that there are three cases to consider:
        \begin{itemize}
            \item If $\sigma^{(i)}, \tau^{(j)} \in \Gamma(T_rA, Q_A)$, then by~\autoref{corollary: vector bundle weightings in terms of Q} it follows that $\sigma \in \G(A)_{(i)}$ and $\tau \in \G(A)_{(J)}$. Since the weighting is infinitesimally multiplicative, we have $[\sigma, \tau] \in \G(A)_{(i+j)}$, hence
                \[ [\sigma^{(i)}, \tau^{(j)}] = [\sigma, \tau]^{(i+j)} \in \Gamma(T_rA, Q_A).\]
            Therefore, $[f\sigma^{(i)}, g\tau^{(j)}] \in \Gamma(T_rA, Q_A)$ in this case. 

            \item If $\sigma^{(i)}, \tau^{(j)} \notin \Gamma(T_rA, Q_A)$, then both $f$ and $g$ must vanish on $Q_M$. By the Jacobi identity, $[f\sigma^{(i)}, g\tau^{(j)}]$ also vanishes along $Q_M$, hence belongs to $\G(T_rA, Q_A)$ in this case. 

            \item Finally, suppose that $\sigma^{(i)} \notin \Gamma(T_rA, Q_A)$ and $\tau^{(j)}\in \Gamma(T_rA,Q_A)$. In this case, we must have that $f$ vanishes along $Q_M$. Since $T_ra(\tau^{(j)}) = (a(\tau))^{(j)}$ is tangent to $Q_M$, it follows that $\mathcal{L}_{T_ra(\tau^{(j)})}f$ vanishes on $Q_M$ and therefore $[f\sigma^{(i)}, g\tau^{(j)}]$ vanishes along $Q_M$ as well. 
        \end{itemize}
      In each of these cases we see that $[f\sigma^{(i)}, g\tau^{(j)}] \in  \Gamma(T_rA, Q_A)$, which shows that $\Gamma(T_rA, Q_A)$ is a Lie subalgebra of $\Gamma(T_rA)$.

    Conversely, suppose that $Q_A\subseteq T_rA$ is a Lie subalgebroid. In order to show that conditions (a) and (b) in~\autoref{definition: IM weighting} are satisfied, it is sufficient, since the weighting is concentrated in non-positive degree, to consider indices $-i, -j$ with $0\leq i,j \leq r$. If $\sigma \in \Gamma(A)_{(-i)}$, and $\tau \in \Gamma(A)_{(-j)}$, then $\sigma^{(-i)} \in \Gamma(T_rA, Q_A)$ and $\tau^{(-j)}\in \Gamma(T_rA, Q_A)$, whence $[\sigma^{(-i)}, \tau^{(-j)}]=[\sigma, \tau]^{(-i-j)} \in \Gamma(T_rA, Q_A)$ as $Q_A$ is a Lie subalgebroid. This shows that $[\sigma, \tau]\in \Gamma(A)_{(-i-j)}$. Similarly, let $\sigma \in \Gamma(A)_{(-i)}$. Then $\sigma^{(-i)}\in \Gamma(T_rA, Q_A)$ and $T_ra(\sigma^{(-i)}) = (a(\sigma))^{(-i)}$ is tangent to $Q_N$, which shows that $a(\sigma)\in \mathfrak{X}(M)_{(-i)}$, as needed. 
\end{proof}

\begin{proposition}
    Suppose that $G\toto M$ is multiplicatively weighted along $H$, with graded bundle $Q_G \sset T_rG$. Then the graded bundle for the weighting of $A = \mathrm{Lie}(G)$ is given by $Q_A = \mathrm{Lie}(Q_G)$. 
\end{proposition}
\begin{proof}
    Recall that the weighting of $A$ is given by $\sigma \in \G(A)_{(i)}$ if and only if $\sigma^L \in \mathfrak{X}(G)_{(i)}$. The result follows since $(\sigma^{(-i)})^L = (\sigma^L)^{(-i)}$ for $0\leq i\leq r$, which follows because they are both left invariant and agree along $T_rM$.  
\end{proof}

\section{Integration of Lie Algebroid Weightings Revisited}

In this section we will return to the integration problem for infinitesimally multiplicative weightings. More specifically, suppose that $H \rightrightarrows N$ is an $s$-connected Lie subgroupoid of $G\rightrightarrows M$, and suppose that $A = \mathrm{Lie}(G)$ has an infinitesimally multiplicative weighting along $B = \mathrm{Lie}(H)$. We will show that if the graded bundle $Q_A \subseteq T_rA$ corresponding to this weighting integrates to an $s$-connected Lie subgroupoid $Q_G \subseteq T_rG$ then this subgroupoid defines a weighting of $G$ along $H$. In order to do this, we will use a weighted Lie algebroid spray to define an exponential map which allows us to find weighted coordinates for $G$ near the object space. Flowing these coordinates along the $s$-fibres using left invariant vector fields of filtration degree 0 gives weighted coordinates everywhere along $H$, proving that $Q_G$ indeed defines a weighting.

\subsection{The spray exponential for a weighted Lie algebroid}

Let $A\Rightarrow M$ be a weighted Lie algebroid and let $G\toto M$ be a Lie groupoid integrating $A$. We briefly review how to construct a partially defined exponential map for $A \dashrightarrow G$; for a complete discussion, see~\autoref{appendix: spray exponential}. The data to define an exponential map $A\to G$ is a Lie algebroid spray:

\begin{definition}[{\cite[Definition 3.1]{cabrera2020local}}]
    Let $A\Rightarrow M$ be a Lie algebroid. A \emph{Lie algebroid spray} on $A$ is a vector field $V\in \ger{X}(A)$ such that 
    \begin{enumerate}
        \item[(a)] $\kappa_t^*V = tV$ for all $t\neq 0$, where $\kappa_t$ denotes scalar multiplication by $t$, and 
            
        \item[(b)] for all $\xi\in A$, one has $T\pi(V_\xi) = a(\xi)$, where $\pi:A\to M$ is the vector bundle projection and $a:A\to TM$ is the anchor. 
    \end{enumerate}
\end{definition}

If $V\in \ger{X}(A)$ is a Lie algebroid spray, then there is an open neighbourhood $U_V\sset A$ of the zero section for which the flow $\phi_V^t$ is defined for $|t|\leq 1$. For each $u\in U_V$, the path $t\to u(t)= \phi_V^t(u)$ is an $A$-path, and therefore integrates to a $G$-path $\tilde{u}(t)$. The \emph{spray exponential} is the map 
    \[ \exp_V : U_V \to G, \quad u\mapsto \tilde{u}(1).  \]
Following the notation in~\cite{cabrera2020local} we use the dotted arrow notation 
    \[ \exp_V:A \dashrightarrow G \]
to denote that $\exp_V$ is only defined on an open neighbourhood of the zero section. 

\begin{example}
    Let $G$ be a Lie group and $\ger{g} = \mathrm{Lie}(G)$. Then $0 \in \ger{X}(\ger{g})$ is a Lie algebroid spray for $\ger{g}$ and $\exp_0$ is the standard exponential map $\exp: \ger{g} \to G$. 
\end{example}

If $A$ is a weighted Lie algebroid, the the appropriate sprays to consider are ones having filtration degree zero. 

\begin{lemma}
    For any weighted Lie algebroid $A$, there exists a spray $V\in \ger{X}(A)_{(0)}$.
\end{lemma}
\begin{proof}
    Note that any convex combination of sprays is again a spray. Therefore, by pulling back a partition of unity on $M$ it suffices to work locally. 

    Let $w_i$ and $v_j$ be the base and vertical weights for $A$, respectively. If $\G_{ij}^k \in C^\infty(M)$ denote the structure functions for $A$ corresponding to a local weighted frame then recall from~\autoref{lemma: structure constants for Lie algebroids} that $\G_{ij}^k \in C^\infty(M)_{(v_i+v_j-v_k)}$. If $x_a, p_b$ are weighted vector bundle coordinates for $A$ then
        \[ V = a_{ij}(x)p_i\frac{\partial}{\partial x_j} + \Gamma^k_{ij}(x)p_ip_j\frac{\partial}{\partial p_k} \]
    defines a spray for $A$. Since $a:\Gamma(A)_{(-i)}\to \mathfrak{X}(M)_{(-i)}$, it follows that $a_{ij}\in C^\infty(M)_{(v_i+w_j)}$, hence $V\in \ger{X}(A)_{(0)}$. 
\end{proof}

Let $G\toto M$ be a Lie groupoid with a multiplicative weighting along $H\toto N$, let $A = \mathrm{Lie}(G)$ and let $B = \mathrm{Lie}(H)$. Let $Q_G \sset T_r G$ be the graded bundle corresponding to the weighting, and let $Q_A = \mathrm{Lie}(Q_G) \sset T_rA$. 

\begin{lemma}
\label{lemma: weighted spray exponential is a weighted morphism}
    Let $V\in \ger{X}(A)_{(0)}$ be a Lie algebroid spray on $A$. Then the spray exponential is a partially defined map of pairs 
        \[ T_r\exp_{V} : (T_rA, Q_A) \dashrightarrow (T_rG, Q_G). \]
    It is a diffeomorphism of pairs in a sufficiently small open neighbourhood of $T_rM \sset T_rG$. In particular, if $\exp_V$ is globally defined it is a weighted morphism. 
\end{lemma}
\begin{proof}
    By~\autoref{proposition: tangent and exp commute}, we have that $T_r\exp_{V} = \exp_{V^{(0)}}$. Since $V\in \ger{X}(V)_{(0)}$, it follows by~\autoref{theorem: weighting-graded bundle correspondences} (d) that $V^{(0)}$ is tangent to $Q_A$, hence by~\autoref{lemma: exp maps subgroupoids to subgroupoids}
        \[ \exp_{V^{(0)}} : (T_r A, Q_A) \dashrightarrow (T_rG, Q_G) \]
    is a (partially defined) map of pairs.
\end{proof}

This immediately implies the following result. 

\begin{proposition}
    For any weighted Lie group $G$, the exponential map is a weighted morphism. It is a weighted diffeomorphism on any sufficiently small open neighbourhood of the origin in $\ger{g}$. 
\end{proposition}
\begin{proof}
    The zero vector field on $\ger{g}$ is a weighted Lie algebroid spray, and the corresponding spray exponential is the standard exponential map. Hence the result follows by~\autoref{lemma: weighted spray exponential is a weighted morphism}.
\end{proof}

\begin{remark}
    This can also be proved using weighted paths in $\ger{g}$, but the proof is considerably longer. 
\end{remark}

\subsection{Integration of Lie algebroid weightings revisited} 

Let $G\toto M$ be a Lie groupoid with Lie algebroid $A = \mathrm{Lie}(G)$. We now prove the main result of this section, which says that if the graded bundle corresponding to an infinitesimally multiplicative weighting of $A$ integrates to a subgroupoid of $T_rG$, then this is automatically a graded subbundle which defines a multiplicative weighting of $G$.

\begin{theorem}
\label{theorem: integration of multiplicative weightings}
    Suppose that $G\toto M$ is a $s$-connected Lie groupoid and $H\toto N$ is a $s$-connected Lie subgroupoid with Lie algebroids $A = \mathrm{Lie}(G)$ and $B = \mathrm{Lie}(H)$, respectively, and suppose that $A$ is infinitesimally multiplicatively weighted along $B$. Let $Q_A \Rightarrow Q_M$ be the graded subbundle of $T_rA \Rightarrow T_rM$ corresponding to the weighting of $A$ along $B$. If $Q_A$ integrates to an $s$-connected subgroupoid $Q_G \toto Q_M$ of $T_rG\toto T_rM$, then $Q_G$ is a graded subbundle of $T_rG \toto T_rM$ which defines a multiplicative weighting of $G$ along $H$. 
\end{theorem}
\begin{proof}
    Let $Q_A \sset T_rA$ be the graded bundle associated to the weighting of $A$ along $B$, and let $Q_G \sset T_rG$ be the Lie subgroupoid integrating it. To show that $Q_G$ is a GB-subgroupoid we just have to show that it is closed under scalar multiplication. To do this, let $V \in \mathfrak{X}(A)_{(0)}$ be a Lie algebroid spray. Let $\kappa_t : T_rG \to T_rG$ be scalar multiplication by $t\in \R$; abusing notation we let $\kappa_t$ denote scalar multiplication on $T_rA$ as well. Since $V^{(0)} = T_rV$, it follows that $V^{(0)} \sim_{\kappa_t} V^{(0)}$, hence the following diagram commutes 
        \begin{equation*}
        \xymatrix{
            T_rA \ar[r]^{\kappa_t} \ar@{-->}[d]_{\exp_{V^{(0)}}} & T_rA \ar@{-->}[d]^{\exp_{V^{(0)}}} \\
            T_rG \ar[r]_{\kappa_t} & T_rG.
        }
        \end{equation*}
    Using this, the fact that both
        \[ \exp_{V^{(0)}} : (T_rA, Q_A) \dashrightarrow (T_rG, Q_G) \quad  \text{and} \quad \kappa_t:(T_rA, Q_A)\to (T_rA, Q_A) \]
    are maps of pairs, and the assumption that $Q_G$ is $s$-connected implies that $Q_G$ is a graded subbundle. Since 
        \[ \mathrm{Lie}(\kappa_0(Q_G)) = \kappa_0(\mathrm{Lie}(Q_G)) = B \]
    it follows that $H = \kappa_0(Q_G)$, since $H$ is the Lie subgroupoid of $G$ integrating $B$. 
    
    It remains to show that $Q_G \sset T_rG$ defines a weighting. To do this, we will show that near any point $p\in H$ one can always find local coordinates $x_a$ for $G$ such that $Q_G$ is locally defined by~\autoref{equation: cut out functions for Q}. There are open neighbourhoods $U_A \sset A$ and $U_G\sset G$ containing the objects for which $\exp_V:U_A \to U_G$ is a diffeomorphism. By the previous section, we have 
        \begin{equation*}
            T_r\exp_V(Q_A \cap T_rU_A) = Q_G \cap T_rU_G
        \end{equation*}
    In particular, if $x_a$ are weighted coordinates defined on $U\subseteq U_A$, then $Q_G \cap T_rU$ is defined by the the equations 
        \begin{align*}
            ((\exp_V)_*x_a)^{(0)} = & (\exp_{V^{(0)}})_*x_a^{(0)} = 0 \quad \text{for } w_a > 0,\\
            ((\exp_V)_*x_a)^{(1)} = & (\exp_{V^{(0)}})_*x_a^{(1)} = 0 \quad \text{for } w_a > 1, \\
            & \vdots \\ 
            ((\exp_V)_*x_a)^{(r-1)} = & (\exp_{V^{(0)}})_*x_a^{(r-1)} = 0 \quad \text{for } w_a > r-1.
        \end{align*}
    This gives weighted coordinates for $G$ near $M$. 
    
    We will now construct weighted coordinates away from $M$. Let $h\in H$ and assume, for simplicity, that $h = \phi_{\sigma^L}^t(s(h))$ for some $\sigma \in \Gamma(A)_{(0)}$. If $x_a$ are weighted coordinates for $G$ defined on $U\sset G$ containing $s(h)$, let $\tilde{x}_a = (\phi_{\sigma^L}^t)_*x_a$; these give coordinates on $\tilde{U} = \phi_{\sigma^L}^{-t}(U)$, and we claim that 
        \[ Q_G \cap T_r\tilde{U} = \{\tilde{x}_a^{(i)} = 0 : w_a > i,\ i = 0, \dots, r-1\}. \]
    To see this, note that by~\autoref{lemma: flows and lifts}
        \begin{align*}
            \tilde{x}_a^{(i)} = ((\phi_{\sigma^L}^t)_*x_a)^{(i)} = (T_r\phi_{\sigma^L}^t)_*x_a^{(i)} =  (\phi_{(\sigma^L)^{(0)}}^t)_*x_a^{(i)}.
        \end{align*}
    Since $\sigma \in \Gamma(V)_{(0)}$, it follows that $(\sigma^L)^{(0)}$ is tangent to $Q_G$ hence $Q_G$ is invariant under its flow. This shows that 
        \begin{align*}
            Q_G \cap T_r\tilde{U} & = \phi_{(\sigma^L)^{(0)}}^t(Q_G \cap T_rU) \\
            & = \phi_{(\sigma^L)^{(0)}}^t(\{x_a^{(i)} = 0 : w_a > i,\ i = 0, \dots, r-1\}) \\
            & = \{(\phi_{(\sigma^L)^{(0)}}^t)_*x_a^{(i)} = 0 : w_a > i,\ i = 0, \dots, r-1\}) \\
            & = \{\tilde{x}_a^{(i)} = 0 : w_a > i,\ i = 0, \dots, r-1\},
        \end{align*}
    as claimed. Using induction and the fact that $H$ is $s$-connected, we see that any $h\in H$ is contained in a neighbourhood on which weighted coordinates are defined, which completes the proof. 
\end{proof}

\begin{remark}
    The question of whether or not the subalgebroid $Q_A \Rightarrow Q_M$ integrates to a subgroupoid of $T_rG \toto T_rM$ is a little subtle. At first one might think that such an integration always exists, since any subalgebroid of an integrable Lie algebroid is always integrable. However,~\cite{moerdijk2006integrability} shows that this is \emph{not} always the case. Taking into account the graded bundle structure, one might wonder if it is enough to assume that the base algebroid integrates to a subgroupoid. However this is also not sufficient, as explained in~\cite{cabrera2018obstructions}. We leave it as an open problem to give conditions on when such an integration is possible.  
\end{remark}

%% file: 6_outlook.tex
\chapter{Outlook}
\label{chapter: Outlook}

We conclude this thesis by examining some possible directions this research can lead. The following is speculative and vague, yet our aim is to inspire readers to explore the topic more thoroughly.

\section{More on Linear Weightings}

\subsection{Linear singular Lie filtrations}

The first avenue we think is worth pursuing is an adaptation of~\autoref{section: singular Lie filtrations} to vector bundles. Let $V\to M$ be a vector bundle. Recall that the \emph{linear} vector fields on $V$ is the $C^\infty(M)$-module 
    \[ \ger{X}_{[0]}(V) = \{ X \in \ger{X}(V) : [E, X] = 0 \},  \]
where $E\in \ger{X}(V)$ is the Euler vector field. The linear vector fields on $V$ can be identified with first-order differential operators acting on $\G(V)$ having scalar principal symbol. We propose the following definition.

\begin{definition}
\label{definition: linear singular Lie filtrations}
    \begin{enumerate}
        \item[(a)] A \emph{linear singular distribution} on a vector bundle $V\to M$ is a sheaf of $C^\infty_M$-submodules $\mathcal{L} \sset \mathfrak{X}_{[0]}(V)$ which is locally finitely generated in the sense that each $p\in M$ has an open neighbourhood $U\sset M$ such that $\mathcal{L}(V|_U)$ is finitely generated. 

        \item[(b)] A \emph{linear singular Lie filtration} of order $r$ is a filtration of the sheaf of linear vector fields 
            \begin{equation}
            \label{equation: LSLF}
                \mathfrak{X}_{[0]}(V) = \mathcal{L}_{(-r)} \supseteq \mathcal{L}_{(-r+1)} \supseteq \cdots \supseteq \mathcal{L}_{(-1)} \supseteq 0
            \end{equation}
        by linear singular distributions such that 
            \[ [\mathcal{L}_{-i}, \mathcal{L}_{-j}] \sset \mathcal{L}_{-i-j} \]
        for all $i, j$. 
    \end{enumerate} 
\end{definition}

The linear vector fields on $V$ are the sections of a Lie algebroid $\mathrm{At}(V) \Rightarrow M$, called the Atiyah algebroid of $V$, which fits in to an exact sequence 
    \begin{equation}
    \label{equation: wtd exact sequence}
        0 \corr{} \mathrm{End}(V) \corr{} \mathrm{At}(V) \corr{} TM \corr{} 0. 
    \end{equation}
In particular, we see that a linear singular Lie filtration on $V$ determines a filtration of $\G(\mathrm{End}(V))$ by $C^\infty(M)$-submodules. The space of linear differential operators acting on sections of $V$ can be identified with the quotient 
    \[ \mathcal{U}(\mathrm{At}(V))/I \]
where $\mathcal{U}(\mathrm{At}(V))$ is the universal enveloping algebra of the Lie algebroid $\mathrm{At}(V)$ (cf.~\cite{moerdijk2010universal}) and $I\sset \mathcal{U}(\mathrm{At}(V))$ is the ideal generated by elements of the form $\eta\otimes \phi - \eta\phi$, where $\eta, \phi \in \Gamma(\End(V))$. Therefore, a linear singular Lie filtration~\eqref{equation: LSLF} determines an algebra filtration
    \[ \cdots \mathrm{DO}(V)_{(i)} \supseteq \mathrm{DO}(V)_{(i+1)} \supseteq \cdots \]
of $\mathrm{DO}(V)$, so long as the induced filtration of $\G(\mathrm{End}(V))$ (using the exact sequence~\eqref{equation: wtd exact sequence}) is multiplicative.

Let $N\sset M$ be a closed submanifold and suppose that 
    \begin{equation}
    \label{equation: clean subfiltration}
        \cdots \supseteq (V|_N)_{(i)} \supseteq (V|_N)_{(i+1)} \supseteq
    \end{equation}
be a filtration of $V|_N$ be subbundles $(V|_N)_{(i)} \to N$. With the filtration of $\mathrm{DO}(V)$ as above, let 
    \begin{equation}
    \label{equation: induced linear weighting}
        \G(V)_{(i)} = \{ \sigma \in \G(V) : D \in \mathrm{DO}(V)_{(q)} \implies D\sigma \in \G(V, (V|_N)_{(i+q)} \}.
    \end{equation}
This defines a filtration of $\G(V)$. 

\begin{problem}
    Find the correct compatibility conditions between the linear singular Lie filtration~\eqref{equation: LSLF} and the filtration~\eqref{equation: clean subfiltration} so that the filtration~\eqref{equation: induced linear weighting} defines a linear weighting of $V$. 
\end{problem}

One possible approach to this problem would be to mimic the technique sketched in~\autoref{remark: integration of weighting to frame bundle} and find conditions for which the bundle of filtered frames is clean. We intend to pursue this in a future work. 

\begin{remark}
    This problem appears to be closely related to the works of Higson and Yi~\cite{higson2019spinors} and Braverman and Haj~\cite{BravHaj}, where a filtration of $\mathrm{DO}(V)$ is defined using a connection on $V$ and a filtration of $\mathrm{End}(V)$. 
\end{remark}

\subsection{Weightings for associative algebroids}

Let $G$ be a Lie groupoid. An \emph{associative algebroid} (cf.~\cite[Letter 6]{vsevera2017letters}) is a vector bundle $V \to G$ with an associative (in the appropriate sense) fibrewise product
    \[ V_g\otimes V_h \to V_{g\circ h}, \]
whenever $g, h\in G$ are composable. There is an obvious notion of weightings for associative algebroids. 

\begin{problem}
    Explain whether or not the weighted deformation bundle of a weighted associative algebroid has a canonical associative algebroid structure. 
\end{problem}

If yes, then there is a well-defined convolution product on sections of the weighted deformation bundle. 

\begin{problem}
    Understand multipliers of the convolution algebra of sections of an associative algebroid as algebras of pseudodifferential operators.
\end{problem}

Combining this with the the previous question, it would be interesting to try to define a rescaled spinor bundle on a manifold with boundary, i.e. a deformation bundle over the $b$-tangent groupoid, and try to understand Getzler rescaling (\cite{getzler1983pseudodifferential}) in this context. 

\section{Tangent ($k$-fold) Groupoid of a Multifiltered Manifold}

\subsection{Multifiltered manifolds}

Recall that a filtered manifold is a smooth manifold $M$ together with a filtration by subbundles 
    \[ TM = F_{-r} \supseteq F_{-r+1} \supseteq \cdots \supseteq F_{-1} \supseteq 0 \]
satisfying $[\G(F_{-i}), \G(F_{-j})] \sset \G(F_{-i-j})$. This generalizes to a \emph{multifiltered} manifold. 

\begin{definitions}[{\cite[Definitions 3.4.1 and 3.4.2]{yuncken2018pseudodifferential}}]
    \begin{enumerate}
        \item[(a)] Let $V\to M$ be a vector bundle. A \emph{$k$-multifiltration} on $V$ is a family of subbundles $F_{-\mathbf{i}} \sset V$ indexed by $\mathbf{i} \in \N^k$ such that 
            \begin{itemize}
                \item[(i)] $F_{\mathbf{0}} = 0$ and $F_{-\mathbf{n}} = V$ for some $\mathbf{n}\in \N^k$

                \item[(ii)] $F_{-\mathbf{i}} \cap F_{-\mathbf{j}} = F_{-\mathbf{i} \wedge \mathbf{j}}$ for all $\mathbf{i}, \mathbf{j} \in \N^k$, where $\mathbf{i}\wedge\mathbf{i}$ is the entry-wise minimum of the two multi-indices. 
            \end{itemize}

        \item[(b)] A \emph{$k$-multifiltered manifold} is a smooth manifold $M$ with an $k$-multifiltration of $TM$ by subbundles $F_{-\alpha}$ satisfying 
            \begin{align}
            \label{equation: lie multifiltration}
                [\G(F_{-\mathbf{i}}), \G(F_{-\mathbf{j}})] \sset \G(F_{-\mathbf{i}-\mathbf{j}}).
            \end{align}
        The $k$-multifiltration~\eqref{equation: lie multifiltration} will be referred to as a \emph{Lie multifiltration}. 
    \end{enumerate}
\end{definitions}

\begin{examples}[{cf.~\cite[Section 3.4]{yuncken2018pseudodifferential}}]
    \begin{enumerate}
        \item[(a)] Any filtered manifold is a 1-multifiltered manifold. 

        \item[(b)] The full flag manifold of a complex semisimple Lie group $G$ is an $k$-multifiltered manifold, where $k$ is the rank of G.
    \end{enumerate}
\end{examples}

\subsection{Multiweightings}

Let $M$ be an $m$-dimensional manifold. A $k$-\emph{multiweight} sequence for $M$ is a list $\mathbf{w}_1, \dots, \mathbf{w}_m \in \Z^k_{\geq 0}$. Any upper bound $r \in \Z_{\geq 0}$ for the sequence $|\mathbf{w}_a|$ is called its \emph{order}. A monomial 
    \[ x^s = x_1^{s_1} \cdots x_m^{s_m} \]
has multiweight $s\cdot \mathbf{w} = \sum_a s_a\mathbf{w}_a$. For an open set $U\sset \R^m$ and a vector $\mathbf{i}\in \Z^k_{\geq 0}$, we define $C^\infty(U)_{(\mathbf{i})}$ to be the ideal generated by the monomials 
    \[ x^s \quad \text{with } s\cdot \mathbf{w} \geq \mathbf{i},\]
where we use the standard component-wise partial ordering on $\mathbb{Z}^k_{\geq 0}$. 

\begin{definition}[{\cite[Section 9.2]{loizides2023differential}}]
    A $k$-\emph{multiweighting} of a manifold $M$ is a mutlifiltration of the sheaf of smooth functions by ideals $C^\infty_{M, (\mathbf{i})}$ with the property that each point $p\in M$ has an open neighbourhood $U \sset M$ so that $C^\infty(U)_{(\mathbf{i})}$ is the ideal defined above. 
\end{definition}

Let $e_a$ be the standard basis vectors of $\Z^k$. A $k$-multiweighting of $M$ determines $k$ weightings $C^\infty_{M,(i\cdot e_a)}$ of $M$ along submanifolds $N_1, \dots, N_k$. We say that $M$ is \emph{multiweighted} along $(N_1, N_2, \dots, N_k)$. Completely analogously to the work above, one obtains the \emph{$k$-fold weighted normal bundle}
    \[ \nuw(M, N_1, \dots, N_k) = \aHom(\gr(C^\infty(M)), \R), \]
which is a $k$-fold graded bundle over the intersection $N_1\cap \cdots \cap N_k$. Similarly, one obtains the $k$-fold weighted deformation space
    \[ \defw(M,N_1, \dots, N_k)\to \R^k. \]
    
\autoref{theorem: wide weighted groupoids are filtered} says, in particular, that Lie filtrations of a manifold $M$ are the same things as multiplicative weightings of $\mathrm{Pair}(M)$ along the units. Therefore, it is reasonable to expect that a $k$-multifiltration corresponds to a $k$-fold multiweighting of the $k$-fold pair groupoid $\mathrm{Pair}^k(M) = \prod_{i=1}^{2k}M$. 

\begin{problem}
    Explain the correspondence between $k$-multifiltered manifolds and multiplicative weightings of the $k$-fold pair groupoid $\mathrm{Pair}^k(M)$ along the various diagonals. 
\end{problem}

Let us take a minute to speculate how this might look in the case $k=2$. A 2-multifiltration on a manifold $M$ is given by a family of subbundles 
    \begin{equation}
    \label{equation: bifiltration}
        \begin{array}{cccc}
             TM = F_{(-r_1, -r_2)} & F_{(-r_1+1, -r_2)} & \cdots &  F_{(0, -r_2)}\\
             F_{(-r_1, -r_2+1)} & F_{(-r_1+1, -r_2+1)} & \cdots & F_{(0, -r_2+1)} \\
             \vdots & \vdots & \ddots & \vdots \\
             F_{(-r_1, 0)} & F_{(-r_1 + 1, 0)} & \cdots & F_{(0, 0)} = 0.
        \end{array}
    \end{equation}
Within this, there are two Lie filtrations $F^h_\bullet$ and $F^v_\bullet$ of $M$, given by the top row and left column of~\eqref{equation: bifiltration}. Correspondingly, there are two tangent groupoids $\T_{F^v}M$ and $\T_{F^h}M$. Our conjecture is that the 2-multifiltration~\eqref{equation: bifiltration} will determine a 2-multiweighting of the double pair groupoid $\mathrm{Pair}^2(M)$ along the diagonals 
    \[ \Delta^h = \{ (m_1, m_0, m_1, m_0) \in \mathrm{Pair}^2(M)\}, \quad \Delta^v = \{ (m_1, m_1, m_0, m_0) \in \mathrm{Pair}^2(M)\}  \]
such that the double deformation space has the structure of a double groupoid 
    \begin{equation}
    \xymatrix{
        \defw(\mathrm{Pair}^2(M), \Delta^h, \Delta^v) \ar@<-.5ex>[d] \ar@<.5ex>[d] \ar@<-.5ex>[r] \ar@<.5ex>[r] & \T_{F^h}M \ar@<-.5ex>[d] \ar@<.5ex>[d] \\
        \T_{F^v}M \ar@<-.5ex>[r] \ar@<.5ex>[r] & M\times \R.
    }
    \end{equation}
It therefore our belief that $\delta(\mathrm{Pair}^2(M), \Delta^h, \Delta^v)$ is the right notion of the tangent (double) groupoid of the filtered 2-multifiltered manifold $M$. 

\subsection{$C^*$-algebraic perspective}

As mentioned in the introduction, an important feature of Connes' tangent groupoid is that it defines a continuous field of $C^*$-algebras. 

\begin{problem}
    Given a weighted groupoid $G\toto M$, determine when the deformation groupoid $\defw(G,H)\toto \defw(M,N)$ determines a continuous field of $C^*$-algebra
\end{problem}

Given a Lie groupoid, one can form the convolution algebra and from that a groupoid $C^*$-algebra (\cite{renault2006groupoid}). Some work towards determining the convolution algebra of a double groupoid has been done in~\cite{roman2023convolution}, but there appears to be much work to be done in this direction. Related to the discussion in the previous section, we also pose the following.

\begin{problem}
    Determine what $C^*$-algebraic information is contained in a double groupoid. 
\end{problem}

%% file: A-equivalence_of_weighted_groupoids.tex
\chapter{Morita Equivalence of Weighted Groupoids}
\label{appendix: weighted Morita equivalence}

\section{Weighted Morita Equivalence}
\label{A-section: weighted morita equivalence}

A Morita equivalence is a special instance of generalized morphisms, called \emph{Hilsum-Skandalis morphisms}, which are invertible in the appropriate sense. The following approach to Morita equivalence is adapted from~\cite{moerdijk2003introduction, lerman2010orbifolds}. 

\subsection{Definition of Morita Equivalence}

\begin{definitions}
\label{defintion: Morita equivalence}
Let $G\toto M$ and $H\toto N$ be Lie groupoids.
    \begin{itemize}
        \item[(a)] A \emph{right action} of $G\toto M$ on a manifold $Q$ is given by a map $\mu:Q\to M$ together with an action map 
            \[ Q\times_M G = \{(q,g) : \mu(q) = t(g)\} \to Q, \quad (q, g) \mapsto q\cdot g \]
        such that $\mu(q\cdot g) = s(g)$, $(q\cdot g_1)\cdot g_2 = q\cdot (g_1\circ g_2)$, and $q\cdot \mu(q) = q$ for all $g_1, g_2 \in G$, whenever this is defined. A \emph{left action} is defined analogously, swapping the roles of $t$ and $s$. 

        \item[(b)] A \emph{right principal $G$-bundle over $B$} is given by a diagram
            \begin{equation*}
            \xymatrixcolsep{1.5pc}\xymatrix{
                P \ar[d]_{\pi} \ar[drr]_{\mu} & \ar@(dr,ur)[] & G \ar@<-.5ex>[d] \ar@<.5ex>[d] \\
                B & & M
             }
        \end{equation*}
        where $(P, \mu)$ is a right $G$-space such that
            \begin{itemize}
                \item[(i)] $\pi : P\to B$ is a surjective submersion,
                \item[(ii)] $\pi(p\cdot g) = \pi(p)$ whenever $\mu(p) = t(g)$, and 
                \item[(iii)]  the map 
                    \[ P\times_M G \to P\times_B P, \quad (p, g) \mapsto (p, p\cdot g) \]
                is a diffeomorphism. 
            \end{itemize}

        \item[(c)] A \emph{Hilsum-Skandalis morphism} $P:G\to H$ is represented by a diagram 
            \begin{equation*}
            \xymatrixcolsep{1.5pc}\xymatrix{
                G \ar@<-.5ex>[d] \ar@<.5ex>[d] & \ar@(dl,ul)[] & P \ar[dll]^{\mu_G} \ar[drr]_{\mu_H}  & \ar@(dr,ur)[] & H \ar@<-.5ex>[d] \ar@<.5ex>[d] \\
                M & & & & N
            }
            \end{equation*}
        where $(P, \mu_G)$ is a left $G$-space and $(\mu_G:P\to M, \mu_H)$ is a right principal $H$-bundle.

        \item[(d)] A \emph{Morita equivalence} $P:G\to H$ is a Hilsum-Skandalis morphism as above such that $(\mu_H:P\to H, \mu_G)$ is a left principal $G$-bundle. 
    \end{itemize}  
\end{definitions}

\begin{remark}
    Hilsum-Skandalis morphisms can be composed, but their composition is only associative up to isomorphism. 
\end{remark}

\begin{examples}
\begin{itemize}
    \item[(a)] Let $M$ and $N$ be spaces (thought of as trivial Lie groupoids), and let $F:M\to N$ be a smooth map. Then 
        \[ N\times_M M \to M, \quad (n, m) \mapsto m \]
    is a left action of $N$ on $M$ along $F$, and 
        \begin{equation*}
        \xymatrixcolsep{1.5pc}\xymatrix{
            M \ar@<-.5ex>[d] \ar@<.5ex>[d] & \ar@(dl,ul)[] & M \ar[dll]^{\mathrm{id}} \ar[drr]_{\mathrm{F}}  & \ar@(dr,ur)[] & N \ar@<-.5ex>[d] \ar@<.5ex>[d] \\
                M & & & & N
        }
        \end{equation*}
    is a Hilsum-Skandalis morphism $M:M\to N$, which is a Morita equivalence if and only if $F$ is a diffeomorphism. In fact, \emph{any} Hilsum-Skandalis morphism between spaces is of this form. 

    \item[(b)] Two Lie \emph{groups} $G$ and $H$ are Morita equivalent if and only if they are isomorphic. 

    \item[(c)] For any manifold $M$, the canonical action 
        \[ \mathrm{Pair}(M)\times_M M \to M, \quad 
        ((m_0, m_1), m_1) \mapsto m_0\]
    is a left action of $\mathrm{Pair}(M)$ on $M$ and 
        \begin{equation*}
        \xymatrixcolsep{1.5pc}\xymatrix{
            \mathrm{Pair}(M) \ar@<-.5ex>[d] \ar@<.5ex>[d] & \ar@(dl,ul)[] & M \ar[dll]^{\mathrm{id}} \ar[drr]  & \ar@(dr,ur)[] & \mathrm{pnt} \ar@<-.5ex>[d] \ar@<.5ex>[d] \\
            M & & & & \mathrm{pnt}
        }
        \end{equation*}
    is a Morita equivalence. 
\end{itemize}
\end{examples}

\subsection{Weighted Morita Equivalence}

We now state the analogue of~\autoref{defintion: Morita equivalence} for weighted groupoids. 

\begin{definitions}
\label{defintion: Morita equivalence}
Let $G\toto M$ and $H\toto N$ be weighted Lie groupoids.
    \begin{itemize}
        \item[(a)] A right action of $G\toto M$ on a weighted manifold $Q$ along $\mu$ is weighted if both $\mu$ and the action map 
            \[ Q\times_M G = \{(q,g) : \mu(q) = t(g)\} \to Q, \quad (q, g) \mapsto q\cdot g \]
        are weighted. We say that $Q$ is a \emph{weighted right $G$-space}. 

        \item[(b)] A right principal $G$-bundle over $B$ 
            \begin{equation*}
            \xymatrixcolsep{1.5pc}\xymatrix{
                P \ar[d]_{\pi} \ar[drr]_{\mu} & \ar@(dr,ur)[] & G \ar@<-.5ex>[d] \ar@<.5ex>[d] \\
                B & & M
             }
        \end{equation*}
        is weighted if both $P$ and $B$ are weighted and
            \begin{itemize}
                \item[(i)] $\pi : P\to B$ is a weighted submersion,
                \item[(ii)]  the map 
                    \[ P\times_M G \to P\times_B P, \quad (p, g) \mapsto (p, p\cdot g) \]
                is a weighted diffeomorphism. 
            \end{itemize}

        \item[(c)] A \emph{weighted Hilsum-Skandalis morphism} $P:G\to H$ is a diagram 
            \begin{equation*}
            \xymatrixcolsep{1.5pc}\xymatrix{
                G \ar@<-.5ex>[d] \ar@<.5ex>[d] & \ar@(dl,ul)[] & P \ar[dll]^{\mu_G} \ar[drr]_{\mu_H}  & \ar@(dr,ur)[] & H \ar@<-.5ex>[d] \ar@<.5ex>[d] \\
                M & & & & N
            }
            \end{equation*}
        where $(P, \mu_G)$ is a weighted left $G$-space and $(\mu_G:P\to M, \mu_H)$ is a weighted right principal $H$-bundle.

        \item[(d)] A \emph{weighted Morita equivalence} $P:G\to H$ is a weighted Hilsum-Skandalis morphism as above such that $(\mu_H:P\to H, \mu_G)$ is a weighted left principal $G$-bundle. 
    \end{itemize}  
\end{definitions}

\begin{remarks}
\begin{itemize}
    \item[(a)] In the definition of right weighted groupoid action, note that since $\mu:Q\to M$ is weighted and $s:G\to M$ is a weighted submersion, the fibre product $Q \times_{M} G$ inherits a weighting as weighted submanifold of $Q\times G$. Therefore, this definition makes sense.

    \item[(b)] In the definition of a weighted right principal $G$-bundle, note that since the action map $P\times_{M} G \to P$ factors through 
        \[ P\times_{M} G \to P\times_{B} P \corr{\mathrm{pr_2}} P  \]
    is follows that the action map is automatically weighted. 
\end{itemize}
\end{remarks}

\begin{examples}
\label{examples: weighted morita equivalence}
\begin{itemize}
    \item[(a)] If $B$ is a weighted manifold and $\mu:B\to M$ is a weighted morphism, then 
        \begin{equation*}
        \xymatrixcolsep{1.5pc}\xymatrix{
            B\times_M G \ar[d]_{\mathrm{pr}_2} \ar[drr]_{s} & \ar@(dr,ur)[] & G \ar@<-.5ex>[d] \ar@<.5ex>[d] \\
            B & & M
        }
        \end{equation*}
    is the \emph{trivial weighted right principal $G$-bundle over $B$}. 
    
    \item[(b)] If $G\toto M$ is any weighted Lie groupoid, then 
        \begin{equation*}
        \xymatrixcolsep{1.5pc}\xymatrix{
            G \ar@<-.5ex>[d] \ar@<.5ex>[d] & \ar@(dl,ul)[] & G \ar[dll]^{t} \ar[drr]_{s}  & \ar@(dr,ur)[] & G \ar@<-.5ex>[d] \ar@<.5ex>[d] \\
            M & & & & M
        }
        \end{equation*}
    defines a weighted Morita equivalence from $G$ to itself.

    \item[(c)] If $(M, N)$ is a weighted pair, then 
        \begin{equation}
        \label{equation: morita equivalence from pair to point}
        \xymatrixcolsep{1.5pc} \vcenter{\hbox{\xymatrix{
            \mathrm{Pair}(M) \ar@<-.5ex>[d] \ar@<.5ex>[d] & \ar@(dl,ul)[] & M \ar[dll]^{\mathrm{id}} \ar[drr]  & \ar@(dr,ur)[] & \mathrm{pnt} \ar@<-.5ex>[d] \ar@<.5ex>[d] \\
            M & & & & \mathrm{pnt}
        }}}
        \end{equation}
    is a weighted Morita equivalence when $\mathrm{Pair}(M)$ is weighted along $\mathrm{Pair}(M)$. 

    \item[(d)] If $\mathrm{Pair}(M)$ is trivially weighted along its units, then~\eqref{equation: morita equivalence from pair to point} is \emph{not} a weighted Morita equivalence. This is because  
    \begin{equation*}
        \xymatrixcolsep{1.5pc}\xymatrix{
            M \ar[d]_{\mathrm{id}} \ar[drr]_{\mathrm{id}} & \ar@(dr,ur)[] & \mathrm{Pair}(M) \ar@<-.5ex>[d] \ar@<.5ex>[d] \\
            \mathrm{pnt} & & M
        }
    \end{equation*}
    is \emph{not} a weighted principal $\mathrm{Pair}(M)$-bundle as the action map 
        \[ M\times_M \mathrm{Pair}(M) \to M\times_\mathrm{pnt} M \]
    is not a \emph{weighted} diffeomorphism. Indeed, $M\times_M\mathrm{Pair}(M)$ is weighted along the diagonal $M \into M\times_M\mathrm{Pair}(M)$, whereas $M\times_\mathrm{pnt} M$ is weighted along itself, so the inverse map is not a weighted morphism. 
\end{itemize}
\end{examples}

We now work towards showing that weighted Morita equivalence is an equivalence relation on the set of (isomorphism classes of) weighted Lie groupoids, the main point being that weighted Hilsum-Skandalis morphisms can be composed. Since the composition of Hilsum-Skandalis morphisms is defined as a quotient, it is not immediately clear that their composition is weighted in any canonical way. The main tool in establishing that the quotient is naturally weighted is the following lemma, which says that weighted principal $G$-bundles are locally trivial in a weighted sense. 

\begin{lemma}
\label{lemma: weighted principal bundles are locally trivial}
    If     
        \begin{equation*}
        \xymatrixcolsep{1.5pc}\xymatrix{
            P \ar[d]_{\pi} \ar[drr]_{\mu} & \ar@(dr,ur)[] & G \ar@<-.5ex>[d] \ar@<.5ex>[d] \\
            B & & M
        }
        \end{equation*}
    is a right principal $G$-bundle, then any point $p\in B$ has an open neighbourhood $U \sset B$ together with a weighted morphism $U\to M$ such that 
        \[ P|_U = \pi^{-1}(U) \cong U\times_{M} G \]
    as weighted principal $G$-bundles. 
\end{lemma}
\begin{proof}
    Since $\pi$ is a weighted submersion, we can find a weighted section $\sigma:U\to P$ defined near $p$. The map $U \to M$ is defined as the composition 
        \[ U \corr{\sigma} P|_U \to M, \]
    which is weighted. 

    The map 
        \begin{align}
        \label{equation: local triv 1}
            U\times_{B} G \to P|_U, \quad (u, g) \mapsto \sigma(u).g 
        \end{align}
    is a weighted, being the composition of weighted morphisms. We will show that it is a weighted diffeomorphism by constructing an explicit weighted inverse. The map  
        \[ \delta: P\times_B P \corr{\cong} P\times_{M} G \corr{\mathrm{pr}_2} G \]
    is a weighted morphism and therefore so is 
        \begin{align}
        \label{equation: local triv 2}
            P|_U \to U\times_{M} G, \quad 
            p \mapsto (\pi(p), \delta(\sigma(\pi(p)), p)).
        \end{align}
    The maps~\eqref{equation: local triv 1} and~\eqref{equation: local triv 2} are inverse to one another, which proves the claim. 
\end{proof}

\begin{theorem}
\label{A-theorem: composition of weighted HS-morphisms}
    Let $G\toto G_0$, $H\toto H_0$, and $K\toto K_0$ be weighted Lie groupoids. If $P:G\to H$ and $Q:H\to K$ are weighted Hilsum-Skandalis morphisms then there is a unique weighting on the composition $P\circ Q = (P\times_{H_0} Q)/H$ such that 
        \begin{itemize}
            \item[(a)] the quotient map $\pi:P\times_{H_0} Q \to P\circ Q$ is a weighted submersion and 
            \item[(b)] $P\circ Q : G \to K$ is a weighted Hilsum-Skandalis morphism. 
        \end{itemize}
\end{theorem}
\begin{proof}
    Let
        \begin{equation}
        \label{equation: hilsum-skandalis weighting}
            C^\infty(P\circ Q)_{(i)} = \{ f\in C^\infty(P\circ Q) : \pi^*f \in C^\infty(P\times_{H_0} Q)_{(i)} \}.
        \end{equation}
    In order to show that this defines a weighting of $P\circ Q$, we must produce weighted coordinates. 

    The work in the previous sections shows that $g_0 \in G_0$ has an open neighbourhood $U\sset G_0$ with a weighted morphism $U\to H_0$ such that 
        \[ (P\times_{H_0} Q)|_U = P|_U\times_{H_0}Q \cong U\times_{H_0} H\times_{H_0} Q \]
    as weighted $H$-spaces. Moreover, the map 
        \[ U\times_{H_0} H\times_{H_0} Q \to U\times_{H_0} Q, \quad (u,h,q) \mapsto (u,q) \]
    is an $H$-equivariant weighted submersion which descends to a diffeomorphism $(U\times_{H_0} H\times_{H_0} Q)/H \cong U\times_{H_0}Q$ making the following diagram commute 
        \begin{equation*}
        \xymatrix{
            U\times_{H_0}H \times_{H_0}Q \ar[d] \ar[dr] & \\
            (U\times_{H_0}H \times_{H_0}Q)/H \ar[r]_-{\cong} & U\times_{H_0}Q
        }
        \end{equation*} 
    In particular, submersion coordinates for $U\times_{H_0}H \times_{H_0}Q\to U\times_{H_0}Q$ induced weighted coordinates on $(U\times_{H_0}H \times_{H_0}Q)/H$ for the filtration~\eqref{equation: hilsum-skandalis weighting} so that the quotient map $P\times_{H_0} Q \to P\circ Q$ is a weighted submersion. 
\end{proof}

\begin{corollary}
    Weighted Morita equivalence is an equivalence relation on the set of isomorphism classes of weighted groupoids. 
\end{corollary}

\section{Weighted Principal Bundles and their Associated Bundles}

For completeness, we explain the associated bundle construction in the weighted setting. 

\begin{proposition}
\label{proposition: principal weightings}
    Let $(G,H)$ be a weighted Lie groupoid pair. If $P \to M$ is a weighted right principal $G$-bundle, then the submanifold $Q$ along which $P$ is weighted is a right principal $H$-bundle.
\end{proposition}
\begin{proof}
    Since the action map $P\times G \to P$ is weighted, it follows that $H$ acts on $Q$. Furthermore, since the projection $P\to M$ is a weighted submersion, it follows that the restriction $Q\to N$ is a submersion. Finally, since the action map $P\times_M G \to P\times_B P$ is a weighted diffeomorphism, its restriction 
        \[ Q\times_N H \to Q\times_N Q \]
    is an isomorphism as well.  
\end{proof}

Next, we specialize to the case when $(G,H)$ is a weighted Lie group pair. Let $V$ be a weighted vector space, and suppose that 
    \begin{equation}
    \label{equation: weighted linear group action}
        \rho : G\times V \to V
    \end{equation}
is a linear action of $G$ on $V$, which is weighted in the sense $\rho$ is a weighted morphism between weighted vector bundles. We remark that for any $h\in H$ the map
    \[ V_{(i)} \to  V, \quad v \mapsto h.g  \]
is a weighted morphism, whence $H$ acts on $V$ by filtration preserving automorphisms.  Let $P\times_G V \to M$ be the associated bundle and let 
    \[ q:P\times V \to P\times_G V \]
be the quotient map.

\begin{proposition}
\label{proposition: weighted principal bundles define linear weightings of associated bundles}
    The filtration 
        \[ C^\infty_{pol}(P\times_G V)_{(i)} = \{ f\in C^\infty_{pol}(P\times_G V) : q^*f \in C^\infty_{pol}(P\times V)_{(i)} \} \]
    defines a linear weighting of $P\times_G V$. The filtration of $(P\times_G V)|_N$ is given by 
        \[ ((P\times_G V)|_N)_{(i)} = Q\times_H V_{(i)}.  \]
\end{proposition}
\begin{proof}
    Given $p\in N$ we will construct weighted vector bundle coordinates near $p$. Choose an open neighbourhood $U\sset M$ containing $p$ and an isomorphism $P|_U \cong U\times G$ of weighted principal bundles. Since the diagram 
        \begin{align*}
        \xymatrix{
            U\times G \times V \ar[d]_q \ar[dr] & \\
            (U\times G)\times_G V \ar[r]_-\cong & U\times V
        }
        \end{align*}
    commutes, it follows that any choice of weighted vector bundle coordinate system for $U\times V$ does the trick. 

    For the filtration statement, let $\sigma : U \to P$ be a weighted section, and let $e_1, \dots, e_k$ be a weighted frame for $V$. Then the sections 
        \[ \sigma_a : U\to P \times_G V \quad u \mapsto [(\sigma(u), e_a)] \]
    form a weighted frame for $(P\times_G V)|_U$, hence $((P\times_G V)|_{N\cap U})_{(i)}$ is spanned by $\sigma_a$ with $v_a\geq i$. Since $Q$ is an $H$-principal subbundle of $P$ the map 
        \[ Q\times_H G \to P|_N, \quad [(q,g)] \mapsto q.g \]
    is a reduction of the structure group of $P|_N$ from $G$ to $H$. In particular, $(P\times_G V)|_N \cong Q\times_H V$. Since $\sigma(N\cap U) \sset Q$, if follows that $Q|_{U\cap N}\times_H V_{(i)}$ is exactly the image of $((P\times_G V)|_{U\cap N})_{(i)}$ under this isomorphism.
\end{proof}

\begin{remark}
\label{remark: integration of weighting to frame bundle}
    Let $V\to M$ be a rank $k$ complex vector bundle. One might wonder if a linear weighting of $V\to M$ induces a principal weighting of its frame bundle. The answer to this question is yes. This can be seen as an application of~\autoref{theorem: wide integration}, which we outline now. 
    
    The weight vector defines a filtration of $\C^k$ by subspaces, and this filtration, in turn, defines a Lie filtration of $\ger{gl}(\C^k)$. The space of filtration preserving endomorphisms of the filtered vector space $\C^k$, denoted $\ger{gl}(\C^k)_{(0)}$, is the Lie algebra of a connected Lie group $\mathrm{GL}(\C^k)_{(0)}\sset \mathrm{GL}(\C^k)$. Hence, by~\autoref{theorem: wide integration}, $\mathrm{GL}(\C^k)$ has a canonical multiplicative weighting along $\mathrm{GL}(\C^k)_{(0)}$. Furthermore, the canonical action 
        \[ \mathrm{GL}(\C^k)\times \C^k \to \C^k \]
    is weighted (using~\autoref{proposition: weighted morphisms of weighted lie filtrations}).

    Next, let $F(V) \to M$ denote the frame bundle of $V$, and let 
        \[ F_W(V) = \{ (p,f) \in N\times_N F(V) |\ f:\C^k \to V_p \text{ is filtration preserving}\} \]
    be the bundle of filtration preserving frames. This is a principal $\mathrm{GL}(\C^k)_{(0)}$-subbundle of $F(V)$. The weighting of $V$ induces an infinitesimally multiplicative weighting of the Atiyah algebroid $\mathrm{At}(V)\Rightarrow M$ such that the follow sequence of weighted vector bundles 
        \begin{equation}
        \label{A-equation: weighted exact sequence}
            0 \corr{} \mathrm{End}(V) \corr{} \mathrm{At}(V) \corr{} TM \corr{} 0
        \end{equation}
    is \emph{weighted exact}, in the sense that one can find a weighted splitting. Identifying sections of $\mathrm{At}(V)$ with the $\mathrm{GL}(\C^k)$-invariant vector fields on $F(V)$ we see that the linear weighting of $V$ induces a singular Lie filtration on $F(V)$. Moreover, using that~\eqref{A-equation: weighted exact sequence} is weighted exact, we have that $F_W(V)$ is a clean submanifold with respect to this singular Lie filtration. By~\autoref{theorem: singular Lie filtrations and weightings}, this defines a weighting of $F(V)$ along $F_W(V)$. Again, using~\eqref{A-equation: weighted exact sequence} one can deduce that this weighting is principal. 
\end{remark}

%% file: A-spray_exponential.tex
\chapter{The Spray Exponential}
\label{appendix: spray exponential}

For a Lie group $G$ with Lie algebra $\ger{g}$, there is a canonical exponential map $\exp:\ger{g} \to G$ which is a diffeomorphism near the origin. For a general Lie groupoid $G\toto M$ with Lie algebroid $A\Rightarrow M$ there is no longer a canonical exponential map relating the two. Nonetheless, it is shown in~\cite{cabrera2020local} that after a choice of a \emph{Lie algebroid spray} $V\in \ger{X}$ is made an exponential map $\exp_V: A\dashrightarrow G$ defined on an open neighbourhood of the zero section in $A$ can be defined. In this appendix we review this construction. For the following, let $G\toto M$ be a Lie groupoid with Lie algebroid $A = \mathrm{Lie}(G) \Rightarrow M$. 

\section{$A$-paths and $G$-paths}

Let $J\sset \R$ be an open neighbourhood of the origin. Since $J$ is one dimensional, a Lie algebroid morphism is a vector bundle morphism $\varphi : TJ \to A$ with the property that
    \[ (a\circ \varphi )\left(\frac{\bd}{\bd t}\bigg|_{t=s} \right) = T_s\varphi_M\left(\frac{\bd}{\bd t}\bigg|_{t=s} \right) \]
for all $s \in J$. In particular, $\varphi$ defines a path 
    \[ \gamma : J \to A, \quad s \mapsto \varphi\left( \frac{\bd}{\bd t}\bigg|_{t=s} \right) \]
which lifts the path $\varphi_M$. We summarize these properties in the following definition. 

\begin{definition}
    An \emph{$A$-path} is a path $\gamma: J\to A$ with the property that 
        \begin{equation}
        \label{equation: A-path definition}
            a(\gamma(s)) =  T_s\gamma_M\left( \frac{\bd}{\bd t}\bigg|_{t=s} \right)
        \end{equation}
    for all $s\in J$, where $\gamma_M:J\to M$ is the base map of $\gamma$. 
\end{definition}

Thus any Lie algebroid morphism $\gamma:TJ \to A$ defines an $A$-path. The converse is also true: given an $A$-path $\gamma:\R \to A$ we let $\Tilde{\gamma}:TJ\to A$ be the vector bundle morphism defined by 
    \[ \Tilde{\gamma}\left(\frac{\bd}{\bd t}\bigg|_{t=s} \right) = \gamma(s).  \]
This is a morphism of anchored vector bundles precisely because of~\eqref{equation: A-path definition}. Thus, we have established a 1-1 correspondence between $A$-paths $J\to A$ and Lie algebroid morphisms $TJ\to A$. 

The dual notion of an $A$-path is a $G$-path: 

\begin{definition}
    A \emph{$G$-path} is a path $\gamma: J\to G$ such that 
        \[ \mathrm{t}(\gamma(t)) = \gamma(0) \in M \]
    for all $t\in J$. 
\end{definition}

Given a $G$-path $\gamma: J\to G$, the map 
    \begin{equation*}
    \label{equation: pair groupoid morphism associated to G-path}
        \Tilde{\gamma} : \mathrm{Pair}(J) \to G, \quad (t,s) \mapsto \gamma(t)^{-1}\gamma(s)
    \end{equation*}
is a Lie groupoid morphism. Conversely, any groupoid morphism $\gamma : \mathrm{Pair}(J)\to G$ defined a $G$-path by 
    \[ t \mapsto \gamma(0,t). \]
Thus, there is a 1-1 correspondence between $G$-paths $J\to G$ and groupoid morphisms $\mathrm{Pair}(J) \to G$. 

\section{The Maurer-Cartan Form}

Let $T^tG = \mathrm{ker}(Tt)$ be the tangent space to the $t$-fibres. Given $g\in G$ with $s(g)=x$ and $t(g)=y$, left multiplication by $g$ defines a diffeomorphism $L_g:t^{-1}(x) \to t^{-1}(y)$. For each $g\in G$, this defines an isomorphism 
    \[ T_{g}L_{g^-1} : T^t_{g}G \to T^t_{s(g)}G = A_{s(g)} \]
which fits together to define a vector bundle map
    \[ \theta^G : T^tG \to A \]
covering $s:G\to M$. 

\begin{definition}
    The vector bundle morphism $\theta^G: T^tG \to A$ is the (left-invariant) Maurer-Cartan form for $G$. 
\end{definition}

We can use the Maurer-Cartan form to differentiate $G$-paths to $A$-paths. Let $\gamma:J\to G$ be a $G$-path. Since $\mathrm{t}(\gamma(t)) = \gamma(0)$ for all $t\in J$, it follows that 
    \[ \dot{\gamma}(t) : J \to TG, \quad t\mapsto T_{t}\gamma\left( \frac{\bd}{\bd t}\bigg|_{t}\right)  \]
defines a path $J\to T^\mathrm{t}G$. Hence we can define 
    \[ D^L\gamma : J \to A, \quad t\mapsto \theta^G(\dot{\gamma}(t)). \]
Alternatively, we can identify $\gamma$ with a groupoid morphism $\mathrm{Pair}(J)\to G$ and let $D^L\gamma:J\to A$ be the $A$-path defined by the induce Lie algebroid morphism $TJ \to A$ (see~\cite[Propostion 1.1]{crainic2003integrability}). Using Lie's second theorem for Lie algebroids (~\cite[Theorem A.1]{mackenzie2000integration}), any $A$-path also integrates (uniquely) to a $G$-path. Summarizing, 

\begin{proposition}(cf.~\cite[Propostion 1.1]{crainic2003integrability}) 
    The assignment $\gamma\mapsto D^L\gamma$ defines a 1-1 correspondence between $G$-paths and $A$-paths. 
\end{proposition}

In particular, the $G$-path $\varphi:J\to G$ integrating a given $A$-path $\gamma:J\to A$ with $\gamma(0) \in A_x$ is the unique solution to the initial value problem 
    \begin{equation}
    \label{equation: ODE describing integration of $A$-paths}
        \theta^G(\dot{\varphi}(t)) = \gamma(t), \quad \varphi(0) = x. 
    \end{equation}
    
\section{The Spray Exponential}

We want to use the correspondence between $A$-paths and $G$-paths to define an exponential map for Lie groupoids. In order to do this, we need a way of defining an $A$-path through a given $a \in A$. The data needed for this is a Lie algebroid spray (cf.~\cite[Definition 3.1]{cabrera2020local}).

\begin{definition}
\label{definition: Lie algebroid spray}
    A \emph{spray} on a Lie algebroid $A\rightarrow M$ is a vector field $V\in \mathfrak{X}(A)$ such that 
        \begin{enumerate}
            \item[(a)] $\kappa_t^*V = tV$ for all $t\neq 0$, where $\kappa_t$ denotes scalar multiplication by $t$, and 
            
            \item[(b)] for all $a\in A$, one has $T\pi(V_a) = \rho(a)$, where $\pi:A\to M$ is the vector bundle projection and $\rho:A\to TM$ is the anchor. 
        \end{enumerate}
\end{definition}

\begin{remarks}
Let us summarize some facts about Lie algebroid sprays, as explained in~\cite[Section 3.1]{cabrera2020local}.
    \begin{itemize}
        \item[(a)] The first condition implies that $V$ vanishes along $M$. Hence, there is an open neighbourhood $M \sset U_V \sset A$ for which the flow $\phi_V^t$ up to time $t=1$ is defined for all $v \in U_V$. 

        \item[(b)] The second condition implies that for all $v\in U_V$ the map 
            \[ \phi_v :J \to A, \quad t\mapsto \phi_V^t(v) \]
        is an $A$-path, where $J\sset \R$ is an open interval centered at 0 and containing 1.   
    \end{itemize}
\end{remarks}

In particular, a choice Lie algebroid spray $V \in \ger{X}(A)$ determines an $A$-path starting at $v$ for any $v\in A$ sufficiently close to the zero section. This motivated the following definition, cf.~\cite[Definition 3.20]{cabrera2020local}.

\begin{definition}
    Let $A= \mathrm{Lie}(G)$ and let $V\in \ger{X}(A)$ be a Lie algebroid spray. Given $v\in A$ sufficiently close to the zero section, let $\exp_V(v) = \phi_G(1) \in G$ where $\phi_G(t)$ is the $G$-path integrating the $A$-path $\phi_V^t(v)$. The \emph{spray exponential} is the map 
        \[ \exp_V: A\dashrightarrow G, \quad v \mapsto \exp_V(v) \]
    defined on an open neighbourhood $U_V\sset A$ containing the zero section. 
\end{definition}

\begin{remark}
    In~\cite{cabrera2020local} they define the spray exponential to be the solution of the initial value problem~\eqref{equation: ODE describing integration of $A$-paths}, but by the discussion in the previous section we see they give the same result. 
\end{remark}

\begin{lemma}
\label{lemma: exp maps subgroupoids to subgroupoids}
    Let $H \toto N$ be a Lie subgroupoid of $G\toto M$, and let $B = \mathrm{Lie}(H)$. If $V\in \ger{X}(A)$ is tangent to $B$, then $\exp_V(v)\in H$ for all $v\in B$ sufficiently close to the zero section. 
\end{lemma}
\begin{proof}
    Since $B$ is a Lie subalgebroid, the restriction of $V$ to $B$ is a Lie algebroid spray for $B$. Thus, for $v\in B$ sufficiently close to zero, $\phi_V^t(v)$ is an $A$-path whose image lies in $B$, hence integrates to an $G$-path whose image lies in $H$. 
\end{proof}

\section{Applying the Higher Tangent Functor}

Finally, we are interested in examining how applying the higher tangent functor affects to corresponding exponential map. 

\begin{lemma}
\label{lemma: flows and lifts}
    Suppose that $X\in \ger{X}(M)$ with flow $\phi_X$, defined for time $|t|\leq \epsilon$. Then 
        \[ T_r\phi_X^t = \phi_{X^{(0)}}^t. \]
\end{lemma}
\begin{proof}
    For $t,s$ sufficiently small we have by functoriality that 
        \[  T_r\phi_X^0 = \mathrm{id} \quad \text{and} \quad T_r\phi_X^{t+s} = T_r(\phi_X^{t}\phi_X^{s}) = T_r\phi_X^tT_r\phi_X^s, \]
    hence $T_r\phi_X^t$ is the flow of a vector field. To verify that this vector field is $X^{(0)}$, let $f\in C^\infty(M)$ and observe that
        \begin{align*}
            \frac{d}{dt}\bigg|_{t=0}(T_r\phi_X^t)^*f^{(i)} = \frac{d}{dt}\bigg|_{t=0} ((\phi^t_X)^*f)^{(i)} = (Xf)^{(i)} = X^{(0)}f^{(i)},
        \end{align*}
    for $i=0,1,\dots, r$. 
\end{proof}

\begin{remark}
    In particular, $\phi_{X^{(0)}}^t$ is also defined for $|t|\leq \epsilon$. 
\end{remark}

Recall that if $A\Rightarrow M$ is a Lie algebroid, then there is a canonical Lie algebroid structure $T_rA \Rightarrow T_rM$ such that
    \[ a_{T_rA} = T_ra \quad \text{and} \quad [\sigma^{(-i)}, \tau^{(-j)}] = [\sigma, \tau]^{(-i-j)}.  \]
The canonical identification $T_r(TG) \cong T(T_rG)$ restricts to an identification $T_r(T^\mathrm{t}G|_M)\cong T^{T_r\mathrm{t}}(T_rG)|_{T_rM}$. In particular, we see that $T_rA$ is the Lie algebroid of $T_rG$. 
    
\begin{proposition}
\label{proposition: tangent and exp commute}
    \begin{itemize}
        \item[(a)] If $V\in \ger{X}(A)$ is a Lie algebroid spray, then $T_rV = V^{(0)}$ is a Lie algebroid spray for $T_rA$.
        
        \item[(b)] For any Lie algebroid spray $V\in \ger{X}(A)$, the partially defined maps 
        \[ T_r\exp_V : T_rA \dashrightarrow T_rG \quad \text{and} \quad \exp_{V^{(0)}}:T_rA \dashrightarrow T_rG \]
    agree. 
    \end{itemize}    
\end{proposition}
\begin{proof}
    \begin{itemize}
        \item[(a)] Let $E\in \ger{X}(A)$ denote the Euler vector field for $A$. Then $E^{(0)} \in \ger{X}(T_rA)$ is the Euler vector field for $T_rA$ and we have that 
            \[ [E^{(0)}, V^{(0)}] = [E, V]^{(0)} = V^{(0)},  \]
        so $T_rV$ satisfies~\autoref{definition: Lie algebroid spray} (a). Condition (b) in~\autoref{definition: Lie algebroid spray} follows by functoriality. 

        \item[(b)] Let $U_V \sset A$ be an open neighbourhood of the zero section for which $\phi_V^t$ is defined for $|t| \leq 1$. By applying~\autoref{lemma: flows and lifts} to $M = U_V$ we see that, for any $u\in T_rU_V$, the $T_rA$-paths $T_r\phi_V^t(u)$ and $\phi_{V^{(0)}}^t(u)$ agree. Therefore, they integrate to the same $T_rG$-path, which proves the claim. \qedhere
    \end{itemize}
\end{proof}

%% file: A-technical_proofs.tex
\chapter{Proofs of Various Lemmas}
\label{appendix: technical proofs}

\section{Proof of~\autoref{lemma: graded modules gives sections of weighted normal bundle}}
\label{A-section: graded modules gives sections of weighted normal bundle}

\begin{lemma}[=~\autoref{lemma: graded modules gives sections of weighted normal bundle}]
    For $\sigma \in \G(V)_{(i)}$, the $i$-th homogeneous approximation is a smooth section of $\nuw(V)$ which is homogeneous of degree $-i$ and which depends only on the class of $\sigma$ in $\G(V)_{(i)}/\G(V)_{(i+1)}$. Moreover, for any $f\in C^\infty_{[n]}(V)_{(j)}$ and $g\in C^\infty(M)_{(k)}$ one has 
        \begin{align}
        \label{equation: homogeneous approximation properties}
        \begin{split}
            f^{[j]}\circ \sigma^{[i]} & = (f\circ \sigma)^{[j+ni]} \in C^\infty(\nuw(M,N)) \quad \text{and} \\
            g^{[k]}\sigma^{[i]} & = (g\sigma)^{[i+k]} \in \G(\nuw(V)).
        \end{split}
        \end{align}
\end{lemma}
\begin{proof}
    The equation $f^{[j]}\circ \sigma^{[i]} = (f\circ\sigma)^{[j+ni]}$ follows from~\eqref{equation: homogeneous approximation definition}, and it follows from this that $\sigma^{[i]}:\nuw(M,N) \to \nuw(V)$ is smooth. To see that it is a section, we note that 
        \[ \nuw(\pi)\circ \sigma^{[i]} = \nuw(\pi\circ \sigma) = \mathrm{id}_{\nuw(M,N)}. \]
    Given $g \in C^\infty(M)_{(k)}$, $f\in C^\infty_{[n]}(V)_{(j)}$ and $\varphi \in \nuw(M,N)$, using that $f$ is homogeneous of degree $n$ we compute 
        \begin{align*}
            ((g^{[k]}\sigma^{[i]})(\varphi))(f^{[j]}) & = (\nuw(\kappa_{\varphi(g^{[k]})})\sigma^{[i]}(\varphi))(f^{[j]}) \\
            & = (\sigma^{[i]}(\varphi))(\kappa_{\varphi(g^{[k]})}^*f^{[j]}) \\
            & = (\sigma^{[i]}(\varphi))(\varphi(g^{[k]})^nf^{[j]}) \\
            & = \varphi(g^{[k]})^n\varphi((f\circ \sigma)^{[j+ni]}) \\
            & = \varphi((g^n)^{[nk]}(f\circ \sigma)^{[j+ni]}) \\
            & = \varphi((g^n(f\circ \sigma))^{[j+ni+nk]}) \\
            & = \varphi((f\circ (g\sigma))^{[j+n(i+k)]}) \\
            & = ((g\sigma)^{[i+k]}(\varphi))(f^{[j]}),
        \end{align*}
    which establishes~\eqref{equation: homogeneous approximation properties}. To show that $\sigma^{[i]}$ only depends on the class of $\sigma \in \G(V)_{(i)}/\G(V)_{(i+1)}$, it suffices to show that if $\sigma \in \G(V)_{(i+1)}$ then $\sigma^{[i]} = 0$. Specifically, what this means is that for any $f\in C^\infty_{pol}(V)_{(j)}$ and $\varphi \in \nuw(M,N)$ one has
        \begin{equation}
        \label{equation: what the zero section is}
            (\sigma^{[i]}(\varphi))(f^{[j]}) = \varphi(\kappa_0^*f^{[j]}).
        \end{equation}
    If $f\in C^\infty_{[n]}(V)_{(j)}$ with $n\geq 1$, one has that 
        \[ f\circ \sigma \in C^\infty(M)_{(j+n(i+1))} \implies (f\circ \sigma)^{[j+ni]} = 0, \]
    so~\eqref{equation: what the zero section is} follows in this case. For $f\in C^\infty_{[0]}(V)_{(j)} = C^\infty(M)_{(j)}$ one has that 
        \[ (\sigma^{[i]}(\varphi))(f^{[j]}) = \varphi((f\circ \sigma)^{[j]}) = \varphi(f^{[j]}) =  \varphi(\kappa_0^*f^{[j]}), \]
    since $f$ is homogeneous of degree zero. Since any polynomial is a sum of monomials, it follows that $\sigma^{[i]} = 0$.  Finally, to see that $\sigma^{[i]} \in \G_{[-i]}(\nuw(V))$, we compute 
            \begin{align*}
                (\alpha_\lambda(\sigma^{[i]}(\varphi)))(f^{[j]})  & = (\sigma^{[i]}(\varphi))(\lambda^jf^{[j]})  \\
                & = \varphi(\lambda^j(f\circ \sigma)^{[j+ni]}) \\
                & = \varphi(\lambda^{j+ni}(f\circ (\lambda^{-i}\sigma))^{[j+ni]}) \\
                & = (\alpha_\lambda(\varphi))((f\circ (\lambda^{-i}\sigma))^{[j+ni]}) \\
                & = (\lambda^{-i}\sigma^{[i]}(\alpha_\lambda(\varphi)))(f^{[j]}). \qedhere
            \end{align*}
\end{proof}

\section{Proof of~\autoref{lemma: action is weighted morphism}}
\label{A-section: weighted action proof}

We now prove the following. 

\begin{lemma}[=~\autoref{lemma: action is weighted morphism}]
    If $\mathrm{Pair}(P)$, $\mathrm{Pair}(\mathrm{Spin}(k))$ are given the doubled trivial weighting along the diagonal, and $\Cl(\R^k)$ is given the linear weighting defined by its filtration by subspaces, then the group action 
        \begin{equation}
        \label{equation: severa action is weighted}
            \mathrm{Pair}(\mathrm{Spin}(k)) \times \mathrm{Pair}(P) \times \Cl(\R^k) \to \mathrm{Pair}(P) \times \Cl(\R^k)
        \end{equation}
    is a weighted morphism.
\end{lemma}

As indicated above, the proof uses a characterization of weighted morphisms in terms of weighted paths. We being by explaining this characterization and then give a proof of the lemma. The contents of this section are based on communications with Gabriel Beiner, Yiannis Loizides, and Eckhard Meinrenken. 

Let us briefly recall the set up. Let $V \to M$ be a rank $k$ vector bundle with inner product and given spin structure. Let $P\to M$ be the principal $\mathrm{Spin}(k)$-bundle specified by the spin structure, and recall that the action of $\mathrm{Pair}(\mathrm{Spin}(k))$ on $\mathrm{Pair}(P)\times \Cl(\R^k)$ is given by
    \[ (g_1, g_2).(p_1, p_2, v) = (p_1.g_1, p_2.g_2, g_1vg_2^{-1}), \]
where we are making use of the inclusion $\mathrm{Spin}(k) \sset \Cl(\R^k)$. In order to show that this action is weighted, we make use of weighted paths. 

The weighting on both $\mathrm{Pair}(P)$ and $\mathrm{Pair}(\mathrm{Spin}(k))$ is the doubled weighting along the diagonal, where a function has filtration degree $2i$ if it vanishes to order $i$ along the diagonal. Giving $\mathrm{Pair}(M)$ the doubled weighting along the diagonal, one finds that $\mathrm{Pair}(P)$ can be identified locally with $\mathrm{Pair}(M)\times \mathrm{Pair}(\mathrm{Spin}(k))$ as weighted manifolds. Thus, to show that the action of $\mathrm{Pair}(\mathrm{Spin}(k))$ on $\mathrm{Pair}(P)$ is weighted, it is sufficient (in fact, equivalent) to show that the action of $\mathrm{Pair}(\mathrm{Spin}(k))$ on itself is a weighted morphism. 

\begin{lemma}
    Suppose that $G$ is a Lie group and $H\sset G$ is a closed subgroup. If $G$ is given the doubled weighting along $H$, then the map 
        \begin{align*}
            a: G\times G & \to G \\
            (g_1, g_2) & \mapsto g_1g_2^{-1}
        \end{align*}
    is a weighted morphism. 
\end{lemma}
\begin{proof}
    By definition of the doubled weighting, the map $a$ is weighted if and only if $a^*\van{H} \sset \van{H\times H}$. From this it follows that $a$ is a weighted morphism, since $H$ is a subgroup.  
\end{proof}

It remains to show that the action of $\mathrm{Pair}(\mathrm{Spin}(k))$ on $\Cl(\R^k)$ is weighted, which we accomplish by making use of weighted paths. It is sufficient to consider path $\gamma$ in $\mathrm{Pair}(\mathrm{Spin}(k)) \times \Cl(\R^k)$ of the form 
    \[ \gamma(t) = \left(g\exp(\xi t)\exp(\xi_1(t)), g\exp(\xi t)\exp(\xi_2(t)), \sum_{i=0}^k c_it^i\right),  \]
where $g \in \mathrm{Spin}(k)$, $\xi \in \mathfrak{spin}(k)$, $\xi_1(t), \xi_2(t) = O(t^2)$, and $c_i \in \Cl_i(\R^k)$. Composing this with the action $\mathrm{Pair}(\mathrm{Spin}(k))\times \Cl(\R^K)\to \Cl(\R^k)$ gives 
    \begin{equation}
    \label{A-equation: weighted path calculation}
        \sum_{i=0}^k \mathrm{Ad}_{g\exp(\xi t)}\left(\exp(\xi_1(t))c_i\exp(-\eta_2(t))t^i\right).
    \end{equation}
Using that the adjoint action of $\mathrm{Spin}(k)$ on $\Cl(\R^k)$ is filtration preserving, the identification $\mathfrak{spin}(k) = \Cl_2(\R^k)$, and the fact that $\xi_1(t), \xi_2(t) = O(t^2)$, it follows from the power series expansion of the exponential map that 
    \[ \sum_{i=0}^k \mathrm{Ad}_{g\exp(\xi t)}\left(\exp(\xi_1(t))c_i\exp(-\eta_2(t))t^i\right) = \sum_{i=0}^k v_i t^i + O(t^{k+1}), \]
for some $v_i \in \Cl_i(\R^k)$. In particular,~\eqref{A-equation: weighted path calculation} is a weighted path. By~\autoref{A-proposition: characterization of weighted morphisms}, this completes the proof of~\autoref{lemma: action is weighted morphism}.

\section{Proof of~\autoref{lemma: annilator is a weighted subbundle}}
\label{A-section: annihilator is weighed subbundle}

Let $V\to M$ and $W\to M'$ be weighted vector bundles. 

\begin{proposition}
\label{A-proposition: kernels or linear submersions}
    Suppose that $\varphi:V\to W$ is a weighted vector bundle morphism such that the base map $\varphi_M:M\to M'$ is a weighted submersion and the maps 
        \[ (V|_N)_{(i)} \to (W|_{M'})_{(i)}  \]
    are fibrewise surjective. The $\ker(\varphi)$ is a weighted subbundle of $V$. 
\end{proposition}
\begin{proof}
    As in the proof of~\autoref{theorem: linear weighted morphism in terms of graph}, we may assume that the weightings of both $V$ and $W$ are concentrated in negative degree. In this case, the assumptions on $\varphi$ ensure that it is a weighted submersion. Since $M'$ is a weighted submanifold of $W$, it follows that 
        \[ \mathrm{ker}(\varphi) = \varphi^{-1}(M') \]
    is a weighted subbundle of $V$ by~\autoref{corollary: inverse image of mfld is mfld}. 
\end{proof}

\begin{corollary}[=~\autoref{lemma: annilator is a weighted subbundle}]
    Let $V\to M$ be a weighted vector bundle and $W\to R$ a weighted subbundle. Then $\mathrm{ann}(W)$ is a weighted subbundle of $V^*$. 
\end{corollary}
\begin{proof}
    Let $i:R\into M$ be the inclusion. Then $i^*V^*\to R$ is a weighted vector bundle by~\autoref{subsection: linear constructions}. The induced map 
        \[ i^*V^*\to W^* \]
    satisfies the assumptions of~\autoref{A-proposition: kernels or linear submersions} and has kernel equal to $\mathrm{ann}(W)$. 
\end{proof}

\begin{remark}
\label{A_remark: sharp annihilator}
    By a similar argument, if $V_1$ and $V_2$ are weighted vector bundles, and $R\sset V_2\times V_1$ is a weighted subbundle, then $\mathrm{ann}^\sharp(R) \sset V_2^*\times V_1^*$ is a weighted subbundle. Indeed, it is the kernel of the map 
        \[ i^*(V_2^*\times V_1^*) \to R^*, \quad (\xi_2, \xi_1)\mapsto \xi_2-\xi_1. \]
\end{remark}

%% file: bibliography.tex
\bibliographystyle{amsalpha}

\bibliography{thesis}

%% file: colophon.tex
\thispagestyle{plain}

\hfill

\vfill

\pdfbookmark[0]{Colophon}{colophon}
\section*{Colophon}
\noindent This thesis was typeset using the typographical look-and-feel\\
\texttt{classicthesis} developed by Andr\'e Miede and Ivo Pletikosić.\bigskip

\noindent The style was inspired by Robert Bringhurst's seminal book\\
on typography ``\emph{The Elements of Typographic Style}''.\bigskip

